\documentclass[10pt]{amsart}
\usepackage{geometry}                
\usepackage{graphicx}
\usepackage{amssymb}
\usepackage{epstopdf}
\usepackage{amsmath}
\usepackage{graphicx}  \usepackage{epstopdf}
 \usepackage[colorlinks=true]{hyperref}
\hypersetup{urlcolor=blue, citecolor=red}

\usepackage{placeins} 
\usepackage{flafter} 

\usepackage[usenames,dvipsnames]{xcolor}

\usepackage[normalem]{ulem}

\newtheorem{theorem}{Theorem}[section]
\newtheorem*{theo}{Theorem}
\newtheorem{corollary}[theorem]{Corollary}
\newtheorem*{main}{Main Theorem}
\newtheorem{lemma}[theorem]{Lemma}
\newtheorem{proposition}[theorem]{Proposition}

\theoremstyle{definition}
\newtheorem{definition}[theorem]{Definition}
\newtheorem{remark}[theorem]{Remark}

\newtheorem{example}{Example}[section]
\newtheorem*{theorem*}{Theorem}

\usepackage{mathtools}
\usepackage{amssymb}
\usepackage{enumitem,mathrsfs}
\usepackage{stackrel}

\DeclarePairedDelimiter\floor{\lfloor}{\rfloor}
\DeclarePairedDelimiter\ceil{\lceil}{\rceil}
\setlist{noitemsep,topsep=0pt,partopsep=0pt}

\newcommand{\MP}{\ensuremath{{\mathcal P}}}
\newcommand{\MZ}{\ensuremath{{\mathcal Z}}}
\newcommand{\ZZ}{\ensuremath{{\mathbb Z}}}
\newcommand{\CC}{\ensuremath{{\mathbb C}}}
\newcommand{\CW}{\ensuremath{{\widehat{\mathbb C}}}}
\newcommand{\RR}{\ensuremath{{\mathbb R}}}
\newcommand{\NN}{\ensuremath{{\mathbb N}}}
\newcommand{\QQ}{\ensuremath{{\mathbb Q}}}
\newcommand{\HH}{\ensuremath{{\mathbb H}}}

\newcommand{\E}{\ensuremath{{\mathscr E}}}
\newcommand{\R}{\ensuremath{{\mathcal R}}}

\font\myfont=cmr10 at 12pt
\newcommand{\e}{{\text{\myfont e}}}

\newcommand{\abs}[1]{\left\lvert#1\right\rvert}

\renewcommand{\Re}[1]{{\mathfrak{Re}\left(#1\right)}}
\renewcommand{\Im}[1]{{\mathfrak{Im}\left(#1\right)}}
\renewcommand{\bar}[1]{{\overline{#1}}}
\newcommand{\del}[2]{\frac{\partial #1}{\partial #2}}
\newcommand{\ddel}[2]{\dfrac{\partial #1}{\partial #2}}
\newcommand{\Arg}[1]{\text{Arg}\left(#1\right)}

\newcommand{\ent}[2]{{}_{#1}\mathscr{E}_{#2}}

\usepackage{tikz,etoolbox}
\newcommand{\circled}[1]{ \text{\tikz[baseline=(char.base)]{
    \node[shape=circle,draw,inner sep=1pt] (char) {$#1$};}} }
\robustify{\circled}
\newcommand{\raiz}[1]{ \text{\tikz[baseline=(char.base)]{
    \node[shape=circle,draw,inner sep=0.001pt] (char) {
    \text{\tikz[baseline=(char.base)]{
    \node[shape=circle,draw,inner sep=1pt] (char) {$#1$};}}
    };}} }
\robustify{\raiz}
\newcommand{\msigma}{\ensuremath{{\mathfrak{a}}}}
\newcommand{\mrho}{\ensuremath{{\mathfrak{r}}}}
\newcommand{\mvarrho}{\ensuremath{{\mathfrak{z}}}}
\usepackage{graphicx}
\makeatletter
\newcommand*\bigcdot{\mathpalette\bigcdot@{.5}}
\newcommand*\bigcdot@[2]{\mathbin{\vcenter{\hbox{\scalebox{#2}{$\m@th#1\bullet$}}}}}
\makeatother

\title[Dynamics of vector fields with essential singularities II]
{Dynamics of singular complex analytic vector fields 
with essential singularities II}

\author[Alvaro Alvarez--Parrilla and Jes\'us Muci\~no--Raymundo]{}

\makeatletter
\@namedef{subjclassname@2020}{%
\textup{2020} Mathematics Subject Classification}
\makeatother



 \email{alvaro.uabc@gmail.com}
 \email{muciray@matmor.unam.mx}

\thanks{$^*$ Corresponding author}

\begin{document}
\maketitle

\centerline{\scshape Alvaro Alvarez--Parrilla$^*$}
\medskip
{\footnotesize
 \centerline{Grupo Alximia SA de CV}
   \centerline{Ensenada, Baja California, CP 22800, M\'exico}
} 

\medskip

\centerline{\scshape Jes\'us Muci\~no--Raymundo}
\medskip
{\footnotesize
\centerline{Centro de Ciencias Matem\'aticas}
\centerline{Universidad Nacional Aut\'onoma de 
M\'exico, Morelia, M\'exico} 
}

\begin{abstract}
The singular complex analytic vector fields 
$X$ on the Riemann sphere $\CW_{z}$ belonging to the family
$$
\E(r,d)=\left\{ 
X(z)=\frac{1}{P(z)}\ \e^{E(z)}\del{}{z} 
\ \Big\vert \
P, E\in\CC[z]
\right\},
$$
where $P$ is monic,   
$\deg{P}=r, \ \deg{E}=d, \ r+d \geq 1 $,
have a finite number of poles on the complex plane 
and an isolated essential singularity at infinity
(for $d\geq 1$).  
Our aim is to
describe geometrically 
$X$, 
particularly the singularity at infinity. 
In order to do so, we use the 
natural one to one \emph{correspondence} between $X$,
a global singular analytic distinguished parameter 
$\Psi_X(z)=\int^z P(\zeta)\e^{-E(\zeta)}d\zeta$, and 
the Riemann surface $\R_X$ 
of this distinguished parameter.
We introduce \emph{$(r,d)$--configu\-ra\-tion trees} 
which are 
weighted directed rooted trees. 
An $(r,d)$--configuration tree completely 
encodes 
the Riemann surface $\R_X$
and the singular flat metric associated on $\R_{X}$.
The $(r,d)$--configuration trees provide ``parameters'' 
for the complex manifold $\E(r,d)$,
which give
explicit geometrical and dynamical information;
a valuable tool for the
analytic description of $X \in \E(r,d)$.
Furthermore, given $X$, the phase portrait of 
the associated real vector field $\Re{X}$
on the Riemann sphere
is decomposed into $\Re{X}$--invariant components: 
half planes and 
finite height strips. 
The germ of $X$ at infinity 
is described as a combinatorial word 
(consisting of hyperbolic, elliptic, parabolic and entire
angular sectors
having the point at infinity of $\CW_z$ as center). 
The structural stability, under perturbation in $\E(r,d)$, 
of the phase portrait of $\Re{X}$ 
is characterized by using the 
$(r,d)$--configuration trees. 
We provide explicit conditions, in terms of $r$ and $d$, 
as to when the number of topologically equivalent 
phase portraits of $\Re{X}$ 
is unbounded.
\end{abstract}

\newpage
\setcounter{tocdepth}{2} 
\tableofcontents

\section{Introduction}

Motivated by the nature of 
meromorphic and essential singularities of 
complex analytic vector fields on Riemann surfaces 
\cite{MR}, \cite{MucinoValero},
\cite{HockettRamamurti},
\cite{AlvarezMucino}, 
\cite{AlvarezMucino2},
we study the larger families 
\begin{equation}\label{familiaErd}
\E(r,d)=\left\{ X(z)=\frac{1}{P(z)}\e^{E(z)}\del{}{z}
\ \Big\vert\ 
\begin{array}{c}
P, E\in\CC[z], 
\ P \ \text{ monic},
\\
\deg{P}=r, \ \deg{E}=d, 
\ r+d \geq 1
\end{array}
\right\}, 
\end{equation}

\noindent 
of vector fields 
on the Riemann sphere $\CW$,
having $r$ poles on
$\CC$ and an isolated essential 
singularity at $\infty$ (for $d\geq 1$). 
The appearance of poles is one of the main new features,
respect to the previous
work in \cite{AlvarezMucino}.

\noindent 
Each
$X\in\E(r,d)$ is provided with a
global singular analytic 
distinguished parameter
\begin{equation}
\Psi_{X} (z)= \int^z P(\zeta) \e^{-E (\zeta )} d\zeta  
: \CC_{z}\longrightarrow \CW_{t},
\end{equation}

\noindent 
which in turn has an associated Riemann surface 
\begin{equation}
\R_{X}=\{(z,\Psi_{X}(z))\},
\end{equation}

\noindent
whose origin can be traced back to the seminal work of
R. Nevanlinna \cite{Nevanlinna1}, \cite{Nevanlinna2}.

\noindent  
Thus there is a bijective correspondence,
induced by $\Psi_{X}$,
between
\begin{equation}\label{Xcubiertas}
X \in \E(r,d)
\ \
\longleftrightarrow
\ \
\left\{
\begin{array}{l}
\text{branched 
coverings
} 

\pi
:\R_X\longrightarrow \big(\CW_{t},\del{}{t}\big) \text{ having } \\
d \text{ logarithmic branch points over } \infty, \\
d \text{ logarithmic branch points over } \{a_{\sigma}\}\subset\CC_{t}, \\
r \text{ 
finitely
ramified branch points
over } \{\widetilde{p}_{\iota}\}\subset\CC_{t}
\end{array} \right\},
\end{equation}
where 
$\pi(z,t)=t$. Along this work,
$(\CW_{z}, X)$ denotes a pair, Riemann sphere and a singular complex 
analytic vector field, 
$\big((\CC_{z}, z_0), X\big)$ denotes a germ of a singular analytic vector field $X$ on $(\CC_{z}, z_0)$.
The work of M. Taniguchi 
\cite{Taniguchi1}, \cite{Taniguchi2}
is a keystone for understanding 
the right side of
\eqref{Xcubiertas}. 
Analogous correspondences  
have been previously used in 
\cite{AlvarezMucino}, \cite{AlvarezMucino2}, \cite{AlvarezMucinoSolorzaYee}.  
The function $\Psi_X$ is single valued, 
as a consequence there exists a
biholomorphism
$\CC_z \cong \R_X$
and  $(\Psi_X)_* X = \del{}{t}$ providing 
a \emph{global flow box} for $X$, 
see Lemma \ref{maximalOmega}.

\noindent 
The Riemann surface $\R_{X}$ can be naturally described 
by gluing
half planes $\HH^2$ and finite height strips 
$\{ 0 < \Im{t} < h\}$, see Lemma 
\ref{descomposicion-planos-bandas}.

Three natural cases arise from the values $(r,d)$:

\smallskip 
\noindent
$\bullet$
$X\in\E(r,0)$
has $r$ poles on $\CC_{z}$ and 
$\Psi_X$ is a polynomial map. 
See W. M. Boothby \cite {Boothby1}, \cite{Boothby2} for pioneering work and 
S. K. Lando {\it et al.} 
\cite{LandoZvonkin} chapters 1 and 5 for advances in the combinatorial direction.

\smallskip 
\noindent
$\bullet$
$ X\in\E(0,d)$
has an isolated essential singularity at $\infty\in\CW_{z}$, no zeros or poles.
$\Psi_{X}$ is an infinitely ramified covering map
as in \eqref{Xcubiertas}; 
see R. Nevanlinna \cite{Nevanlinna1} chapter XI, M. Taniguchi \cite{Taniguchi1}, \cite{Taniguchi2}; 
and \cite{AlvarezMucino}.

\smallskip 
\noindent 
$\bullet$
$X\in\E(r,d)$, $r, \, d \geq 1$, 
has $r$
poles on $\CC_{z}$ and an isolated essential singularity 
at $\infty\in\CW_{z}$.
$\Psi_X$ is an infinitely ramified covering map as in \eqref{Xcubiertas}.

\medskip
\noindent

Obviously, 
$\E(r,d)$
is an open complex submanifold of $\CC^{r+d+1}$, 
see \cite{AlvarezMucino2}.
However for the study of analytical, geometrical and topological aspects 
of $X$ and $\Psi_X$, ``geometrical parameters'' 
that shed light on the description 
of the vector fields are desirable. 
Recall for instance the role of the critical 
value map $\{f(z)_c=z^2 + c \}  \mapsto c $, as 
dynamical parameters for the Mandelbrot set
of the quadratic family;
our search is for 
parameters with analogous properties.
In our case, even though the map 

\medskip

\centerline{
$\big\{\text{coefficients of }P(z), E(z) \text{ from }X \big\}
\longmapsto
\left\{
\begin{array}{c}
\text{critical and asymptotic values }
\\
\{\widetilde{p}_\iota\}\cup\{a_\sigma\} \subset \CC_t \text{ of }\Psi_{X}
\end{array}
\right\}$}

\medskip

\noindent
is holomorphic, 
\emph{the critical and asymptotic values 
of $\Psi_X$ are insufficient to 
completely describe the family $\E(r,d)$}; 
see 
Example \ref{ejemplo-Exp3}, 
Figure \ref{fig-Exp3}
for an instance in $\E(0,3)$ and 
Corollary \ref{Psi-con-valores-criticos-asintoticos-preasignados}.

The classical notion
of divisor for $X\in\E(r,0)$,
as a meromorphic section of the tangent line bundle
$T\CW_z$,
provides a finite collection of pairs; poles and zeros with their orders. 
The divisor characterizes the vector field up to multiplicative factor,
see Lemma \ref{Brill--Noether-elemental}.
However, for $d \geq 1$ the essential 
singularity of $X$ at $\infty$ encodes
the information related to the exponential. 
Following the idea of divisor,
for the transcendental 
case $d\geq 1$, 
we introduce a non--Hausdorff compactification of $\CC_z$ with 
$2d$ copies of $\infty$.
This allows us to obtain a finite collection of triplets; branch points
in \eqref{Xcubiertas}
with their ramification index.
The triplets play the role of the divisor for $X\in\E(r,d)$,
see \S\ref{divisores-ternas} and Definition 
\ref{divisor-reducido}.
An advantage is that the information contained in the triplets includes 
the poles of $X$, the critical and asymptotic 
values of $\Psi_X$.

However, the information contained in the triplets is not enough for a complete 
description of $X$.
With this in mind,
in Definition 
\ref{d-confTree}, 
we introduce \emph{$(r,d)$--configuration trees 
$\Lambda_X $} which are  
weighted directed rooted trees 
that completely 
encode the branched Riemann surfaces $\R_X$, for $X\in\E(r,d)$.
They provided explicit 
``geometrical parameters'', 
which allows us to obtain a 
\emph{complete global analytical and 
geometrical description} for 
$X$.

\noindent
The \emph{vertices} of $\Lambda_{X}$ are the 
triplets associated to 
the branch points in $\R_{X}$,
over $\CC_t$ as in \eqref{Xcubiertas}, 
including their ramification index.

\noindent
The \emph{weighted edges} of $\Lambda_{X}$ provide us with 
the relative position of the branch points on $\R_X$, 
\emph{i.e.} two pieces of information:
\\
1) each \emph{edge} specifies which pair of branch points share the same 
sheet of $\R_{X}$, 
\\
2) the \emph{weight} of the edge tells us the relative number of  
sheets of $\R_{X}$, we must go ``up or down'' on the surface in order to find another sheet containing other
branch points.

Letting

\centerline{
$\E(r,d)^*\doteq
\left\{ X\in\E(r,d)\ \Big\vert \ 
\begin{array}{c}
\R_X 
\text{has at least two branch points over  
different}
\\
\text{critical and asymptotic values 
of }\Psi_X 
\end{array}
\right \}
$,}

\noindent
we have:

\begin{main}[$(r,d)$--configuration trees as 
parameters
for $\E(r,d)$]
\label{ClassificationThm}
\hfill

\noindent 
For each pair $(r,d)$, $r+d\geq 2$,
there is an isomorphism, as complex manifolds of dimension $r+d+1$, 
between $\E(r,d)^*$ 
and equivalence
classes of $(r,d)$--configuration trees
with at least two vertices,
{\it i.e.}
$$
\E(r,d)^* \cong \left\{ \big[ \Lambda_{X}\big] 
\ \big| \ \Lambda_{X}\text{ is a }(r,d)
\text{--configuration tree 
with at least two vertices}
\right\}.
$$  
\end{main}

\smallskip

The vector fields avoided in  $\E(r,d)^*$ are of two kinds; 
those in $\E(0,1)$ and those in $\E(r,0)$ with a unique pole. 
In \S\ref{concreteexamples} 
explicit examples of $\Lambda_{X}$ 
can be found, while in \S\ref{dificultades}
a digression on some of the 
difficulties encountered in the proof of the Main Theorem,
are presented. 
Moreover, these difficulties require the consideration of classes $[ \Lambda_{X} ]$ of 
$(r,d)$--configuration trees;
the description of the classes, roughly speaking, originates from a re--labelling 
of the vertices of $\Lambda_{X}$, see Definition \ref{skeleton}.
The proof is presented in \S\ref{Mainclassification}, with 
the description of the equivalence relation and their classes 
$[\, \cdot \, ]$ in \S\ref{clasesE(d)}.
Our Main Theorem extends results of \cite{AlvarezMucino} 
\S8.5
in the families $\E(0, d)$, to the families $\E(r,d)$ with $r+d \geq 1$.

We decode the information at $\infty$, that is 
we shall answer the following question:

\smallskip
\emph{How can we describe the 
essential singularity of $X$ at $\infty\in\CW_{z}$, 
for $X\in\E(r,d)$?}

\smallskip

The classical idea, which 
has its roots in the work of
I. Bendixon, 
A. A. Andronov and F. Dumortier \emph{et al.} 
(see 
\cite{Andronov-L-G-M} p.~304, \cite{Arnold-Ilyashenko} p.~84
and 
theorem 5.1 in \cite{AlvarezMucino}),
is to look at the germ at infinity 
$\big( (\CW_z, \infty), X \big)$
and try to split into a finite union of  
hyperbolic $H$, elliptic $E$, parabolic $P$
and entire $\ent{}{}$ angular sectors, this last based upon 
$\e^z \del{}{z}$ at infinity; 
see Equation \eqref{de-germen-a-palabra} and
Figure \ref{los-cuatro-sectores}.
Thus obtaining a cyclic word 
$\mathcal{W}_X$ associated to 
$\big( (\CW_z, \infty), X \big)$.

For the essential singularity of $X$ at $\infty$  
the analytic/topological nature of 
the inva\-riants of $X$
is certainly a novel aspect, see
\S\ref{epilogo}, recalling that:

\noindent 
$\bigcdot$
The germ at infinity
$\big( (\CW_z, \infty),X \big)$,  which 
is a local analytic invariant 
of $X$ 
under biholomorphism germs
of $(\CW_z, \infty)$, and also under complex
affine transformations 
$Aut(\CC) \subset PSL(2, \CC)$ of $\CC_z$.

\noindent 
$\bigcdot$
The cyclic word 
$\mathcal{W}_{X}=W_{1} W_{2} \cdots W_{k}$, which 
is a local topological invariant 
of $\Re{X}$
under local homeomorphism of $(\CW_z, \infty)$ preserving the 
orientation. 

\noindent 
The following theorem answers the above posed question, as well as the
dynamical description of 
$X$ and equivalently for
the phase portrait of 
its associated real vector field
$\Re{X}$.

\begin{theo}[Dynamical applications]
\label{teorema-aplicaciones-dinamicas}
Let be $X\in\E(r,d)$.
\begin{enumerate}[label=\arabic*),leftmargin=*]
 
\item 
The cyclic word $\mathcal{W}_{X}$ associated to $X$ at $\infty$
is recognized as 
\begin{equation}
\big( (\CW_{z},\infty),\Re{X} \big) 
\longmapsto \mathcal{W}_X=W_{1} W_{2} \cdots W_{k}, 
\quad W_{\iota}\in\{ H,E,P,\ent{}{} \},
\end{equation}
with exactly $2d$ letters $W_{\iota}=\ent{}{}$.

\item 
The cyclic word $\mathcal{W}_X$  
is a topological invariant
of the germ
$\big( (\CW_{z},\infty),\Re{X} \big)$.

\item
Let 
$\big((\CC_z ,\infty), Y \big)$ be a 
singular complex analytic vector field germ,
the following are equivalent:

\noindent
$\bullet$
The germ
$\big((\CC_z ,\infty), Y \big)$
is analytically equivalent to 
the restriction of a vector field $X \in \E(r,d)$ for $d\geq 1$.

\noindent
$\bullet$
The cyclic word $\mathcal{W}_{Y}$ of the germ
$\big((\CC_z ,\infty), \Re{Y} \big)$ 
satisfies that 
\begin{enumerate}[label=\roman*)]
\item the residue of 
the 1--form of time $\omega_Y$ of $Y$ is 
$Res(\omega_{Y}, \infty ) = 0$,

\item the Poincar\'e--Hopf index 
of $Y$ is $PH(Y, 0 ) = 2 + r$, 

\item the word $\mathcal{W}_{Y}$ has exactly $2d\geq 2$ 
entire sectors $\ent{}{}$.
\end{enumerate}

\item
The phase portrait of $\Re{X}$ is structurally stable 
(under perturbations in $\E(r,d)$) if and only if

\begin{enumerate}[label=\roman*)]
\item
$X$ has only simple poles, 

\item
all edges of $\Lambda_X$ have 
non--zero imaginary component.
\end{enumerate}

\item
The number of non topologically equivalent phase portraits 

\centerline{$\{\Re{X} \ \vert\ X\in\E(r,d) \}$ \, on $\CW_z$}

\noindent
is  infinite if and only if 

\centerline{
$(r,d)\in\big\{(r\geq2,1), (r\geq1,2),(r\geq0,d\geq3)\big\}$.}

\end{enumerate}
\end{theo}

For the accurate assertions and proofs, see 
Theorem \ref{corresp-germen-palabra}, 
Theorem \ref{estabilidadestructural} and 
Theorem \ref{numbertop} respectively.
A stronger version of the decomposition of the phase 
portrait into $\Re{X}$--invariant components, 
can be found as Theorem \ref{HorizontalStripStructures}. 

Throughout this work, the objects previously described are 
related via the diagram

\begin{center}
\begin{picture}(180,85)(0,10)
\put(-35,78){$X\in \E(r,d) $}
\put(-65,35){$[\Lambda_X]$}
\put(15,35){$(r,d)\hbox{--soul}$}

\put(-50,47){\vector(1,1){25}}
\put(-32,65){\vector(-1,-1){20}}

\put(-15,37){\vector(1,0){25}}
\put(-10,37){\vector(-1,0){25}}

\put(20,47){\vector(-1,1){25}}
\put(2,65){\vector(1,-1){20}}

\put(70,35){$
\longrightarrow 
\underbrace{\big( (\CW_z, \infty), X \big) }_{
{\tiny{ 
\begin{array}{c}
\hbox{local  analytic}
\\ 
\hbox{invariant of }X \hbox{ at } \infty
\end{array}
} } 
}
\longrightarrow 
\underbrace{\mathcal{W}_{X}=W_{1} W_{2} \cdots W_{k}}_{
{\tiny{ \begin{array}{c}
\hbox{local topological}
\\ 
\hbox{invariant  of  } \Re{X} \hbox{ at } \infty
\end{array}
} }
}$.}
\end{picture}
\end{center}

\noindent 
The \emph{soul of $X$}, Definition
\ref{soul}, is the smallest flat Riemann surface 
inside $\R_X$
that encodes all 
the combinatorial information 
of $X$ (the analogous idea appears in riemannian 
geometry \cite{Cheeger-Gromoll} 
as the soul,
and vector fields 
\cite{MucinoValero} \S5.2
as dynamical locus). 
Summing up, 
the Main Theorem 
provides the global, on $\CW$, analytic 
bijection between 

\smallskip

\noindent $\bigcdot$
a vector field $X\in\E(r,d)$, 

\noindent $\bigcdot$
a class $[\Lambda_X]$ of $(r,d)$--configuration trees, and 

\noindent $\bigcdot$
an $(r,d)$--soul.  

\smallskip 
\noindent
Clearly, the germ
$\big( (\CW_z, \infty),X \big)$, 
does not determine the class of $X$ in $\E(r,d)/ Aut(\CC)$,
see Remark
\ref{el-germen-al-infinito-no-determina-la-clase-de-X}.

Moreover, the topological classification of functions $\Psi_X$
is coarser than the 
topological 
classification of phase 
portraits of vector fields $\Re{X}$, for $\E(r,d)$, 
see Remark \ref{function-coarser-vectorfield}.
In particular, for $X\in\E(r,d)$ the Riemann surface
$\R_{X}$ admits an infinite number of half planes 
$\bar{\HH}^2$
if and only if $d\geq1$.
However, following R. Nevanlinna,
Example \ref{TresValoresAsintoticos-texto}
provides a Riemann surface admitting a decomposition in an 
infinite number of half planes, 
where the corresponding vector field does not belong
to any $\E(r,d)$.

\smallskip
The study of complex functions and vector fields
under geometric tools (in our context: combinatorial with 
complex weights)
is possible due to the richness of their geometric structure, 
the roots of which goes back to H. Schwarz \cite{Schwarz} and F. Klein \cite{Klein-2}.
Our Main Theorem 
provides a geometrical 
characterization of the vector fields $X$, 
functions $\Psi_X$ and Riemann surfaces $\R_X$ originated 
from the families $\E(r,d)$. 
It enhances 
the work of A. Speiser \cite{Speiser2}, 
R. Nevanlinna \cite{Nevanlinna1}, \cite{Nevanlinna2} p.~291 and 
G. Elfving \cite{Elfving} 
on the classification, via line complexes, 
of (simply connected) Riemann surfaces $\R_{X}$ related to 
meromorphic functions $\Psi_{X}$. 
M. Taniguchi \cite{Taniguchi1} \& \cite{Taniguchi2} and 
K. Biswas {\it et al.} \cite{Biswas-PerezMarco-1}, \cite{Biswas-PerezMarco-2} 
\& \cite{Biswas-PerezMarco-3}
develop analytic aspects of the functions $\Psi_X$, for $d \geq 1$. 
More recently, the study of parameter spaces for
complex analytic vector fields is a current subject of interest; 
\emph{e.g.}
J. Muci\~no--Raymundo \emph{et al.} \cite{MucinoValero}, 
\cite{MR} in the rational case;
B. Branner \emph{et al.} \cite{Branner-Dias}, 
M.--E. Fr\'ias--Armenta \emph{et al.} \cite{Frias-Mucino},
K. Dias \emph{et al.} \cite{Dias-Tan}, 
M. Kilme\v{s} \emph{et al.} \cite{Kilmes-Rousseau} in the polynomial case.

\smallskip 

Some of the proofs presented are based upon 
technical results of \cite{AlvarezMucino}, however
the minimal previous results,
evidence and examples provided in this work allow for a self 
contained reading and understanding.

As future work,
see \S \ref{future-work}:
in the combinatorial framework
recall the fruitful ideas of 
Bely{\u \i} functions and
dessin's d'enfants, promoted by A. Grothendieck;
$(r,d)$--configuration trees follow this, see 
\S \ref{Dessin's-d'-enfants}. 
The extension of these ideas to $\E(r,d)$ will be appear elsewhere. 
A clear topological description 
of $\R_X$ as a ramified covering, see 
\eqref{Xcubiertas}, is missing
for the more general vector field $Y(z)=(Q(z)/P(z)) 
\e^{E (z)} \del{}{z}$ having zeros: it remains 
for future projects.
The possible construction of effective local parameters
for $\E(r,d)$, avoiding the equivalence classes
in $\{[ \Lambda_X] \}$ are discussed in the Epilogue
\S \ref{epilogo}. 

\newpage
\section{Different facets for singular analytic vector fields $X\in\E(r,d)$}\label{sec:prelim}

\subsection{Vector fields, differential forms, 
orientable quadratic differentials, flat metrics, distinguished parameters, Riemann surfaces}
\label{subsec:equivalencias}
Let
\begin{equation}\label{campo-X-con-P-y-E}
X(z)=\frac{1}{P(z)}\ \e^{E(z)}\del{}{z} 
\in \E(r,d), 
\ \ \ \deg{P}=r, \ \deg{E}=d, \ r+d\geq 1,
\end{equation}
be a singular complex analytic vector field in the family
\eqref{familiaErd}. 
The polynomials describing it can be expressed as
\begin{equation}
\label{coeficientes-P-y-E}
\begin{array}{c}
P(z)=(z- p_1) \cdots (z- p_r) 
\doteq 
z^r + b_1 z^{r-1}+ \cdots + b_r , \ \
\\
\vspace{-.2cm}
\\
E(z)=\mu\, (z- e_1) \cdots (z- e_d) 
\doteq
\mu\, \big(z^d + c_1 z^{d-1}
+ \cdots + c_d \big), 
\ \ \ \mu\in\CC^*.
\end{array}
\end{equation}

\noindent
Note that 
if $d=0$ then $P(z)$ is non necessarily monic,
so in this case, let 
\begin{equation}\label{def-lambda}
\lambda\doteq\e^{E(z)}=\e^{\mu c_0}\in\CC^*.
\end{equation}
\noindent
We denote by 
\begin{equation}\label{los-polos}
\MP \doteq \{p_1, \ldots, p_\iota, \ldots , p_r \}
\end{equation}
\noindent 
the set of \emph{poles of $X$}, allowing repetitions.

\noindent
The \emph{associated singular analytic} differential form, 
\begin{equation}\label{difform}
\omega_{X}=P(z)\, \e^{-E(z)} dz,
\end{equation}
is such that $\omega_{X}(X)\equiv 1$,
also called the 1--form of time of $X$.

\noindent
A \emph{singular analytic quadratic differential} $\mathcal{Q}$ on $\CW_{z}$ is 
\emph{orientable} if it is globally given as $\mathcal{Q}=\omega\otimes\omega$, 
for some singular analytic differential form $\omega$ on $\CW_{z}$.
In our case we have the quadratic differential, 
\begin{equation}\label{cuaddif}
\mathcal{Q}_{X}=\omega_{X}\otimes\omega_{X}
= P^{2}(z)\, \e^{-2E(z)} dz^{2} .
\end{equation}

\noindent 
The  singular \emph{horizontal trajectories} 
of $\mathcal{Q}_X$ on $\CC_{z}\backslash\MP$
are equivalent 
to the trajectories of the 
real vector field $\Re{X}$, 
see for instance equation (2.2) of \cite{AlvarezMucino}. 

\noindent
Since $\omega_{X}$ is holomorphic on $\CC_{z}$, the local notion 
of \emph{distinguished parameter}
can be extended as below
(see \cite{Jenkins}, \cite{Strebel}
for the local case).
\begin{definition}
Let $X\in\E(r,d)$, the map 

\centerline{
$\Psi_{X}(z)=\int_{z_{0}}^{z} P(\zeta)\, \e^{-E(\zeta)} d\zeta : 
\CC_{z} \longrightarrow \CW_{t} $} 

\noindent 
is a \emph{global distinguished parameter for $X$}
(note the dependence on $z_0\in\CC_{z}$). 
\end{definition}

\noindent
The \emph{singular flat Riemannian metric $g_{X}$ with singular set $\MP\subset\CC_{z}$} 
on $\CC_{z}\backslash\MP$ is defined as the pullback under 
$\Psi_{X}:(\CC_{z}\backslash\MP , g_{X})\rightarrow (\CC_{t},\delta)$, 
where $\delta$ is the usual flat
Riemannian metric on $\CC_{t}$. 
The singularities of $g_{X}$ at $p_{\iota}\in\MP$ are cone points with angle $(2\nu_{\iota}+2)\pi$, where $-\nu_{\iota}\leq-1$ is the order of the pole $p_{\iota}$ of $X$.
The trajectories of $\Re X$ and $\Im  X$ are unitary geodesics in 
$(\CC_{z}\backslash\MP , g_{X})$. 

\noindent
The graph of $\Psi_{X}$

\centerline{$
\R_{X}= \{(z,t) \ \vert \  t=\Psi_{X}(z) \} \subset \CC_{z}\times\CW_{t}
$}

\noindent 
is a Riemann surface. 
Let $\pi_{X,1}$ and $\pi_{X,2}$ be the projections 
from $\R_{X}$ to $\CC_{z}$ and $\CW_{t}$, respectively.
The flat metric
on $\big(\R_{X},\pi_{X,2}^{*}(\del{}{t})\big)$ is induced by the usual metric on
$\big(\CW,\delta\big)$, equivalently
$\big(\CW_{t},\del{}{t}\big)$, 
via the projection of $\pi_{X,2}$.
Since $\pi_{X,1}$, as in Diagram \ref{diagramaRX}, is an isometry.
\begin{lemma}\label{LemmaRX}
The following diagram commutes
\begin{center}
\begin{picture}(180,80)(0,10)

\put(-90,40){\vbox{\begin{equation}\label{diagramaRX}\end{equation}}}

\put(10,75){$\big(\CW_{z},X\big) $}

\put(115,75){$\big(\R_X,\pi^*_{X,2}(\del{}{t})\big)$}

\put(108,78){\vector(-1,0){60}}
\put(65,85){$\pi_{X,1}$}

\put(133,65){\vector(0,-1){30}}
\put(138,47){$ \pi_{X,2} $}

\put(38,65){\vector(2,-1){73}}
\put(55,39){$ \Psi_X $}

\put(115,20){$\big(\CW_t,\del{}{t}\big) $.}

\end{picture}
\end{center}

\noindent 
Moreover, 
$\Psi_X$ is single valued,
by removing 
$\infty\in\CW_{z}$.
The projection 
$\pi_{X,1}$ is a biholomorphism between

\centerline{ 
$\big( \R_{X},\pi^*_{X,2}(\del{}{t}) \big)$ 
\ and  \   $(\CC_{z},X)$.}
\hfill\qed
\end{lemma}

\noindent 
In what follows, unless explicitly stated, we shall use the abbreviated form $\R_{X}$ instead of the more 
cumbersome $\big(\R_{X},\pi^*_{X,2}(\del{}{t})\big)$, see Figures \ref{figejemplo-un-polo},
\ref{figejemplo-dos-polos}, \ref{figejemplo-un-va} and \ref{figejemplo-alma-no-trivial}.

\noindent
In Diagram \ref{diagramaRX} we abuse notation slightly by saying that the domain of $\Psi_{X}$ is $\CW_{z}$. 
This is a delicate issue. 
\begin{remark}
By integrating along 
asymptotic paths 
associated to asymptotic values of 
$\Psi_{X}$ at the essential singularity 
$\infty\in\CW_{z}$, 
the choice of initial $z_0$ and end points $z$ 
for the integral defining $\Psi_X$ 
can be relaxed to include $\infty\in\CW_{z}$ as end point,
see 
Definition \ref{defasymptoticvaluepath} 
and Figure \ref{distribucion-exponentialtracts}.
\end{remark}

\begin{lemma}\label{maximalOmega}
1. The map $\Psi_X$ is a global flow box of $X$, 
\emph{i.e.}

\centerline{
$(\Psi_{X})_{*} X = \del{}{t} $ \ \ \ \ \ on the whole $\CC_z$.}

\noindent  
2. For fixed initial
condition $z_{0}\in\CC_{z}\backslash\MP$ , the maximal (under analytic continuation) time domain 
of the complex flow $\varphi$ of $X$ is
provided by $\R_{X}$, that is

\centerline{
$\varphi( z_{0} ,\,  \cdot \, ): 
\R_{X}\backslash 
\cup_{p_\iota\in\MP} \big\{(p_{\iota},\widetilde{p}_{\iota})\big\} \longrightarrow \CC_{z}\backslash\MP$,}

\noindent
is a maximal complex trajectory solution.
\end{lemma}

\begin{proof}
For assertion 2, note that 
the punctured 
$\R_X \backslash 
\cup_{p_\iota\in\MP} 
\big\{(p_{\iota},\widetilde{p}_{\iota})\big\} $
is a translation Riemann surface, 
following  
\cite{Thurston} \S 3.3 and
\cite{MasurTabachnikov}. Moreover, 
$\R_X$ is provided with singular 
horizontal and vertical foliations 
$\Re{ \pi^*_{X,2} (\del{}{t}) }$, 
$\Im{ \pi^*_{X,2} (\del{}{t})}$, of real and imaginary time. In the 
spirit of Riemann surface theory, the complex trajectory 
$\varphi(z_0, \,  \cdot \,  ) \doteq \pi_{X,1}( \, \cdot \,  )$ 
is  holomorphic and single valued
function of the variable in this punctured Riemann surface.
\end{proof}

\subsection{The singular complex analytic dictionary}
\noindent

\begin{proposition}[Dictionary between the singular analytic objects 
originating from $X\in \E(r,d)$, 
\cite{AlvarezMucino} p. 137] \label{basic-correspondence}
The following diagram describes a canonical one--to--one correspondence 
between its objects

\begin{picture}(200,138)(-70,-40)
\put(55,86){$X(z)=\frac{1}{P(z)}\, \e^{E(z)}\frac{\partial}{\partial z} $}
\put(-14,53){$\omega_X (z)= P(z)\, \e^{-E(z)}dz$}
\put(120,53){$\Psi_X (z)= \int\limits^z P(\zeta)\, \e^{-E(\zeta)}d\zeta$}
\put(60,-35){$\big((\CC ,g_X ), \Re X\big)$ \ .}

\put(58,80){\vector(-1,-1){15}}
\put(44,66){\vector(1,1){15}}

\put(152,66){\vector(-1,1){15}}
\put(138,80){\vector(1,-1){15}}

\put(58,-24){\vector(-1,1){15}}
\put(44,-10){\vector(1,-1){15}}

\put(151,-10){\vector(-1,-1){15}}
\put(138,-23){\vector(1,1){15}}

\put(0,0){$\mathcal{Q}_X= P^{2}(z)\, \e^{-2E(z)}dz^{2}
$}
 \put(128,0){$\big(\R_X, \pi^{*}_{X,2} (\frac{\partial}{\partial t})\big)$}

\put(36,37){\vector(0,-1){20}}
\put(36,22){\vector(0,1){20}}

\put(155,37){\vector(0,-1){20}}
\put(155,22){\vector(0,1){20}}

\put(-104,15){\vbox{\begin{equation}\label{diagramacorresp}\end{equation}}}
\end{picture}
\hfill\qed
\end{proposition}

\begin{remark}\label{comentariosCorrrespondencia}
The correspondence 
must be understood up to choice of initial point
$z_{0}$ for the integral defining the global distinguished 
parameter. 
Thus, $\Psi_X$ and $\Psi_X+ {\tt t }_0$ are considered
the same object.  
\end{remark}

\section{Analytic characteristics of $X\in \E(r,d)$}\label{sec:analytic}

\subsection{Order of growth at a singular point of $X$}
In the classic literature, the \emph{order of growth or growth order at $\infty$}
is defined for entire functions, 
these invariants extend for vector fields, 
see \cite{AlvarezMucino} \S4.1
and references therein.
In the present work, 
we only require the use of the \emph{1--order}.
\\
Let $\psi:(\CC\backslash\{0\},0)\rightarrow\CC$ be a germ of a complex 
analytic function with an isolated singularity at $z=0$; {\it i.e.} $\psi$ has 
a pole or an isolated essential singularity  
at the origin. 
We can define for $\varepsilon>0$

\centerline{
$
M_{\varepsilon}(\psi)=\max_{\abs{z}=\varepsilon}\{\log\abs{\psi(z)} \}.
$
}

\noindent
When the number $\rho\in \RR$ determined by

\centerline{
$
\rho(\psi)=\limsup_{\varepsilon\to0}
\frac{\log(M_{\varepsilon}(\psi))}{-\log(\varepsilon)}
$
}

\noindent
exists, it is called 
the \emph{1--order of growth of $\psi$ at $0$}.

\begin{definition}\label{class-order-type}
Let $\big((\CC,0),X(z)=f(z)\del{}{z}\big)$ be a germ of a singular 
analytic vector field, 
with 0 an isolated singularity of $X$.  
The \emph{1--order of $X$ at $0$} is the
corresponding 1--order of $f$,  
{\it i.e.} 
$\rho (X) := \rho (f)$. 
Analogously, if $\omega(z)=dz/f(z)$ is a germ of a differential form with $0$ an isolated singularity of 
$\omega$, 
then $\omega$ inherits the \emph{1--order} from that of the function $1/f$.
\end{definition}

\begin{lemma}[\cite{AlvarezMucino} p. 144]
\label{orden-de-crecimiento-de-X}
If 
$X=\frac{1}{P(z)}\, \e^{E(z)} \del{}{z}\in\E(r,d)$,  
then at $z= \infty$, 

\centerline{$X$ has 1--order $\rho(X) = deg(E(z)) = d$.}

\noindent  
In this case the 1--order of $\omega$ and $\Psi_{X}$ agree and is the negative of the 1--order of $X$.
\hfill $\Box$
\end{lemma}

\subsection{Asymptotic values of $\Psi_X$}\label{asymptoticvalues}
Asymptotic values for meromorphic functions in the classical setting appear in many instances, 
see \cite{HuaYang} 
p.~66, \cite{Nevanlinna1} pp.~298--303,   
we follow W. Bergweiler {\it et al.} \cite{BergweilerEremenko}, 
essentially verbatim from Definition \ref{eremenko1} 
to Definition \ref{logbranchpoint}, below.

Let $\Psi : \CC_{z} \longrightarrow \CW_{t}$ be a meromorphic function,
a priori not related to some vector field.
The inverse function $\Psi^{-1}$ can be defined on a 
Riemann surface which is conformally equivalent to 
$\CC$ via $\Psi^{-1}$.
We want to study the singularities of $\Psi^{-1}$. 
This can be done by adding to $\CC_{z}$ 
some ideal points and defining neighborhoods of 
these points.

\begin{definition}\label{eremenko1}
Take $a\in\CW_{t}$ and denote by 
$D(a,\rho)$ the disk of radius $\rho > 0$ (in the spherical metric) 
centred at $a$. 
For every 
$\rho > 0$,
choose a component $U(\rho)$ of 
$\Psi^{-1}(D(a,\rho))$ in such a way that 
$\rho_{1} < \rho_{2}$ implies $U(\rho_1) \subset U(\rho_{2})$. 
Note that the function $U : \rho \to U(\rho)$ 
is completely determined by its germ at 0. 

\noindent
Two possibilities can occur for the germ of $U$:
\begin{enumerate}[label=\arabic*),leftmargin=*] 
\item 
$\cap_{\rho>0} U(\rho)=\{z_{0}\},\, z_{0}\in\CC_{z}$. 
In this case $a=\Psi(z_{0})$. 

\noindent
Moreover, if $a\in\CC_{t}$ and $\Psi'(z_{0})\neq0$, 
or 
$a = \infty$ and $z_{0}$ is a simple pole of $\Psi$, 
then $z_{0}$ is called an \emph{ordinary point}. 

\noindent
On the other hand, 
if $a\in\CC_{t}$ and $\Psi'(z_{0}) = 0$, 
or 
if $a = \infty$ and $z_{0}$ is a multiple pole of $\Psi$, 
then $z_{0}$ is called a \emph{critical point} and $a$ is called a \emph{critical value}. 
We also say that the critical point $z_{0}$ \emph{lies over $a$}.

\item $\cap_{\rho>0}U(\rho) = \varnothing$. 
Then we say that our choice $\rho \to
U(\rho)$ defines a \emph{transcendental 
singularity of $\Psi^{-1}$},
and that the transcendental singularity 
$U$ \emph{lies over $a$}. 
\\
For every 
$\rho > 0$, 
the open set $U(\rho) \subset \CC_{z}$ is called a 
\emph{neighborhood of the transcendental singularity $U$}. 
So if $z_{k} \in\CC_{z}$, we say that $z_{k} \to
U$ if for every $\rho>0$ 
there exists $k_{0}$ such that $z_{k} \in U(\rho)$ for $k\geq k_{0}$.
\end{enumerate}
\end{definition}

\begin{definition}\label{defasymptoticvaluepath}
If $U$ is a transcendental singularity of $\Psi^{-1}$ 
then $a$ is an \emph{asymptotic value of $\Psi$}, which 
means that there exists an \emph{asymptotic path}
$\alpha(\tau ):(0, \infty) \longrightarrow \CC_z$ tending to 
$\infty$ 
such that  $\lim_{\tau \to \infty} \Psi( \alpha(\tau) ) = a$. 
\end{definition}

\noindent
In particular, it follows that 
every neighborhood $U(\rho)$ of a transcendental singularity $U$ is unbounded. 

\noindent
If $a$ is an asymptotic value of $\Psi$, then 
there is at least one transcendental singularity over $a$.
Certainly there can be many different transcendental singularities 
as well as critical and ordinary points over the same point $a$. 

\begin{definition}
A transcendental singularity $U$ over $a$ is called 
\emph{direct} 
if there exists $\rho > 0$ such that $\Psi(z)\neq a$ 
for $z \in U(\rho)$, this is also true for all smaller values of $\rho$.

\noindent
Moreover, $U$ is called
\emph{indirect} if it is not direct, 
{\it i.e.} for every $\rho > 0$ the function $\Psi$ takes the value $a$ in $U(\rho)$, 
in which case the function $\Psi$ takes the value $a$ infinitely often in $U(\rho)$.
\end{definition}

\begin{definition}\label{logbranchpoint}
The transcendental singularity $U$ is 
a \emph{logarithmic branch point over $a$} 
if $\Psi : U (\rho) \longrightarrow
D(a, \rho) \backslash\{a\}$ is a universal covering for some $\rho > 0$.
The (unbounded) neighborhoods $U(\rho)$ are called \emph{exponential tracts}. 
\end{definition}

The simplest case of a direct singularity is a logarithmic branch point, see Example \ref{ejemplolog}.

A simple example of an indirect singularity is given by the inverse function of $\sin(z) / z$, 
where the asymptotic value $0$ is a limit point of critical values. 
In fact W.~Bergweiler and A.~Eremenko prove that 
\emph{``If $f$ is a meromorphic function of finite order, then every indirect singularity of $f^{-1}$ 
is a limit of critical points''.}
However more complicated examples show that this is not always the case. 
See theorem 1 of \cite{BergweilerEremenko} for further discussion and details.

\subsection{Poles and zeros of $X$}
When we apply the above definitions to germs
of the distinguished parameter $\Psi_{X}$,
centered at the isolated singularities 
$\{p_1, \ldots , p_r, \infty \}$ of $X\in\E(r,d)$,
three cases appear;
poles, zeros and essential singularities. 
The local analytic normal forms of poles and zeros
of $X$ are well known, due to several authors.  
Figure \ref{fig:forma-normal}
shows their real phase portraits,
for further details see  
\cite{MR}, \cite{AlvarezMucino} p.~133 and examples 4, 5 
and 6 in \cite{AlvarezMucinoSolorzaYee}.

\begin{remark}\label{forma-normal-polos}
\emph{Analytic normal form for poles.}
The point $p_\iota \in \mathcal{P}$ is a pole of $X$, 
having order\footnote{
We convene that the order $-\nu_\iota$ of a pole 
$p_\iota$ is to be negative. }
$-\nu_\iota \leq -1$.
This corresponds
to a critical point of $\Psi_{X}$, thus a
\emph{finite covering}
$\Psi_X: U (\rho) \longrightarrow
D(\widetilde{p}_\iota, \rho) 
\backslash\{ \widetilde{p}_\iota \}$
where
$$
\widetilde{p}_\iota=\Psi(p_\iota)
\
\hbox{ is the critical value.}
$$ 
\noindent 
Furthermore, 
because of the local analytic normal forms, 
up to local biholomorphism 
\begin{equation}
\label{forma-normal-polo}
X(z)=\frac{1}{(z-p_\iota )^{\nu_\iota}}\del{}{z}
\ \ \text{ and } \ \
\Psi_{X}(z)=\frac{(z-p_\iota )^{\nu_\iota +1}}{
\nu_\iota +1 \ }
\ \ \ \hbox{ on }  (\CC, p_\iota).
\end{equation}
 
\noindent 
The local phase portrait of $\Re{X}$ at $p_\iota$
has 
$2(\nu_\iota + 1)$ hyperbolic sectors.
\end{remark}

\begin{remark}\label{forma-normal-ceros}
\emph{Analytic normal form for zeros.}
The point $\infty\in\CW_{z}$ is a
zero for $X \in \E(r,d)$ if and only if $r \geq 1$ and $d=0$. 
In which case 
$\infty$ is the unique zero of 
$X$ and has order  $s \doteq r+2 \geq 3$. 
$\Psi_{X}$ is a polynomial and $\infty$ 
is a pole of it. 
Using $\{ w \}$ as a local chart at $\infty$, 
the local analytic normal forms of $X$
and $\Psi_X$ on $(\CC_w, 0$) are 
\begin{equation}
\label{forma-normal-cero}
X(z)=w^s \del{}{w}
\ \text{ and } \
\Psi_{X}(z)= 
\frac{w^{1-s} }{(1-s) }, 
\ \ \ s \geq 3.
\end{equation}

\noindent
The local phase portrait of $\Re{X}$ at $\infty$
has $2(s - 1)$ elliptic sectors.
\end{remark}

\begin{figure*}[htbp]
\centering
\includegraphics[width=0.8\textwidth]{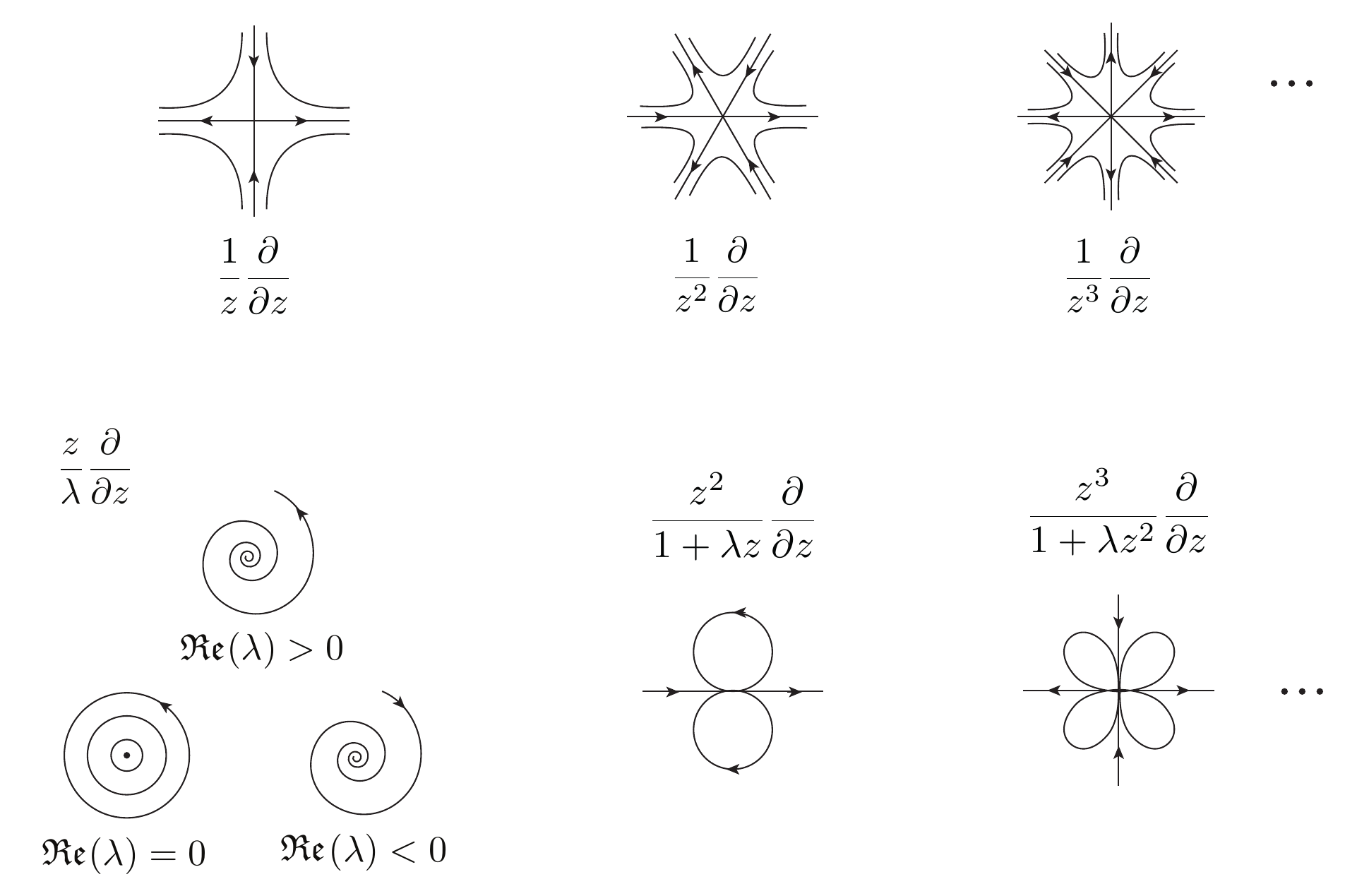}
\caption{
Normal forms of $X$ and phase portraits of 
$\Re{X}$ at poles or zeros in $z=0$.
Top row: for a pole of order $-\nu_\iota\leq -1$, the phase portrait has $2(\nu_\iota+1)$ separatrices 
arriving or leaving the pole and $2(\nu_\iota+1)$ hyperbolic sectors. 
Bottom row: simple zeros and zeros of order $s\geq 2$, here $\lambda=Res(\omega_{X},0)$. 
For simple zeros, the phase portrait is the pullback via $\Psi_{X}(z)=\lambda \log z$ of the constant vector 
field $Y(t)=\del{}{t}$. 
For $s\geq2$ the trajectories of $\Re{X}$ 
form a flower with $2(s-1)$ elliptic 
sectors. 
In our case $X \in \E(r, 0)$, note that
$\lambda=0$ and $s\geq 3$.
} 
\label{fig:forma-normal}
\end{figure*}

Summing up, and recalling \eqref{def-lambda}.
\begin{proposition}[Topological 
properties of $X \in \E(r,0)$]
\label{propRacional}
Let 

\centerline{$X(z)=\dfrac{\lambda}{P(z)} \del{}{z} \in \E(r, 0),
\ \ \ \deg{P}=r \geq 1$, $\lambda\in\CC^*$,} 

\noindent 
be a rational vector field, the following properties 
hold. 
\begin{enumerate}[leftmargin=*,label=\arabic*)]

\item
$X$ has a zero of order $r+2 \geq 3$ at $\infty \in \CW_z$.
 
\item 
The local phase portrait of $\Re{X}$ at $\infty$
has 
$2(r+1)$ elliptic sectors.

\item 
$X$ has a pole of order $-\nu_\iota$ at $p_\iota$ 
(a zero of $P(z)$ of order $\nu_\iota$).

\item 
The local phase portrait of $\Re{X}$ at $p_\iota$
has 
$2(\nu_\iota + 1)$ hyperbolic sectors.

\item 
The global phase portrait of $X$ has a decomposition 
into  

\noindent $\bigcdot$
$2r$ half planes $(\bar{\HH}^2_{\pm}, \del{}{t})$ and 

\noindent $\bigcdot$
$M$
finite height horizontal strips of the form
$\big(\{ 0 \leq \Im{t} \leq h \}, \del{}{t} \big)$,
where $0 \leq M \leq r-1$ and each $h > 0 $.
\end{enumerate}
\end{proposition}

\begin{proof}
Assertion (5) follows by a topological description
of the separatrices of $\Re{X}$ from 
the saddle points 
$p_\iota$; see
Figure \ref{EjemploArbolesE30} for examples of the 
assertions (4)--(5), in the case of three simple poles.
More detail about the decomposition in (5), will be provided
in Lemma \ref{descomposicion-planos-bandas}.
\end{proof}

A simple example that will be used throughout follows.
\begin{example}\label{campo-racional-un-polo}
Consider the vector field
\begin{equation}\label{X-1-polo}
X(z)= \dfrac{\lambda}{ (z-p_1)^{r} } \ddel{}{z} \in \E(r,0),
\ \ \ r\geq 1, 
\end{equation}
and its distinguished parameter
\begin{equation}\label{Psi-1-polo}
\Psi_{X}(z)=
\frac{1}{\lambda} 
\int\limits_{z_0}^{z} (\zeta-p_{1})^{r} d\zeta
=\frac{1}{\lambda(r + 1)} 
\big( (z-p_1)^{r + 1} - (z_0-p_1)^{r + 1}\big). 
\end{equation}
the pole $p_{1}$ of $X$ is the critical point of 
$\Psi_X$ and its critical value is
\begin{equation}\label{val-crit-1-polo}
\widetilde{p}_1=\Psi_{X}(p_{1})= 
-\frac{1}{\lambda (r + 1)}  (z_0-p_1)^{r + 1}.
\end{equation}

\noindent 
See also Example \ref{ejemplo-dos-polos}.
The vector field in Equation \eqref{X-1-polo} 
is such that, $\R_X$ has only one branch point.
Whence these vector fields
define vector fields in $\E(r,d) \backslash \E^*(r,d)$ 
which are forbidden in the Main Theorem.
\end{example}

\section{Branch points of $\R_X$}
\label{etiqueta-pendiente}

\subsection{Local ramification data for $\R_X$}\label{localram}

For $d\geq 1$, the point $\infty\in\CW_{z}$ is
an isolated essential singularity of
$$
X(z)=\frac{1}{P(z)}\ \e^{E(z)}\del{}{z},
$$

\noindent 
and the distinguished parameter $\Psi_{X}$,
belongs to the family 
\begin{equation}\label{structFinite}
SF_{r,d}=
\left\{
\Psi_X (z) =
\int_{z_0}^{z} P(\zeta)\, \e^{-E(\zeta)} d\zeta\  \ \Big{\vert} \ 
P, E\in\CC[z], \ \deg{P}=r, \ \deg{E}=d\right\},
\end{equation}
of \emph{structurally finite entire functions of type $(r,d)$}, see \cite{Taniguchi1}.
We recall the simplest object.

\begin{example}\label{ejemplolog}
Consider the vector field 
\begin{equation}\label{La-exponencial}
X(z)= \e^{\mu (z + c_1)}\del{}{z} \in \E(0,1)
\end{equation}
and its corresponding distinguished parameter
\begin{equation}\label{Psi-1-va}
\Psi_{X}(z)
=\frac{1}{\mu} \int_{z_0}^{z} \e^{-\mu (\zeta + c_1)} d\zeta=
\frac{1}{\mu} \big(\e^{-\mu (z_0+c_1)}-\e^{-\mu (z+c_1)} \big),
\end{equation}
with $\mu\in\CC^*$, $c_1\in\CC$ as in \eqref{coeficientes-P-y-E}.
Of course $\infty\in\CW_{z}$ is an isolated essential singularity of both.  
Moreover, $\Psi_{X}$ has two asymptotic values 
\begin{equation}\label{va-exponencial}
a_1= \frac{1}{\mu} \e^{-\mu (z_0 + c_1)} \in\CC_t
\quad \text{ and } \quad
a_2=\infty\in\CW_t
\end{equation} 
with exponential tracts the half planes 

\centerline{
$U_{1}(\rho)=\{ z\in\CC_z \ \vert\ \Re{\mu z}>\rho \}$ \ and \
$U_{2}(\rho)=\{ z\in\CC_z \ \vert\ \Re{\mu z}<-\rho \}$}

\noindent
respectively. 
The multivalued function 

\centerline{$
\Psi_{X}^{-1}(t)=
\dfrac{1}{\mu} \log \left(-\dfrac{\e^{\mu  (z_0+c_1)}}{t \, \mu \,  \e^{\mu  (z_0+c_1)}-1}\right) - c_1
$}

\noindent  
has two logarithmic branch points: 
one over the finite asymptotic value $a_1$ and the 
other over the asymptotic value $a_2=\infty$. 
Note that Equation \eqref{La-exponencial} defines a forbidden 
stratum in the whole family $\E(r,d)$ of the Main Theorem.
\end{example}

In order to determine the Riemann surface $\R_{X}$ precisely, one needs the knowledge of
the branch points 
under $\pi_{X,2}$
\begin{equation}
\label{puntos-ramificacion-RX}
\{( z_{\msigma},t _{\msigma} )\}
\subset\R_{X}, 
\ \ \
z_{\msigma}\in \{ p_1, p_2, \ldots, p_r, \infty \},
\ \ \
t _{\msigma} \doteq \Psi_X(z_{\msigma}),
\end{equation}
the subindex $\msigma$ will be very useful in 
several constructions. 
The next result clearly explains the singularities of $\Psi_{X}^{-1}$. 

\begin{lemma}
[The existence of finitely ramified and logarithmic branch points]
\label{pareja-finita-infinita}
Let $\Psi_X:\CC_{z}\to\CW_{t}$ be a 
structurally finite entire function of type $(r,d)$.
Then  
\begin{enumerate}[leftmargin=*,label=\arabic*)]
\item $\Psi_X$ has $r$ critical values $\{\widetilde{p}_\iota\}\subset\CC_{t}$ 
(counted with multiplicity),
\item $\Psi^{-1}_X$ has $d$ direct singularities corresponding to $d$ 
logarithmic branch points over $d$ finite asymptotic 
values $\{a_{\sigma}\}\subset \CC_t$ of $\Psi_X$, 
and 
\item $\Psi^{-1}_X$ has $d$ direct singularities corresponding to $d$ 
logarithmic branch points over $\infty\in\CW_{t}$.
\end{enumerate}
Furthermore, $\Psi^{-1}_X$ has no indirect singularities.
\end{lemma}

\begin{proof}
Case $(r,d)$ with $d\geq 1$ can be found as lemma 8.4 in \cite{AlvarezMucino} 
with a proof that relies heavily on the work of M.~Taniguchi \cite{Taniguchi1}, \cite{Taniguchi2}.
\end{proof}


\subsection{Branch point enumeration}\label{divisores-ternas}

As motivation, let $T\CW_z$ be  
the holomorphic tangent bundle of
$\CW_z$. 
For 
$X \in \E(r,0)$, the \emph{divisor of 
the meromorphic section $X:\CW_z \longrightarrow T\CW_z$} 
is the assignment\footnote{
A formal sum of pairs 
$(p, \nu)$ denoting a point in $\CW_z$ and its order
in $\ZZ^*$, positive for a zero of the vector field, negative for a pole.} 
\begin{equation}
\label{divisor-de-X-clasico}
X \
\longmapsto \ 
(\infty, r+2 ) \cup 
\big\{
(p_\iota, -\nu_{\iota})
\big\}_{\iota= 1}^{n} , 
\end{equation}

\noindent 
where 
$\sum_{\iota=1}^{n} \nu_{\iota} = r $
and $n \leq r$, the equality holds if and only if
all the poles of $X$ are simple. 
A very useful result for meromorphic vector fields 
on $\CW_z$ of the so called 
Brill--Noether theory in algebraic geometry 
(or the Poincar\'e--Hopf theory if you prefer); 
the zeros and poles determine a meromorphic vector field on $\CC_Z$
up to scalar factor.

{\begin{lemma}
\label{Brill--Noether-elemental}
Let $\nu_j, \, \nu_\iota \in \NN$, 
a divisor  

\centerline{$ 
\big\{
(q_j, \nu_j)
\big\}_{\j= 1}^{{\tt s}} 
\cup 
\big\{
(p_\iota, -\nu_{\iota})
\big\}_{\iota= 1}^{n}
\ 
\subset \CW_z \times \ZZ^*$}

\noindent 
determines a family of meromorphic vector fields 
$\{ \lambda Y  \ \vert \ \lambda \in \CC^* \}$
on $\CW_z$
if and only if 

\centerline{$
\sum_{j=1}^{\tt s} \nu_{j} - 
\sum_{\iota=1}^{n} \nu_{\iota} 
= s-r=2.$}
\hfill \qed
\end{lemma}

We would like to extend Lemma \ref{Brill--Noether-elemental} to the case of $X\in\E(r,d)$, $d\geq 1$.
Of course the divisor \eqref{divisor-de-X-clasico} 
is not well defined since $X$ is not meromorphic on $\CW_z$.

\noindent
In order to accomplish this 
an accurate enumeration of the
branch points 
$\{(z_{\msigma},t_{\msigma})\}\subset\R_{X}$ 
in \eqref{puntos-ramificacion-RX}
is required.
In what follows, the reader might find it helpful to follow along with Figures 
\ref{figejemplo-un-polo}--\ref{figejemplo-dos-polos}, 
\ref{figejemplo-un-va}--\ref{fig7ejemploscampos} in \S\ref{concreteexamples}.

\smallskip

For $r\geq1$, the point $z_\msigma$ is a \textbf{pole} $p_{\iota}\in\CC_{z}$ of $X$ 
if and only if its 
image $\widetilde{p}_{\iota}=\Psi_{X}(p_{\iota})\in\CC_{t}$ is a critical value of $\Psi_{X}$.
Moreover $(p_{\iota},\widetilde{p}_{\iota})\in\R_{X}$ is a finitely ramified branch point 
(under $\pi_{X,2}$) with ramification 
index $\nu_{\iota}+1\geq2$, where 
$-\nu_{\iota} \leq -1$ is the order of the pole 
$p_{\iota}$, recall \eqref{forma-normal-polo}. 

\noindent
Enlarging the pairs in the divisor,
we enumerate the corresponding finitely ramified branch points using \emph{triplets} 
\begin{equation}
\label{poleenumeration}
\left\{ 
\circled{\, \iota\, }
\doteq 
(p_{\iota},\widetilde{p}_{\iota}, - \nu_\iota)
\right\}_{\iota=1}^{n}\subset\R_{X}, 
\text{ with order } 
-\nu_{\iota}\leq-1 
\text{ and }
\sum_{\iota=1}^{n} \nu_{\iota} = r.
\end{equation}

\noindent 
Abusing notation, we say that the triplets
are in $\R_X$.

\smallskip 
For $d \geq 1$, after enumerating the poles as above, 
we also need to consider the logarithmic branch points in $\R_X\subset\CC_z \times \CC_t$.
We shall use two compactifications: the usual one for $\CC_t$, namely the Riemann sphere $\CW_t$,
and a non--Hausdorff one for $\CC_z$ as follows.

\begin{definition}\label{singesen}
The \emph{non--Hausdorff closure} 
\begin{equation}\label{CnoHausdorff}
\overline{\CC}_{z} \doteq
\Big(
\big(\CW\times\{1\}\big) \sqcup \big(\CW\times\{2\}\big) \sqcup \cdots \sqcup \big(\CW\times\{2d\}\big)
\Big) \Big\slash \sim
\end{equation}
is the sphere with $2d\geq2$ infinities, 
that is the disjoint union of $2d$ copies of the 
Riemann sphere $\CW$ with the equivalence relation  
$(z,\sigma)\sim (z,\eta)$ 
for all $\sigma,\eta \in \{1,\dots,2d\}$ 
when $z\neq\infty$.
\end{definition}

The point
$\infty\in\CW_{z}$ is an \textbf{isolated essential singularity of $X$}.
Hence,
we will denote the $2d$ different infinities, 
referred to in \eqref{CnoHausdorff} and Lemma \ref{pareja-finita-infinita}, by
\begin{equation}\label{notacion-infinitos}
\left\{
\infty_{\sigma} \doteq (\infty, \sigma ) 
\right\}_{\sigma= 1}^{2d}\subset\overline{\CC}_{z}.
\end{equation}
In order to include the 
logarithmic branch points,
$\R_X$ extends to 
$\overline{\R}_X \subset\overline{\CC}_z \times \CW_t$.
By simplicity, we shall use the same notation $\R_X$
for this extension. 
Lemma \ref{pareja-finita-infinita}
allows us to accurately denote the distinct
asymptotic values of $\Psi_X$ by
\begin{equation}\label{essenenumeration}
\begin{array}{rl}
\{a_{j}\}_{j=1}^{m}\subset\CC_{t} 
&
\text{ with multiplicities }\{ {\tt m}_j \geq 1 \}_{j=1}^{m}, \
\sum_{j=1}^{m} {\tt m}_{j}=d 
\text{ and }
\\[7pt]
a_{m+1} =\infty\in\CW_{t}  &
\text{ with multiplicity } d.
\end{array}
\end{equation}

\noindent 
Thus, $\pi_{X,2}^{-1}(a_{j})$ contains
${\tt m}_j $ (resp. $d$) 
logarithmic branch points
for each exponential tract 
$U_\sigma (\rho)$ 
associated to the asymptotic value $a_{j}$, $j=1,\ldots,m$ 
(resp. $j=m+1$).

\noindent
Recalling the 
multiplicities, 
the correspondence between 
indices is given by
\begin{equation}\label{correspjalphaIndices}
\begin{array}{r}
\sigma  \in 
\underbrace{1,\,  \ldots\, , {\tt m}_{1}} 
, 
\
\underbrace{{\tt m}_{1}+1\ ,\, \ldots\, , {\tt m}_{1}+{\tt m}_{2}} 
,\,\ldots\, ,
\,
\underbrace{d-{\tt m}_{m}+1, \, \ldots\, , d} 
,
\, \underbrace{d+1, \, \ldots\, , 2d} 
,
\\
j=j(\sigma) \in  
\quad\ \ \,
1 
\quad\ \ \,
,
\qquad \quad \quad \ 
2
\quad \quad \quad \quad  \,
,\,\ldots\, , 
\quad \quad \quad  \ \,
m 
\quad \quad \quad \ \ 
, 
\ \,
\quad 
m+1 \
\quad\ \, 
, 
\\
\end{array}
\end{equation}

\noindent
where
$\sigma$ enumerates the logarithmic branch points, 
while $j=j(\sigma)$  enumerates the distinct asymptotic values $a_{j}\in\CW_t$,
in accordance with \eqref{essenenumeration}.

\begin{remark}
1. The asymptotic paths $\alpha_{\sigma}(\tau)$ lie in the 
non--Hausdorff closure $\overline{\CC}_z$. 
If $\alpha_{\sigma}(\tau):(0, \infty) \longrightarrow \overline{\CC}_{z}$ is an asymptotic path
approaching $\infty_{\sigma}\in\overline{\CC}_{z}$ 
associated to the asymptotic value $a_{j(\sigma)}$ then
we may assume that $\alpha_{\sigma}(\tau)$ is restricted to one exponential tract (the one containing 
$\infty_{\sigma}\in\overline{\CC}_{z}$).
The actual choice of $\alpha_\sigma(\tau)$ inside the exponential tract $U_\sigma(\rho)$ will be 
made explicit in Proposition \ref{prop-equidistribution}.4 and 
Remark \ref{rem-equidistribution}.1.

\noindent
2. Each asymptotic path $\alpha_{\sigma}(\tau)$ together with 
the distinguished parameter $\Psi_{X}$ 
gives rise to the asymptotic value
\begin{equation}
a_{j(\sigma)}=
\lim_{\substack{z\to\infty_{\sigma}\\ z\in\alpha_\sigma }} 
\Psi_{X}( z )=
\lim_{\tau \to \infty } 
\int\limits_{z_0}^{\alpha_{\sigma}(\tau)} 
P(\zeta) \e^{-E(\zeta)} d\zeta \in\CW_t.
\end{equation}

\noindent
3.
Because of the multiplicity of $a_{j(\sigma)}$,
when $j(\sigma)\neq m+1$
there are exactly ${\tt m}_{j}$ asymptotic paths 
$\alpha_{\sigma}(\tau)$
and ${\tt m}_{j}$ exponential tracts for each 
of the $m$ distinct finite asymptotic values $a_{j(\sigma)}\in\CC_{t}$,
see Figure \ref{distribucion-exponentialtracts}. 
Moreover, there are $d$ asymptotic paths and $d$
exponential tracts for the asymptotic value $a_{m+1}=\infty\in\CW_{t}$.
\end{remark}

\smallskip
\noindent
Recalling that logarithmic branch points are infinitely ramified, and
using the notation provided by Equation \eqref{notacion-infinitos},
we will denote the logarithmic branch points 
over the 
asymptotic values 
$a_{j(\sigma)}\in\CC_{t}$
of $\Psi_X$, 
as \emph{triplets} 
\begin{equation}\label{essenvert}
\circled{\scalebox{0.75}{$n \text{$+$} \sigma$}  }
\doteq
(\infty_{\sigma},a_{j(\sigma)}, - \infty)
\in \R_{X}, 
\
\text{ for } \sigma\in\{1,\ldots,2d\}.
\end{equation}

\noindent 
Abusing notation, we say that the triplets
are in $\R_X$.

Note that the above discussion proves the following.
\begin{lemma}\label{sigma-enumera-todo}
For $X\in\E(r,d)$, there is a correspondence between: 

\noindent
$\cdot$ the $2d$ logarithmic branch points $\circled{\scalebox{0.75}{$ n \text{$+$} \sigma$}}$ in $\R_{X}$, 

\noindent
$\cdot$ the $2d$ asymptotic values $a_{\sigma}$ 
of $\Psi_X$
(counted with
multiplicities) in $\CW_{t}$, 

\noindent
$\cdot$ the $2d$ exponential tracts $U_\sigma (\rho)$ in $\CW_{z}$ and 

\noindent
$\cdot$ the $2d$ 
asymptotic paths $\alpha_{\sigma}(\tau)$ in $\overline{\CC}_{z}$.
\hfill \qed
\end{lemma}

\begin{remark}\label{correspexponentialtract}
1. An immediate advantage of the correspondence observed 
in Lemma \ref{sigma-enumera-todo} 
is that the index $\sigma$ simultaneously enumerates all of 
the objects in question.
In particular, from now on we agree that  

\centerline{$a_{\sigma}$ is referring to 
$a_{j(\sigma)}\in\CW_t$, as in \eqref{correspjalphaIndices}.}

\noindent  
For instance 
$\{a_{j(\sigma)}\}$ corresponds to one point in 
$\CC_{t}$ for $\sigma={\tt m}_{1}+1,\ldots, {\tt m}_{1}+{\tt m}_{2}$.

\noindent
2. Each exponential tract when considered in $\CW_z$
is an angular sector about $\infty$. 
Hence the exponential tracts have a natural counterclockwise
ordering about $\infty\in\CW_{z}$ arising from 
$S^1=\{\e^{i\theta}\}$. 
The ordering will be made explicit in Proposition \ref{prop-equidistribution}.3;
see also Figure \ref{distribucion-exponentialtracts} and 
the last row in Figure \ref{figEjemploE23-1-10-20-50}.
\end{remark}

We have the following \emph{ad hoc} notion, that expands 
the notion of the divisor of $X$ as meromorphic section, 
Equation \eqref{divisor-de-X-clasico}.

\begin{definition}\label{eldivisor}
Let $X \in \E(r,d)$, 
$d\geq 1$, the assignment 
\begin{multline}\label{eqdivisor}
X 
 \longmapsto 
\underbrace{
\Big\{
\circled{\, \iota \,}=
\big( p_{\iota},\widetilde{p}_{\iota}, -\nu_{\iota} \big)
\Big\}_{
\iota=1 
}^{n}
}_{\hbox{\tiny{pole vertices} }}
\cup
\underbrace{
\Big\{ 
\circled{\scalebox{0.75}{$n\text{+} \sigma$} }=
\big( \infty_{\sigma},a_{\sigma},-\infty \big)
\Big\}_{\sigma=1}^{d}
}_{ \hbox{\tiny{essential vertices} }}
\\
\cup
\underbrace{
\Big\{ 
\circled{ \scalebox{0.75}{$n\text{+} \sigma$} }=
\big( \infty_{\sigma}, \infty ,-\infty \big)  
\Big\}_{\sigma=d+1}^{2d} 
}_{ \hbox{\tiny{vertices over} } \infty }
\end{multline}

\noindent 
is the \emph{divisor of $X$}. 
\end{definition}

The notations in Equation \eqref{eqdivisor} 
will be the useful at several stages of the proof 
of the main result.
The following section explains the equidistribution of the exponential tracts and thus
provides a natural 
ordering/enumeration for 
the asymptotic values.

\subsection{Approximation of $X \in \E(r,d)$, $d \geq 1$, via rational vector fields $X_{\tt n}$}
\label{aproximando-via-Euler}
The vector fields $X \in \E(r,d)$, $d \geq 1$ can be approximated 
by rational vector fields of the form 
$X_{\tt n}(z)=
\frac{1}{ {\tt P}_{\tt n} (z)}\del{}{z}$.
Analogous ideas for other differential 
equations are applied in \cite{Masoero}. 
Moreover, as will be shown, the construction behaves nicely providing insight
into de combinatorial/geometrical structure of $X$, $\Psi_{X}$ and $\R_{X}$.

Let $X$ be as in \eqref{campo-X-con-P-y-E} 
and recall Euler's formula
\begin{equation}\label{formula-euler}
\e^{-E(z)} = \lim\limits_{{\tt n}\to\infty}
\Big( 1 - \dfrac{E(z)}{{\tt n}} \Big)^{\tt n}.
\end{equation}
Thus 
\begin{equation}\label{X-Psi-aproximantes}
X_{\tt n}(z)=\dfrac{ 1 }{ P(z) \big( 1 - \frac{E(z)}{{\tt n}} \big)^{\tt n} }\del{}{z} 
\ \ \ \text{ and  } \ \ \
\Psi_{\tt n} (z)=\int\limits^{z}_{z_0} P(\zeta) \Big( 1- \dfrac{E(\zeta)}{{\tt n}} \Big)^{\tt n} \, d\zeta
\end{equation}
converge to 

\begin{equation}\label{X-Psi}
X(z)=\dfrac{\e^{E(z)}}{P(z)} \del{}{z}
\ \ \ \text{ and  } \ \ \
\Psi_{X}(z)=\int_{z_0}^{z} P(\zeta) \e^{-E(\zeta)} d\zeta
\end{equation}
uniformly on compact sets of $\CC_{z}$,
while 

\centerline{
$\R_{X_{\tt n}}=\{(z,\Psi_{\tt n}(z)) \ \vert\ z\in\CC_z \}$
}

\noindent
converges to $\R_X$,
in the Caratheodory topology; see 
K. Biswas  \emph{et al.}
\cite{Biswas-PerezMarco-1}, 
\cite{Biswas-PerezMarco-2}
for details on Caratheodory convergence.

\noindent
Because of the Dictionary Proposition \ref{basic-correspondence}
(see also Remarks \ref{forma-normal-polos} and \ref{forma-normal-ceros}),
as ${\tt n}\to\infty$, 
the successions $\{ X_{\tt n} \}$, $\{ \Psi_{\tt n} \}$ 
and $\{ \R_{X_{\tt n}} \}$ 
enjoy the following analogous features:

\noindent
Each $X_{\tt n}$ has: 
\begin{itemize}[label=$\bigcdot$,leftmargin=*]
\item
$r$ poles at the roots $\{ p_\iota \}_{\iota=1}^{r}$ of $P(z)$ with orders $\{-\nu_\iota\}$,

\item 
$d$ poles at the roots $\{ \widehat{e}_\sigma=\widehat{e}_\sigma({\tt n}) \}_{\sigma=1}^{d}$
of ${\tt n}-E(z)$, each of order $-{\tt n}$ and

\item
a zero of order $r+d{\tt n}+2$ at $\infty\in\CW_{z}$.
\end{itemize}

\noindent
In consequence, 
each $\Psi_{\tt n}$ has: 
\begin{itemize}[label=$\bigcdot$,leftmargin=*]
\item
$r$ critical points $\{ p_\iota \}_{\iota=1}^{r}$ at 
the roots of $P(z)$,
with $r$ critical values $\{\widetilde{p}_\iota({\tt n}) \doteq \Psi_{\tt n}(p_\iota) \}_{\iota=1}^{r}$,

\item
$d$ critical points $\{ \widehat{e}_\sigma=\widehat{e}_\sigma({\tt n}) \}_{\sigma=1}^{d}$ 
at the roots of ${\tt n}-E(z)$ 
with critical values 
$\{ \widetilde{e}_{\sigma}({\tt n})\doteq\Psi_{\tt n}(\widehat{e}_\sigma) \}_{\sigma=1}^{d}$ 
and

\item
a pole of order $-(r+d{\tt n}+2)$ at $\infty\in\CW_{z}$. 
\end{itemize}

\noindent 
Each $\R_{X_{\tt n}}$ has: 
\begin{itemize}[label=$\bigcdot$,leftmargin=*]
\item
$r$ finitely ramified branch points 
$\{ (p_\iota,\widetilde{p}_\iota({\tt n}), -\nu_{\iota} ) \}_{\iota=1}^{r}$ 
with ramification index corresponding 
to $\nu_{\iota}+1$ where $-\nu_{\iota}$ is the order of the pole $p_{\iota}$,

\item
$d$ finite ramification points 
$\{ (\widehat{e}_\sigma({\tt n}),\widetilde{e}_{\sigma}({\tt n}), -{\tt n}) \}_{\iota=1}^{d}$ 
with ramification index ${\tt n}+1$
and 

\item
a finite ramification point $(\infty,\infty, r+{\tt n} d +2)$ with ramification index $r+{\tt n} d+3$.
\end{itemize}

\smallskip

Clearly the critical points $p_{\iota}$ do not change as ${\tt n}\to\infty$, 
however the critical values 
$ \widetilde{p}_{\iota}({\tt n})\in\CC_{t} $ do, but remain finite without
changing their ramification index, 
thus each finitely ramified branch point $ (p_{\iota},\widetilde{p}_\iota({\tt n}),-\nu_{\iota}) $ converges
to the finitely ramified branch point 
$(p_\iota, \widetilde{p}_\iota,-\nu_{\iota})=(p_\iota,\Psi_{X}(p_\iota),-\nu_{\iota})$.

As the critical points $\{ \widehat{e}_\sigma({\tt n}) \}$ 
approach $\infty\in\CW_{z}$, the corresponding critical values $\{ \widetilde{e}_\sigma({\tt n}) \}$ 
converge to the finite 
asymptotic values $\{a_\sigma\}$ 
of $\Psi_X$, 
as follows.

\noindent 
In $\CW_z$, arguments (directions) 
and angular sectors at $\infty$ are well defined.
The succession  $\widehat{e}_\sigma({\tt n})$
converges to a ray with constant argument $\theta_{\sigma}$ 
starting at $\infty \in \CC_z$.
Moreover, 
the rays $\theta_{\sigma}$ and $\theta_{\sigma+1}$ are exactly $2\pi/d$ radians apart.

\smallskip

A key point is the careful examination of the 
phase portraits of the rational 
vector fields $\Re{X_{{\tt n}} }$
(for the description of the phase portrait 
in the vicinity of poles and zeros, recall Proposition \ref{propRacional} 
and Figure \ref{fig:forma-normal}), 
having the following features:

\noindent i) the zero at $\infty$ of $X_{{\tt n}}$
determines 
$2r + 2 {\tt n} d+2$ elliptic sectors, and the same number of separatrices,

\noindent ii)
the $d$ poles at $\widehat{e}_\sigma$ of
$X_{{\tt n}}$
determine $2{\tt n} + 2$ hyperbolic  sectors, 
and the same number of separatrices.

The above is summarized in 
Figure \ref{distribucion-exponentialtracts}
and the following.

\begin{proposition}\label{prop-equidistribution}
Let $d\geq1$ and $X\in\E(r,d)$.
Then, the sequence of polynomial distinguished parameters $\Psi_{X_{\tt n}}$ 
given by \eqref{X-Psi-aproximantes}, converges to $\Psi_X$.
Furthermore:
\begin{enumerate}[leftmargin=*,label=\arabic*)]
\item 
Each of the $r$ sequences of finitely ramified branch points 
$(p_{\iota},\widetilde{p}_\iota({\tt n}), -\nu_{\iota}) \in \R_{X_{\tt n}}$ converges
to the finitely ramified branch point $(p_\iota, \widetilde{p}_\iota, -\nu_{\iota}) \in \R_{X}$.

\item
Each of the $d$ sequences of finitely ramified branch points
$(\widehat{e}_\sigma({\tt n}),\widetilde{e}_{\sigma}({\tt n}), -{\tt n})\in\R_{X_{\tt n}}$ 
converges to the logarithmic branch point
$(\infty_{\sigma},a_{\sigma},-\infty)\in\R_{X}$.

\item 
The $2d$ exponential tracts of $\Psi_{X}$ are angular sectors of angle 
$\pi/d$ about $\infty\in\CW_{z}$ and they 
alternate so that each exponential tract corresponding to a finite asymptotic value is in between two 
exponential tracts corresponding to the asymptotic value $\infty\in\CW_{t}$.

\item 
There exist $2d$ 
asymptotic paths $\alpha_{\sigma}(\tau)$ associated to 
the asymptotic values $\{a_\sigma\}_{\sigma=1}^{2d}$
of $\Psi_{X}$ which
are angularly equidistributed
about $\infty\in\CW_{z}$. 
\hfill $\Box$
\end{enumerate}
\end{proposition}
\begin{remark}\label{rem-equidistribution}
1. By recalling that 

\centerline{$
E(z)=
\mu\, \big(z^d + c_1 z^{d-1} + \cdots + c_d \big)
$ in\eqref{formula-euler},}

\noindent
a simple calculation shows that:

\noindent
a) for the finite asymptotic values $a_{\sigma}\in\CC_{t}$
of $\Psi_X$, 
without loss of generality we can choose
asymptotic paths $\alpha_{\sigma}(\tau)$ arriving at $\infty\in\CW_{z}$ with angle 
$\theta_{\sigma}=\Arg{\mu^{1/d}} + \frac{2\pi}{d} \sigma$ 
for $\sigma=1,\ldots,d$;

\noindent
b) likewise, the asymptotic paths $\alpha_{\sigma}(\tau)$ 
corresponding to the asymptotic value $\infty\in\CW_{t}$
arrive at $\infty\in\CW_{z}$ with angle
$\theta_{\sigma}=\Arg{\mu^{1/d}} + \frac{2\pi}{d}(\sigma-d) + \frac{\pi}{d}$
for $\sigma=d+1,\ldots,2d$. 

\noindent
2. Figure \ref{distribucion-exponentialtracts} shows the angular equidistribution of the exponential tracts 
around $\infty$ for $X\in\E(r,d)$. 
\end{remark}
\begin{figure}[htbp]
\begin{center}
\includegraphics[width=0.50\textwidth]{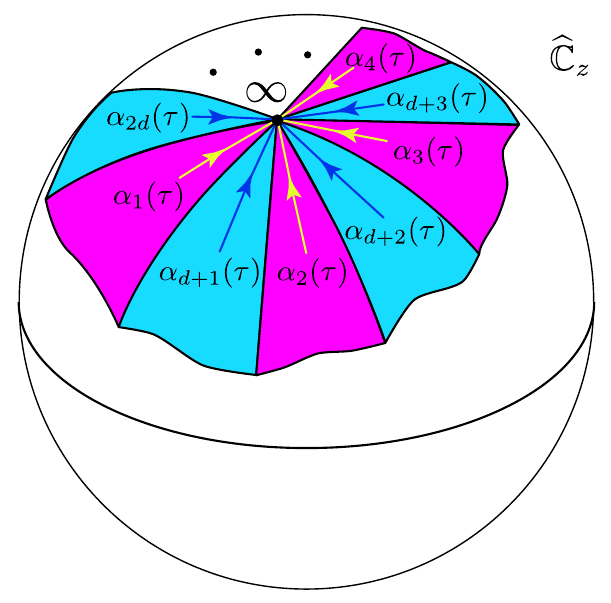}
\caption{
Angular equidistribution of the 
$2d$ exponential tracts about $\infty \in \CW_z$,
corresponding to the 
asymptotic values of $\Psi_{X}(z)$ for $X\in\E(r,d)$, 
$d\geq 1$.  
Purple angular sectors represent exponential tracts corresponding to finite asymptotic values 
$a_\sigma \in \CC_t$, in yellow their asymptotic paths $\alpha_{\sigma}(\tau)$, for $\sigma=1,\ldots,d$. 
Blue angular sectors represent exponential tracts corresponding to the asymptotic value $\infty\in \CW_{t}$,
in dark blue their asymptotic paths, $\alpha_{\sigma}(\tau)$, for $\sigma=d+1,\ldots,2d$.
Once again, the $2d$ asymptotic paths are equally distributed about $\infty\in\CW_{z}$.
Note that for the family $\E(r,d)$ this equidistribution property is independent of $r$.
}
\label{distribucion-exponentialtracts}
\end{center}
\end{figure}

\begin{example}\label{E23-1-10-20-50}
Let
$$
X(z)= \frac{\e^{z^d}}{z^r} \del{}{z}\in\E(r,d), 
\ \ \
\text{ for }d\geq1.
$$
Euler's formula provides the approximation of $X$ by the vector fields
$$X_{\tt n}(z)=\dfrac{ 1 }{ z^r \big( 1 - \frac{z^d}{{\tt n}} \big)^{\tt n} }\del{}{z},
\ \ \ 
\text{ for } {\tt n} \geq 1,
$$ 
so
$$\Psi_{\tt n} (z)=\int\limits^{z}_{0} \zeta^r \big( 1 - \frac{\zeta^d}{{\tt n}} \big)^{\tt n} \, d\zeta
=
\frac{z^{r+1} \, _2F_1\left(-{\tt n},\frac{r+1}{d};\frac{r+1}{d}+1;\frac{z^d}{\tt n}\right)}{r+1},
$$
where $ _2F_1$ is the classical Gauss's hypergeometric function 
(see for instance \cite{Abramowitz} ch.~15).

\noindent
The poles of $X_{\tt n}$ are $0$, of order $-r$, and 
$\{ \widehat{e}_\sigma({\tt n}) \doteq e^{\frac{2 i \pi \sigma}{d}} {\tt n}^{1/d} \}_{\sigma=1}^{d}$, 
of order $-{\tt n}$. 
Of course the poles of $X_{\tt n}$ are the critical points of $\Psi_{\tt n}$,

\noindent
On the other hand, the critical values of $\Psi_{\tt n}$ are given by $\Psi_{\tt n} (0)=0$ and
\begin{equation}\label{val-criticos-aproximacion}
\widetilde{e}_\sigma({\tt n}) \doteq \Psi_{\tt n} (e^{\frac{2 i \pi \sigma}{d}} {\tt n}^{1/d})
= 
e^{2 i \pi \sigma \frac{(r+1)}{d}} \,
\frac{\Gamma \left(\frac{r+1}{d}+1\right)}{(r+1)} \  
\frac{{\tt n}^{(r+1)/d}\ \Gamma ({\tt n}+1) }{ \Gamma \left({\tt n}+\frac{r+1}{d}+1\right)},
\end{equation}
for $\sigma=1,\ldots,d$.

\noindent
Furthermore $X_{\tt n}$ has a unique zero at $\infty\in\CW_{z}$ of order $r+{\tt n} d+2$.

\noindent
Hence each Riemann surface 

\centerline{
$\R_{X_{\tt n}}=\{(z,\Psi_{\tt n}(z)) \ \vert\ z\in\CC_z \}$
}

\noindent
has $(0,0,-r)\in\R_X$ as a branch point with ramification index $r+1$ and $d$ branch points
$$
\left( \widehat{e}_\sigma({\tt n}), \widetilde{e}_\sigma({\tt n},-{\tt n}) \right)\in\R_X,
\ \ \
\text{ for } \sigma=1,\ldots,d,
$$
with ramification index ${\tt n}+1$.

\begin{figure}[htbp]
\begin{center}
\includegraphics[height=0.9\textheight]{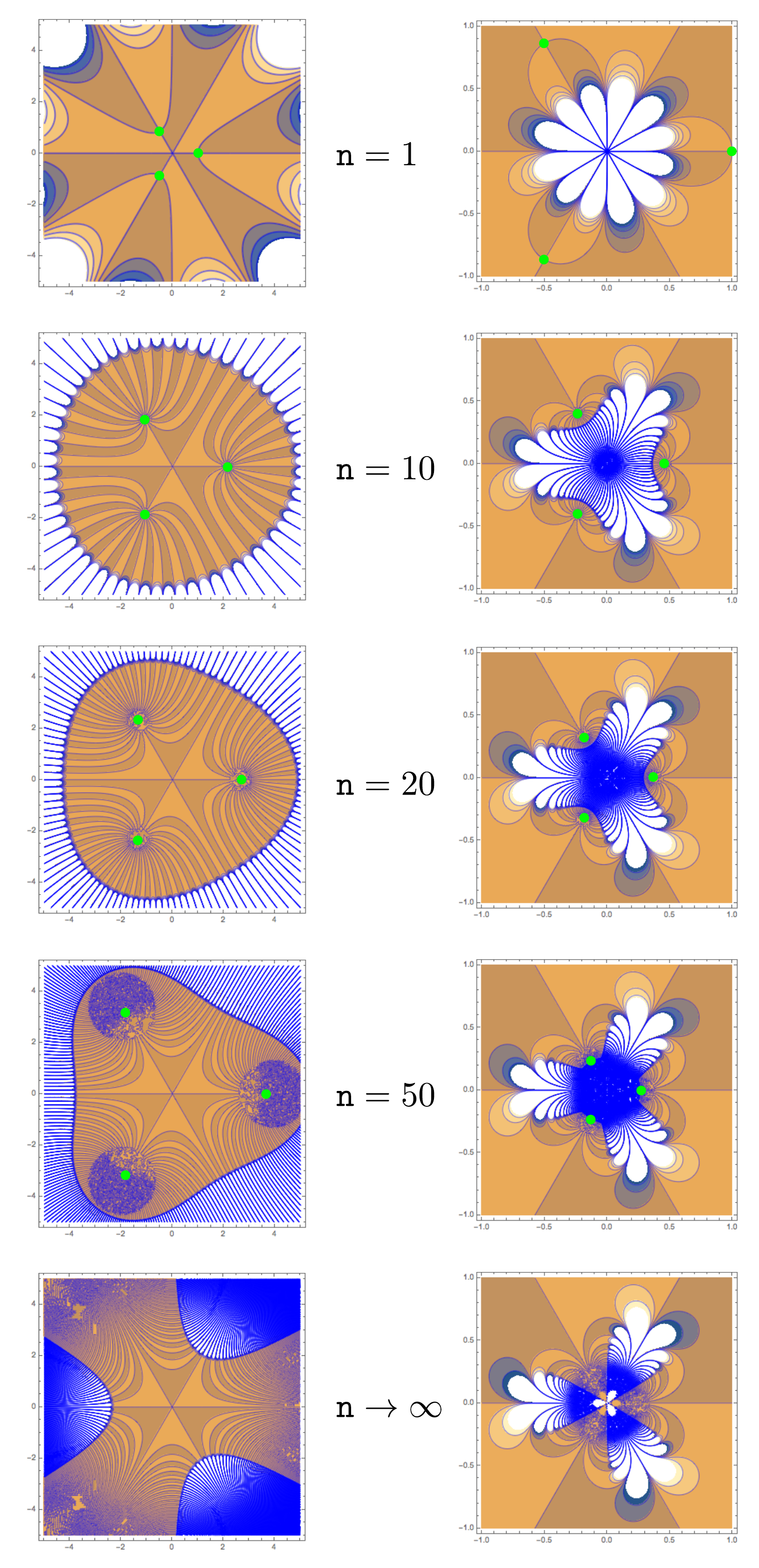}
\caption{Phase portrait of $\Re{X_{\tt n}}$ for ${\tt n}=1, 10, 20, 50$ converging to $\Re{X}$ 
with $X\in\E(2,3)$ as in Example \ref{E23-1-10-20-50}. 
Left hand side portrays a neighborhood of the origin, and the right hand side a neighborhood
of $\infty\in\CW_{z}$.
Note that by approaching $\infty\in\CW_{z}$ along paths that avoid the poles 
$\{ e^{\frac{2 i \pi \sigma}{d}} {\tt n}^{1/d} \}_{\sigma=1}^d$ (green dots), the value of 
$\Psi_{\tt n}(z)$ converges to $\infty\in\CW_t$. 
}
\label{figEjemploE23-1-10-20-50}
\end{center}
\end{figure}

Letting ${\tt n}\to\infty$ and since
$$
\lim_{{\tt n}\to\infty}
\frac{\Gamma ({\tt n}+1)\ {\tt n}^{\frac{r+1}{d}}}{\Gamma \left({\tt n}+\frac{r+1}{d}+1\right)} = 1,
$$
we conclude that the critical values $\widetilde{e}_\sigma ({\tt n})$ converge, along
the asymptotic paths $\alpha_{\sigma}(\tau)=\tau \e^{2 i \pi \sigma / d}$ 
suggested by the sequence of 
critical points $\{\widehat{e}_\sigma({\tt n})\}_{{\tt n}=1}^{\infty}$, 
to the finite asymptotic values $a_{\sigma}$ of 
$\Psi_{X}(z)=\int\limits_{0}^{z} \zeta^r \e^{-\zeta^d} d\zeta$, given by 
\begin{equation}\label{val-asintoticos-aproximacion}
a_{\sigma} = e^{2 i \pi \sigma \frac{(r+1)}{d}} \,
\frac{\Gamma \left(\frac{r+1}{d}+1\right)}{(r+1)} \in\CC_{t},
\ \ \
\text{ for } \sigma=1,\ldots,d.
\end{equation}
Furthermore, traveling along the asymptotic paths 
$\alpha_{\sigma}(\tau)=\tau \e^{2 i \pi (\sigma -d) / d} \e^{i\pi/d}$,
that arrive at $\infty\in\CW_{z}$ with angle $\frac{2\pi}{d}(\sigma-d)+\frac{\pi}{d}$, for
$\sigma=d+1,\ldots,2d$,
we see
that $\Psi_{\tt n}(z)$ converges to $\infty\in\CW_t$, thus there are $d$ 
(classes of) asymptotic paths that give rise to the asymptotic value $\infty\in\CW_t$.

\noindent
In Figure \ref{figEjemploE23-1-10-20-50} we visualize 
(using the techniques presented in \cite{AlvarezMucinoSolorzaYee}), for $r=2$, $d=3$,
the phase portraits of $\Re{X_{\tt n}}$, for ${\tt n}=1,10,20,50$, and $\Re{X}$. 
The poles $\{ e^{\frac{2 i \pi \sigma}{d}} {\tt n}^{1/d} \}_{\sigma=1}^d$ are portrayed as green dots.
Note that at $\infty\in\CW_{z}$ there is a zero of $X_{\tt n}$ of order 
exactly $r+d{\tt n}+2=3{\tt n}+4$.
\end{example}

\section{The geometry of the Riemann surface $\R_{X}$; setup for the proof of Main Theorem}
\label{RX}
This section has two goals: 
understanding the geometry of the Riemann surface $\R_{X}$ for $X\in\E(r,d)$
and
setting up the geometrical/combinatorial elements in order to define
vertices, edges and weights of the $(r,d)$--configuration trees. 

\subsection{Branch points of $\R_{X}$ as vertices}
Since the  branch points of $\R_{X}$ over 
$\infty\in\CW_{t}$ are independent
of $X\in\E(r,d)$, they will not enter the proof of the Main Theorem,
thus the following concept is natural.
\begin{definition}
\label{divisor-reducido}
For $d \geq 1$,
the \emph{reduced divisor of $X \in \E(r,d)$} is 
\begin{multline}\label{eq-divisor-reducido}
X\longmapsto
\underbrace{
\left\{ 
\circled{\, \iota \,} =
(p_{\iota},\widetilde{p}_{\iota},-\nu_{\iota})
\right\}_{\iota =1}^{n}
}_{\hbox{\tiny{pole vertices} } }
\cup
\underbrace{
\left\{
\circled{\scalebox{0.75}{$n \text{+} \sigma$} }=
(\infty_{\sigma},a_{\sigma},-\infty) 
\right\}_{\sigma =1}^{d}
}_{\hbox{\tiny{essential vertices} }}
\\
=
\Big\{ 
\circled{ \msigma } 
= 
(z_{\msigma},t_{\msigma}, - \nu_{\msigma}) 
\Big\}_{\msigma =1}^{n+d}
.
\end{multline}
\end{definition}

\begin{remark}
The corresponding critical points of $\Psi_X$ (and 
transcendental singularities of $\Psi^{-1}_X$) are
\begin{equation}\label{puntos-criticos-sing-trascen}
z_{\msigma}
\in \{p_1, \ldots , p_\iota, \ldots  p_n, 
\infty_{1},\ldots, \infty_{\sigma} 
\ldots 
\infty_{d}\}\subset\overline{\CC}_{z}
\end{equation}

\noindent 
once again 
$n \leq r$, with equality if and only if all the poles
of $X$ are simple.
Recalling 
Remark \ref{correspexponentialtract}.2,
their $n+m$ critical and finite asymptotic values 
of $\Psi_X$ 
are to be denoted by
\begin{equation}\label{valores-criticos-asintoticos}
t_{\msigma}
\in 
\{  \widetilde{p}_{1},\ldots ,\widetilde{p}_{\iota},
\ldots ,\widetilde{p}_{n},
a_{1},\ldots, a_{j(\sigma)}, \ldots, a_{m} \}
\subset\CC_{t},
\end{equation}

\noindent 
where
$m \leq d$, with equality if and only if
all the finite asymptotic values of $\Psi_X$ are of multiplicity one.

\noindent 
The main features of $\circled{\msigma}$ are
summarized in Table \ref{tablaTriada}.
\end{remark}

\begin{table}[htp]
\caption{Branch points of $\R_{X}$.}
\begin{center}
\begin{tabular}{|c|c|l|}
\hline
&& \vspace{-.3cm} 
\\
Branch point
 & Vertex & Notation in $(r,d)$--configuration trees; 
\\
$(z_{\msigma},t_{\msigma}) \in \R_{X}$ & 
$\circled{\msigma}=
( z_{\msigma},t_{\msigma}, -\nu_{\msigma})$ & 
meaning \\[4pt]
\hline
\hline
&& \vspace{-.3cm}\\
& & Pole vertices if $r\geq 1$;
\\
$(p_{\iota},\widetilde{p}_{\iota})$ & 
$\circled{\,\iota\,}=
( p_{\iota},\widetilde{p}_{\iota}, 
-\nu_{\iota})$  
& $\widetilde{p}_{\iota} = \Psi_{X}(p_{\iota})$ is a critical value 
of $\Psi_{X}$,
\\
& & $p_{\iota}$ is a pole of $X$ having order 
$-\nu_{\iota} \leq -1$,
\\
\eqref{poleenumeration} & & hence $(p_{\iota},\widetilde{p}_{\iota})$ is a branch point with \\
& & ramification index $\nu_\iota+1\geq2$.\\[2pt]
\hline
&& \vspace{-.3cm}
\\
$(\infty,\infty)$ &$(\infty,\infty, s)$ & Zero vertex when $d=0$, $r \geq 1$ in which
\\
&& case $s=2+r$;
\\
&& $\infty\in\CW_z$ is a zero of $X$ having order $s$.
\\
&& \vspace{-.3cm}
\\
\hline
&& \vspace{-.3cm}\\
& & Essential vertices if $d \geq 1$;\\
$(\infty_{\sigma},a_{\sigma})$ & 
$\circled{\scalebox{0.75}{$n \text{$+$} \sigma$} }=
(\infty_{\sigma},a_{\sigma},-\infty)$ & 
$\infty\in\CW_{z}$ 
is an essential singularity of $X$,\\
& & $a_{\sigma} \in \CC_{t}$ being a finite asymptotic 
value
\\
\eqref{essenvert} & & of $\Psi_{X}$, with exponential tract $U_\sigma (\rho)$, so\\
& & $(\infty_{\sigma},a_\sigma,-\infty)$ is a logarithmic branch\\
& &  point in (the closure of) $\R_{X}\subset\overline{\CC}_{z}\times\CW_{t}$. \\[2pt]
\hline
\end{tabular}
\end{center}
\label{tablaTriada}
\end{table}

\subsection{The families of surfaces $\R_{X}$ 
having only one vertex}
We describe the families of vector fields avoided in the 
Main Theorem, {\it i.e.} those
having exactly one branch point over $\CC_t$.

\begin{lemma}\label{oneasymptoticvalue}
Let $X\in\E(r,d)$. 
The associated $\Psi_{X}$ has exactly one 
finite critical or asymptotic value $t_{1}\in\CC_{t}$ 
if and only if 
\begin{equation*}
(r,d)=\begin{cases}
(r\geq1,0) & X \text{ has a unique pole of order }-r, \\
& \text{in which case } t_{1} \text{ is the finite critical value,} \\
(0,1) & X \text{ has an isolated essential singularity at }\infty\in\CW_{z}, \\
& \text{in which case }t_{1}\text{ is the finite asymptotic value.} 
\end{cases}
\end{equation*}
\end{lemma}
\begin{proof}
$(\Leftarrow)$ 
The case 
$X$ has a unique pole of order $-r$, 
is obvious: the required distinguished parameter 
is given by Equation \eqref{Psi-1-polo} 
and the unique finite critical value is given by Equation \eqref{val-crit-1-polo}.

\noindent 
When $(r,d)=(0,1)$, the distinguished parameter is given by Equation \eqref{Psi-1-va} and 
the unique finite asymptotic value is given by Equation \eqref{va-exponencial}.

\smallskip
\noindent
$(\Rightarrow)$ 
By Lemma \ref{pareja-finita-infinita}, 
$\Psi_{X}^{-1}$ has $d$ logarithmic branch points over $d$ finite asymptotic values 
of $\Psi_X$, 
$d$ logarithmic branch points over 
$\infty\in\CW_{t}$ and $r$ critical values 
(counted with multiplicity).

\noindent 
Let $\{(z_{\msigma}, t _1 )\}
\subset
\pi_{X,2}^{-1}( t_1 )
\subset\R_{X}$ 
be all the branch points over  $t_1$.
The set $\{(z_{\msigma}, t _1 )\}$ consists of 

\noindent 
$\bigcdot$ 
exactly $d$ logarithmic branch points over $t_1$
and 

\noindent 
$\bigcdot$
$n$ ($\leq r$) finitely 
ramified branch points $(p_{\iota}, t_1)$ of ramification indices $\nu_{\iota}+1$ with 
$r=\sum_{\iota=1}^{n\leq r}\nu_{\iota}$.

\noindent 
Note that there are $d+n \geq 1$ 
\emph{distinct} branch points of $\R_X$.

Moreover,  $\R_X$ is a connected Riemann surface
(it is the graph of $\Psi_X$). 
The restriction of the second projection 
over the punctured plane 

\centerline{$
\pi_{X, 2}: \R_X \backslash 
\{ \pi_{X, 2}^{-1}(t_1 )\}
\longrightarrow 
\CC \backslash \{ t_1 \}
$}

\noindent 
is a holomorphic cover without ramification. 
The subgroups $G$ of $\pi_1(\CC \backslash \{ t_1\})
\cong \ZZ$ classify topologically these covers.

\noindent
For $G= \ZZ_{r+1}$ with $r \geq 1$, 
the cover is finite cyclic and  
we can recognize that 
$\Psi_X (z)$ 
is as in \eqref{Psi-1-polo},
using Riemann's removable singularity Theorem. 
This implies $X \in \E(r, 0)$ has a unique pole
as in the first assertion.  

\noindent 
While for $G= \ZZ$, 
using Lemma \ref{pareja-finita-infinita} assertion 3
we can recognize 
$\Psi_X(z)$ as in \eqref{Psi-1-va}, 
which in turn provides the corresponding $X \in \E(0,1)$. 

\noindent
The case $G=id$, gives $\Psi_X(z)= \frac{1}{\lambda} (z-p_1)$
and the constant vector field 
$X(z)= \lambda \del{}{z}$, 
that does not belong to $\E(r,d)$.
\end{proof}

\subsection{Diagonals of $\R_X$ as edges}
Because of Lemma \ref{oneasymptoticvalue} from now on, unless explicitly noted,
we assume that there are two or more finite critical or asymptotic values $t_{\msigma}\in\CC_t$
of $\Psi_X$.

In order to completely describe $\R_{X}$, we also require information of 
the relative position of the branch points of the surface,
recall Diagram \ref{diagramaRX}  
and Definition \ref{eldivisor};
let 
$\{ 
\circled{\msigma}= (z_\msigma, t_\msigma, -\nu_{\msigma})
\}_{\msigma=1}^{n+d}$
as in Equation \eqref{eq-divisor-reducido}.

Consider the oriented straight line segment $\overline{t_{\msigma}t_{\mrho}}\subset\CC_{t}$. 
The inverse image

\centerline{
$\pi_{X,2}^{-1}\big(\overline{t_\msigma t_\mrho}\big)
= \{ \Delta_{\vartheta \msigma \mrho}\} \subset\R_{X}$ }

\noindent
is a set consisting of 
a finite (when $m=0$, equivalently $d=0$) or 
an infinite (when $m\geq1$) 
number of copies of $\overline{t_{\msigma}t_{\mrho}}$. 
A priori, for each segment 
$\Delta_{\vartheta \msigma \mrho}$,
the projection 
$\pi_{X,1}(\Delta_{\vartheta \msigma \mrho})\subset \overline{\CC}_{z}$ can have
regular points at its end points.

\begin{definition}
\label{diagonal}
1. A segment 
\, $\Delta_{\vartheta \msigma \mrho}$ 
\emph{ is a diagonal of $\R_{X}$} 
\,
when the interior 
of $\pi_{X,1}(\Delta_{\vartheta \msigma \mrho})$ is in $\CC_{z}$ and 
the endpoints of
$\pi_{X,1}(\Delta_{\vartheta \msigma \mrho})$
are $z_{\msigma},z_{\mrho}
\in 
\{p_1, \ldots , p_n,  
\infty_{1}, \ldots,$ 
$\infty_{d}\} \subset \overline{\CC}_{z}$;
critical points of $\Psi_X$ or transcendental 
singularities of $\Psi_X^{-1}$. 
\\
2. A given diagonal 

\centerline{
$\Delta_{\vartheta \msigma \mrho}$, 
\emph{starts at 
$\circled{\msigma}=(z_{\msigma}, t_{\msigma}, -\nu_{\msigma})$ 
and ends at 
$\circled{\mrho}=(z_{\mrho}, t_{\mrho}, -\nu_{\mrho})$}.}

\noindent 
We shall say that the branch points of $\pi_{X, 2}$,
$\circled{\msigma}$ and 
$\circled{\mrho}$,
\emph{share the same sheet}
$\CC_{\Delta_{\vartheta \msigma \mrho}} 
\backslash \{\text{suitable branch cuts}\}$ in 
$\R_{X}$.
\end{definition}

\begin{remark}\label{casosVAf}
1. By notational simplicity, 
if we drop the index $\vartheta$ from 
$\Delta_{\vartheta \msigma \mrho}$ we are specifying the
unique diagonal $\Delta_{\msigma \mrho}$.
The following identification will be useful 
in the proof of the Main Theorem
\begin{equation}\label{diagonal-arista}
\underbrace{\Delta_{\msigma \mrho} \subset \R_X
\text{ with endpoints }
\circled{\msigma},
\,
\circled{\mrho}}_{
\text{\tiny{diagonal}} }
\ \  \longleftrightarrow 
\underbrace{
\Delta_{\msigma \mrho}}_{\substack{\text{\tiny oriented}\\ \text{\tiny edge}}
}\hskip-5pt.
\end{equation}

\noindent
2. Note that since $\Psi_{X}$ is a uni--valued function, 
there can not be homoclinic trajectories of $\Re{X}$ from a pole $p$ to itself; 
{\it i.e.} there does not exist a diagonal 
$\Delta_{\iota \kappa}$ whose endpoints are
the finitely ramified branch points
$(p_{\iota},\widetilde{p}_{\iota},-\nu_{\iota})$ and $(p_{\kappa},\widetilde{p}_{\kappa},-\nu_{\kappa})$,
with $\pi_{X,1}(p_{\iota})=\pi_{X,1}(p_{\kappa})=p$.

\noindent 
3. The
diagonals $\Delta_{\msigma \mrho}$ 
can have endpoints 
as follows:

\begin{enumerate}[label=\arabic*),leftmargin=*]
\item
$\Delta_{\iota \kappa}$ has as endpoints 
two finitely ramified branch points (corresponding to pole vertices):
$(p_{\iota},\widetilde{p}_{\iota},-\nu_{\iota})$ and $(p_{\kappa},\widetilde{p}_{\kappa},-\nu_{\kappa})$,
$\iota\neq \kappa$. 
For an example see Figure \ref{figejemplo-dos-polos} in \S\ref{concreteexamples}.

\item
$\Delta_{\iota \sigma}$ has as endpoints a 
finitely ramified branch point and a logarithmic 
branch point 
(corresponding to a pole and an essential vertex) or viceversa:
$(p_{\iota},\widetilde{p}_{\iota},-\nu_{\iota})$ and 
$(\infty_{\sigma},a_{\sigma},-\infty)$. 
For an example see Figure \ref{figejemplo-alma-no-trivial} in \S\ref{concreteexamples}.

\item
$\Delta_{\sigma \rho}$ has as endpoints two 
logarithmic branch points (corresponding to essential vertices) 
with finite asymptotic values $a_{\sigma}$ and $a_{\rho}$ with exponential tracts 
$\alpha_\sigma$ and $\alpha_\rho$: 
$(\infty_{\sigma},a_{\sigma},-\infty)$ to $(\infty_{\rho},a_{\rho},-\infty)$, 
$\sigma\neq \rho$, where the subscripts 
are as in \eqref{correspjalphaIndices}.
For an example see Figure \ref{fig-Exp3} in \S\ref{concreteexamples}.
\end{enumerate}
\end{remark}

\noindent
Following the notation in \cite{Frias-Mucino},
for $\Delta_{\msigma \mrho}$ a diagonal
the associated \emph{semi--residue} is
\begin{equation}\label{diagonalsemiresidue}
S(\omega_{X}, z_{\msigma}, z_{\mrho})
\doteq
\int_{z_{\msigma}}^{z_{\mrho} } 
P(\zeta)\e^{-E(\zeta)} d\zeta
= 
t_{\mrho}-t_{\msigma}. 
\end{equation}

\begin{lemma}[Existence of diagonals in $\R_X$]
\label{pareja-diagonales}
Suppose that there are at least two branch points
$
\big\{ 
\circled{ \msigma } 
= 
(z_{\msigma},t_{\msigma}, - \nu_{\msigma}) 
\big\}_{\msigma =1}^{n+d}
\subset\R_{X}$.
Then every branch point 
$\circled{\msigma}$
is an endpoint for at 
least two diagonals,
an incoming diagonal and an outgoing diagonal.
\end{lemma}

\begin{proof}
Consider any branch point 
$\circled{\msigma}=(z_{\msigma},t_{\msigma},-\nu_{\msigma})
\in\pi_{X,2}^{-1}(t_{\msigma})$, with
$t_{\msigma}$ 
as in Equation \eqref{valores-criticos-asintoticos}. 
Suppose that there is no diagonal 
$\Delta_{\msigma \mrho}$ with endpoint 
$\circled{\msigma}$. 
This implies that $\circled{\msigma}$ does not share a sheet, 
$\CC_{t}\backslash\{$suitable branch cuts$\}$,
with any other branch point 
$\circled{\mrho}=(z_{\mrho},t_{\mrho},-\nu_{\mrho})\in\pi_{X,2}^{-1}(t_{\mrho})$, 
for some finite asymptotic or critical value 
$t_{\mrho}\neq t_{\msigma}$
(note that the existence of $t_{\mrho}$ is guaranteed by Lemma \ref{oneasymptoticvalue}). 
In other words the only sheets, $\CC_{t}\backslash\{\text{suitable branch cuts}\}$, of $\R_{X}$ 
containing the branch point $\circled{\msigma}$ are of the form 
$\CC_{t}\backslash\{L_{\msigma}\}$, for $L_{\msigma}=[t_{\msigma},\infty)$.
Hence by the same arguments as in Lemma \ref{oneasymptoticvalue}, 
the Riemann surface 
$\R_{X}$ will have at least 2 connected 
components (one containing $\circled{\msigma}$ and 
the other containing $\circled{\mrho}$), 
leading to a contradiction.
\end{proof}

\subsection{Geometrical building blocks and the weights of edges} 
Our elementary building blocks are pairs, 
Riemann surface and singular complex analytic vector fields, as follows.

\begin{definition}\label{piezasFatou}
The pair 
$\big(\bar{\HH}^2_\pm,\del{}{t}\big)$ 
will be called a \emph{(closed) half plane};
the closure of $\bar{\HH}^2_\pm$ is considered in $\CC$, hence its 
boundary is $\RR$.

\noindent
Likewise, 
$\big(\{0\leq\Im{t}\leq h\}, \del{}{t}\big)$ will be called a \emph{(closed) finite height horizontal strip}. 
\end{definition}

In the flat surface category, 
surgery tools are widely used, {\it v.g.} 
\cite{Strebel} p.~~56 ``welding of surfaces'', 
\cite{MR}, 
or \cite{Thurston} \S 3.2.--3.3 for general 
discussion.

\begin{corollary}[Isometric glueing]
\label{pegado-isometrico}  
Let $(M^0, g_X)$, $(N^0, g_Y)$ be two flat surfaces arising 
from two singular complex analytic vector fields $X$ and $Y$. 
Assume that both spaces $M^0$, $N^0$
have as geodesic boundary components of the same length: the trajectories
$\sigma_1 (\tau)$, $\sigma_2(\tau)$ of $\Re{X}$ and $\Re{Y}$, 
$\tau \in I \subset \RR$. 
The isometric glueing of them 
along these geodesic boundary,  
is well defined, and provides a new flat surface 
on $M^0 \cup N^0$ arising from a new complex analytic vector field.
\hfill $\Box$
\end{corollary}

\begin{lemma} 
\label{descomposicion-planos-bandas}
Let $X \in \E(r,d)$.

\begin{enumerate}[leftmargin=*,label=\arabic*)]
\item 
The Riemann surface $\R_{X}$ can be constructed
by isometric glueing of

\noindent $\bigcdot$
half planes $(\bar{\HH}^2_{\pm}, \del{}{t})$ and 

\noindent $\bigcdot$
finite height horizontal strips 
$\big(\{ 0 \leq \Im{t} \leq h\}, \del{}{t} \big)$, 
where $h \geq 0 $.

\item
There exists a one--to--one correspondence 
$$
\begin{array}{ccc}
\left\{
\begin{array}{c}
\hbox{finite height horizontal strips}
\\ 
\big(\{ 0 \leq \Im{z} \leq h\}, \del{}{t} \big)
\end{array}
\right\}
& 
\longleftrightarrow 
& 
\left\{
\begin{array}{c}
\hbox{diagonals } \Delta_{\msigma \mrho}
\hbox{ with}
\\ 
\big\vert 
\Im{  \int_{z_{\msigma}}^{z_\mrho} \omega_X} 
\big\vert
= h \geq 0 
\end{array}
\right\}
\end{array} ,
$$
\noindent 
here
$\circled{\msigma} =
(z_{\msigma},t_{\msigma},-\nu_{\msigma})$ 
and 
$\circled{\mrho}=
(z_{\mrho},t_{\mrho},-\nu_{\mrho})$.

\item
The case when $\Im{  \int_{z_{\msigma}}^{z_\mrho} \omega_X}  =0$ 
corresponds to the finite height horizontal strip
degenerating into a segment of trajectory of $\Re{X}$ between the branch points 
$\circled{\msigma}$ and $\circled{\mrho}$,
{\it i.e.} a horizontal diagonal $\Delta_{\msigma \mrho}$.
\end{enumerate}
\end{lemma}
\begin{proof}
Assertion (1) follows by using Lemma \ref{pareja-finita-infinita},
we can consider the pullback under $\pi_{X, 2}$
of the vector field $\del{}{t}$ to 
$\R_X$. Thus, the phase portrait of $\Re{X}$
determines the decomposition.

\noindent
Assertion (2) and (3) follow directly from Definition \ref{diagonal}.
\end{proof}

Definitions \ref{horizontal-branch-cut}, 
\ref{rem-diagonal} below
apply for singular flat Riemann
surfaces (not necessarily of type $\R_X$). 

\begin{definition}\label{horizontal-branch-cut}
Let $\{ t_{\tt k} \}_{{\tt k}=1}^{\tt r} \subset \CC_t$
be a finite set of distinct points.
A {\it sheet} is a copy of $\CC_{t}$ 
with ${\tt r}\geq 1$ branch cuts $L_{\tt k}$; {\it i.e.}
$\CC_t$ is cut along horizontal right segments 
$L_{{\tt k}}= [t_{\tt k}, \infty)$, 
remaining connected
\begin{equation}\label{sheetminuscuts}
\CC_{t}\backslash \{ L_{\tt k}\}_{{\tt k} = 1}^{\tt r} 
\ \doteq \ 
\left[ \CC_{t}\backslash \big( \cup_{{\tt k}=1}^{{\tt r}} [t_{\tt k},\infty) \big) \right]
\cup_{{\tt k}=1}^{{\tt r}}\{[t_{\tt k},\infty)_{+},[t_{\tt k},\infty)_{-}\},
\end{equation}
having $2{\tt r}$ horizontal boundaries
where the subscripts $\pm$ refer 
to the obvious 
upper or lower boundary using $\Im{t}$.
\end{definition}

See Figure \ref{figPiezasBasicas} for examples of sheets.

\begin{remark}
1. The sheets $\CC_{t}\backslash \{ L_{\tt k}\}_{{\tt k} = 1}^{\tt r}$ include the upper and lower boundaries $[t_{\tt k},\infty)_{\pm}$.

\noindent
2. Note that cuts (and the corresponding boundaries) need not be to the right, 
they could be more general simple paths, 
however for notational simplicity and ease of the proofs we shall only use right cuts 
$[t_{\tt k}, \infty)_{\pm}$ as in \eqref{sheetminuscuts}. 
\end{remark}

\begin{definition}\label{rem-diagonal}
A \emph{diagonal of the sheet} $\CC_{t}\backslash\{L_{\tt k} \}_{{\tt k} = 1}^{\tt r}$ is 
an oriented straight line segment
\begin{equation}
\label{diagonal-first-version}
\Delta_{\msigma \mrho}=\overline{ t_{\msigma} t_{\mrho} }
\subset \CC_{t}\backslash\{L_{\tt k} \}_{{\tt k} = 1}^{\tt r}  ,
\end{equation}

\noindent  
starting at $t_{\msigma}$ and ending at $t_{\mrho}$, here
$\mrho,\msigma$ are as in Equation 
\eqref{valores-criticos-asintoticos}. 
\end{definition}

The abuse of notation in Equations  
\eqref{diagonal-arista}, 
\eqref{diagonalsemiresidue},
and 
\eqref{diagonal-first-version}, 
will be cleared in the proof of the main Theorem.
See Figures \ref{figPiezasBasicas}.b, \ref{figPiezasBasicas}.c 
and Figure
\ref{figejemplo-dos-polos} for examples of diagonals
of sheets.
\begin{figure}[htbp]
\begin{center}
\includegraphics[width=\textwidth]{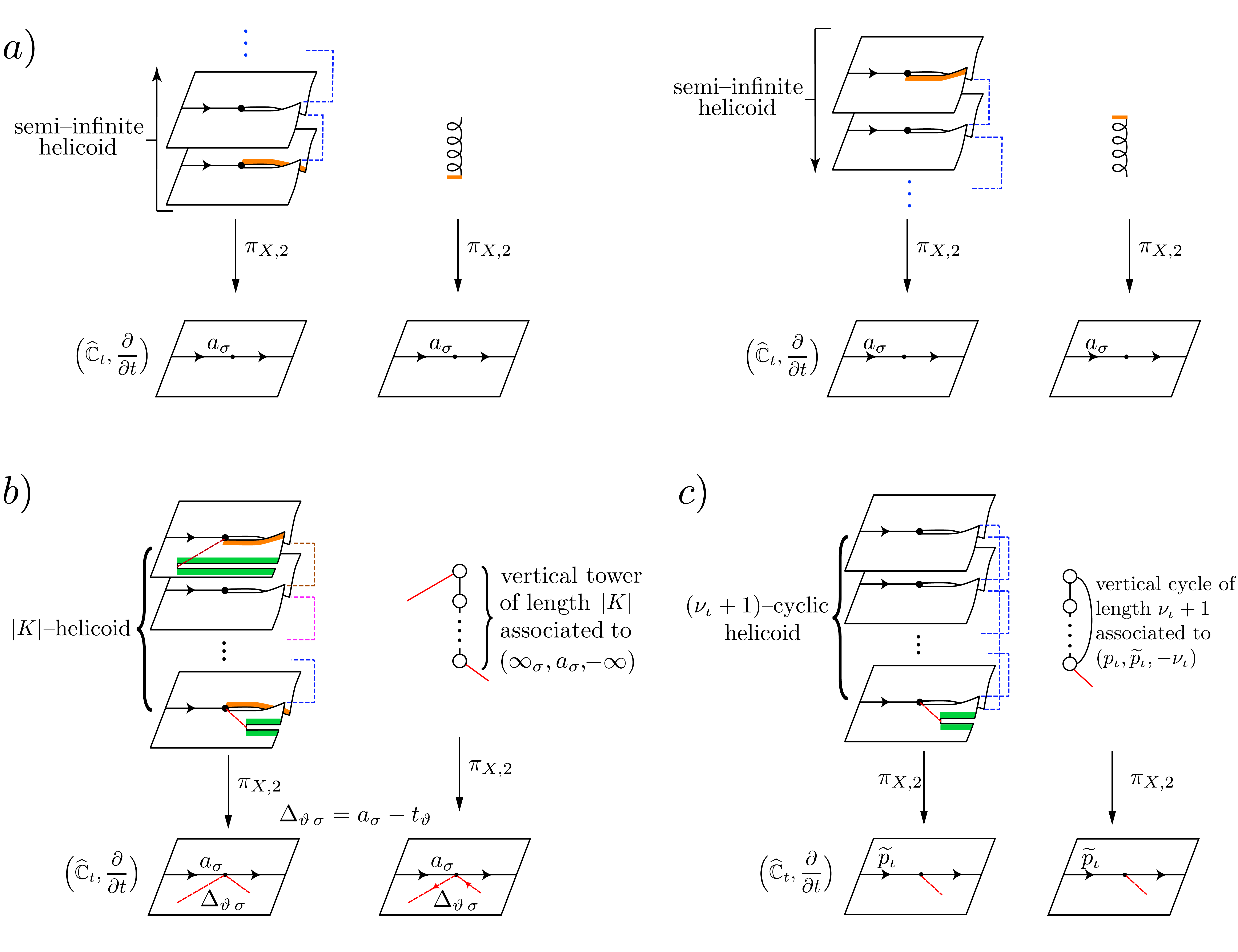}
\caption{In this figure we have the representation of the non elementary building blocks
as sheets with branch cuts glued appropriately in $\R_{X}$, and schematically as a graph.
(a) On the top row we have the semi--infinite (up and down) helicoids.
On the bottom row we have (b) on the left a $\abs{K}$--helicoid 
and (c) on the right an $(\nu_\iota +1)$--cyclic helicoid, for $(\nu_\iota +1)\geq 2$.
In the
$\abs{K}$--helicoids and $(\nu_\iota+1)$--cyclic
helicoids the red segments are diagonals.
}
\label{figPiezasBasicas}
\end{center}
\end{figure}

\smallskip

We introduce,
four non elementary building blocks, 
pictured in Figure \ref{figPiezasBasicas}, 
their Definitions \ref{helicoide-semiinfinito}, 
\ref{helicoide-finito}, \ref{helicoide-ciclico}
include a geometrical and a combinatorial framework.

\smallskip

Recalling Example \ref{ejemplolog} we propose the following.

\begin{definition}\label{helicoide-semiinfinito}
A \emph{semi--infinite helicoid} is the Riemann surface of 

\centerline{$
\int\limits_{z_0}^z \e^{-\mu(\zeta + c_1)} d\zeta 
- a_\sigma = 
-\dfrac{1}{\mu}\e^{-\mu (z+c_1)}
: \overline{\HH}^2_\pm 
\subset \CC_z \longrightarrow \CC_t
$}

\noindent
having a horizontal boundary $[a_\sigma,\infty)_\pm$,  
geometrically it is
an infinite succession of half--planes 
\\ 
\centerline{
$\Big( \big(\bar{\HH}^2_{\pm} \cup \bar{\HH}^2_{\mp} 
\cup \ldots \big),\del{}{t} \Big)$.}

\noindent 
glued in the usual way, Corollary \ref{pegado-isometrico}.
See Figure \ref{figPiezasBasicas}.a.
\end{definition}

\begin{remark}
\label{coils}
In the combinatorial framework, 
for $X \in \E(r,d)$, each essential vertex 
$\circled{\scalebox{0.75}{$n \text{$+$} \sigma$} }
= (\infty_\sigma, a_\sigma, -\infty)$
has associated two semi--infinite helicoids,
up and down, respectively.

\noindent 
Pictorically, we will represent each one by a small coil
in our figures. 
However, the semi--infinite helicoids
and their coils do not appear in the 
formal $(r,d)$--configuration trees.
\end{remark}

Geometric characteristics:
a semi--infinite helicoid

\noindent
$\bigcdot$ lies over the 
finite asymptotic value 
$a_\sigma \in \CC_t$,

\noindent 
$\bigcdot$
has an infinite number of sheets,

\noindent 
$\bigcdot$
its horizontal boundary is coloured in orange, 

\noindent
$\bigcdot$ 
is up or down 
(resp. the domain of its integral is 
$\bar{\HH }_{+}^2$ or $\bar{\HH }_{-}^2$).


\smallskip
Recalling the Example \ref{ejemplolog}, 
we have the following.

\begin{definition}\label{helicoide-finito}
For $K \in \ZZ$,  
a \emph{$\abs{K}$--helicoid} 
is the Riemann surface of 

\centerline{$
\int\limits_{z_0}^z \e^{-\mu(\zeta + c_1)} d\zeta 
- a_\sigma = 
-\dfrac{1}{\mu}\e^{-\mu (z+c_1)}:
\mathcal{D}
\subset \CC_z \longrightarrow \CC_t
$}

\noindent 
where 
$$
\mathcal{D} =\left\{ 
\begin{array}{ll}
\{0 \leq 
\Im{\mu z} 
\leq 2 \pi (K +1)\}, & K  \geq 0
\\
& \vspace{-.3cm}
\\
\{ 2 \pi K \leq 
\Im{\mu z} 
\leq 2\pi  \}, &
K < 0,
\end{array}
\right.
$$

\noindent
having a finite number $\ell \geq 2$ of branch cuts
$\{ L_{\tt k}\}_{{\tt k}=2}^{\ell}$,
geometrically it as 
a succession of $2(\abs{K} +1)$ 
half--planes 
\\ 
\centerline{
$\Big( \big(\bar{\HH}^2_{\pm} 
\cup\bar{\HH}^2_{\mp}\cup  \ldots\cup\, 
\bar{\HH}^2_{\mp} \big),\del{}{t}\Big) $ }

\noindent 
glued in the usual way, with $2\ell$ horizontal boundaries. 

\noindent
In the combinatorial framework, for $X \in \E(r,d)$,
the following objects are equivalent:

\noindent 
a $\abs{K}$--helicoid, 

\noindent 
a 
\emph{vertical tower of length $\abs{K}$}.

\noindent 
See Figure \ref{figPiezasBasicas}.b.
\end{definition}

\begin{remark}
\label{}
Each essential vertex 
$\circled{\scalebox{0.75}{$n \text{$+$} \sigma$} }
= (\infty_\sigma, a_\sigma, -\infty)$
determines at least one $\abs{K(\sigma)}$--helicoid.
\end{remark}

\noindent
Geometric characteristics: a $\abs{K}$--helicoid

\noindent
$\bigcdot$ 
lies over the finite asymptotic value 
$a_\sigma \in \CC_t$,

\noindent 
$\bigcdot$
has $\vert K \vert +1 $ sheets and there
are three subcases.

\noindent 
If $K > 0$ or $K < 0$, the $\abs{K}$--helicoid:

\noindent 
$\bigcdot$
goes up or down depending on the sign of $K \neq 0$,

\noindent 
$\bigcdot$
has $\ell \geq 2$ diagonals with common
extreme points over $a_\sigma \in \CC_t$,

\noindent 
$\bigcdot$ in particular, 
there are always 2 diagonals present:
one on the image of the strip  

\centerline{
$\{0 \leq \Im{z} \leq 2 \pi  \} $,}

and the other on the image of the strip 

\centerline{
$\{2\pi K \leq \Im{z} \leq 2 \pi (K+1)  \} $}

$K$ sheets above/below,

\noindent 
$\bigcdot$
has $2\ell \geq 4$ horizontal boundaries 
(arising from diagonals, colored green) and 

2 horizontal boundaries (orange),

\noindent
If $K=0$, the $\abs{K}$--helicoid:

\noindent 
$\bigcdot$
has $\ell \geq 0$ diagonals having common
extreme points over $a_\sigma \in \CC_t$,

\noindent 
$\bigcdot$
has $2\ell \geq 0$ horizontal boundaries from diagonals (green) and 

2 horizontal boundaries (orange).

\smallskip
Recalling the case of a pole, Example \ref{campo-racional-un-polo}, 
we have the following.

\begin{definition}\label{helicoide-ciclico}
A \emph{$(\nu_\iota+1)$--cyclic helicoid} 
is the Riemann surface of 

\centerline{
$\frac{1}{\lambda} \int\limits_{z_0}^z  (\zeta-p_\iota)^{\nu_\iota}\,d\zeta
+ \widetilde{p}_\iota =
\frac{1}{\lambda} \dfrac{(z-p_\iota)^{\nu_\iota+1}}{\nu_\iota+1} 
: \CC_z \longrightarrow \CC_t$}

\noindent
having a finite number $\{ L_{\tt k}\}_{{\tt k}=0}^{\ell} 
$, $\ell \in \NN \cup \{0 \}$,  of branch cuts,
geometrically it is 
a finite succession of $2\nu_\iota +2$ half--planes 
\\ \centerline{
$\Big( \big(\bar{\HH}^2_{\pm} \cup\bar{\HH}^2_{\mp}\cup  \ldots\cup\, \bar{\HH}^2_{\mp} \big),\del{}{t}\Big) $ }
glued in the usual way, with $2\ell \geq 0 $ horizontal boundaries.

\noindent 
In the combinatorial framework for $X \in \E(r,d)$, 
the following three objects are equivalent:  

\noindent 
a $(\nu_\iota+1)$--cyclic helicoid, 

\noindent 
a
\emph{vertical cycle of length $\nu_\iota+1$},

\noindent 
a pole vertex
$\circled{\, \iota \, }=
(p_\iota, \widetilde{p}_\iota, - \nu_\iota)$. 

\noindent 
See Figure \ref{figPiezasBasicas}.c. 
\end{definition}

\noindent 
Geometric characteristics: a $(\nu_\iota+1)$--cyclic helicoid

\noindent
$\bigcdot$ lies over the
finite critical value $\widetilde{p}_\iota \in \CC_t$,
 
\noindent 
$\bigcdot$
has $\nu_\iota+1$ sheets, 

\noindent
$\bigcdot$
has 
$\ell \geq 0$ diagonals with common
extreme points over $\widetilde{p}_\iota$.

\begin{remark}[Cut and paste of the different 
geometric pieces for $X \in \E(r,d)$]
The paste of a
semi--infinite helicoid and a cyclic helicoid
is forbidden; 
and the paste 
of two semi--infinite helicoids 
will appear essentially only for $X\in \E(0,1)$, 
as in Example 
\ref{ejemplolog}.

\noindent
Semi--infinite helicoids and 
$\abs{K}$--helicoids 
are glued along their orange horizontal boundaries,
see Figure \ref{figPiezasBasicas}.
\end{remark}

\begin{remark}
[The weights of the edges]
\label{pesos-como-info-geometrica}
1. In order to construct the surface $\R_X$
by glueing the geometric pieces described in 
Lemma \ref{descomposicion-planos-bandas},
we shall need to specify not only the branch points
$
\big\{ 
\circled{ \msigma } 
= 
(z_{\msigma},t_{\msigma}, - \nu_{\msigma}) 
\big\}_{\msigma =1}^{n+d}
$
and the corresponding diagonals
$\{ \Delta_{\msigma \mrho} \}$ of the sheets,
but also how many sheets 
each geometric piece has.

\noindent
2. In what follows, the term ``vertical'' refers to the 
$z$--direction in 
$\overline{\CC}_{z}\times\CC_{t}$.
Let 
$(z_\msigma, t_\msigma, - \nu_\msigma)
\in\R_{X}$ be a fixed branch point,
we consider the usual lifting 

$$\beta(\theta)=
\pi_{X,2}^{-1}\big( t_{\msigma}+\rho\,\e^{i2\pi\theta}) \subset \R_{X} 
\ \hbox{ where } \
\theta\in[\theta_{min},\theta_{max}]\subset\RR,
$$ 

\noindent 
for appropriate $\theta_{min}$, $\theta_{max}$ 
and small enough $\rho>0$.

\noindent
3. 
As is usual, three cases appear:

\noindent 
i) Going around a branch point 
counterclockwise corresponds to going 
\emph{up} on the ramified surface 
$\R_X$ and hence 
the number that 
separates the sheets is positive. 

\noindent 
ii)
Similarly going around the branch point clockwise 
corresponds to going \emph{down}. 

\noindent 
iii)
Furthermore going $K$ times around a finitely ramified branch point 
of $\R_X$ 
of ramification index $\nu$ is 
equivalent to going around it $K \pmod{\nu}$ times.

\noindent
4. Whenever there are two (or more) diagonals sharing the same branch point, the 
number of sheets that separate the diagonals in question can be counted on the 
Riemann surface $\R_{X}$ as  
number of planes
traversed by a small enough 
circular path $\beta(\theta)$ around the common branch point.
In this way, by choosing a \emph{local zero level sheet} for each branch point
of $\R_X$, 
we will be able to assign 
a weight to each edge/diagonal attached to the branch point (relative to the local zero level sheet).
\end{remark}

\section{Why is the geometrical description of $\E(r,d)$ difficult?}
\label{dificultades}

Let $X$ be in $\E(r,d)$. 
The graph of $\Psi_{X}$ is the flat 
Riemann surface $\R_{X}$.
In order to get an accurate combinatorial description
of $\R_X$ for the family $\E(r,d)$, 
we shall need to specify two basic sets:

\noindent
i) the vertices or branch points  
$$
\big\{ 
\circled{ \, \iota \,}=
(p_{\iota},\widetilde{p}_{\iota},-\nu_{\iota})
\big\}_{\iota =1}^{n}
\cup
\left\{
\circled{\scalebox{0.75}{$n \text{$+$} \sigma$} }
=
(\infty_{\sigma},a_{\sigma},-\infty) 
\right\}_{\sigma =1}^{d}
=
\Big\{ 
\circled{ \msigma } 
= 
(z_{\msigma},t_{\msigma}, - \nu_{\msigma}) 
\Big\}_{\msigma =1}^{n+d}
$$

\noindent 
of $\R_{X}$, {\it i.e.} the reduced
divisor of $X$, Definition \ref{divisor-reducido},

\noindent
ii) a subset of the diagonals 
$\{ \Delta_{\msigma \mrho} \}$,
Equation \eqref{diagonal-arista},
connecting the branch points 
(in particular it will be useful to know 
which branch points share specific sheets of $\R_{X}$).

The implicit combinatorial obstacles are:

\smallskip

\begin{enumerate}[leftmargin=0.7cm]
\item[D.1] No canonical choice for the initial 
integration point of 
$\Psi_X (z)=\int_{z_0}^z \omega_X$
can be given. 

\item[D.2]
There is no preferred/canonical 
\emph{global zero level}, denoted (GZL),

\centerline{ 
$\CC_{\Delta_{\msigma \mrho}} \backslash \{\text{suitable branch cuts}\}
\ \ \subset \ \     
\R_{X}$, }

\noindent 
that is to be chosen to start the description of 
$\R_{X}$ as a combinatorial object.

\item[D.3]
No canonical order can be given to the 
branch points 
$\big\{ 
\circled{ \msigma } 
= 
(z_{\msigma},t_{\msigma}, - \nu_{\msigma}) 
\big\}_{\msigma =1}^{n+d}$ 
of $\R_{X})$.

\item[D.4] {\it A priori}, the choice of
minimal subset of the diagonals 
required to describe
$\R_X$ is non canonical. 
 
\end{enumerate}

\smallskip

In particular, note that difficulty D.1 will have a repercussion on the labeling of the vertices in Definition 
\ref{d-confTree}, while difficulties D.2 and D.3
are associated to the choice of a suitable root $\raiz{1}$ in the reduced divisor.
Difficulty D.4 will require a certain ordering of the
subset of the diagonals on each sheet of $\R_{X}$
(as will appear in Definition 
\ref{arbol-lineal-enraizado}).

The resolution of these choices/conventions
motivates the notion of equivalence classes $[\Lambda_X]$ as in our Main 
Theorem, see \S\ref{clasesE(d)}.

There are also analytical obstructions/obstacles:

\begin{enumerate}[leftmargin=0.7cm]
\item[D.5] 
Not all the vertices
$\big\{ 
\circled{\, \iota \,}=
(p_{\iota},\widetilde{p}_{\iota},-\nu_{\iota})
\big\}_{\iota =1}^{n}
\cup
\left\{
\circled{\scalebox{0.75}{$n \text{$+$} \sigma$} }
=
(\infty_{\sigma},a_{\sigma},-\infty) 
\right\}_{\sigma =1}^{d}$
are possible 
as branch points for $\R_{X}$, only those that are a solution to the 
system of  transcendental equations
\begin{equation}
\left\{
\begin{aligned}\label{sistema-realizable}
\Psi_{X}(p_\iota) & = \widetilde{p}_\iota   
\\
\Psi_{X}^{(\ell)}(p_\iota) & = 0 \hskip2.0cm 1 \leq \ell \leq \nu_{\iota},\ \ \iota=1,\ldots,n, 
\\
\lim\limits_{\tau\to\infty} \alpha_{\sigma}(\tau) 
& =\infty_{\sigma} 
\\
\ 
\lim\limits_{\tau\to\infty}\Psi_{X}\big(\alpha_{\sigma}(\tau)\big) & =a_{\sigma}  \hskip1.9cm \sigma=1,\ldots,d.
\end{aligned}
\right.
\end{equation}
\end{enumerate}
The last two equalities are analytical expressions of the geometrical 
structure of 
Figure \ref{distribucion-exponentialtracts}.
From a general point of view, we can consider abstract divisors
(finite collection of triplets). 

\begin{definition}
We shall say that an abstract collection of 
$n+d$ vertices 

\centerline{
$\big\{ 
\circled{\, \iota \,}=
(p_{\iota},\widetilde{p}_{\iota},-\nu_{\iota})
\big\}_{\iota =1}^{n}
\cup
\left\{
\circled{\scalebox{0.75}{$n \text{$+$} \sigma$} }
=
(\infty_{\sigma},a_{\sigma},-\infty) 
\right\}_{\sigma =1}^{d}
$
}

\noindent
is \emph{realizable } 
if it is a solution of \eqref{sistema-realizable}, for some $\Psi_{X}$, $X\in\E(r,d)$. 
\end{definition}

\section{Combinatorial objects: $(r,d)$--configuration trees}
\label{CombinatorialTrees}

In order to 
make precise the choices required to resolve the issues  
D.1, D.3, D.4  and D.5 cited in \S \ref{dificultades},
we introduce the following auxiliary concepts:
weighted directed zero--rooted trees and
$(r,d)$--configuration trees, which can be understood
as associated to some $\R_X$
(even though they are abstract graphs,
a priori 
not necessarily associated to the Riemann 
surface $\R_{X}$).

\noindent
To accomplish the above, we shall use some basic notions of graph theory, namely trees, oriented or not, 
with and without roots and weights, see 
\cite{LandoZvonkin} pp.~46, 379 for standard concepts.
It is natural to use 
the branch points of $\R_{X}$, 
described by the reduced divisor, Definition \ref{divisor-reducido}, 

\centerline{
$
\Big\{ 
\circled{ \msigma } 
= 
(z_{\msigma},t_{\msigma}, - \nu_{\msigma}) 
\Big\}_{\msigma =1}^{n+d}
$ 
\quad with \quad
$z_\msigma\in\overline{\CC}_z$, $t_\msigma\in\CC_t$ and $\nu_\msigma\in\NN\cup\{\infty\}$,
}

\noindent 
as vertices of our graphs, 
{\bf with a possible re--labelling when needed}.

To resolve issue D.1, we introduce 
\emph{directed rooted trees} which have one vertex designated as the 
\emph{root} $\raiz{\varrho}$ and its edges are assigned 
an orientation
\begin{equation}\label{arbol-cero}
\Lambda
=
\Big\{ 
\underbrace{\circled{1},\ldots,
\circled{\mrho},\ldots,
\circled{\msigma},\ldots,
\circled{ \mathfrak{m} }}_
{\hbox{ \tiny{vertices}} }
\, ; \,  
\underbrace{\raiz{\varrho}}_{\hbox{\tiny{root}} }
\, ; \, 
\underbrace{ \Delta_{\msigma \mrho},  \ldots
}_{\substack{\tiny \mathfrak{m}-1\\
\hbox{\tiny{oriented edges}} }}
\Big\}  ,
\end{equation}

\noindent
where $\mathfrak{m} \leq  n+d$, 
{\it i.e.} not all vertices in Definition \ref{divisor-reducido} 
are necessarily considered.   
We shall consider directed rooted trees that have an orientation away from the root
and convene that 
$\Delta_{\msigma \mrho}$ denotes the edge starting at 
$\circled{\msigma}$ and ending at $\circled{\mrho}$. 
The \emph{tree--order} is the partial ordering on 
the vertices of  \eqref{arbol-cero} such that,
$ \circled{\iota} <  \circled{\mrho}$ if and only if 
the unique path from the root $\raiz{\varrho}$ to $ \circled{\mrho}$ passes 
through $ \circled{\iota}$. 
The \emph{depth} of the vertex $ \circled{\mrho}$ 
is the length of the path (number of edges) from the root.

As noted in Remark \ref{pesos-como-info-geometrica}, 
the building blocks provide weights for the edges of the 
directed rooted trees.

\begin{definition}
Given a directed rooted tree as in 
\eqref{arbol-cero}, 
by assigning a weight 
$K(\msigma,\mrho)\in\ZZ$ to each edge $\Delta_{\msigma \mrho}$, 
we obtain a \emph{weighted directed rooted tree} 
\begin{equation*}
\Big\{ 
\underbrace{\circled{1},\ldots,\circled{\mathfrak{m}}}_
{\hbox{ \tiny{vertices}} }
\, ; \,  
\underbrace{\raiz{\varrho}}_{\hbox{\tiny{root}} }
\, ; \, 
\underbrace{ (\Delta_{\mrho \msigma},K(\mrho,\msigma)), \ldots
}_{\substack{\tiny \mathfrak{m}-1\\
\hbox{\tiny{weighted edges}} }}
\Big\}  .
\end{equation*}

\noindent
A \emph{zero parent of a vertex $\circled{\msigma}\neq\raiz{\varrho}$} is the 
unique vertex $\circled{\mrho}$ connected to $\circled{\msigma}$ on the 
path to the root,
whose edge $(\Delta_{\mrho\msigma},K(\mrho,\msigma))$ has in addition 
$K(\mrho,\msigma)=0$.
\end{definition}

For some weighted directed rooted trees the root is not a zero parent.
\begin{example}
For the weighted directed rooted tree with all its weights equal to 1, namely
\begin{equation*}
\Big\{ 
\circled{1},\ldots,\circled{\mathfrak{m}}
\, ; \,  
\raiz{\varrho}
\, ; \, 
(\Delta_{1\, 2},1), \ldots, (\Delta_{\mathfrak{m}-1\, \mathfrak{m}},1)
\Big\}  ,
\end{equation*}
none of the vertices is a zero parent.
\end{example}

\begin{definition}
A weighted directed rooted tree
\begin{equation}
\label{cero-arbol-enraizado-con-pesos}
\Lambda_{\circled{\varrho }}
=
\Big\{ 
\underbrace{\circled{1},\ldots,\circled{\mathfrak{m}}}_
{\hbox{ \tiny{vertices}} }
\, ; \,  
\underbrace{\raiz{\varrho}}_{\hbox{\tiny{root}} }
\, ; \, 
\underbrace{ (\Delta_{\mrho \msigma},K(\mrho,\msigma)), \ldots
}_{\substack{\tiny \mathfrak{m}-1\\
\hbox{\tiny{weighted edges}} }}
\Big\}  ,
\end{equation}
whose root $\raiz{\varrho}$ is a zero parent will be called a 
\emph{weighted directed zero--rooted tree}.
\end{definition}
\begin{definition}
A \emph{zero child of a vertex $\circled{\mrho}$} is a vertex $\circled{\msigma}$ of which 
$\circled{\mrho}$ is the zero parent.
A \emph{zero descendant of a vertex $\circled{\mrho}$} is any vertex which is either 
the zero child of $\circled{\mrho}$ or is (recursively) the zero descendant of any of the zero
children of $\circled{\mrho}$.
\end{definition}
From the above definitions we immediately obtain the following.
\begin{lemma}\label{lema-arboles-horizontales}
Let $\Lambda_{\circled{\varrho}}$ be a weighted directed zero--rooted tree.

\noindent
1. The zero descendants of the root $\raiz{\varrho}$ form the 
\emph{horizontal rooted subtree of the root}, denoted by
$\Lambda_{H(\varrho)}$. 
Note that the root of $\Lambda_{H(\varrho)}$ is once again $\raiz{\varrho}$.

\noindent
2. Each weighted edge $(\Delta_{\mrho \msigma}, K(\mrho,\msigma))$ 
with $K(\mrho,\msigma)\neq 0$, 
defines a \emph{horizontal rooted subtree $\Lambda_{H(\mrho, \msigma)}$,} 
with root $\raiz{\mrho}$,
of the \emph{incoming edge $(\Delta_{\mrho \msigma}, K(\mrho,\msigma))$},
with vertices 

\centerline{
$V_{H(\mrho, \msigma)}= \{ \circled{\mrho}, \circled{\msigma} \}
\cup
\{ \circled{\iota} \ \vert \ \circled{\iota} \text{ is a zero descendant of }
\circled{\msigma} \}$
}

\noindent
and edges  

\centerline{
$E_{H(\mrho, \msigma)}= \{ (\Delta_{\mrho \msigma}, K(\mrho,\msigma)) \}
\cup
\{ \text{edges that end on the zero descendants of } \circled{\msigma} \}$.
}
\hfill\qed
\end{lemma}

Note that, on each horizontal rooted subtree, the incoming edges are the only edges with 
non--zero weight. 

\noindent
\begin{example}\label{ejemplo-descomposicion}
Consider the weighted directed zero--rooted tree
\begin{multline}
\Lambda_{\circled{1}}=\Big\{\circled{1},\ldots,\circled{12}\, ;\raiz{1}\, ; 
\\
(\Delta_{1\,2},0), (\Delta_{2\,4},0), (\Delta_{4\,7},-3), (\Delta_{7\,12},0), (\Delta_{7\,11},4),
\\
(\Delta_{1\,3},1), (\Delta_{3\,5},0), (\Delta_{3\,6},0), (\Delta_{6\,9},0), (\Delta_{6\,10},0),
(\Delta_{5\,8},1) 
\Big\}.
\end{multline}
The root is $\raiz{1}$, and the incoming edges are 
$\Delta_{1\,3}$, $\Delta_{4\,7}$, $\Delta_{7\,11}$ and $\Delta_{5\, 8}$. 
Then 

\centerline{
$\Lambda_{\circled{1}} 
= \Lambda_{H(1)} 
\cup \Lambda_{H(1,3)}
\cup \Lambda_{H(4,7)}
\cup \Lambda_{H(7,11)}
\cup \Lambda_{H(5,8)}$,
}

\noindent
provides the decomposition into horizontal rooted subtrees as in Lemma \ref{lema-arboles-horizontales}.

\noindent
In Figure \ref{IncomingEdges}, 
the horizontal rooted subtree of the root is colored red;
the horizontal rooted subtrees, corresponding to the incoming edges, are colored orange, 
blue, green, purple respectively.
Note that the weight of each of the incoming edges could be any non--zero integer.
The vertices $\circled{1}$, $\circled{4}$, $\circled{5}$ and $\circled{7}$ belong to more than one
horizontal subtree.

\begin{figure}[htbp]
\begin{center}
\includegraphics[width=0.4\textwidth]{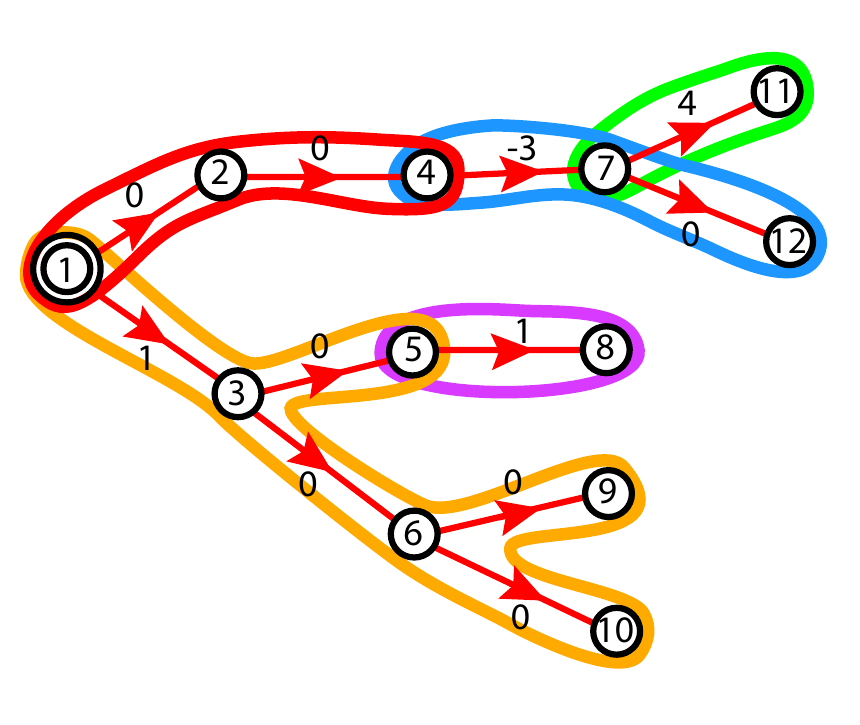}
\caption{The decomposition into horizontal rooted subtrees of the weighted directed rooted tree of
Example \ref{ejemplo-descomposicion}.
The weights are placed beside the corresponding edges.
The rooted subtree colored red  
is by definition the \emph{global zero level subtree}.
Recalling that the vertices correspond to branch points of $\R_{X}$, 
each of the horizontal subtrees is to represent the sheets of $\R_{X}$ that contain the 
corresponding subset of branch points.	
}
\label{IncomingEdges}
\end{center}
\end{figure}
\end{example}

\begin{remark}
The decomposition of $\Lambda_{\circled{\varrho}}$, given by Lemma \ref{lema-arboles-horizontales},
provides a disjoint partition on the set of edges. 
This relates to the fact that on $\R_{X}$ the diagonals between branch points are partitioned into disjoint sets 
according to the sheet they share.
\end{remark}

In order to overcome difficulty D.4, we shall need one more concept.

\noindent 
Consider the linear (weighted) directed tree 

\centerline{$G=\{ \mathcal{V}; \mathfrak{E}\}$}

\noindent 
with $\mathfrak{m}$ vertices 

\centerline{
$\mathcal{V}=
\{ \circled{\msigma} \doteq (z_\msigma, t_\msigma, -\nu_\msigma) \}_{\msigma=1}^\mathfrak{m}$, 
\ \ \
$z_\msigma\in\overline{\CC}_{z}$, 
$t_\msigma\in\CC_{t}$, 
$\nu_\msigma\in\NN\cup\{\infty\}$,
}

\noindent
where $\{ t_\msigma \}$ are different points, 
labelled so that 

\centerline{
$\Im{t_\msigma}\geq \Im{t_{\msigma+1}}$, \quad 
$\Re{t_{\msigma}}\leq\Re{t_{\msigma+1}}$,
} 

\noindent
and $\mathfrak{m}-1$ oriented weighted edges

\centerline{
$\mathfrak{E}=
\left\{ \big( \Delta_{(\msigma-1)\, \msigma}, K(\msigma-1,\msigma) \big)
\right\}_{\msigma=2}^\mathfrak{m}$, 
} 

\noindent
where $\Delta_{(\msigma-1)\, \msigma} \doteq (t_{\msigma } - t_{\msigma-1})$, 
and $K(\msigma-1,\msigma)\in\ZZ$.

\smallskip
\noindent 
In particular, $G$ can be understood as 
embedded in $\CC_{t}$. 

\noindent
Note that the vertices are connected from left to right and 
top to bottom;
starting with the top--\&--left--most vertex and ending at the bottom--\&--right--most vertex.

\begin{definition}\label{arbol-lineal-enraizado}
1. The $G$ 
constructed as above, is the 
\emph{left--right--top--bottom linear (weighted) directed tree of the 
vertices $\mathcal{V}$}. 
The underlying undirected 
linear graph will be called the 
\emph{undirected left--right--top--bottom linear (weighted) tree of the vertices $\mathcal{V}$}.

\noindent 2.
Moreover, 
for any choice of $\circled{\mrho}\in \mathcal{V}$,
the rooted tree

\centerline{
$G_{\circled{\mrho}}=
\left\{ \mathcal{V}; \raiz{\mrho}\,; \widehat{\mathfrak{E}}
\right\}$
}

\noindent
where 

\centerline{
$\widehat{\mathfrak{E}}=
\left\{ \Delta_{\mrho\, \mrho+1}, \Delta_{\mrho+1\, \mrho+2},
\ldots, \Delta_{\mathfrak{m}-1\, \mathfrak{m}},
\Delta_{\mrho\, \mrho-1}, \Delta_{\mrho-1\, \mrho-2},
\ldots, \Delta_{2\, 1} 
\right\}$,
}

\noindent
is called a \emph{linear (weighted) directed rooted tree 
with incoming vertex 
$\circled{\mrho}$}.
\end{definition}

Note that 
$G$ and $G_{\circled{\mrho}}$ have the same vertices, however different oriented edges
$\mathfrak{E}$ and $\widehat{\mathfrak{E}}$.  
Figure \ref{ArbolDirigidoHorizontal} 
provides an example with seven vertices.

\begin{figure}[htbp]
\begin{center}
\includegraphics[width=\textwidth]{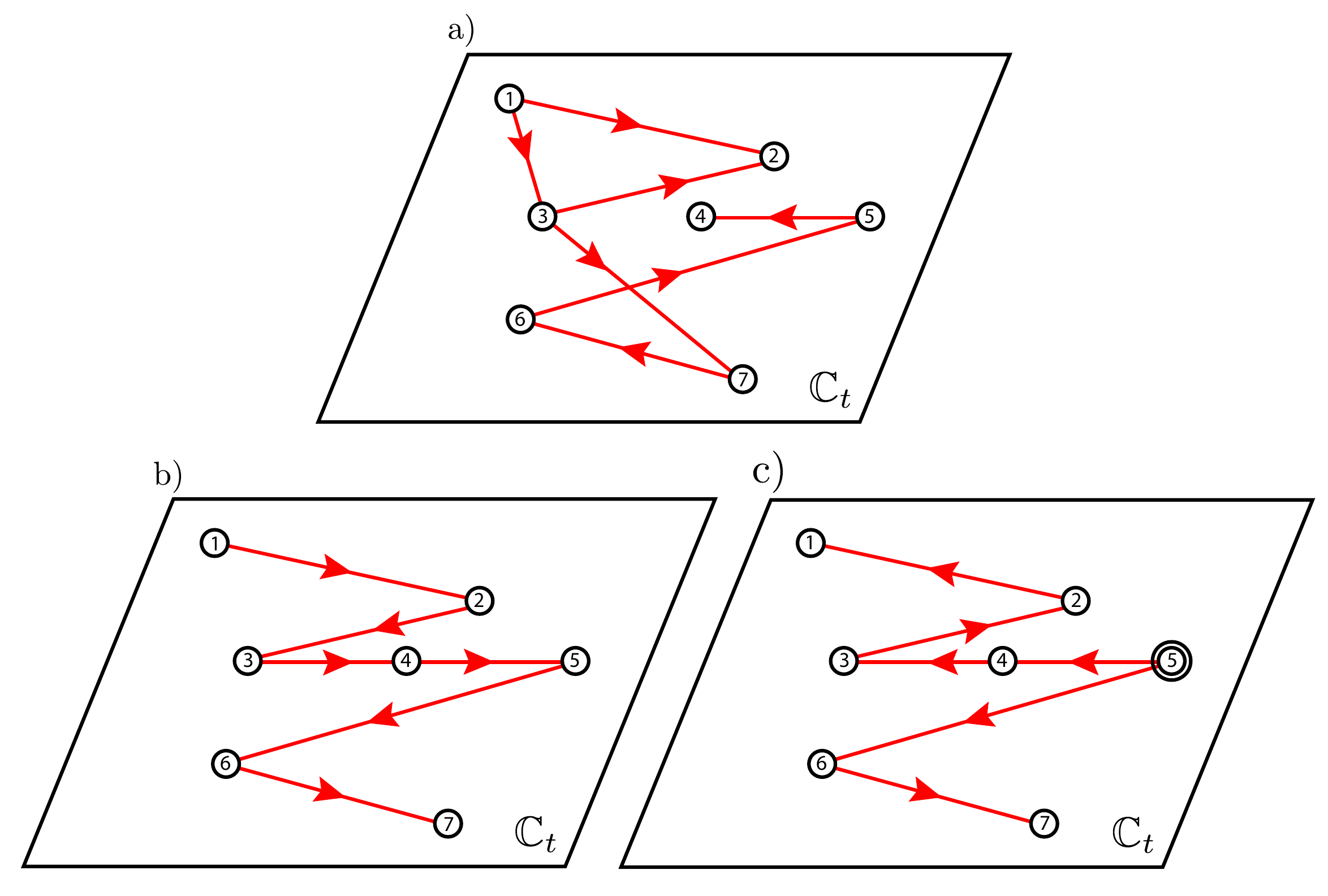}
\caption{In (a) we have a 
directed graph with $7$ vertices
$\mathcal{V}=\{ \circled{1}, \ldots , \circled{7} \}$.
In (b) we have eliminated and added some edges to obtain a
left--right--top--bottom linear directed tree $G$ of vertices
$\mathcal{V}$
with edges $\mathfrak{E}=\{ \Delta_{1\,2}, \Delta_{2\,3}, \Delta_{3\,4}, \Delta_{4\,5},
\Delta_{5\,6}, \Delta_{6\,7} \}$.
In (c) we have the linear directed rooted tree 
$G_{\circled{5}}$ where $\raiz{5}$ is both the incoming vertex and the root,
having edges $\widehat{\mathfrak{E}}=\{ \Delta_{5\,6}, \Delta_{6\,7}, 
\Delta_{5\,4},  \Delta_{4\,3}, \Delta_{3\,2}, \Delta_{2\,1} \}$, 
note the change of directions of the edges.
}
\label{ArbolDirigidoHorizontal}
\end{center}
\end{figure}

Recalling Table \ref{tablaTriada}, 
we are now ready to to introduce the particular weighted directed 
zero--rooted trees 
that will encode the information needed to specify the Riemann surfaces $\R_{X}$.
Issue D.5 will be dealt with by condition (1) of the following definition.

\begin{definition}\label{d-confTree}
For $r+d\geq1$,
a \emph{$(r,d)$--configuration tree} 
is a weighted directed 
zero--rooted tree 

\centerline{
$
\Lambda=\Big\{ V; 
\raiz{\varrho}\,;\,
E \Big\}
$}

\noindent
with: 

\noindent
$\bullet\ n+d$ \emph{vertices} 

\centerline{
$ V=
\Big\{ 
\circled{ \msigma } 
= 
(z_{\msigma},t_{\msigma}, - \nu_{\msigma}) 
\Big\}_{\msigma =1}^{n+d} \ ,
$ 
\quad \quad
$z_\msigma\in\overline{\CC}_z$, $t_\msigma\in\CC_t$ and $\nu_\msigma\in\NN\cup\{\infty\}$,
}

\noindent where 
$\sum\limits_{-\nu_\msigma \neq -\infty} \nu_{\msigma}=r$;

\noindent
$\bullet\ n+d-1$ \emph{weighted 
oriented edges} 
\begin{equation*}
E=\big\{ (\Delta_{\msigma \mrho},K(\msigma, \mrho)) \ \vert\ 
\Delta_{\msigma \mrho} \text{ starts at } 
\circled{\msigma} \text{ and ends at } \circled{\mrho},\ 
K(\msigma, \mrho)\in\ZZ \big\} ,
\end{equation*}
with the orientation of the edges being away from the root.

\noindent
In addition, the following conditions must be satisfied: 

%
%
\begin{enumerate}[label=\arabic*),leftmargin=*]
\item The set of vertices
$\big\{ \circled{\msigma} =(z_{\msigma}, t_{\msigma}, -\nu_{\msigma})
\big\}_{\msigma=1}^{n+d}$
must be realizable, {\it i.e.} 
they must satisfy 
the system of equations \eqref{sistema-realizable}.

\item
\emph{Types of vertices}.
Concerning the position of the vertices in 
$\overline{\CC}_z$;
$z_{\msigma}\in\CC_{z}$ if and only if $\nu_{\msigma}\in\NN$.
Thus, the vertices 
are classified in two types 
$$
\circled{\, \iota \, } =
\underbrace{\big(p_{\iota},\widetilde{p}_{\iota},-\nu_{\iota} 
\big)}_{\text{pole vertex}}
\in\CC_z
\quad\text{and}
\quad
\circled{ \scalebox{0.75}{$n \text{+} \sigma$} } 
=
\underbrace{\big( \infty_{\sigma},a_{\sigma}, -\infty\big)}_{
\text{essential vertex}} 
.
$$

\item
If $\Lambda$ consists of only one vertex, then
\begin{itemize}
\item[] the $(r,0)$--configuration trees are 
$\Big\{\circled{1}=\big(p_{1},\widetilde{p}_1, -r \big);
\raiz{1}\,;\,
\varnothing\Big\}$,
\item[]
the $(0,1)$--configuration trees are
$\Big\{\circled{1}=\big(\infty_{1},a_{1},-\infty\big);
\raiz{1}\,;\,
\varnothing\Big\}$.
\end{itemize}

\item \emph{Root condition}. 
If $r=0$ the root $\raiz{1}$ is the essential vertex 
$\circled{1}=\big(\infty_{1},a_{1},-\infty\big)$.

\noindent
If $r\neq 0$ the root $\raiz{ \varrho}$ is the pole vertex 
$\big(p_{\varrho},\widetilde{p}_{\varrho},-\nu_{\varrho} \big)$, 
whose $z$--coordinate is top \& left most:
{\it i.e.} 
$\Im{\widetilde{p}_\varrho}\geq\Im{\widetilde{p}_\iota}$ and 
when equality is achieved 
it is required that 
$\Re{\widetilde{p}_\varrho}<\Re{\widetilde{p}_\iota}$ for $1< \iota \leq n$.

\item 
\emph{Equality of vertices}. 
Given 
$\circled{\msigma}=(z_\msigma , t_\msigma , -\nu_\msigma)$ and 
$\circled{\mrho}=(z_\mrho , t_\mrho , -\nu_\mrho)$,
if $z_\msigma=z_\mrho$ then 
$t_\msigma = t_\mrho$ and $\nu_\msigma =  \nu_\mrho$,
{\it i.e.} 
necessarily 
$\circled{\msigma}=\circled{\mrho}$ in $\Lambda$.

\item \emph{Existence of edges}.
There is no edge between the vertices   
$\circled{\msigma}=
\big(z_{\msigma},t_{\msigma},-\nu_{\msigma}\big)$ and 
$\circled{\mrho}=\big(z_{\mrho},t_{\mrho},-\nu_{\mrho}\big)$
when $t_{\msigma}=t_{\mrho}$.

\item \emph{Horizontal subtree structure}.
We require that each of the horizontal rooted subtrees of $\Lambda$,
be a linear (weighted) directed rooted subtree 
$G_{\circled{\mrho}}$ with incoming vertex $\circled{\mrho}$,
as in Definition \ref{arbol-lineal-enraizado}.

\end{enumerate}
\end{definition}

\noindent
\begin{remark}\label{remArbol}
1. The root condition (4) allows us to make a canonical choice
of the root vertex, see Remark 
\ref{correspexponentialtract}.2.

\noindent
2. Condition (5) of Definition \ref{d-confTree} is equivalent to saying that $t_{\msigma}$ and 
$\nu_{\msigma}$ are functions of $z_{\msigma}$.
For instance in Examples \ref{ejemplo-E23} and \ref{ejemplo-E33}, 
even though the finite asymptotic values 
of $\Psi_X$ have
multiplicity 3, we can use 
$z_{\msigma}\in\overline{\CC}_z$
to label the vertices of $\Lambda_{X}$. 

\noindent
3. As will be seen in the proof of the Main Theorem,
condition (7) of Definition \ref{d-confTree} provides, for 
each sheet of $\R_{X}$, a choice of the diagonals that 
connect the branch points that
share the same sheet.
This choice will enable us to define appropriately the class $[\Lambda_{X}]$ of 
$(r,d)$--configuration trees.

\noindent 
4. In case that there is only one horizontal rooted subtree for $\Lambda$, 
Lemma \ref{lema-arboles-horizontales} ensures that the only horizontal rooted subtree 
is $\Lambda_{H(\rho)}$.

\noindent 
5. When $r=0$, Definition \ref{d-confTree} reduces 
to the definition of a $d$--configura\-tion tree presented in \S8.3 of \cite{AlvarezMucino}, 
for $X\in\E(0,d)$. 
The equivalence becomes explicit by observing that the essential vertices 
$\circled{\sigma}=
\big( \infty_{\sigma}, a_{\sigma},-\infty \big)$ 
of $(0,d)$--configuration trees correspond to 
the vertices $(\infty_{\sigma},a_{\sigma})$,
pairs in \cite{AlvarezMucino},
of $d$--configuration trees.
\end{remark}

\section{Low degree significative examples, from $X$ to $\Lambda_X$}
\label{concreteexamples}

We provide examples of the correspondence 
from vector fields to Riemann surfaces and 
configuration trees

\centerline{
$X \ \longmapsto \  \R_X 
\ \longmapsto \  \Lambda_X,$
}

\noindent 
using the basic geometric/combinatorial pieces
described in Figure \ref{figPiezasBasicas}.

{\bf About the meaning of the different 
data of $\Lambda_X$.}

\begin{enumerate}[label=\arabic*),leftmargin=*]

\item 
The vertices of $\Lambda_{X}$ are in bijection with 
the reduced divisor of $X$
(recall Definition \ref{divisor-reducido}, Table 
\ref{tablaTriada}),  and with 
the  branch points of $\R_X$ 
with respect to $\pi_{X,2}$, restricted over $\CC_t$, 
recall Diagram \ref{diagramaRX}. 

\item
The edges correspond to a subset of the diagonals, connecting the
branch points, necessary to describe completely 
the Riemann surface $\R_{X}$, 
recall Definition \ref{diagonal}.1.

\item 
The root vertex 
$\raiz{\varrho} \doteq(z_\varrho, t_\varrho, -\nu_\varrho)$, as usual 
in graph theory,  means 
the initial vertex in order to construct a tree,
as in \eqref{arbol-cero}. 
From the analytic and geometric point of view,
the root determines the initial point of
$\Psi_X (z)= \int_{z_1}^z \omega_X$.

\item 
The weights
$K(\msigma, \mrho)\in\ZZ$ in Definition \ref{d-confTree}:
If we can describe all the branch points of $\R_X$ using
only one sheet (Definition \ref{horizontal-branch-cut}), 
then $K(\msigma, \mrho)=0$. 
If several sheets are required, 
the weight of the edge $K(\msigma, \mrho) \in \ZZ \backslash \{ 0 \}$ 
tells us the relative number of  
sheets of $\R_{X}$, 
we must go ``up or down'' on the surface in order 
to find another sheet containing other
branch points. 
Vector fields having $K(\msigma, \mrho)\neq 0$
appear in 
Example \ref{ejemplo-E23} and
\S\ref{campos-3-0-y-sus-arboles}.

\noindent  
In particular, if there are only two branch points
then 
there is no need to go up or down at the starting branch point, 
so the weight of the only edge is 0. 

\item 
The \emph{global zero level sheet} denoted \emph{GZL}
(which is in general non canonical),
indicates a subset of the branch points that 
share the same sheet as the root $\raiz{1}$.

\end{enumerate}

\smallskip

\begin{remark}\label{primera-vez-soul}
The notion of \emph{skeleton} associated to 
$\Lambda_X$ will be described in 
Definition \ref{skeleton} and the notion of
\emph{$(r,d)$--soul} in Definition \ref{soul}.
They play an active role in the 
proof of the Main Theorem.

\noindent $\bigcdot$
For $X \in \E(r, d)$, $d \geq 1$, 
the $(r,d)$--soul is necessarily a flat Riemann 
surface with boundary, 
and is obtained from
$\R_X$ by removing its semi--infinite helicoids.

\noindent $\bigcdot$
In the case  $X \in \E(r, 0)$, 
the $(r,d)$--soul coincides with
$\R_X$. 
\end{remark}

\subsection{$X\in\E(r,0)$ has $r\geq1$ poles on $\CC_{z}$ and $\Psi_X$ is a polynomial map}

\begin{example}\label{ejemplo-un-polo}
Consider the vector field of Example \ref{campo-racional-un-polo}, 

\centerline{$X(z)=\dfrac{\lambda}{(z-p_1)^{r}} \ddel{}{z}$}

\noindent
and recall the notion of $(\nu_{\iota} +1)$--cyclic helicoid Definition \ref{helicoide-ciclico}.
The $(r,0)$--configuration tree consists of one pole vertex and no edges
\\
\centerline{
$\Lambda_{X}=\Big\{ 
\circled{1}=(p_{1},\widetilde{p}_{1}, -r);
\raiz{1}\, ;\,
\varnothing
\Big\},$}

\noindent
where $\widetilde{p}_{1}$ as in \eqref{val-crit-1-polo}
is the critical value. 
In Figure \ref{figejemplo-un-polo} the 
case $r=2$ is pictured:
on the left hand side the phase portrait of $X$ is shown, clearly there are 6 hyperbolic sectors 
each corresponding to a half plane;
on the second column the Riemann surface is shown; 
finally on the rightmost column the combinatorial 
objects are portrayed.
For the general case $-r \leq -1$ see 
Figure \ref{figPiezasBasicas}.c.
\begin{figure}[htbp]
\begin{center}
\includegraphics[width=\textwidth]{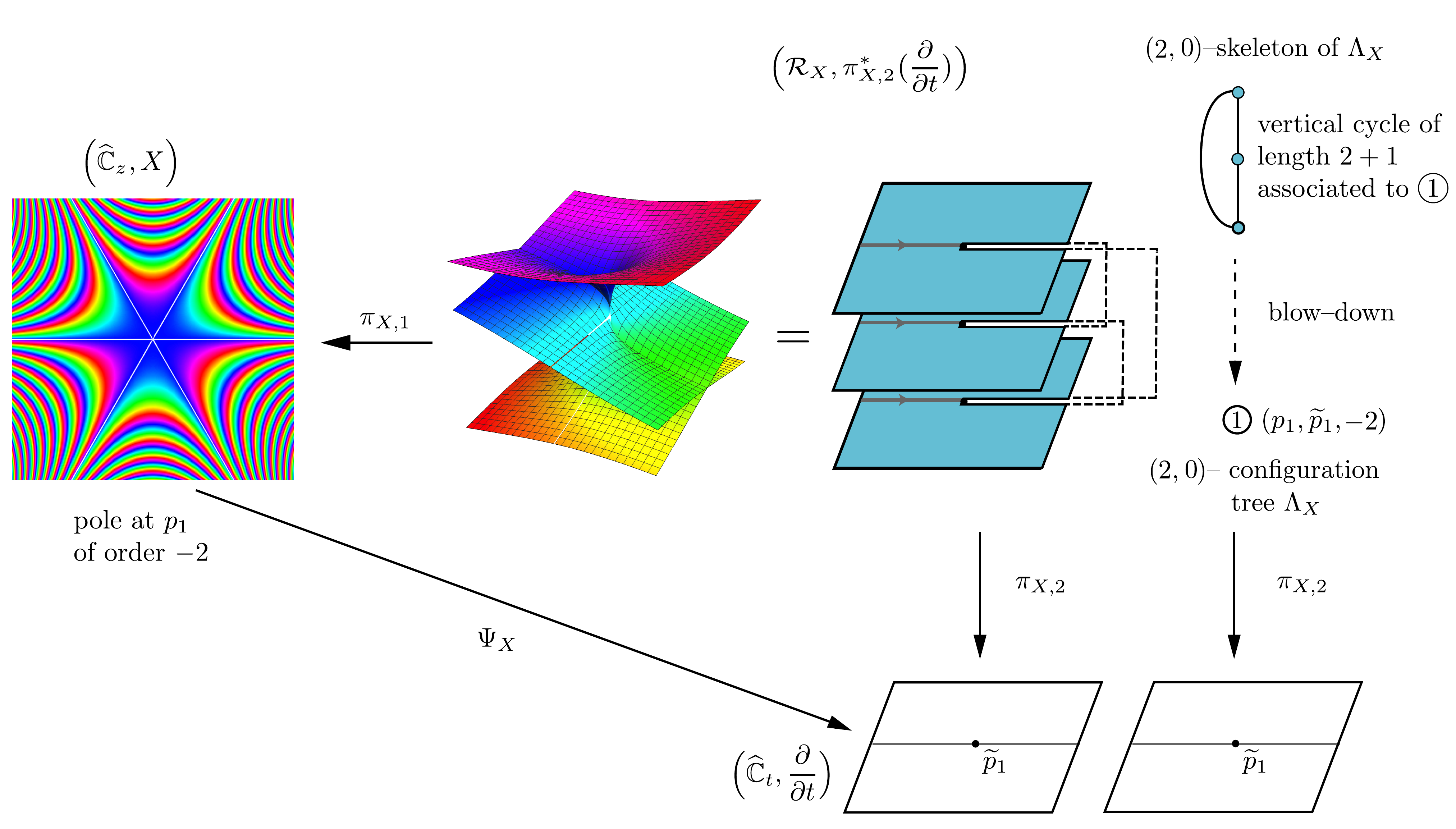}
\caption{{\bf Vector field 
$X(z)=\dfrac{1}{(z-p_1)^2 } \ddel{}{z}$ 
with a pole of order $-2$ at $p_{1}\in\CC_{z}$.}
The surface $\R_{X}$ consists of three sheets with a branch cut, glued together to form a 
$(2+1)$--cyclic helicoid. 
On the right column, 
the $(2,0)$--configuration tree consisting of one vertex,
and
its corresponding $(2,0)$--skeleton are portrayed; 
see \S\ref{confTree-to-skeleton} and Definition \ref{skeleton}.
}
\label{figejemplo-un-polo}
\end{center}
\end{figure}

\end{example}

In the next examples we shall consider 
$(r,d)$--configuration trees with two or more 
vertices, hence weighted edges shall appear. 

\begin{example}\label{ejemplo-dos-polos}
Consider the vector field
\begin{equation}
\label{ecuacion-dos-polos}
X(z)= \dfrac{\lambda}{(z-p_{1})^{\nu_{1}} (z-p_{2})^{\nu_{2}}} \ddel{}{z} \in \E(r,0), \
\ \nu_{1}+\nu_{2}=r, \ \nu_1,\nu_2\geq 1, 
\end{equation}

\noindent 
and its distinguished parameter
\\
\centerline{
$\Psi_{X}(z)=
\frac{1}{\lambda}
\int_{z_{0}}^{z} (\zeta-p_{1})^{
\nu_{1}} (\zeta-p_{2})^{\nu_{2}} d\zeta$. }

\noindent 
Without loss of generality, we assume that
the critical values
$\widetilde{p}_{j}=\Psi_{X}(p_{j})$, $j=1,2$,
satisfty
$\Im{\widetilde{p}_1} \geq \Im{\widetilde{p}_2}$.
In this case the $(r,0)$--configuration tree has two pole vertices and one edge
\\
\centerline{
$\Lambda_{X}=\Big\{ \circled{1}=(p_{1},\widetilde{p}_{1}, -\nu_{1}), 
\circled{2}=(p_{2},\widetilde{p}_{2}, -\nu_{2});
\raiz{1}\, ;\,
(\Delta_{1\, 2},0) 
\Big\},$}
where   
the edge $\Delta_{1\, 2}$ is the semi--residue 
$S(\omega_{X},p_{1},p_{2},\gamma)=\widetilde{p}_{2}-\widetilde{p}_{1}$, 
which according to \eqref{diagonalsemiresidue} is equivalent to the diagonal 
with the same name. 
Finally, the weight of the edge is 0,
since $\Delta_{1\, 2}$ 
is in the global zero level sheet.
See Figure \ref{figejemplo-dos-polos} and  \ref{figPiezasBasicas}.c. 

\noindent 
If $\Im{\widetilde{p}_1}>\Im{\widetilde{p}_2}$, then the semi--residue 
$\Delta_{1\,2}=S(\omega_{X},p_{1},p_{2},\gamma)=\widetilde{p}_{2}-\widetilde{p}_{1} \in \CC \backslash \RR$, 
giving origin to a finite height horizontal strip,
see left drawing in Figure \ref{figejemplo-dos-polos}.

\noindent 
If $\Im{\widetilde{p}_1}=\Im{\widetilde{p}_2}$, then the diagonal
$\Delta_{1\,2}$
coincides, up to orientation, with a saddle connection
of the real vector field
$\Re{X}$.

\begin{figure}[htbp]
\begin{center}
\includegraphics[width=\textwidth]{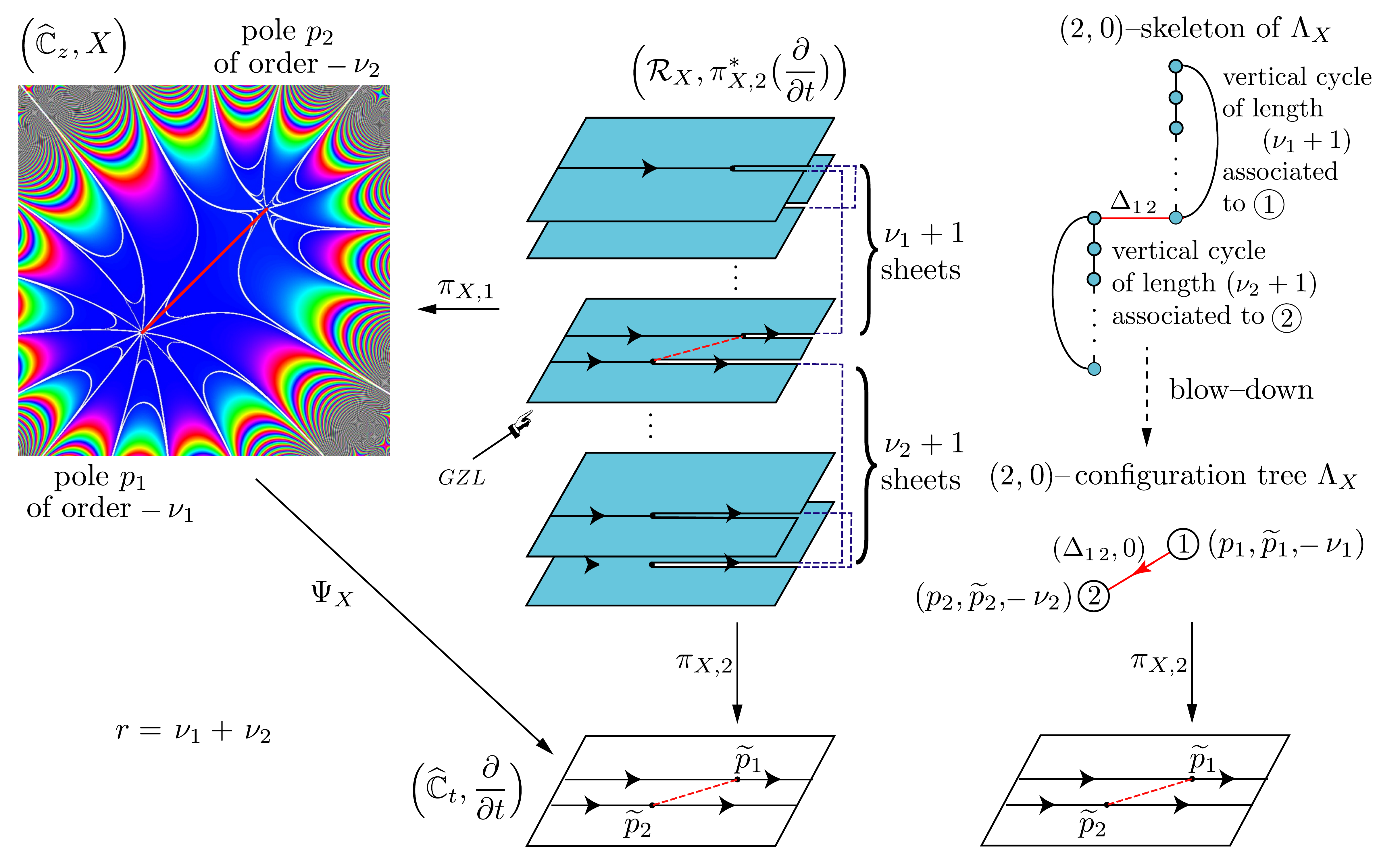}
\caption{{\bf Vector field 
$X(z)=\dfrac{\lambda}{(z-p_{1})^{\nu_{1}} (z-p_{2})^{\nu_{2}}} \del{}{z}$ 
with two poles $p_{\iota}$ of order $-\nu_{\iota}$.} 
The diagonal 
$\Delta_{1\,2}\subset\R_{X}$ 
associated to the finitely ramified branch points and its projections 
via $\pi_{X,1}$ and $\pi_{X,2}$ are coloured red. 
The two branch points are the endpoints of the diagonal 
$\Delta_{1\,2}\subset\R_{X}$ 
on the global zero level sheet.
The phase portrait (left drawing) is the case with 
poles of 
orders $-3=-\nu_2 $ and $-5= -\nu_1$.
See Example \ref{ejemplo-dos-polos}, and \S\ref{confTree-to-skeleton} 
for the drawing on the right. 
}
\label{figejemplo-dos-polos}
\end{center}
\end{figure}
\end{example}

\subsection{Vector fields in $\E(3,0)$ with three simple poles }
\label{campos-3-0-y-sus-arboles}
Consider vector fields 
$X \in \E(3, 0)$ with simple poles, 

\noindent 
$\bigcdot$
we fix two finite ramification values  
$\widetilde{p}_1=0, \, \widetilde{p}_2=1 $, 
and

\noindent 
$\bigcdot$
leave free the third $\widetilde{p}_3$
in the twice punctured plane 
$\CC\backslash\{0,1\}$. 

\noindent 
This gives origin to 
a suitable family of vector fields $\mathfrak{F}$. 
It is to be noted that there is a strong 
analytical and combinatorial dependence on 
the choice of the ramification values of the involved 
distinguised parameters $\Psi_X$.

\noindent 
We shall study the general 
problem of compute $\Psi_X$ starting with
a configuration of 
preassigned critical values 
$\{\widetilde{p}_1, \ldots ,\widetilde{p}_r\}$
in \S\ref{epilogo},   Corollary
\ref{Psi-con-valores-criticos-asintoticos-preasignados}.

\FloatBarrier
\begin{proposition}\label{lema-arbol-3-0}
Let $\mathfrak{F}$ be the family of 
vector fields defined by 
\begin{equation}\label{familia-campos-3-polos}
\begin{array}{rcl}
X(p_3, z):(\CC \backslash \{0, 1/2, 1 \}) \times \CC 
& \longrightarrow &
\E(3,0)
\\
&& \vspace{-.3cm}
\\
(p_3,z) &\longmapsto&
\dfrac{2p_3 - 1}{12 z(z-1)(z- p_3)}
\ddel{}{z}.
\end{array} 
\end{equation}
\begin{enumerate}[label=\arabic*),leftmargin=*]
 
\item 
The corresponding distinguished 
parameters are the polynomials

\begin{equation}\label{familia-polinomios-que-vienen-de-3-polos}
\Psi(p_3, z) = 
\frac{12}{2p_3 - 1}
\Big( \dfrac{1}{4} z^4 + \frac{-p_3 -1}{3} z^3
+ \frac{p_3}{2} z^2 \Big) \in \CC[z].
\end{equation}

\item
The corresponding reduced divisors are
\begin{equation}\label{tres-puntos-de-ramificacion}
X(p_3,z) \quad \longmapsto \quad
(0,0,-1) + (1,1,-1) + 
\Big( p_3, \dfrac{p_3^3(2-p_3)}{2p_3-1} , -1 \Big).
\end{equation}

\item
The 
$(3,0)$--configuration trees 
$\Lambda_{X(p_3, z)}$ for $X(p_3, z)$ are 
given by 
Equations \eqref{arbol-U1}--\eqref{arbol-U67}. 
\end{enumerate}
\end{proposition}

In simple words, each $\Lambda_{X(p_3, z)}$ describes 
the relative position of the branch points 

\centerline{
$\circled{1}=(0,0,-1)$, 
\, $\circled{2}=(1,1,-1)$, 
\, $\circled{3}=(p_3, \widetilde{p}_3,-1)$}

\noindent 
on the Riemann surface $\R_{X(p_3, z)}$.

\begin{remark} 
Motivation for the family
$\mathfrak{F}$.
Let $X \in \E(r, d)$ be a vector field 
with at least two different poles
$p_1$, $p_2$ (so $r \geq 2$), the choice of 

\centerline{
$(p_1, \widetilde{p}_1)=(0,0)$, 
$(p_2, \widetilde{p}_2)=
(1,1) \in \CC^2$ }

\noindent 
as in  
\eqref{ecuacion-dos-polos} 
and 
\eqref{tres-puntos-de-ramificacion}
can be explained as follows. 

\noindent 
1. We consider the complex analytic action 
\begin{equation}\label{LaAccion}
\mathcal{A}:Aut(\CC) \times \E(r,d) \longrightarrow \E(r,d),
\ \ \ (T, X) \longmapsto T^* X,
\end{equation}
of the affine transformation group
$Aut(\CC)$ corresponding to those
$T \in Aut(\CW)= PSL(2, \CC)$ that fix $\infty$,
see \cite{AlvarezMucino2} for general theory.
Using suitable $T$, we obtain $p_1=0$ and $p_2=1$.
(It is to be noted that the affine group $Aut(\CC)$ is the largest 
complex automorphism group that acts on $\E(r,d)$.)

\noindent 2. If $\Psi_X(z)=\int_{p_1=0}^z \omega_X$ 
then $\Psi_X (0) = 0 = \widetilde{p_1}$.

\noindent 
3.
Considering $\{ \lambda X \ \vert \ \lambda \in \CC^* \}$ 
as a projective class,  we normalize
by a suitable $\lambda_0$, thus  

\centerline{
$\int_0^1 \frac{1}{\lambda_0}\omega_X 
= 1 = \widetilde{p}_2$. }

\noindent 
In the particular case $\E(3, 0)$, we get 
Equation \eqref{familia-campos-3-polos}.
\end{remark}

\begin{proof}
{\it First step. A suitable tessellation 
for the third critical point 
$p_3$.}
The degree four rational map
\begin{equation}
\label{funcion-racional-R}
\CC_{p_3} \longrightarrow \CC_{\widetilde{p}_3},
\ \ \ 
p_3 \longmapsto \frac{p_3^3(2-p_3)}{2p_3-1} 
=
\widetilde{p}_3
\end{equation}

\noindent 
determines the behaviour of the third branch point
$(p_{3},\widetilde{p}_3, -1)\in \R_{X (p_3, z)}$.
The rational map \eqref{funcion-racional-R}
has 
$\{0, 1, 1/2\}$ as critical points
with critical values $\{0, 1, \infty \}$, respectively.
The inverse image $\widetilde{p}_3^{\, -1}(\RR)$
is drawn in Figure \ref{dominioK13}
using black, blue and orange 
segments to represent the inverse images of 
$(-\infty,0)$, $[0,1]$ and $(1,\infty)$ respectively.
There are eight open connected components 

\centerline{$\{U_j\}_{j=1}^8
=\CC_{p_3} \backslash \widetilde{p}_3 ^{\,-1}(\RR)$}

\noindent
determining a tessellation of $\CC_{p_3}$.
The regions $U_j$ with even index $j$,
coloured white,
are the 
inverse image of the lower half plane 
${\HH}^2_{-}$
and the odd index regions,
coloured gray,
are the inverse image of 
the upper half plane
${\HH}^2_{+}$.

The meaning and suitability of the tessellation 
of \eqref{funcion-racional-R} in $\CC_{p_3}$
is as follows. 

\begin{figure}[htbp]
\begin{center}
\includegraphics[width=\textwidth]{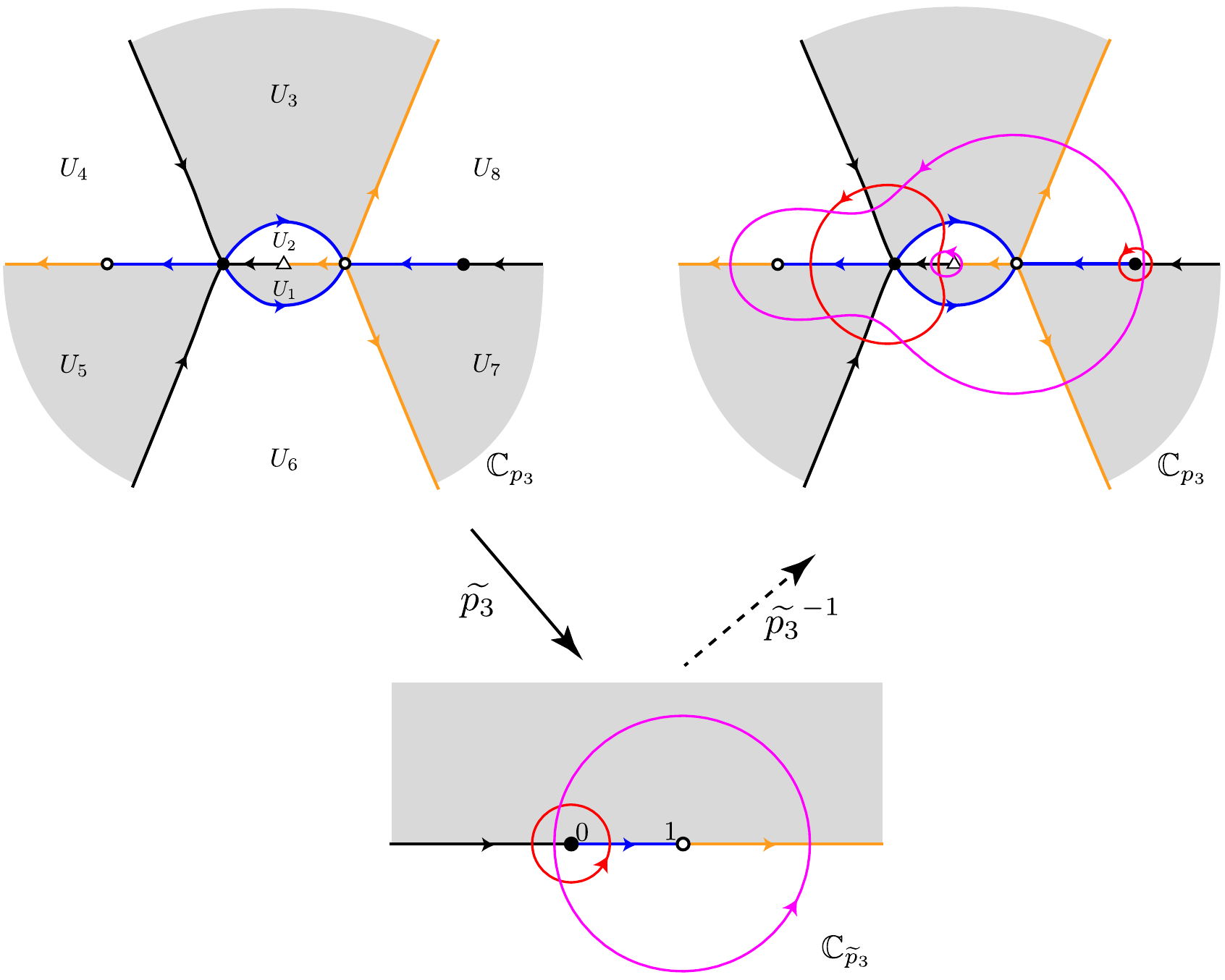}
\caption{
The degree four rational 
map $p_3 \longmapsto \widetilde{p}_3$ gives
origin to eight open regions $U_j$,
forming a tessellation of $\CC_{p_3}$, see
upper left figure. 
The points $0, \, 1, \, 1/2 \in \CC_{p_3}$, 
coloured black, white, triangle vertices,
correspond to the preimages of $0, 1, \infty \in \CC_{\widetilde{p}_3}$,
respectively. \,  
The upper right figure provides a 
description of the lift of the 
red $\alpha$ and magenta $\beta$ 
circles under $\widetilde{p}_3^{\, -1}$.
The blue tree 
$\widetilde{p}_3^{\,-1}\big( [0,1] \big)$ is 
the {\it dessin d'enfant} of the map $\widetilde{p_3}$.
}
\label{dominioK13}
\end{center}
\end{figure}

\smallskip

{\it  A description of the topological behavior 
of the third critical value 
$\widetilde{p}_3$ in loops enclosing
the other two critical values $0$, $1$.}

\begin{enumerate}[label=\arabic*), leftmargin=*]
\item
The red circle 
$\alpha$
describes a loop of $\widetilde{p}_3$
enclosing the ramification value $0$, but not $1$.

\item
The magenta circle 
$\beta$
describes a loop of $\widetilde{p}_3$
enclosing the ramification values $0$ and $1$.
\end{enumerate}

\noindent
These loops generate the fundamental group 
$\pi_1 (\CC_{\widetilde{p}_3} \backslash \{0,1\})$. 
The lifts of the circles $\alpha$, $\beta$
are described in Figure \ref{dominioK13}. 

In order to proceed with the proof,
the idea is as follows. 
If we determine the $(3,0)$--configuration trees
of $X(p_3, z)$ at one point $p_3 \in U_j$,  then 
by a continuity argument 
the analogous 
configuration tree remains valid for all $p_3 \in U_j$.

\smallskip

{\it Second step. 
Computation of the $(3,0)$--configuration trees.} 

\noindent
We shall need to consider the boundaries between the regions $U_j$, $j=1,\ldots,8$.
Let $U_i$ and $U_j$ denote two adjacent regions with 
common boundary (as open segments)

\centerline{ $\partial U_{i, j} \doteq 
(\overline{U_i} \cap \overline{U}_j) \backslash 
\big\{
\widetilde{p}_3^{-1}
\left( \{ 0, 1, \infty  \} \right)
\big\}
$.}

\noindent 
Moreover, we 
assign colors to the vertices and edges 
in Figure \ref{EjemploArbolesE30}, as follows 
$$
\begin{array}{c}
\circled{1}=
(0,0,-1) \, \hbox{ in red,} \
\circled{2}=
(1,1,-1) \, \hbox{ in green,}\
\circled{3}=
(p_3, \widetilde{p}_3,-1)\, \hbox{ in blue,}
\\
\vspace{-.3cm}
\\
\Delta_{1\, 2} =  
\overline{(0,0,-1), (1,1,-1)}, \ 
\Delta_{2\, 1} =  
\overline{(1,1,-1), (0,0,-1)}
\ \hbox{ in dashed black line,}
\\
\vspace{-.3cm}
\\
\Delta_{1\, 3} =  
\overline{(0,0,-1), (p_3, \widetilde{p}_3,-1)}, \
\Delta_{3\, 1} =  
\overline{(p_3, \widetilde{p}_3,-1), (0,0,-1)}
\ \hbox{ in red,}
\\
\vspace{-.3cm}
\\
\Delta_{2\, 3} =  
\overline{(1,1,-1), (p_3, \widetilde{p}_3,-1)}, \
\Delta_{3\, 2} =  
\overline{(p_3, \widetilde{p}_3,-1), (1,1,-1)}
\ \hbox{in green.}
\end{array}
$$

\noindent 
Because of Definition \ref{d-confTree}.4 
there are two possible cases for the root.

\noindent
1) If $\widetilde{p}_3\in\HH^2_{+}\cup(-\infty,0)$ then 
the vertex $\circled{3}$ is the root, 
by Equation \eqref{funcion-racional-R}, this is equivalent to 
\begin{multline}\label{raiz-p3}
p_3\in U_1 \cup U_3 \cup U_5 \cup U_7 \\
\cup \partial U_{3,4} \cup \partial U_{5,6} 
\cup \big(\partial U_{1,2} \cap (0,1/2) \big) 
\cup \big(\partial U_{7,8} \cap (2,\infty) \big).
\end{multline}

\smallskip
\noindent
2) If $\widetilde{p}_3\not\in\HH^2_{+}\cup(-\infty,0)$ then the 
vertex $\circled{1}$ is the root,
once again by \eqref{funcion-racional-R}
\begin{multline}\label{raiz-0}
p_3\in
U_2 \cup U_4 \cup U_6 \cup U_8 \\
\cup \partial U_{2,3} \cup \partial U_{1,6} 
\cup \partial U_{4,5} \cup \partial U_{3,8} \cup \partial U_{6,7} \\
\cup \big(\partial U_{1,2} \cap (1/2,1) \big) 
\cup \big(\partial U_{7,8} \cap (1,2) \big) .
\end{multline}

\begin{figure}[htbp]
\begin{center}
\includegraphics[width=\textwidth]{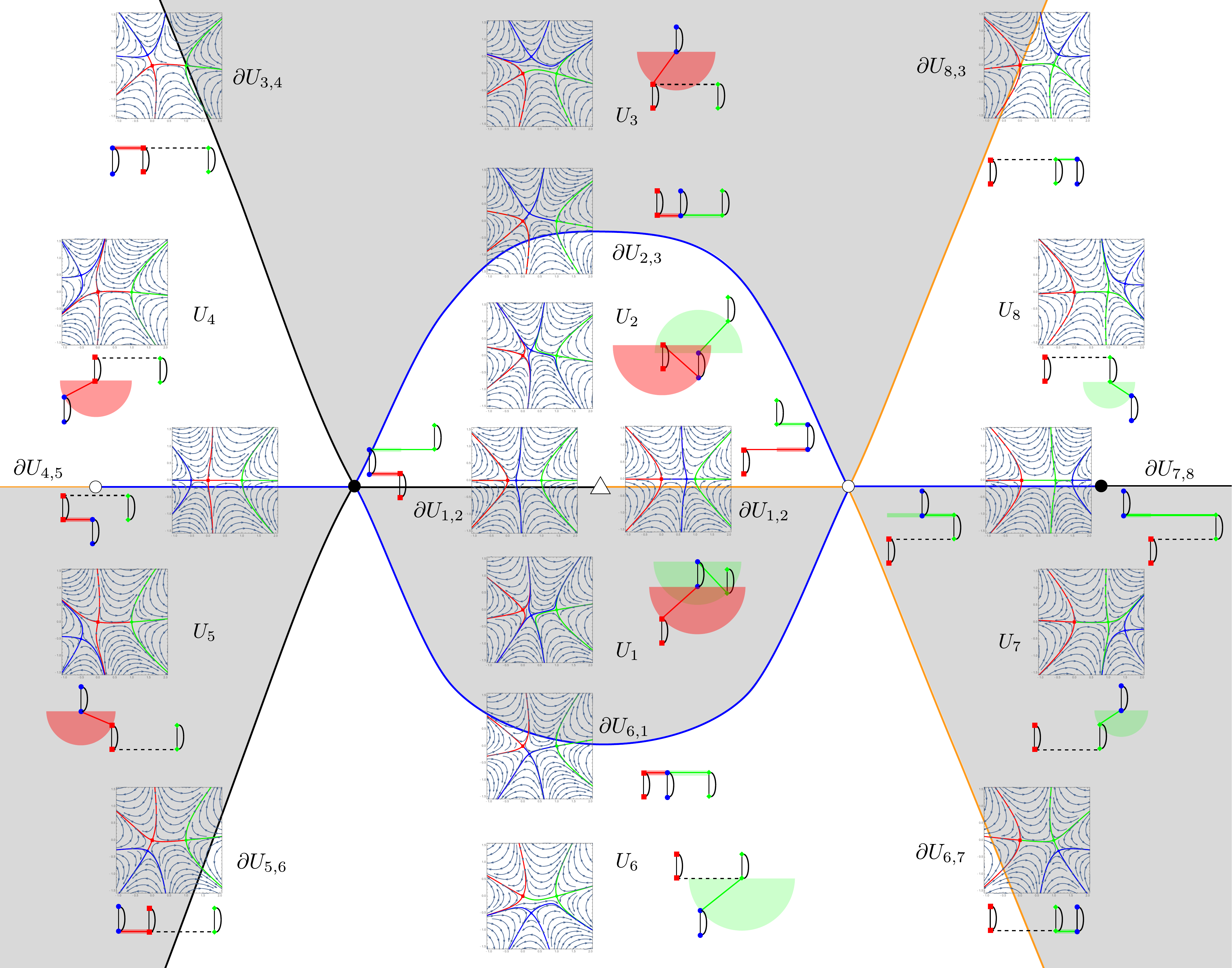}
\caption{
The plane $\CC_ {p_3}$ as 
vector field atlas of $X(p_3, z)$ and
their corresponding 
$(3,0)$--skeletons of 
$\Lambda_{X(p_3, z)}$.
This provides a bifurcation diagram for the phase portraits,
there appear  
12
non equivalent (topologically, 
orientation preserving), 
real vector fields $\Re{X(p_3, z)}$.
The shaded red and green areas represent the half planes or segments
where the diagonal corresponding to $\widetilde{p}_3$ can 
move, its total winding number is 4 coinciding 
with the degree of the map $\widetilde{p}_3$ in
\eqref{funcion-racional-R}.
}
\label{EjemploArbolesE30}
\end{center}
\end{figure}

\smallskip
\noindent
In either case, the diagonals/edges satisfy 
\begin{equation}
\begin{array}{c}
\Delta_{i\, j} = - \Delta_{j\, i} , \quad i,j\in\{1,2,3\}, \ i\neq j ,\\
\Delta_{1\, 2} + \Delta_{2\, 3} + \Delta_{3\, 1} = 0 .
\end{array}
\end{equation}

The $(3,0)$--configuration trees can be deduced upon careful consideration of 
the phase portraits portrayed in Figure \ref{EjemploArbolesE30}. 

\noindent
\emph{Case (1), the root is $\raiz{3}$ and colored blue.} 
From Equation \eqref{raiz-p3}, when
\begin{itemize}[leftmargin=*]
\item $p_3\in U_1$, the $(3,0)$--configuration tree is
\begin{equation}\label{arbol-U1}
\Lambda_{X(p_3, z)} =
\Big\{ \circled{1}, \circled{2}, \circled{3}\, ; \raiz{3} \, ;
(\Delta_{3\,2},0), \, (\Delta_{3\,1}, 1) 
\Big\} , \text{ with } \Delta_{3\,2}, \Delta_{3\,1}\in\HH^{2}_{-};
\end{equation}

\item $p_3\in U_3$, the $(3,0)$--configuration tree is
\begin{equation}\label{arbol-U3}
\Lambda_{X(p_3, z)} =
\Big\{ \circled{1}, \circled{2}, \circled{3}\, ; \raiz{3} \, ;
(\Delta_{3\,1},0), \, (\Delta_{1\,2}, 0) 
\Big\} , \text{ with } \Delta_{1\,2}=1 \text{ and }\Delta_{3\,1}\in\HH^{2}_{-};
\end{equation}

\item $p_3\in U_5$, the $(3,0)$--configuration tree is
\begin{equation}\label{arbol-U5}
\Lambda_{X(p_3, z)} =
\Big\{ \circled{1}, \circled{2}, \circled{3}\, ; \raiz{3} \, ;
(\Delta_{3\,1},0), \, (\Delta_{1\,2}, 1) 
\Big\} , \text{ with } \Delta_{1\,2}=1 \text{ and } \Delta_{3\,1}\in\HH^{2}_{-};
\end{equation}

\item $p_3\in U_7$, the $(3,0)$--configuration tree is
\begin{equation}\label{arbol-U7}
\Lambda_{X(p_3, z)} =
\Big\{ \circled{1}, \circled{2}, \circled{3}\, ; \raiz{3} \, ;
(\Delta_{3\,2},0), \, (\Delta_{2\,1}, 1) 
\Big\} , \text{ with } \Delta_{2\,1}=-1 \text{ and } \Delta_{3\,2}\in\HH^{2}_{-};
\end{equation}

\item $p_3\in \partial U_{1,2} \cap (0,1/2)$, the $(3,0)$--configuration tree is
\begin{equation}\label{arbol-U12-0-1/2}
\Lambda_{X(p_3, z)} =
\Big\{ \circled{1}, \circled{2}, \circled{3}\, ; \raiz{3} \, ;
(\Delta_{3\,1},0), \, (\Delta_{3\,2}, 1) 
\Big\} , \text{ with } \Delta_{3\,2}>0 \text{ and } \Delta_{3\,2}>1;
\end{equation}

\item $p_3\in \partial U_{7,8} \cap (2,\infty)$, the $(3,0)$--configuration tree is
\begin{equation}\label{arbol-U78-2-infty}
\Lambda_{X(p_3, z)} =
\Big\{ \circled{1}, \circled{2}, \circled{3}\, ; \raiz{3} \, ;
(\Delta_{3\,2},0), \, (\Delta_{2\,1}, 1) 
\Big\} , \text{ with } \Delta_{2\,1}=-1 \text{ and } \Delta_{3\,2}>1;
\end{equation}

\item $p_3\in \partial U_{3,4}$, the $(3,0)$--configuration tree is
\begin{equation}\label{arbol-U34}
\Lambda_{X(p_3, z)} =
\Big\{ \circled{1}, \circled{2}, \circled{3}\, ; \raiz{3} \, ;
(\Delta_{3\,1},0), \, (\Delta_{1\,2}, 0) 
\Big\} , \text{ with } \Delta_{1\,2}=1 \text{ and } \Delta_{3\,1}>0;
\end{equation}

\item $p_3\in \partial U_{5,6}$, the $(3,0)$--configuration tree is
\begin{equation}\label{arbol-U56}
\Lambda_{X(p_3, z)} =
\Big\{ \circled{1}, \circled{2}, \circled{3}\, ; \raiz{3} \, ;
(\Delta_{3\,1},0), \, (\Delta_{1\,2}, 0) 
\Big\} , \text{ with } \Delta_{1\,2}=1 \text{ and } \Delta_{3\,1}>0;
\end{equation}

\end{itemize}

\smallskip
\noindent
\emph{Case (2), the root is $\raiz{1}$ and colored red.} 
From Equation \eqref{raiz-0}, when 
\begin{itemize}[leftmargin=*]
\item $p_3\in U_2$, the $(3,0)$--configuration tree is
\begin{equation}\label{arbol-U2}
\Lambda_{X(p_3, z)} =
\Big\{ \circled{1}, \circled{2}, \circled{3}\, ; \raiz{1} \, ;
(\Delta_{1\,3},0), \, (\Delta_{3\,2}, 1) 
\Big\} , \text{ with } \Delta_{1\,3}, \Delta_{3\,2} \in\HH^{2}_{-};
\end{equation}

\item $p_3\in U_4$, the $(3,0)$--configuration tree is
\begin{equation}\label{arbol-U4}
\Lambda_{X(p_3, z)} =
\Big\{ \circled{1}, \circled{2}, \circled{3}\, ; \raiz{1} \, ;
(\Delta_{1\,2},0), \, (\Delta_{1\,3}, 1) 
\Big\} , \text{ with } \Delta_{1\,3}\in\HH^{2}_{-}, \Delta_{3\,2} \in\HH^{2}_{+};
\end{equation}

\item $p_3\in U_6$, the $(3,0)$--configuration tree is
\begin{equation}\label{arbol-U6}
\Lambda_{X(p_3, z)} =
\Big\{ \circled{1}, \circled{2}, \circled{3}\, ; \raiz{1} \, ;
(\Delta_{1\,2},0), \, (\Delta_{2\,3}, 0) 
\Big\} , \text{ with } \Delta_{1\,2}=1 \text{ and } \Delta_{2\,3} \in\HH^{2}_{-};
\end{equation} 

\item $p_3\in U_8$, the $(3,0)$--configuration tree is
\begin{equation}\label{arbol-U8}
\Lambda_{X(p_3, z)} =
\Big\{ \circled{1}, \circled{2}, \circled{3}\, ; \raiz{1} \, ;
(\Delta_{1\,2},0), \, (\Delta_{2\,3}, 1) 
\Big\} , \text{ with } \Delta_{1\,2}=1 \text{ and } \Delta_{2\,3} \in\HH^{2}_{-};
\end{equation}

\item $p_3\in \partial U_{4,5}$, the $(3,0)$--configuration tree is
\begin{equation}\label{arbol-U45}
\Lambda_{X(p_3, z)} =
\Big\{ \circled{1}, \circled{2}, \circled{3}\, ; \raiz{1} \, ;
(\Delta_{1\,2},0), \, (\Delta_{1\,3}, 1) 
\Big\} , \text{ with } \Delta_{1\,2}=1 \text{ and } \Delta_{1\,3} >0;
\end{equation}

\item $p_3\in \partial U_{1,2}\cap (1/2,1)$, the $(3,0)$--configuration tree is
\begin{equation}\label{arbol-U12-1/2-1}
\Lambda_{X(p_3, z)} =
\Big\{ \circled{1}, \circled{2}, \circled{3}\, ; \raiz{1} \, ;
(\Delta_{1\,3},0), \, (\Delta_{3\,2}, 1) 
\Big\} , \text{ with } \Delta_{1\,3}>1 \text{ and } \Delta_{3\,2} < 0;
\end{equation}   

\item $p_3\in \partial U_{7,8} \cap (1,2)$, the $(3,0)$--configuration tree is
\begin{equation}\label{arbol-U78-1-2}
\Lambda_{X(p_3, z)} =
\Big\{ \circled{1}, \circled{2}, \circled{3}\, ; \raiz{1} \, ;
(\Delta_{1\,2},0), \, (\Delta_{2\,3}, 1) 
\Big\} , \text{ with } \Delta_{1\,2}=1 \text{ and } -1<\Delta_{2\,3} < 0;
\end{equation} 

\item $p_3\in \partial U_{1,6}$, the $(3,0)$--configuration tree is
\begin{equation}\label{arbol-U16}
\Lambda_{X(p_3, z)} =
\Big\{ \circled{1}, \circled{2}, \circled{3}\, ; \raiz{1} \, ;
(\Delta_{1\,3},0), \, (\Delta_{3\,2}, 0) 
\Big\} , \text{ with } 0<\Delta_{1\,3}<1 \text{ and } 0<\Delta_{3\,2}<1;
\end{equation} 

\item $p_3\in \partial U_{3,8}$, the $(3,0)$--configuration tree is
\begin{equation}\label{arbol-U38}
\Lambda_{X(p_3, z)} =
\Big\{ \circled{1}, \circled{2}, \circled{3}\, ; \raiz{1} \, ;
(\Delta_{1\,2},0), \, (\Delta_{2\,3}, 0) 
\Big\} , \text{ with } \Delta_{1\,2}=1 \text{ and } \Delta_{2\,3}>0;
\end{equation} 

\item $p_3\in \partial U_{2,3}$, the $(3,0)$--configuration tree is
\begin{equation}\label{arbol-U23}
\Lambda_{X(p_3, z)} =
\Big\{ \circled{1}, \circled{2}, \circled{3}\, ; \raiz{1} \, ;
(\Delta_{1\,3},0), \, (\Delta_{3\,2}, 0) 
\Big\} , \text{ with } 0<\Delta_{1\,3}<1 \text{ and } \Delta_{1\,3} > 0;
\end{equation}   

\item $p_3\in \partial U_{6,7}$, the $(3,0)$--configuration tree is
\begin{equation}\label{arbol-U67}
\Lambda_{X(p_3, z)} =
\Big\{ \circled{1}, \circled{2}, \circled{3}\, ; \raiz{1} \, ;
(\Delta_{1\,2},0), \, (\Delta_{2\,3}, 0) 
\Big\} , \text{ with } \Delta_{1\,2}=1 \text{ and } \Delta_{2\,3} > 0.
\end{equation} 

\end{itemize}
\end{proof}

As an advantage of Proposition \eqref{lema-arbol-3-0},
recalling the complex analytic action \eqref{LaAccion}, 
then the quotient map
$$
\begin{array}{rcl}
\Pi: \E(r,d) & \longrightarrow & \dfrac{\E(r,d)}{Aut(\CC)}
\\  
X & \longmapsto & [X]
\end{array}
$$
\noindent 
determines the analytic classes of
vector fields. In particular, 
$\E(3,0)/Aut(\CC)$
is a complex analytic space of complex dimension two, 
having singularities originated from 
the vector fields in $\E(3,0)$ with non trivial 
isotropy group under the $Aut(\CC)$--action.
See \cite{AlvarezMucino2} for general theory 
on $\E(r,d)$ and 
\cite{Frias-Mucino}, 
\cite{Magana} for the rational case. 
Hence the family $\mathfrak{F}$ gives origin to a
curve $\Pi \circ X (p_2, z )$ of $Aut(\CC)$--classes of vector fields.

\begin{corollary}
\begin{enumerate}[label=\arabic*),leftmargin=*]
 \item 
The complex analytic curve

\centerline{$
\Pi \circ X(p_3, z):
\CC \backslash \{-1, 0, 1/2,  1, 2 \}  
\longrightarrow
\dfrac{\E(3,0)}{Aut(\CC)}
$}

\noindent 
is injective.

\item 
The parameter plane 
$\CC_{p_3} \backslash \{0, \, 1/2, \, 1 \}$ 
provides a bifurcation diagram for the phase portraits
of $\Re{X(p_3, z)}$, there appear  
12 
non equivalent (topologically, 
orientation preserving).
\end{enumerate}
\end{corollary}

The notion of topological equivalence is Definition
\ref{deftopequiv}.
We shall study the general
problem of the number of topological classes of 
$\Re{X}$, for $X \in \E(r,d)$, in 
\S\ref{topoclassification}, see Theorem \ref{numbertop}.

\begin{proof}
We want to describe the $Aut(\CC)$--equivalent
the vector fields  $X(p_3,z)$.
Consider the action $(T,X(p_3, z)) \longmapsto T^* X(p_3,z)$, 
then it lifts to the action on ditinguished parameters, as 

\centerline{
$(T,\Psi(p_3,z)) \longmapsto 
T^* \Psi(p_3,z)= \Psi(p_3, z) \circ T^{-1}(z)$. }

\noindent 
The polynomial $\Psi(p_3, z)$ has critical points
$\{(0,0), (1,1), (p_3, \widetilde{p}_3) \} \subset \CC^2$.
For $T \neq Id \in Aut(\CC)$, 
$\Psi(p_3, z) \circ T^{-1}(z)$ 
gives rise to
a permutation of 
$\{(0,0), (1,1), (p_2, \widetilde{p}_3) \}$. 
This is possible if and only if $\widetilde{p}_3 = 0$ or $1$. 
Using the definition of $\widetilde{p}_3$, 
Equation  \eqref{funcion-racional-R}, 
we have $p_3= -1$ or $2$. 
In fact 
$\Psi(-1,z)$ and $\Psi(2, z)$ are $Aut(\CC)$--equivalent
(using the translation map $T^{-1}$ that sends $\{ -1, 0, 1\}$ to
$\{0, 1, 2 \}$).

The assertion (2) uses careful inspection of 
Figure \ref{EjemploArbolesE30}.
We convene that $\sim$ means topologically equivalence, 
for $p_3$  the topologies correspond to:

\noindent 
$\bigcdot$
$U_1, \ldots ,  U_8$ eight topologies; clearly 
the topology remains without change for $p_3$
on each open set. 
For example, 
$U_1$, $U_2$ determine two horizontal strip flows for 
the corresponding $\Re{X(p_3, z)}$,
however the orientation of the flows make it different.

\noindent
$\bigcdot$
$\partial U_{1,2} \cap (0, 1/2) \sim \partial U_{4,5}$,
$\partial U_{7,8} \cap (1,1/2) \sim 
\partial U_{7,8} \cap (2, \infty)$, two topologies;
in fact the corresponding $\Re{X(p_3, z)}$ 
have two saddle connections, the orientation determines
two non equivalent cases. 

\noindent
$\bigcdot$
$\partial U_{3,4}\sim
\partial U_{1,6}\sim 
\partial U_{3,8}$,
$\partial U_{5,6}\sim 
\partial U_{2,3}\sim
\partial U_{6,7}$, two topologies; 
the corresponding $\Re{X(p_3, z)}$ 
have two saddle connections 
that have a common half plane as boundary
(this is the difference respect to the above case), and
the orientation determines two non equivalent cases. 

\noindent 
Hence we have 12 different 
topological classes.
\end{proof}

\subsection{$ X\in\E(0,d)$ has an isolated essential singularity at $\infty\in\CW_{z}$, no zeros or poles}
The simplest example corresponds to a $(0,1)$--configuration tree; 
only one essential vertex and no edges.

\begin{example}\label{ejemplo-un-va}
Consider once again Example \ref{ejemplolog}, that is

\centerline{
$X(z)= \e^{\mu (z + c_1)}\del{}{z} \in \E(0,1)$,}

\noindent
with $\mu\in\CC^*$, $c_1\in\CC$ as in \eqref{coeficientes-P-y-E}.
We then have an isolated essential singularity at $\infty\in\CW_{z}$ 
with finite asymptotic value $a_1$ 
given by \eqref{va-exponencial}.

\noindent
The $(0,1)$--configuration tree consists of one essential vertex and no edges

\centerline{
$\Lambda_{X}=\Big\{ 
\circled{1}=(\infty_{1},a_{1},
-\infty
); 
\raiz{1}\, ;\,
\varnothing \Big\}.$} 

\noindent 
See Figure \ref{figejemplo-un-va} and \ref{figPiezasBasicas}.a. 
The vertices are branch points of the Riemann surface $\R_{X}$, and since there is only
one branch point/vertex, then no weighted edges appear.
The soul of $\R_X$ is a sheet, recall Remark \ref{primera-vez-soul}, and coincides with the 
global zero level. 
It is to be noted that the semi--infinite helicoids, even though they are part of $\R_{X}$, are not
necessary in the combinatorial description of the surface;
the complete $\R_{X}$ is described
by making the natural convention to glue two semi--infinite helicoids to each vertical tower,
one on the top and one on the bottom. 
Hence the semi--infinite helicoids will have no counterpart in the combinatorial description as graphs.
However, to remind the reader of their existence in $\R_{X}$ we have schematically 
represented them in the figures by the ``springs'' or ``coils'' attached to the vertical towers.

\begin{figure}[htbp]
\begin{center}
\includegraphics[width=0.9\textwidth]{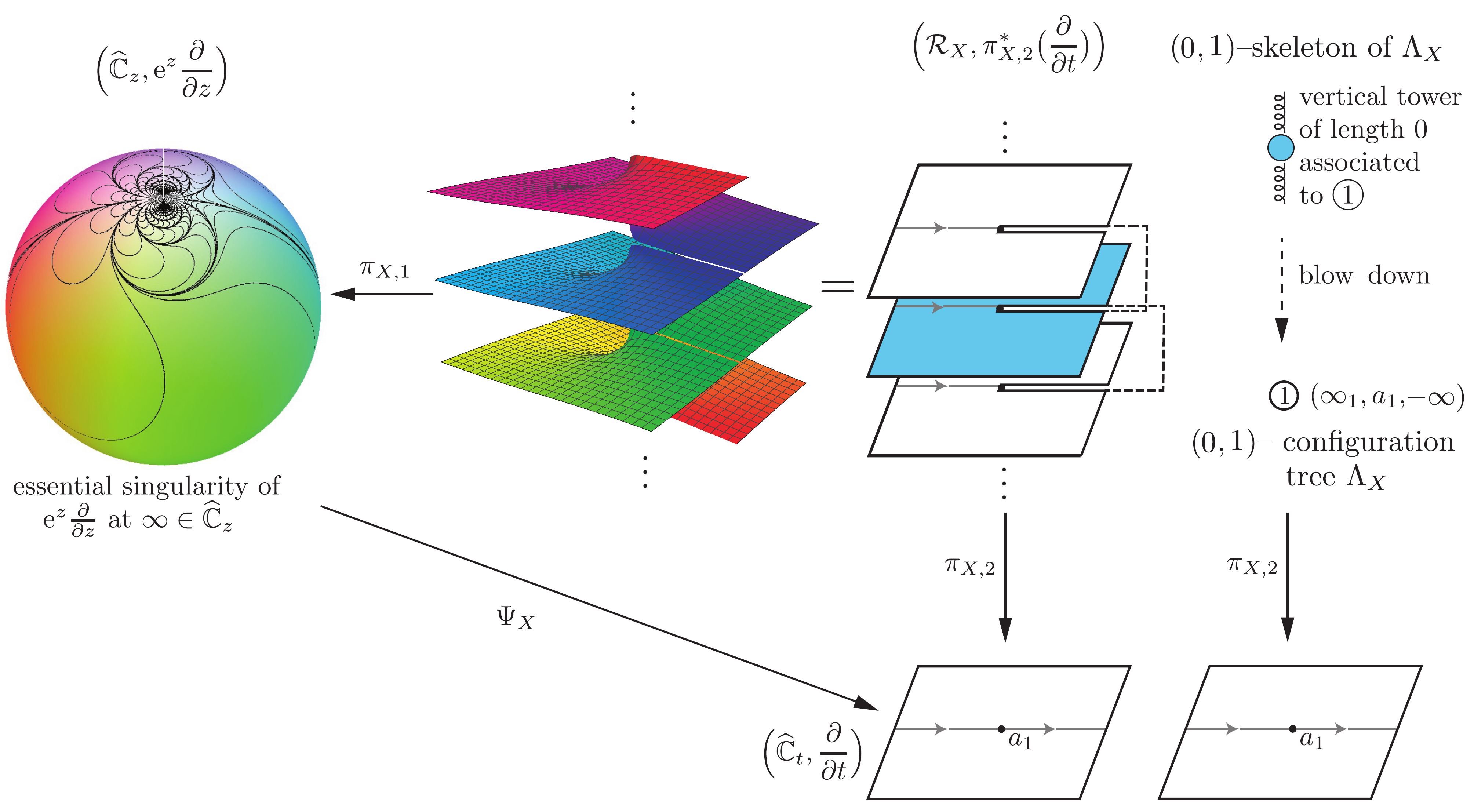}
\caption{{\bf Vector field 
$X(z)= \e^{\mu (z + c_1)}\del{}{z}$ 
with an essential singularity at $\infty\in\CW_{z}$.}
The surface $\R_{X}$ is 
a \emph{logarithmic spiral}
formed by two semi--infinite helicoids 
glued together.
The soul, Definition \ref{soul}, is shaded blue.
See Example \ref{ejemplo-un-va}, and \S\ref{confTree-to-skeleton} 
for the drawing on the right.
}
\label{figejemplo-un-va}
\end{center}
\end{figure}
\end{example}

In the following example there are $d\geq 2$ essential vertices and $d-1$ edges, all sharing
the same sheet.

\begin{example}\label{ejemplo-zd}
Consider the vector field 
\\
\centerline{
$X(z)= \e^{z^{d}}\del{}{z}
\in \E(0, d)$, for $d\geq2$.}
If $z_{0}=0$, the distinguished parameter is 
\\
\centerline{
$\Psi_{X}(z)=\int\limits_{0}^{z}\e^{-\zeta^{d}} d\zeta$.} 
Note that $a_1\doteq \int\limits_0^{\infty} \e^{-\zeta^d} d\zeta\in\RR^+$.
Moreover $\Psi_{X}$ has $d$ finite asymptotic values given by\footnote{
Our numbering of the indices $\sigma$ differ from the ones in \cite{Nevanlinna1} 
so that they agree with the conventions outlined in Remark \ref{correspexponentialtract}.3.
} 
(see \cite{Nevanlinna1} p.~168)

\centerline{
$a_\sigma= \e^{i2\pi (\sigma-1)/d} a_1$ \quad for \quad  $\sigma=1\ldots d$, 
}

\noindent
each with multiplicity one, and
a logarithmic branch point 
$(\infty_{\sigma},a_{\sigma}, -\infty ) \in\R_{X}$ 
over each finite asymptotic value. 
The exponential tracts $A_\sigma$ for each of the finite asymptotic values $a_\sigma$ are given by 

\centerline{
$A_{\sigma}=\left\{z\in\CC_{z} \ \vert\ 
\abs{\arg{z} - \frac{2 \pi (\sigma-1)}{d} } < \frac{\pi}{d} \right\}$,
\ \ \ for $\sigma=1,\ldots,d$.
}

\noindent
Thus the $(0,d)$--configuration tree $\Lambda_{X}$ will have $d$ essential vertices 
$V_H=\{ \circled{\sigma}=(\infty_{\sigma},a_{\sigma},-\infty) \}_{\sigma=1}^{d}$
with root $\raiz{1}=(\infty_1,a_1,-\infty)$.
Recalling Definition \ref{arbol-lineal-enraizado}.1, 
the $d-1$ edges $E_H$ are selected such that we obtain a 
left--right--top--bottom linear directed tree of the vertices $V_H$, 
where the index $\sigma_0$ for the top and left most vertex will be
given by the simple formula 
$$\sigma_0=\ceil*{\floor*{\dfrac{d}{2}} / 2} + 1,$$ 
where $\ceil{\cdot}$ and $\floor{\cdot}$ are the ceiling and floor functions respectively.
Recalling Definition \ref{arbol-lineal-enraizado}.2,  
we obtain $\Lambda_{H,\circled{1}}$ the linear directed rooted tree with incoming
vertex $\circled{1}$. 
Moreover, all $d$ branch points share the same sheet in $\R_{X}$, 
hence by assigning weight 0 to each of the edges 
$\Delta_{\msigma \mrho}\in \widehat{E}_H$
we obtain the $(0,d)$--configuration tree of $X$.
In Figure \ref{fig-ejemplo-z9} the case $d=9$ is illustrated, the $(0,9)$--configuration 
tree is 
\begin{multline}
\Lambda_{X}= \Big\{ \circled{1},\ldots,\circled{9}\,;\raiz{1}\,; 
(\Delta_{1\,5},0), (\Delta_{5\,2},0), (\Delta_{2\,4},0), (\Delta_{4\,3},0), 
\\	
(\Delta_{1\,6},0), (\Delta_{6\,9},0), (\Delta_{9\,7},0), (\Delta_{7\,8},0)
\Big\}.
\end{multline}
\begin{figure}[htbp]
\begin{center}
\includegraphics[width=\textwidth]{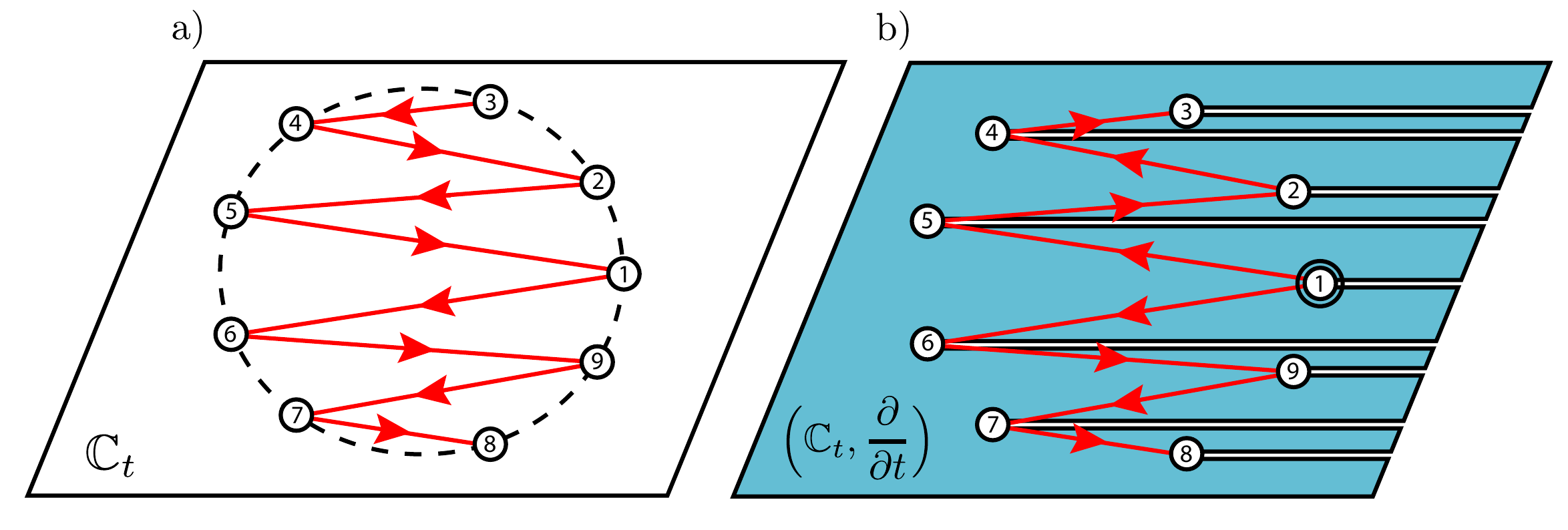}
\caption{{\bf Vector field 
$X(z)=\e^{z^{9}}\del{}{z}$ 
with essential singularity at $\infty$ 
and no poles.}
In this case the top and left most vertex is $\circled{\sigma_0}$ with 
$\sigma_0=\ceil{ \floor{9/2}/2}+1=3$.
(a) Represents the left--right--top--bottom linear directed tree of the vertices
$V_H=\big\{ \circled{1},\ldots,\circled{9} \big\}$.
(b) Represents the $(0,9)$--configuration tree $\Lambda_{X}$ where all the edges have weight 0.
The soul, Definition \ref{soul}, is shaded blue. 
The Riemann surface $\R_{X}$ consists of gluing $18$ semi--infinite helicoids, 
one above and one below each logarithmic branch point/vertex.
All the branch points and diagonals
belong to the global zero level sheet.
}
\label{fig-ejemplo-z9}
\end{center}
\end{figure}

\noindent 
The soul of $\R_X$ is a sheet and coincides with the 
global zero level. 
\end{example}

The following example is a family of vector fields in $\E(0,3)$ whose $(0,3)$--configuration
trees have two edges, one of which can assume a non--zero weight.
Recall that the weight of an edge indicates the number of sheets 
one has to go up or down, at the starting branch point, in order to reach the sheet that shares 
both the starting and ending branch point of the diagonal/edge ({\it i.e.} the sheet containing the
diagonal).

\begin{example}\label{ejemplo-Exp3}
Consider the vector field 
\begin{equation}\label{AnaliticExp3}
X(z)=2\pi i \exp\left(-\frac{1}{3}z^{3}+c_2 z\right)\del{}{z}, 
\ \ \ 
c_2\in\CC.
\end{equation}

\noindent 
We shall need some background on Airy functions and integrals, for full details see
\cite{AlvarezMucino}, pp.~200--203 and references therein.
Let
$$\mathcal{A}i(p,z)=\frac{1}{2\pi i}\int\limits_{\mathcal{L}(z)}  
\e^{\frac{1}{3}\zeta^{3}-p\zeta} d\zeta,
$$
be the \emph{Airy integral}, where $A=\{z\in\CC  \ \vert \ \arg(z)\in(\pi/6,3\pi/6)\}$ and
$\mathcal{L}(z):=\mathcal{L}(z,\tau):[0,1]\longrightarrow A$ is a simple $C^{1}$ path starting at $0$ 
and ending at $z\in A$. 
The relationship between the Airy function $\text{Ai}$ and the Airy integral is given by
\begin{equation}\label{defAiry}
\text{Ai}(p)=\mathcal{A}i(p)- \e^{-i2\pi/3}\mathcal{A}i( \e^{-i2\pi/3}p),
\ \ \ \ \ \  \ \ 
\mathcal{A}i(p)=\lim\limits_{\substack{z\to\infty\\ z\in A}} \mathcal{A}i(p,z).
\end{equation}

\noindent
Choosing $z_0=0$, the distinguished parameter of $X$ is

\centerline{
$\Psi_{X}(z)=\int\limits_{0}^{z} \omega_{X}
= \dfrac{1}{2\pi i} \int\limits_{0}^{z} \e^{\frac{1}{3}\zeta^{3}-c_2\zeta} d\zeta$,}

\noindent
so the 3 finite asymptotic values of $\Psi_{X}(z)$ are given by
\begin{equation}\label{AV3}
a_{j+1}(c_2)= \eta^{j} \mathcal{A}i(\eta^{j} c_2), \qquad 
j=0,1,2,
\quad \eta=\e^{i2\pi/3},
\end{equation}
with asymptotic paths ending in the exponential tracts 
$\eta^{j} A$,
for $j=0$, $1$ and $2$ respectively.

\noindent
The $(0,3)$--configuration trees for $X$ as in \eqref{AnaliticExp3} are
\begin{multline}
\Lambda_{X}=\Big\{ 
\circled{1}= (\infty_1,a_1,-\infty),
\circled{2}=(\infty_2,a_2,-\infty),
\circled{3}=(\infty_3,a_3,-\infty) ; 
\\
\raiz{1}\,; 
(\Delta_{1\,2},0),(\Delta_{1\,3},K(1,3)) 
\Big\},
\end{multline}
\noindent
where
\begin{equation}\label{vad3}
\begin{array}{c}
\Delta_{1\,2}=a_{2}-a_{1}=  \eta\, 
\text{Ai}(\eta\, c_2),
\\[6pt]
\Delta_{1\,3}= a_{3}-a_{1}= - \text{Ai}(c_2),
\end{array}
\end{equation}
with $K(1,3)\in\ZZ$. 
See Figure \ref{fig-Exp3}. 
The dependency of $\Delta_{1\, 2}$ and $\Delta_{1\, 3}$ on $c_2$ is clear from \eqref{vad3},
however the dependency of $K(1,3)$ on $c_2$ is much more intricate, 
any $K(1,3) \in \ZZ$ appears:
for a full description see \cite{AlvarezMucino} \S8.6.1, particularly figure 14.

\begin{figure}[htbp]
\begin{center}
\includegraphics[width=\textwidth]{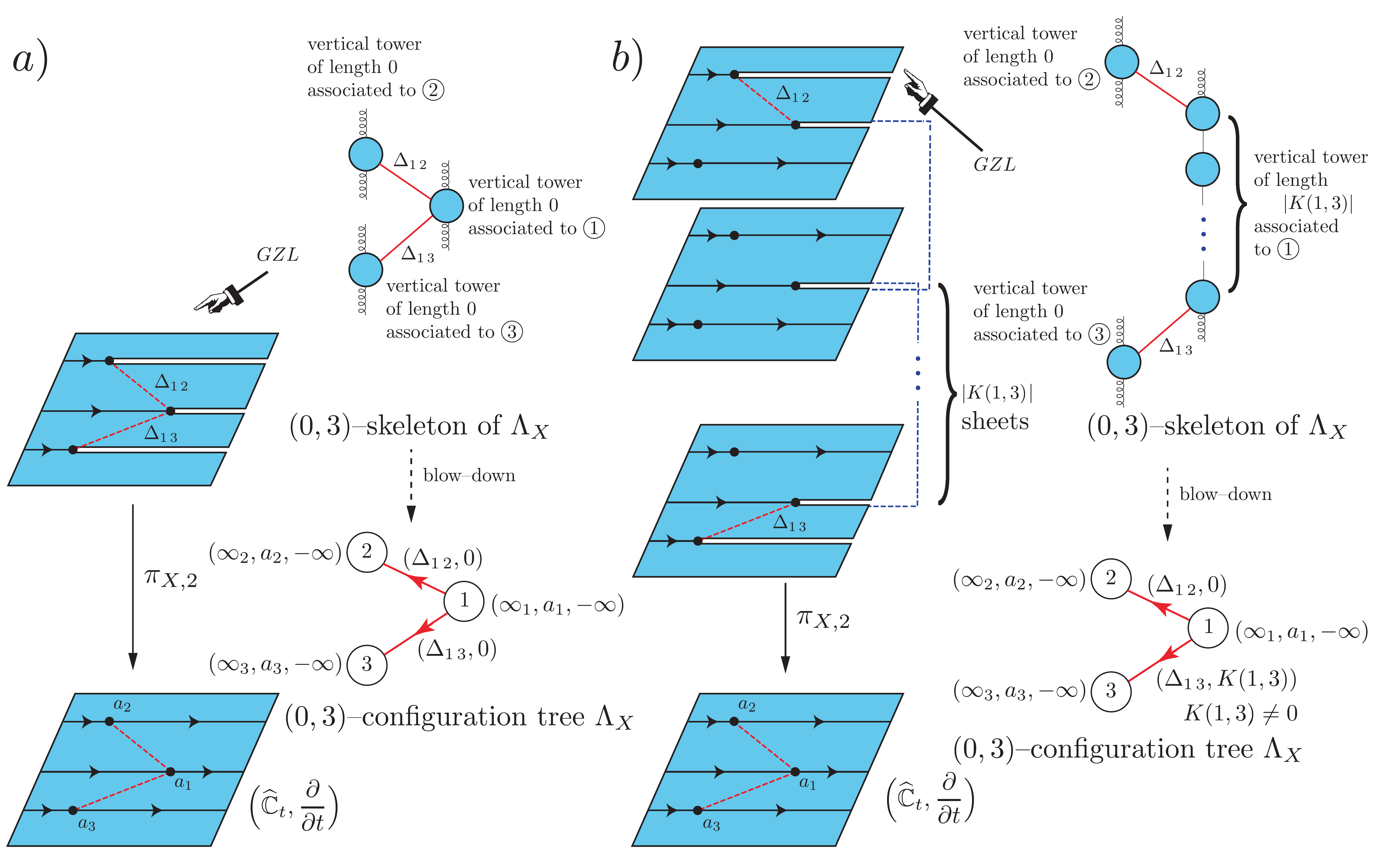}
\caption{
{\bf Vector field 
$X(z)= 2\pi i \exp\left(-\frac{1}{3}z^{3}+c_2 z\right)\del{}{z}\in\E(0,3)$ 
with essential singularity at $\infty$ 
and no poles.}
In (a) we have the case when 
$K(1,2)=K(1,3)=0$ 
so the diagonals $\Delta_{1\, 2}$ and $\Delta_{1\, 3}$
share the same sheet of $\R_{X}$.
Case (b) is when 
$K(1,2)=0$ and 
$K(1,3)\in \ZZ \backslash \{ 0 \}$, 
thus the diagonals $\Delta_{1\, 2}$ and $\Delta_{1\, 3}$ lie
on two different sheets of $\R_{X}$.
When starting on the global zero level sheet (the one that contains the diagonal 
$\Delta_{1\,2}$) 
one needs to go around the branch point
$(\infty_1, a_1, -\infty)$ exactly $K(1,3)$ times (see Remark \ref{pesos-como-info-geometrica}.4) 
in order to reach the sheet containing the diagonal
$\Delta_{1\, 3}$ (thus in this figure $K(1,3)<0$).
Note that for case (b) the global zero level sheet is non canonical:
if we instead choose as the global zero level sheet the one containing the diagonal $\Delta_{1\,3}$,
then the integer parameters would be $K(1,3)=0$ and $K(1,2)>0$.
}
\label{fig-Exp3}
\end{center}
\end{figure}
\end{example}

\subsection{$X\in\E(r,d)$ has $r\geq1$ poles on $\CC_{z}$ and an isolated essential singularity at $\infty\in\CW_{z}$}

The next example shows a simple case where 
the soul is non--trivial: it consists of more than one sheet. 

\begin{example}\label{ejemplo-alma-no-trivial}
Consider the vector field 
\\
\centerline{
$X(z)=\dfrac{\e^{z}}{(z-p_{1}) (z-p_{2})}\ddel{}{z}
\in \E(2, 1)$, 
}
with $p_1= 9 i\frac{\pi}{2}$ and $p_2= -i\frac{\pi}{2}$.
Its distinguished parameter is then
\\
\centerline{
$\Psi_{X}(z)=\int\limits_{z_0}^{z} \omega_{X} 
= \frac{\e^{-z}}{4} \big(-8 - 9 \pi^2 + 16 i \pi (1 + z) - 4 z (2 + z)\big)$. }
The vector field $X$ has an isolated essential singularity at $\infty\in\CW_{z}$ 
and $\Psi_{X}$ has one finite asymptotic value 
\\
\centerline{
$a_{1}=\Psi_{X}(\infty)=0$}

\noindent 
with asymptotic path inside 
the exponential tract
$\{z\in\CC_{z} \ \vert \ \Re{z}>0\}$. 
The poles $p_1, p_2$ have associated critical values $\widetilde{p}_1=-5\pi+2i$ 
and $\widetilde{p}_2=-5\pi-2i$ respectively.
\\
According to the labelling conventions in Definition
\ref{d-confTree},
the 
$(2,1)$--configuration tree 
\begin{multline}
\Lambda_{X}=\Big\{ \circled{1}=(p_1,\widetilde{p}_1,-1),
\circled{2}=(\infty_1,a_1,-\infty),
\circled{3}=(p_2,\widetilde{p}_2,-1) ;
\\
\raiz{1}\, ;\,
(\Delta_{1\, 2}, 0), (\Delta_{2\, 3}, -3)  
\Big\},
\end{multline}
has two pole vertices $\circled{1}$, $\circled{3}$, an essential vertex $\circled{2}$, and 
two edges 

\centerline{
$\Delta_{1\, 2}=a_{1}-\widetilde{p}_{1}=5\pi-2i$ \quad with weight $0$,} 
and 

\centerline{
$\Delta_{2\, 3}=\widetilde{p}_{2}-a_{1}=-5\pi-2i$ \quad with weight $-3$.}

\noindent
Note that the $(2,1)$--soul of $\Lambda_{X}$ is the Riemann surface consisting of 
12 half planes, 2 finite height horizontal strips, two cone points with cone angle $4\pi$ 
(corresponding to the two simple poles), a cone point with cone angle $8\pi$ (corresponding 
to the infinitely ramified branch point 
$(\infty_1,a_1, -\infty)$) and a horizontal branch cut starting at 
the branch point $(\infty_1,a_1, -\infty)$.
See Figures \ref{figejemplo-alma-no-trivial} and \ref{figPiezasBasicas}.
\begin{figure}[htbp]
\begin{center}
\includegraphics[width=\textwidth]{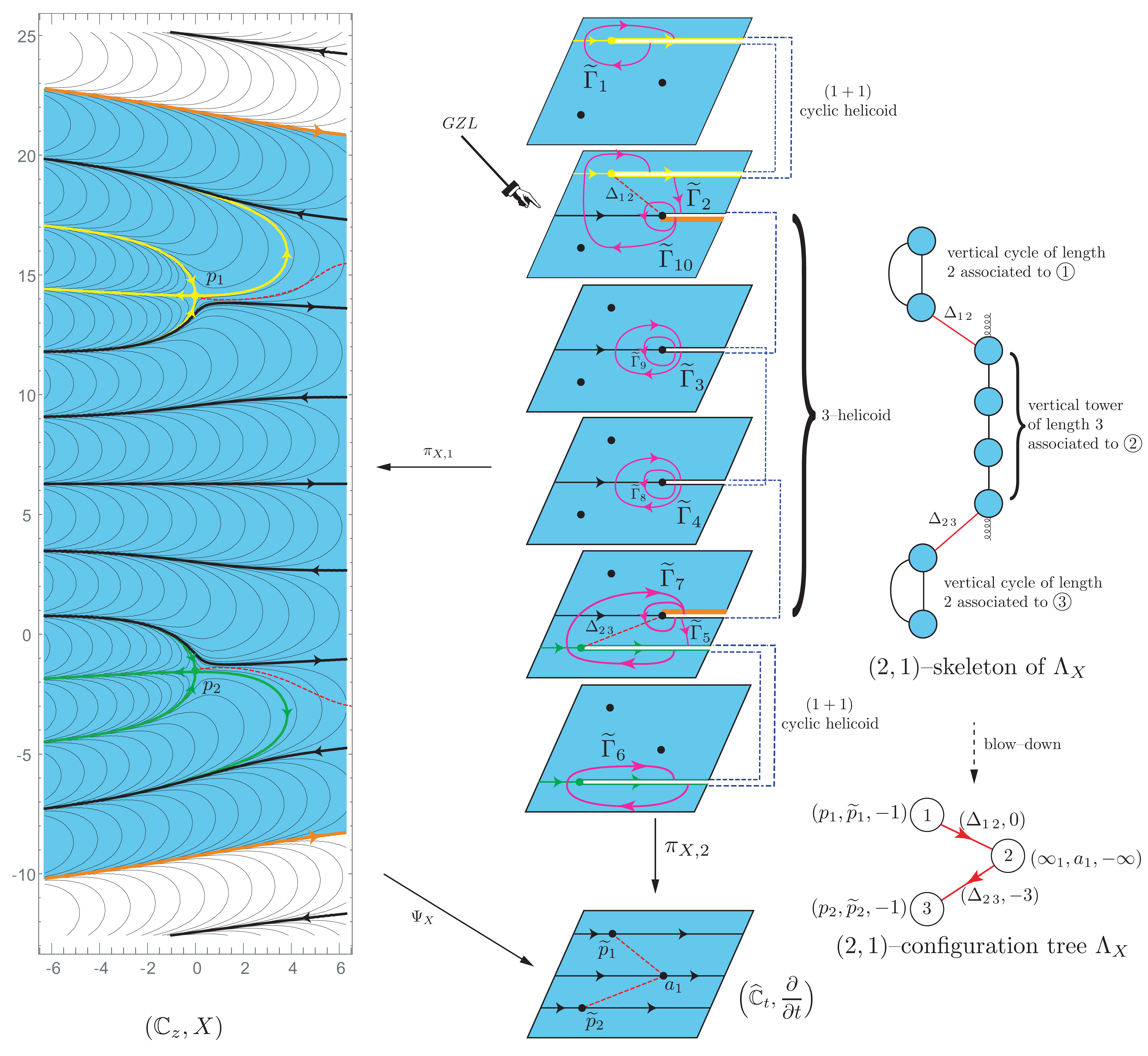}
\caption{{\bf Vector field 
$X(z)=\dfrac{\e^{z}}{(z-9i\frac{\pi}{2}) (z+i\frac{\pi}{2})}\del{}{z}$ 
with essential singularity at $\infty$ 
and two simple poles at $p_1=9i\frac{\pi}{2}$ and $p_2=-i\frac{\pi}{2}$.}
The Riemann surface $\R_{X}$ consists of two semi--infinite helicoids, two $(1+1)$--cyclic helicoids
and a $3$--helicoid. 
The soul, Definition \ref{soul}, is shaded blue and consists of the Riemann surface $\R_{X}$ minus
the semi--infinite helicoids. The boundary (where the two semi--infinite helicoids should be glued) 
is coloured orange.}
\label{figejemplo-alma-no-trivial}
\end{center}
\end{figure}
\end{example}

The next examples present $(r,d)$--configuration trees with more than 
one weighted edge. 


\begin{example}\label{ejemplo-E23}
Consider the vector field 
\\
\centerline{
$X(z)=-\dfrac{\e^{z^{3}}}{3 z^{2}}\ddel{}{z}
\in \E(2, 3)$.}
If $z_{0}=0$ the distinguished parameter is
\\
\centerline{
$\Psi_{X}(z)=\e^{-z^{3}}-1$.} 
Thus the pole $p_{1}=0$ has order $-\nu_{1}=-2$ and critical value $\widetilde{p}_{1}=0$, while
the e\-ssen\-tial singularity at $\infty\in\CW_{z}$ 
has finite asymptotic value $a_{1}=-1$, 
with multiplicity 3. 
In order to distinguish each of the three finite asymptotic values, one has to consider the 
asymptotic path associated to each and see on which of the following exponential 
tracts each lies in
\begin{equation}\label{exptract1}
\begin{array}{cc}
A_{1}=\{ z\in\CC  \, \vert \arg(z)\in (-\pi/6,\pi/6)\, \}, &
A_{2}=\{ z\in\CC  \, \vert \arg(z)\in ( \pi/2,5\pi/6)\, \}, \\
A_{3}=\{ z\in\CC  \, \vert \arg(z)\in (7\pi/6,3\pi/2)\, \}. 
\end{array}
\end{equation}
That is $(\infty_{1},-1, -\infty), (\infty_{2},-1,-\infty), (\infty_{3},-1, -\infty)\in\R_{X}$ are 3 logarithmic
branch points corresponding to the above exponential tracts as in Definition \ref{singesen}.
\\
The $(2,3)$--configuration tree has three essential vertices, and one pole vertex.
According to Definition \ref{d-confTree}.4, since $r\neq 0$ the root must be
the unique pole vertex, so we 
conveniently label the vertices as follows
\begin{equation*}
\begin{array}{rclrcl}
\circled{1}=(z_{1},t_{1},-\nu_{1}) &=& (p_{1},\widetilde{p}_{1},-2), & 
\quad\circled{2}=(z_{2},t_{2},-\nu_{2}) &=& (\infty_{1},a_{1},-\infty),
\\
\circled{3}=(z_{3},t_{3},-\nu_{3}) &=& (\infty_{2},a_{1},-\infty), &
\quad\circled{4}=(z_{4},t_{4},-\nu_{4}) &=& (\infty_{3},a_{1},-\infty).
\end{array}
\end{equation*} 

\noindent 
In this way the $(2,3)$--configuration tree is
\begin{equation}
\Lambda_{X}=\Big\{ \circled{1}, \circled{2}, \circled{3}, \circled{4};
\raiz{1}\, ;\,
(\Delta_{1\, 2}, 0), (\Delta_{1\, 3}, 1), (\Delta_{1	\, 4}, -1) 
\Big\}.
\end{equation}
Once again the edges $\Delta_{1\, 2}$, $\Delta_{1\, 3}$ and $\Delta_{1\, 4}$ correspond to 
the diagonals/semi--residues between the branch points associated to the corresponding 
vertices:
\begin{equation}
\begin{array}{lr}
\Delta_{1\, 2}=\int\limits^{\infty_{1}}_{p_{1}} \omega_{X}
=a_{1}-\widetilde{p}_{1}=-1, &
\Delta_{1\, 3}=\int\limits^{\infty_{2}}_{p_{1}} \omega_{X}=a_{1}-\widetilde{p}_{1}=-1, \\
\Delta_{1\, 4}=\int\limits^{\infty_{3}}_{p_{1}} \omega_{X}=a_{1}-\widetilde{p}_{1}=-1.  
\end{array}
\end{equation}
The weight $K(1,2)=0$ since the branch points corresponding to the root $\raiz{1}$ and the
vertex $\circled{2}$ share the same sheet in $\R_{X}$, 
$K(1,3)=1$ since in order to reach the diagonal $\Delta_{1\,3}$ one has to go up one sheet, 
at the branch point corresponding to $\circled{1}$ ({\it i.e.} $\Delta_{1\,3}$ is one sheet above 
the diagonal $\Delta_{1\,2}$ in $\R_{X}$), and 
$K(1,4)=-1$ since in order to reach the diagonal $\Delta_{1\,4}$ one has to go down one sheet, 
at the branch point corresponding to $\circled{1}$ ({\it i.e.} $\Delta_{1\,4}$ is one sheet below 
the diagonal $\Delta_{1\,2}$ in $\R_{X}$).

\noindent
Recalling Remark \ref{remArbol}.2 and the definition of semi--residues \eqref{diagonalsemiresidue}, 
note that it is possible to calculate the weights 
$K(\msigma,\mrho)$ by considering the phase portrait of $X$: the path of integration from
$p_{1}$ to $\infty_{1}$ stays on the same angular sector about $p_{1}$ so $K(1,2)=0$, 
the path of integration from $\infty_{1}$ to $\infty_{2}$ necessarily crosses two angular sectors 
going counterclockwise around $p_{1}$ corresponding to going up a sheet in $\R_{X}$ 
so $K(1,3)=1$, 
and the path of integration from $\infty_{1}$ to $\infty_{3}$ crosses two angular sectors 
going clockwise around $p_{1}$ corresponding to going down a sheet in $\R_{X}$ 
so $K(1,4)=-1$.

\noindent
In this case the decomposition, provided by Lemma \ref{lema-arboles-horizontales}, 
into horizontal subtrees is

\centerline{
$\Lambda_{X} 
= \Lambda_{H(1)} 
\cup \Lambda_{H(1,3)}
\cup \Lambda_{H(1,4)}
$,
}

\noindent
where
\begin{equation*}
\begin{array}{ll}
\Lambda_{H(1)}=\Big\{ \circled{1},\circled{2};\raiz{1};(\Delta_{1\, 2},0) \Big\}, &
\Lambda_{H(1,3)}=\Big\{ \circled{1},\circled{3};\raiz{1};(\Delta_{1\, 3},1) \Big\},
\\[4pt]
\Lambda_{H(1,4)}=\Big\{ \circled{1},\circled{4};\raiz{1};(\Delta_{1\, 4},-1) \Big\}.
\end{array}
\end{equation*}
See Figure \ref{figejemplo-E23} and the left hand side of Figure \ref{fig7ejemploscampos}.

\begin{figure}[htbp]
\begin{center}
\includegraphics[width=\textwidth]{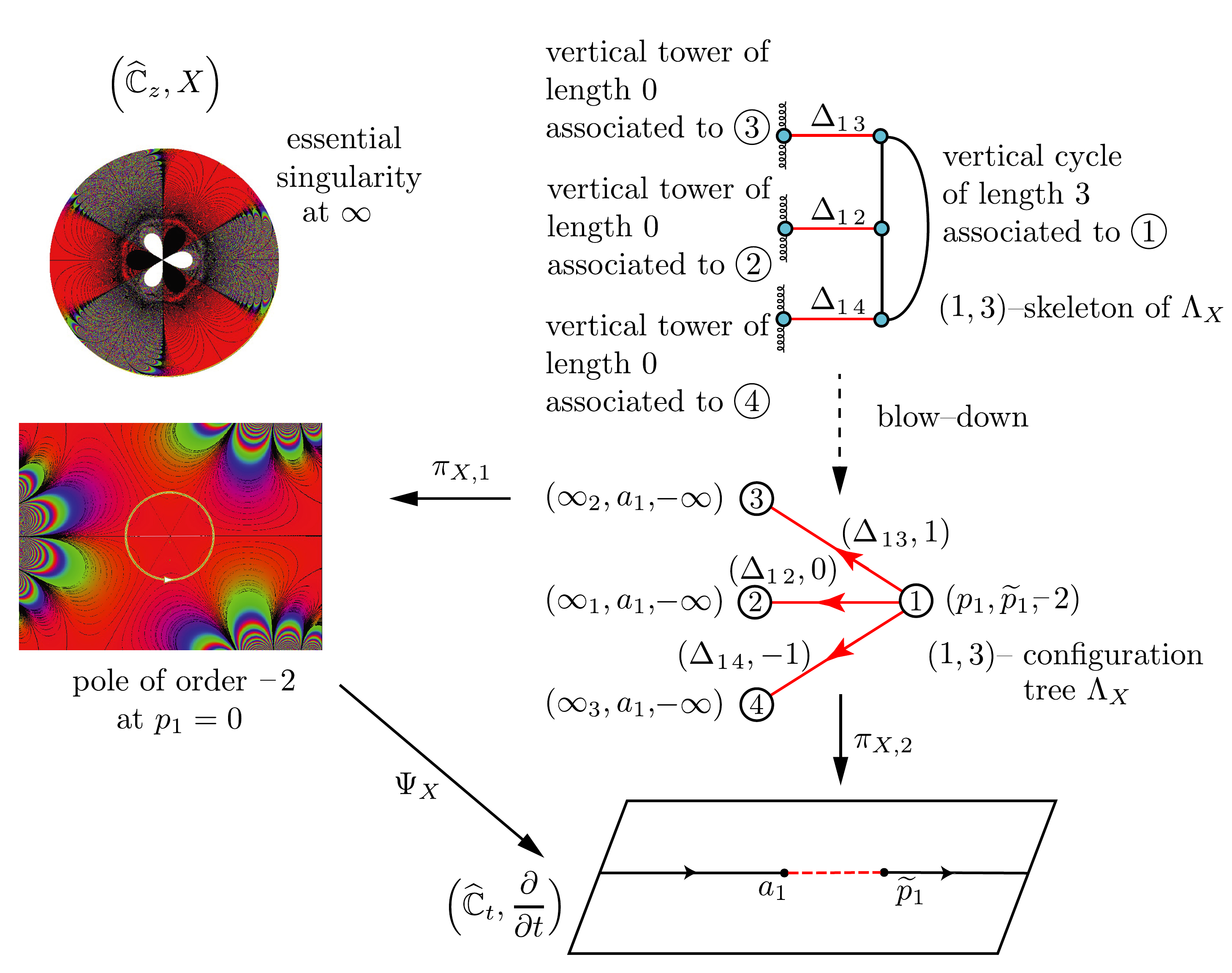}
\caption{{\bf Vector field 
$X(z)=-\dfrac{\e^{z^{3}}}{3 z^{2}}\del{}{z}$ with an essential singularity at $\infty$ 
and pole $p_{1}=0$ of order $-2$.}
The diagonals  
$\Delta_{1\,2}$, $\Delta_{1\,3}$ and $\Delta_{1\,4}$,
and their projections are shown as red segments. 
The Riemann surface $\R_{X}$ is not drawn.
See Example \ref{ejemplo-E23}, and 
\S\ref{confTree-to-skeleton} for the right drawing.
}
\label{figejemplo-E23}
\end{center}
\end{figure}
\end{example}

\begin{example}\label{ejemplo-E33}
In a similar vein as the previous example consider the vector field
\\
\centerline{
$
X(z)=\dfrac{\e^{z^{3}}}{3z^{3}-1}\ddel{}{z} \in \E(3,3),
$}
with simple poles at $p_{1}=\dfrac{1}{\sqrt[3]{3}}$, $p_{2}= \dfrac{\e^{i 2\pi/3}}{\sqrt[3]{3}}$, 
$p_{3}= \dfrac{\e^{-i 2\pi/3}}{\sqrt[3]{3}}$,
and an essential singularity at $\infty\in\CW_{z}$.
Its distinguished parameter is
\\
\centerline{
$\Psi_{X}(z)=\int\limits_{0}^{z} \omega_{X} 
= -z \e^{-z^{3}}$. }

\noindent 
Thus the critical values corresponding to the poles are

\centerline{
$\widetilde{p}_{1}=-\dfrac{1}{\sqrt[3]{3\e}}$, 
\
$\widetilde{p}_{2}= -\dfrac{\e^{i 2\pi/3}}{\sqrt[3]{3\e}}$ 
\ and \ 
$\widetilde{p}_{3}= -\dfrac{\e^{-i 2\pi/3}}{\sqrt[3]{3\e}}$.
}

\noindent
The essential singularity at $\infty$ has 
$a_{1}=0$ as its finite asymptotic value with multiplicity 3, once again with the same exponential tracts as the 
previous example, see equation \eqref{exptract1},
hence $(\infty_{1},0, -\infty), 
(\infty_{2},0, -\infty), (\infty_{3},0, -\infty)\in\R_{X}$ are the 3 logarithmic  
branch points corresponding to the mentioned exponential tracts.
\\
The $(3,3)$--configuration tree has three essential vertices and three pole vertices,
which we 
label as
\begin{equation}\label{verticesejemplo-E33}
\begin{array}{rclrcl}
\circled{1}=(z_{1},t_{1},-\nu_{1}) &=& (p_{2},\widetilde{p}_{2},-1),&
\quad\circled{2}=(z_{2},t_{2},-\nu_{2}) &=& (p_{1},\widetilde{p}_{1},-1),
\\
\circled{3}=(z_{3},t_{3},-\nu_{3}) &=& (p_{3},\widetilde{p}_{3},-1),&
\quad\circled{4}=(z_{4},t_{4},-\nu_{4}) &=& (\infty_{1},a_{1},-\infty),
\\
\circled{5}=(z_{5},t_{5},-\nu_{5}) &=& (\infty_{2},a_{1},-\infty),&
\quad\circled{6}=(z_{6},t_{6},-\nu_{6}) &=& (\infty_{3},a_{1},-\infty).
\end{array}
\end{equation}

\noindent
Thus the $(3,3)$--configuration tree (see Figure \ref{figejemplo-E33}) is\footnote{
The root is $\raiz{3}$ so as to agree with the conventions of Definition \ref{d-confTree}.}
\begin{multline}\label{arbolejemplo-E33}
\Lambda_{X}=
\Big\{ \circled{1},\circled{2},\circled{3},\circled{4},\circled{5},\circled{6};
\raiz{3}\, ;\,
\\
(\Delta_{3\, 6}, 0), 
(\Delta_{3\, 2}, 1),
(\Delta_{2\, 4}, 1),
(\Delta_{4\, 1}, 0), 
(\Delta_{1\, 5}, 1) 
\Big\},
\end{multline}
with edges given by the diagonals/semi--residues
\begin{equation}\label{pesosejemplo-E33}
\begin{array}{ll}
& \Delta_{3\, 6} = \int\limits_{p_{3}}^{\infty_{3}} \omega_{X} = a_{1} - \widetilde{p}_{3} = 
\dfrac{\e^{- i 2\pi/3}}{\sqrt[3]{3\e}},
\\
& \Delta_{3\, 2} = \int\limits_{p_3}^{p_1} \omega_{X} = (\widetilde{p}_1 - \widetilde{p}_3) = 
-\Big(\dfrac{1-\e^{-i 2\pi/3}}{\sqrt[3]{3\e}}\Big), 
\\
& \Delta_{2\, 4} = \int\limits_{p_1}^{\infty_1} \omega_{X} =  a_1 - \widetilde{p}_1= 
\dfrac{1}{\sqrt[3]{3\e}},
\\
& \Delta_{4\, 1} = \int\limits_{\infty_1}^{p_{2}} \omega_{X} =  \widetilde{p}_{2} - a_1 =
-\dfrac{1}{\sqrt[3]{3\e}} \e^{i 2\pi/3},
\\
& \Delta_{1\, 5} = \int\limits_{p_{2}}^{\infty_{2}} \omega_{X} = a_{1} - \widetilde{p}_{2} = 
\dfrac{\e^{ i 2\pi/3}}{\sqrt[3]{3\e}}, 
\end{array}
\end{equation}
and weights $K(3,6)=0$, $K(3,2)=1$, $K(2,4)=1$, $K(4,1)=0$, $K(1,5)=1$.
\begin{figure}[htbp]
\begin{center}
\includegraphics[width=\textwidth]{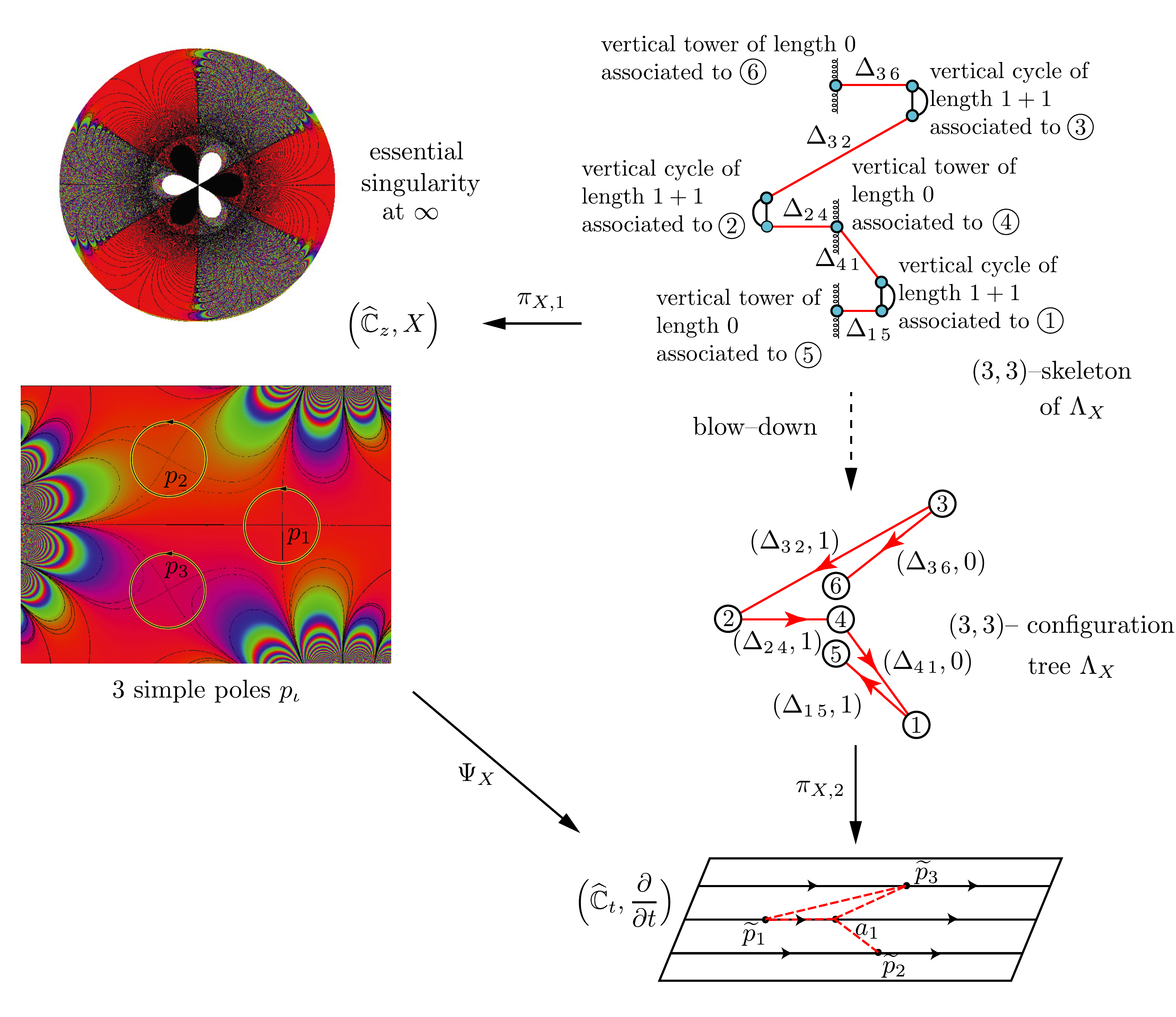}
\caption{{\bf Vector field 
$X(z)=\dfrac{\e^{z^{3}}}{3z^{3}-1}\del{}{z}$ with an essential singularity at $\infty$ and 3 simple poles $p_{\iota}$.}
The five diagonals and their projections are shown in red.
The Riemann surface $\R_{X}$ is not drawn.
See Example \ref{ejemplo-E33}, and \S\ref{confTree-to-skeleton} for the drawing on the right.
}
\label{figejemplo-E33}
\end{center}
\end{figure}

\begin{figure}[htbp]
\begin{center}
\includegraphics[width=\textwidth]{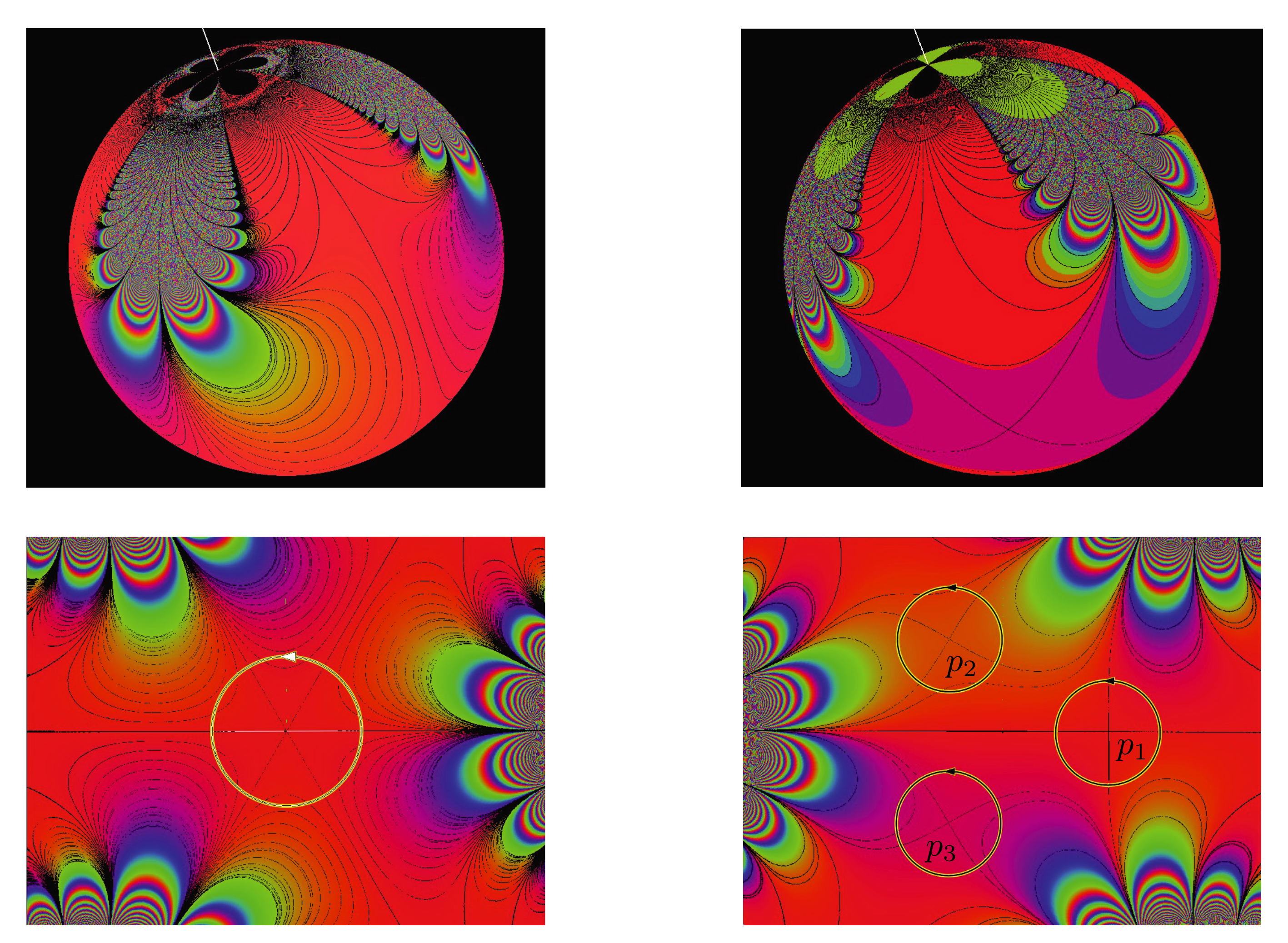}
\caption{{\bf Detail of vector fields in Examples \ref{ejemplo-E23} and \ref{ejemplo-E33}.}
The left hand side shows the vector field $X(z)=-\frac{\e^{z^{3}}}{3z^{2}}\del{}{z}$, the right hand side 
the vector field $X(z)=\frac{\e^{z^{3}}}{3z^{3}-1}\del{}{z}$.
Each angular sector around the poles corresponds to a half plane on $\R_{X}$.
Note that the dynamics of $\Re{X}$ 
in a neighbourhood of $\infty\in\CW$ are different. 
The images contain the information needed to construct the corresponding $(r,d)$--configuration 
trees, as explained in the text.
}
\label{fig7ejemploscampos}
\end{center}
\end{figure}

\noindent
Once again it is instructive to examine in detail how these weights are calculated, 
Figures \ref{figejemplo-E33} and \ref{fig7ejemploscampos} will facilitate the discussion.
\\
Consider the phase portrait of $X$:
for the calculation of the weight $K(3,6)=0$, consider the projection of the diagonal $\Delta_{3\,6}$
onto $\CC_{z}$, clearly this path remains on one angular sector about $p_{3}$ 
(that corresponds to the exponential tract $A_3$ containing $\infty_3\in\overline{\CC}_{z}$).
Thus the logarithmic branch point 
$(\infty_{3},a_{1}, -\infty)$ and the finitely ramified branch point 
$(p_3,\widetilde{p}_3, -1)$ 
share the same sheet on $\R_{X}$.
\\
Now consider the projection onto $\CC_{z}$ of the diagonals $\Delta_{3\,6}$ and $\Delta_{3\,2}$:
the projection of $\Delta_{3\,6}$ lies on the exponential tract $A_3$ associated to $\infty_3$ 
while the projection of the diagonal $\Delta_{3\,2}$ lies on the strip flow determined by $p_1$ and $p_3$.
In order to go from the exponential tract $A_3$ to the strip flow just mentioned, one must go through three
half planes about $p_3$. This is equivalent to going up (or down) one level on $\R_X$, hence $K(3,2)=1$.
\\
Similarly, considering the projections of $\Delta_{3\,2}$ and $\Delta_{2\,4}$ we note that when coming from 
the pole $p_3$ to the pole $p_1$ and then to $\infty_1$ contained in $A_1$ we go around $p_1$ 
and touch upon three half planes (since $\infty_1$ lies on the real axis), thus $K(2,4)=1$.
\\
For $K(2,1)=0$, note that the projection of $\Delta_{4\,1}$ lies entirely within the strip flow 
determined by $p_1$ and $p_2$, 
thus the logarithmic branch point $(\infty_1,a_{1}, -\infty)$ and the finitely ramified branch point 
$(p_{2},\widetilde{p}_{2}, -1)$ share the same sheet on $\R_{X}$. 
\\
Finally, by considering the projections of $\Delta_{4\,1}$ and $\Delta_{1\,5}$ we see that
the projection of $\Delta_{1\,5}$ lies on the exponential tract associated to $\infty_{2}$ 
while the projection of the diagonal $\Delta_{4\,1}$ lies on the exponential tract 
associated to $\infty_{1}$. In order to go from one exponential tract to another one must cross 
at least three angular sectors (but no more than four), this is 
equivalent to the fact that one must go up or down one level on $\R_{X}$ to get from 
the sheet containing $\Delta_{4\,1}$ to the sheet containing $\Delta_{1\,5}$.
Hence $K(1, 5)=1$.

\noindent
In this case the decomposition, provided by Lemma \ref{lema-arboles-horizontales}, 
into horizontal subtrees is

\centerline{
$\Lambda_{X} 
= \Lambda_{H(3)} 
\cup \Lambda_{H(3,2)}
\cup \Lambda_{H(2,4)}
\cup \Lambda_{H(1,5)}
$,
}

\noindent
where
\begin{equation*}
\begin{array}{ll}
\Lambda_{H(3)}=\Big\{ \circled{3},\circled{6};\raiz{3};
(\Delta_{3\, 6},0) \big\}, & 
\Lambda_{H(3,2)}=\Big\{ \circled{2},\circled{3};\raiz{3};(\Delta_{3\, 2},1) \Big\}, 
\\[4pt]
\Lambda_{H(2,4)}=\Big\{ \circled{1}, \circled{2},\circled{4};\raiz{2};(\Delta_{2\, 4},1)
(\Delta_{4\,1},0) \Big\}, &
\Lambda_{H(1,5)}=\Big\{ \circled{1},\circled{5};\raiz{1};(\Delta_{1\, 5},1) \Big\}.
\end{array}
\end{equation*}
\end{example}

\begin{example}\label{Aproximacion-arboles}
Recall Example \ref{E23-1-10-20-50}, where the vector fields

\centerline{
$X(z)= \frac{\e^{z^d}}{z^r} \del{}{z}\in\E(r,d), 
\ \ \
\text{ for }d\geq1$,}

\noindent
are approximated by the rational vector fields

\centerline{
$X_{\tt n}(z)=\dfrac{ 1 }{ z^r \big( 1 - \frac{z^d}{{\tt n}} \big)^{\tt n} }\del{}{z} \in\E(r+{\tt n}d,0),
\ \ \ 
\text{ for } {\tt n} \geq 1$.}

\noindent 
We wish to explore (some of) the $(r,d)$--configuration trees for $X$ and $X_{\tt n}$.

First note that the answer will heavily depend on the parameters $r$ and $d$. 
For instance, \eqref{val-asintoticos-aproximacion} shows that there is only one 
finite asymptotic value $a$
with multiplicity $d$ (as in Example \ref{ejemplo-E23}) if and only if $r=-1\pmod{d}$, 
{\it i.e.} $(r+1)/d=k\in\NN$.

\noindent
In this case the unique finite asymptotic value is
$a = (k-1)! / d$.
Thus $\R_{X}$ has a branch point at 
$\circled{1}=(p_1,\widetilde{p_1}, -\nu_1)=(0,0, -r)$ of 
ramification index $\nu_1=r=k d-1$ and
logarithmic branch points at 
$\circled{\scalebox{0.75}{$1 \text{$+$} \sigma$}}=
(\infty_{\sigma}, a, -\infty)\in\R_{X}$,
with asymptotic paths $\alpha_{\sigma}(\tau)=\tau \e^{i 2\pi\sigma/d}$ for $\sigma=1,\ldots,d$.

\noindent
On the other hand the Riemann surfaces $\R_{X_{\tt n}}$ associated to 
$X_{\tt n}(z)$ have a branch point at $\circled{1}=(p_1,\widetilde{p_1},-r)=(0,0,-r)$ 
of ramification index $r+1=k d$ 
and branch points at

\centerline{
$\circled{\scalebox{0.75}{$1 \text{$+$} \sigma$}}_{\tt n}=
\big(\widehat{e}_\sigma({\tt n}),\widetilde{e}_\sigma({\tt n}), -{\tt n} \big)=
\big( \e^{i2\pi\sigma/d} {\tt n}^{1/d},a\, {\tt n}^{k} {\tt n}! /({\tt n}+k)! , -{\tt n} \big)$
} 

\noindent
of ramification index ${\tt n}+1$, for $\sigma=1,\ldots,d$.
Hence by examining the corresponding phase portraits we can see that the 
$(r,d)$--configuration trees are given by
\begin{multline}
\Lambda_{X_{\tt n}} = \Big\{ \circled{1}, \circled{2}_{\tt n},
\ldots, \circled{\scalebox{0.75}{$1\text{$+$}\sigma$}}_{\tt n}, \ldots,
\circled{\scalebox{0.75}{$1\text{$+$}d$}}_{\tt n}
;\raiz{1}; \\
(\Delta_{1\,2},0), (\Delta_{1\,3},k), 
\ldots, \big(\Delta_{1\,(1+\sigma)}, (\sigma-1) k \big), \ldots, 
\big(\Delta_{1\,(1+d)}, (d-1) k \big) 
\Big\},
\end{multline}
\begin{multline}
\Lambda_{X}=\Big\{ \circled{1},\circled{2},
\ldots, \circled{\scalebox{0.75}{$1\text{$+$}\sigma$}}, \ldots,
\circled{\scalebox{0.75}{$1\text{$+$}d$}};\raiz{1}; \\
(\Delta_{1\,2},0), (\Delta_{1\,3},k), 
\ldots, \big(\Delta_{1\,(1+\sigma)}, (\sigma-1) k \big), \ldots, 
\big(\Delta_{1\,(1+d)}, (d-1) k \big) 
\Big\},
\end{multline}
whose skeletons are as in Figure \ref{esqueletos-aproximacion-menos1-mod-d}.

\begin{figure}[htbp]
\begin{center}
\includegraphics[width=\textwidth]{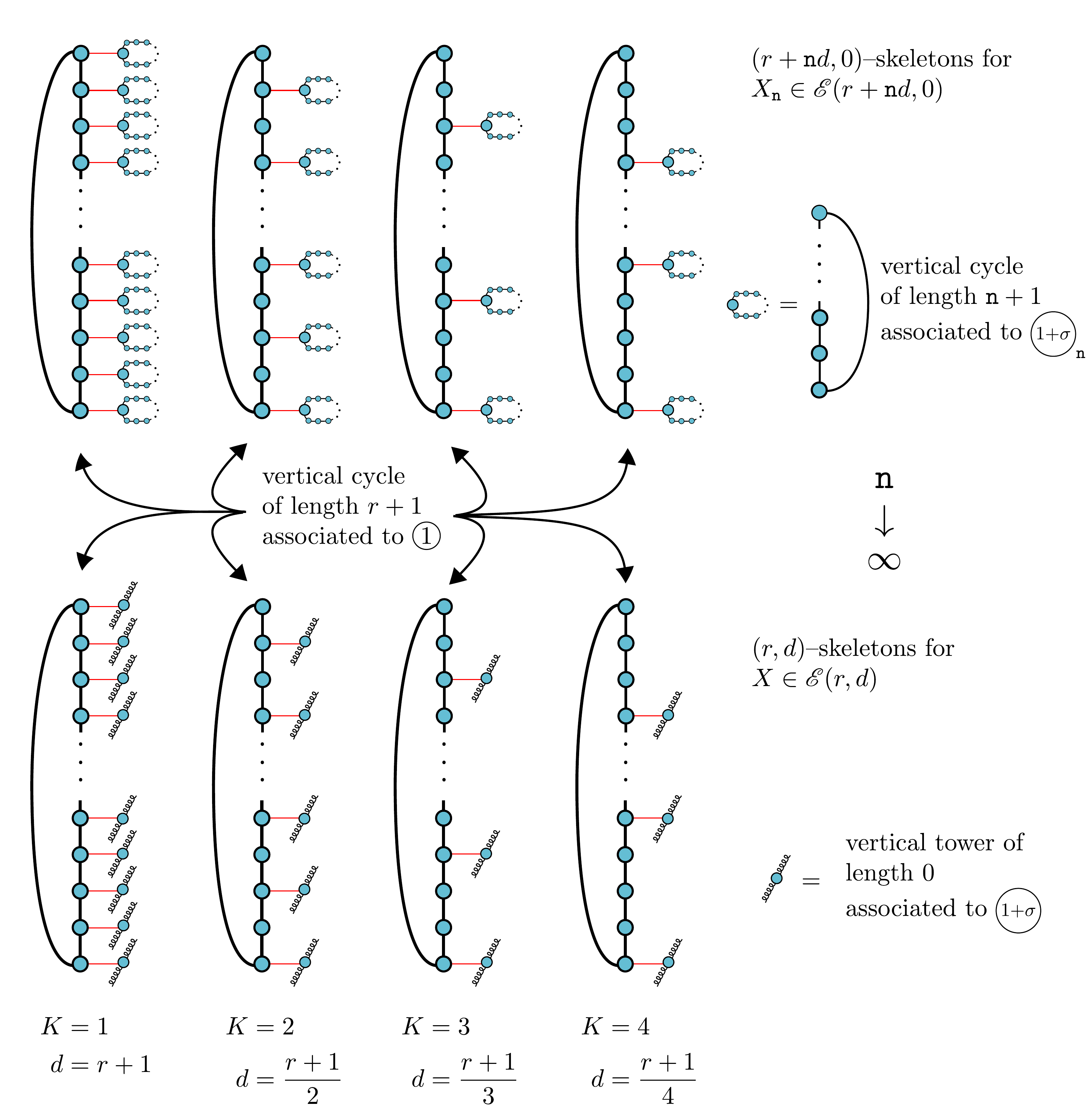}
\caption{Regarding Example \ref{Aproximacion-arboles}; the
top figures are the $(r +{\tt n}d,0)$--skeletons for $X_{\tt n}\in\E(r +{\tt n}d,0)$ 
approximating $X\in\E(r,d)$, when
$r= -1\pmod{d}$.
Bottom figures are the $(r,d)$--skeletons for $X\in\E(r ,d)$.
Note that for $X_{\tt n}$ the critical values are $0$ and $a\, {\tt n}^{k} {\tt n}! /({\tt n}+k)!$ 
and for $X$ the only critical value is $0$ and the only finite asymptotic value is $a=(k-1)! / d$.
Further note that $k=(r+1)/d$ is the number of sheets that separate the different diagonals
both in $\R_{X_{\tt n}}$ and $\R_{X}$ .
}
\label{esqueletos-aproximacion-menos1-mod-d}
\end{center}
\end{figure}
\end{example}

\begin{example}\emph{A prototypical $(r,4)$--configuration tree.}
\label{EjemploPH} 
Consider a vector field 
\begin{equation}
X(z)= \frac{\e ^{E(z)} }{(z-p_1)^{\nu_1}(z-p_2)^{\nu_2}}
\del{}{z},
\ \ \
r=\nu_1+\nu_2,
\end{equation}

\noindent 
$E(z)$ a polynomial of degree 4,
and $\Im{p_1}>\Im{p_2}$. 
The singularity at $\infty\in\CW_z$ has
four finite asymptotic values $(\infty_1, a_1)$, 
$(\infty_2, a_2)$, $(\infty_3, a_2)$ and
$(\infty_4, a_3 )$;
note that two of them differ exclusively 
by their exponential tract, 
sharing the asymptotic value $a_2 \in \CC_z$.
The existence of such a polynomial $E(z)$ will be proved in \S\ref{skeleton-to-RX},
in particular this involves solving the system of equations \eqref{sistema-realizable}.

\noindent 
The following
$(r,4)$--configuration tree $\Lambda_{X}$, 
will be used to exemplify
some constructions and possible complexities that arise in the proof of the Main Theorem.
Thus the $(r,4)$--configuration tree has three essential vertices and two pole vertices, which we 
label as follows
\begin{equation}\label{verticesPoincHopf}
\begin{array}{rclrclrcl}
\circled{1} &=& (p_{1},\widetilde{p}_{1},-\nu_{1}), &
\quad\circled{2} &=& (\infty_{2},a_2,-\infty), &
\circled{3} &=&  (\infty_{1},a_1,-\infty), \\[6pt]
\quad\circled{4} &=& (p_{2},\widetilde{p}_{2},-\nu_{2}), &
\circled{5} &=& (\infty_3, a_2,-\infty), & \circled{6} &=& (\infty_4,a_3,-\infty). \\
\end{array}
\end{equation}
Let
\begin{multline}
\Lambda_{X} = \Big\{ \circled{1}, \circled{2}, \circled{3}, \circled{4},\circled{5}, \circled{6}\ ;
\raiz{1}\, ;
\\
(\Delta_{1\, 2}, 0), \big(\Delta_{1\, 5}, -2),
\big(\Delta_{2\, 3}, K(2, 3)\big), \big(\Delta_{2\, 4}, K(2,4)\big), \big(\Delta_{2\, 6}, K(2,6)\big)
\big) \Big\},
\end{multline}
with edges given by the diagonals/semi--residues
\begin{equation}\label{pesosejemplo7}
\begin{array}{lr}
\Delta_{1\, 2} = \int\limits^{\infty_2}_{p_1} \omega_{X} = a_2 - \widetilde{p}_1,
&
\Delta_{1\, 5} = \int\limits_{p_1}^{\infty_3} \omega_{X} = a_2 - \widetilde{p}_1,
\\[12pt]
\Delta_{2\, 3} = \int\limits_{\infty_2}^{\infty_1} \omega_{X} 
=  a_1 - a_2,
&
\Delta_{2\, 4} = \int\limits_{\infty_2}^{p_2} \omega_{X} 
= \widetilde{p}_2 - a_2,
\\[12pt]
\Delta_{2\, 6} = \int\limits_{\infty_2}^{\infty_4} \omega_{X} 
= a_3 - a_2,
\end{array}
\end{equation}
and weights $K(1,2)=0$, $K(1,5)=-2$, $K(2,6) < K(2,4) < K(2,3) \leq -1$.
Figure \ref{proyeccionArbol} shows the $(r,4)$--configuration tree $\Lambda_{X}$ 
together with the corresponding $(r,4)$--skeleton. 
This figure will be used as a guide in the proof of the Main Theorem.

\noindent
In this case the decomposition, provided by Lemma \ref{lema-arboles-horizontales}, 
into horizontal subtrees is

\centerline{
$\Lambda_{X} 
= \Lambda_{H(1)} 
\cup \Lambda_{H(1,5)}
\cup \Lambda_{H(2,3)}
\cup \Lambda_{H(2,4)}
\cup \Lambda_{H(2,6)}
$,
}

\noindent
where
\begin{equation*}
\begin{array}{ll}
\Lambda_{H(1)}=\Big\{ \circled{1},\circled{2};\raiz{1};(\Delta_{1\, 2},0) \Big\}, & 
\Lambda_{H(1,5)}=\Big\{ \circled{1},\circled{5};\raiz{1};(\Delta_{1\, 5},-2) \Big\},
\\[6pt]
\Lambda_{H(2,3)}=\Big\{ \circled{2},\circled{3};\raiz{2};(\Delta_{2\, 3},K(2,3)) \Big\}, &
\Lambda_{H(2,4)}=\Big\{ \circled{2},\circled{4};\raiz{2};(\Delta_{2\, 4},K(2,4)) \Big\}, 
\\[6pt]
\Lambda_{H(2,6)}=\Big\{ \circled{2},\circled{6};\raiz{2};(\Delta_{2\, 6},K(2,6)) \Big\}.
\end{array}
\end{equation*}

\begin{figure}[htbp]
\begin{center}
\includegraphics[width=\textwidth]{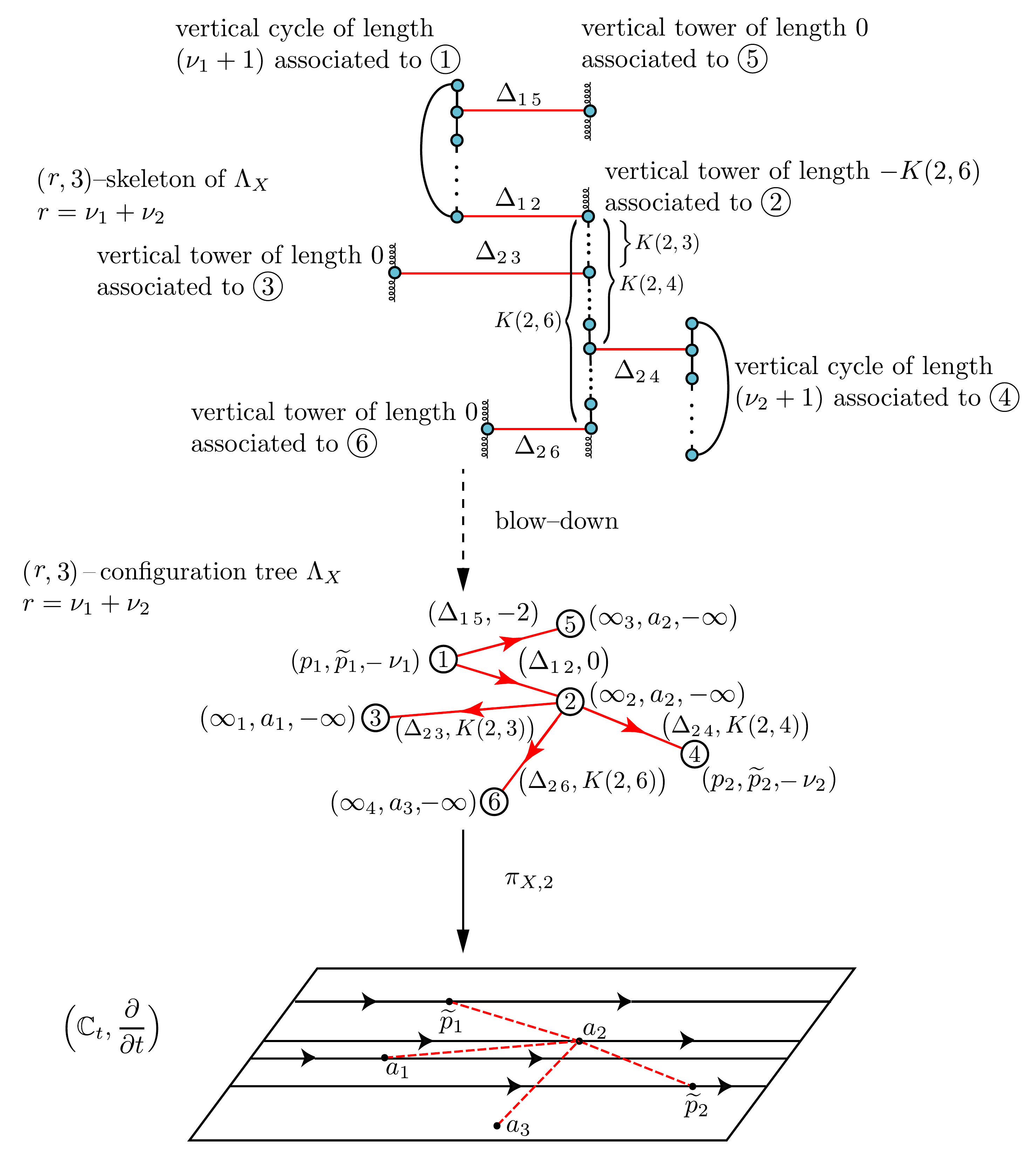}
\caption{
{\bf The $(r,4)$--configuration tree $\Lambda_{X}$ and its 
$(r,4)$--skeleton corresponding to Example \ref{EjemploPH}.} 
The root is $\circled{1}$, the weight $K(1, 2)=0$ 
(hence the branch points 
$\circled{1}$, $\circled{2}$ which are the endpoints of
the diagonal $\Delta_{1\,2}\subset\R_{X}$, share the 
global zero level (GZL) sheet),
$K(1, 5)=-2\pmod{\nu_1+1}$, and 
$K(2,6)<K(2, 4)<K(2,3)\leq-1$.
The information 
about how many sheets we have gone ``up'' or ``down'' on the Riemann surface is 
given by the weights. 
The asymptotic value $a_2$ has multiplicity 2.
}
\label{proyeccionArbol}
\end{center}
\end{figure}
\end{example}

\section{Proof of Main Theorem: description of the family $\E(r,d)$ via combinatorial scheme}
\label{Mainclassification}

\begin{proof}[\textbf{Plan for proof}]
\noindent
That $\E(r,d)$ is a complex manifold of dimension $r+d+1$ is obvious, 
see \cite{AlvarezMucino2}
for more general discussion.

The proof of the bijection $\E(r,d) \cong \left\{ \big[ \Lambda_{X}\big] \right\}$ 
shall procede as follows:

\noindent
1) In \S\ref{X-to-confTree} starting from $\Psi_X$, we construct the 
$(r,d)$--configuration tree $\Lambda_{X}$.

\noindent
2) In \S\ref{confTree-to-skeleton}, we start with an 
abstract $(r,d)$--configuration tree $\Lambda_{X}$ 
Definition \ref{d-confTree} 
and construct the $(r,d)$--skeleton of $\Lambda_{X}$ 
(as an associated combinatorial object).

\noindent
3) In \S\ref{skeleton-to-RX} from the $(r,d)$--skeleton of $\Lambda_{X}$, 
we proceed to construct  
a Riemann surface $\R_{X}$
in $\E(r,d)$.

\smallskip
\noindent
Note that our Dictionary Proposition \ref{basic-correspondence} 
provides the correspondence between 
$\R_X$, $\Psi_{X}$ and $X\in\E(r,d)$.

\medskip
\noindent
The suitable equivalence class $[\Lambda_{X}]$ of $(r,d)$--configuration trees $\Lambda_{X}$ 
will be explained in \S\ref{clasesE(d)}.

\subsection{From $X\in\E(r,d)$ to an $(r,d)$--configuration tree 
$\Lambda_{X}$}\label{X-to-confTree}\hfill

\noindent
Recalling Definition \ref{d-confTree} of $(r,d)$--configuration tree, we have:

\noindent
$\bullet$ The trivial case:
\emph{$\Psi_X$ has exactly one finite asymptotic or critical value:}\\
From Lemma \ref{oneasymptoticvalue}, 
only the following two cases are possible,
\begin{enumerate}[label=\arabic*),leftmargin=*]
\item $X(z)=\frac{\mu}{ (z-p_1)^{r}} \del{}{z}$, {\it i.e.} $(r,d)=(r,0)$, or 
\item $X(z)=\mu^{-1} \e^z \del{}{z}$, {\it i.e.} $(r,d)=(0,1)$,  
\end{enumerate}
where $p_1\in\CC_{z}$ and $\mu\neq0$.

\noindent
For (1), 
$\Lambda_{X}=\Big\{ \circled{1}=(p_{1},\widetilde{p}_{1}, -r);
\raiz{1}\, ;\,
\varnothing
\Big\}$, 
see Example \ref{ejemplo-un-polo}.

\noindent
For (2), 
$\Lambda_{X}=\Big\{ \circled{1}=(\infty_{1},a_{1},-\infty); 
\raiz{1}\, ;\,
\varnothing \Big\}$,
see Example \ref{ejemplo-un-va}.

\smallskip
\noindent
$\bullet$ The non--trivial case:
\emph{$\Psi_{X}$ has two or more finite asymptotic or critical values, {\it i.e.} $d+n\geq2$:}

\noindent
Considering the surface 
$\R_{X}$, recall
Equation \eqref{essenvert} and the reduced
divisor, Definition \ref{divisor-reducido}.

\smallskip
\noindent
\textbf{1. Vertices of $\Lambda_{X}$}.
Let the vertices be those obtained from the reduced divisor
\begin{multline}\label{verticesenumerados}
V =
\Big\{ \circled{\,\iota\,} = 
\big(p_{\iota},\widetilde{p}_{\iota},-\nu_{\iota} \big) \Big\}_{\iota=1}^{n}
\cup 
\Big\{ 
\circled{\scalebox{0.75}{$n \text{$+$} \sigma$}} 
= 
\big( \infty_{\sigma},a_{\sigma}, -\infty\big)
\Big\}_{\sigma=1}^{d}
 \\
=\Big\{ \circled{\msigma}=
\big(z_{\msigma},t_{\msigma},-\nu_{\msigma} \big) \Big\}_{\msigma=1}^{n+d} .
\end{multline}
There are $n+d$ vertices. 

\noindent
Note that the essential vertices 
$\circled{\scalebox{0.75}{$n \text{$+$} \sigma$}}$ 
are labelled as in 
Remark \ref{correspexponentialtract}.2, that is according to the natural counterclockwise order 
of the exponential tracts in $S^1$ about $\infty\in\CW_z$.

\noindent
Root choice:

\noindent
If $r=0$ let the root be 
$\raiz{\varrho} \doteq \circled{1} = \big( \infty_1 ,a_1, -\infty \big)$.

\noindent
If $r\neq 0$ let the vertex 
$\big(p_{\varrho},\widetilde{p}_{\varrho},-\nu_{\varrho} \big)$ be 
such that $\Im{\widetilde{p}_\varrho}\geq\Im{\widetilde{p}_\iota}$ and 
$\Re{\widetilde{p}_\varrho}\leq\Re{\widetilde{p}_\iota}$ 
for $1\leq \iota\leq n$;
{\it i.e.} $\widetilde{p}_\varrho \in \CC_t$ is the top \& left--most 
critical value, 
following the root condition in Definition \ref{d-confTree}.
In this case, choose the root to be 
$\raiz{\varrho}\doteq \big(p_{\varrho},\widetilde{p}_{\varrho},
-\nu_{\varrho} \big)$.

\smallskip
\noindent
\textbf{2. Edges of $\Lambda_{X}$}. 
From Definition \ref{diagonal}, the diagonals, associated to different 
pairs $t_{\msigma},t_{\mrho}$ of 
finite asymptotic or critical values, are \emph{oriented} segments 

\centerline{
$\Delta_{\msigma\mrho}=
\overline{(z_{\msigma},t_{\msigma}, -\nu_{\msigma}) 
(z_{\mrho},t_{\mrho}, -\nu_{\mrho})}$ 
\, in $\R_{X}$,}

\noindent 
whose endpoints project down, via $\pi_{X,2}$, 
to the finite asymptotic or critical values 
$t_{\msigma}$, $t_{\mrho}$. 
From Lemma \ref{pareja-diagonales} it follows that 
there are at least two\footnote{
This is so because diagonals are oriented.
} 
diagonals associated to each finite asymptotic or critical value.
As a first step, we ignore orientation and consider the diagonals as undirected edges\footnote{
We agree to leave only one undirected edge for each pair of oriented edges.
}, 
which without loss of generality 
we shall simply denote by $E=\{ \Delta_{\msigma \mrho} \}$.
In this way we obtain a connected graph, 
say $G=\{ V; E \}$.

\noindent
A subgraph formed by the set of vertices associated to branch points that share the same sheet in 
$\R_{X}$ will be called a \emph{horizontal subgraph}.
Note that each horizontal subgraph, consisting of the vertices 
(branch points) sharing a same sheet of $\R_X$, say
$\{\circled{\ell}=(z_\ell,t_\ell,-\nu_\ell)\}_{\ell=1}^s$ 
together with the corresponding set of edges (undirected diagonals) on the same sheet, 
forms a complete graph $K_s$ with $s(s-1)$ edges.
However, by eliminating appropriate edges from $K_s$ we can always obtain an (undirected)
left--right--top--bottom linear tree of the vertices 
$\{\circled{\ell}=(z_\ell,t_\ell,-\nu_\ell)\}_{\ell=1}^s$,
recall Definition \ref{arbol-lineal-enraizado}.

\noindent
Now replace each of the horizontal subgraphs of $G$ by the corresponding (undirected)
left--right--top--bottom linear tree. 
As a consequence, the diagonals $\Delta_{\msigma \mrho}$
and finite height horizontal strips 

\centerline{
$\left( \big\{ \Im{t_\msigma} \leq \Im{t} \leq \Im{t_\mrho} \big\}, 
\del{}{t} \right) \subset \R_X$}

\noindent
are in bijective correspondence,
as in Lemma \ref{descomposicion-planos-bandas}.2--3. 

This produces a connected (undirected) graph 
$\overline{\Lambda}_X$ whose horizontal subgraphs 
are (undirected) left--right--top--bottom linear subtrees.
The next lemma shows its shape.
\begin{lemma}\label{no-ciclos-entre-pisos}
If $X\in\E(r,d)$ then the graph $\overline{\Lambda}_X$ is a tree. 
\end{lemma}

\begin{proof}
Recall the decomposition of $\CC_z$ in 
half planes and finite height horizontal strips, Lemma \ref{descomposicion-planos-bandas}.

\noindent 
Assume that 
all the asymptotic and critical values $t_\msigma$, associated to the branch points 
$(z_\msigma, t_\msigma)$, lie on different horizontal trajectories of 
$(\CC_{t},\del{}{t})$ 
(we leave the general case for the interested reader, see for instance the discussion preceding
Proposition \ref{soul-of-LambdaX}).

\noindent
Note that the intersection of the interior of any 
two  finite height horizontal strips
is empty. 

\noindent
Since $\CC_z$ is simply connected and 
each diagonal determines a finite height horizontal strip, 
there is no loop/cycle of diagonals. 
\end{proof}

\noindent
Finally, assign an orientation to $\overline{\Lambda}_X$ by choosing the appropriate 
orientation so that the edges point away from the root vertex 
$\raiz{\varrho}$.

\noindent
We thus obtain a non--weighted, 
directed rooted tree
\begin{equation}\label{graficaenCt}
\Big\{
\big\{ \circled{\msigma}=\big(z_{\msigma},t_{\msigma},-\nu_{\msigma} \big) 
\big\}_{\msigma = 1}^{d+n} ;
\,\raiz{\varrho}\,;
\big\{ \Delta_{\msigma \mrho} \big\} \Big\} ,
\end{equation}
\noindent 
satisfying conditions (1--6) of Definition \ref{d-confTree}.

\noindent
By simple inspection, it is clear that each diagonal in \eqref{graficaenCt} 
falls into one of the cases mentioned in Remark \ref{casosVAf}.

\smallskip
\noindent
\textbf{3. Weights of $\Lambda_{X}$}.
At this point of the description of the surface $\R_{X}$, the branch points correspond to vertices 
of a non--weighted tree whose edges correspond to the diagonals between branch points. 
However, as mentioned in Remark \ref{pesos-como-info-geometrica}, an important part of the description 
of the Riemann surface $\R_{X}$ as a pasting of geometric pieces, 
corresponds to the number of sheets in $\R_{X}$ that 
separate the diagonals (pairwise).
Recalling that in a rooted tree there is a unique (simple) path from each vertex to the root,
allows us to assign a 
weight to each edge/diagonal of the tree.

\noindent
As an aid, the reader can follow the construction by considering Example \ref{EjemploPH}.
[We will include such references inside square brackets.]

\noindent
For the assignment of 
weights 
$\{K(\msigma, \mrho)\}$ 
to the edges $\{ \Delta_{\msigma \mrho}\}$
we shall proceed by induction on the depth 
of the vertices as follows.
	
\begin{enumerate}[label=\alph*),leftmargin=*]

	\item 
	We start by considering the vertices contiguous to the root 
	$\raiz{\varrho}$
	({\it i.e.} vertices of depth 1 in \eqref{graficaenCt}), say $\{\circled{\mvarrho}\}$.
	Since all the branch points corresponding to $\{\circled{\mvarrho}\}$ share a 
	sheet with the branch point corresponding to the root 
	$\raiz{\varrho}$, 
	any of the corresponding edges can be assigned a weight zero.
	Choose one, say $\circled{\mvarrho_{0}}$, and 
	assign the weight $K(\varrho,\mvarrho_{0})=0$ to the edge 
	$\Delta_{\varrho \mvarrho_{0}}$ 
	and call this sheet in $\R_{X}$
	the \emph{global zero level (GZL)} sheet.
	
	\noindent
	If there are more vertices in $\{\circled{\mvarrho}\}$ whose branch points share the global zero level
	sheet then the corresponding edges are also assigned the weight 0.
	
	\noindent 
	[Referring to Example \ref{EjemploPH}, our 
	first edge is $\Delta_{1\, 2}$ with 
	weight $K(1,2)=0$; note that the diagonal 
	is as in case (3) of Remark \ref{casosVAf}.
	Moreover, there are no more vertices sharing the GZL sheet.]
	
\item
	Now consider the vertices of depth 1 that do not share the GZL sheet,
	say $\{\circled{\mrho}\}\subset\{\circled{\mvarrho}\}$. 
	For each of these vertices consider their edge 
	$\Delta_{\varrho  \mrho}$: 
	we assign the weight $K(\varrho, \mrho)$ such that
	
\smallskip 
	
\centerline{
$2\pi K(\varrho, \mrho)$ is the 
\emph{argument between the sheets
containing $\Delta_{\varrho \mvarrho_0}$ and $\Delta_{\varrho \mrho}$.}
}
	
\smallskip 	
\noindent 
[Referring to Example \ref{EjemploPH}, the vertex of depth 1 not sharing the GZL sheet is 
$\circled{5}$. In this case the weight $K(1,5)=-2\pmod{\nu_1 +1}$.]
	
	\noindent
	If \eqref{graficaenCt} has only vertices contiguous to the root $\raiz{\varrho}$, then
	we have completed the construction of $\Lambda_{X}$.
		
	\item 
	We now consider the set of vertices of depth 2.
	Of course there is an edge from a vertex of depth 1, say $\circled{\mvarrho}$ 
	to a vertex of depth 2, say $\circled{\msigma}$.
	The associated weight for the edge $\Delta_{\mvarrho\, \msigma}$ is defined 
as $K(\mvarrho,\msigma)$ such that
	
\smallskip 
	
\centerline{
$2\pi K(\mvarrho,\msigma)$ is the 
\emph{argument between the sheets
containing $\Delta_{1\,\mvarrho}$ and $\Delta_{\mvarrho\, \msigma}$.}
}	

\smallskip
		
	\item 
	Continue the assignment of weights as in (c) for all the edges that contain
	vertices of depth 2.
	
	\noindent
	[Referring to Example \ref{EjemploPH}, the weight $K(2, 3)\leq-1$ 
	since on $\R_{X}$ the diagonal $\Delta_{2\,3}$ is $\abs{K(2,3)}$ sheets below the
	diagonal $\Delta_{1\,2}$; 
	similarly the weight $K(2, 4)\leq-2$, 
	since on $\R_{X}$ the diagonal $\Delta_{2\,4}$ is $\abs{K(2,4)}$ sheets 
	below the diagonal $\Delta_{1\,2}$.]
	 
	\item 
	Repeat (c) and (d) with vertices of depth $\geq 3$, assigning the weights 
	until all the vertices are exhausted.
	
	\noindent
	[Referring to Example \ref{EjemploPH}, the last edge to be considered is $\Delta_{3\, 5}$ 
	with corresponding weight $K(3, 5)=-2$.]
		
\end{enumerate}
We have thus constructed an $(r,d)$--configuration tree
\begin{equation}\label{LambdaX}
\Lambda_X
=
\Big\{
\big\{ \circled{\msigma}=\big(z_{\msigma},t_{\msigma},-\nu_{\msigma} \big) 
\big\}_{\msigma = 1}^{d+n} ;
\raiz{\varrho}\, ;
\big\{ (\Delta_{\msigma \mrho},\ K(\msigma, \mrho) ) \big\} \Big\}
\end{equation}
associated to $\Psi_{X}$. 

\begin{remark}
\label{continuous-edge}
[Remark \ref{pesos-como-info-geometrica} revisited.]
The integer weight $K(\msigma,\mrho)$ associated to the edge 
$\Delta_{\msigma \mrho}$ can be incorporated into a continuous ``edge'' by considering

\centerline{
$\widetilde{\Delta}_{\msigma \mrho} 
\doteq \Delta_{\msigma \mrho} \e^{i 2\pi K(\msigma,\mrho)} \in \widetilde{\CC^{*}}$, 
}

\noindent
where $\widetilde{\CC^{*}}=\{\abs{z}\e^{i \arg(z)}  \}$
is the universal cover of $\CC^{*}$ and $\arg(z)$ is the multivalued argument. 
This will provide the compatibility of ``continuous coordinates/parameters'' for the 
manifold structure of $\E(r,d)$.
\end{remark}

\subsection{From an $(r,d)$--configuration tree $\Lambda_{X}$ to the $(r,d)$--skeleton of 
$\Lambda_{X}$} \label{confTree-to-skeleton}

\noindent
Let $\Lambda_{X}$ be an 
abstract 
$(r,d)$--configuration tree as in 
Definition \ref{d-confTree}.
In fact, we want to show the existence of a vector field $X$.
In order to achieve this goal,
we construct  
the \emph{$(r,d)$--skeleton of $\Lambda_{X}$} 
(as a certain \emph{``blow--up''} of $\Lambda_{X}$,
see Definition \ref{skeleton}).

The \emph{$(r,d)$--skeleton of $\Lambda_{X}$} will contain the same information as $\Lambda_{X}$. 
\begin{enumerate}[label=\alph*),leftmargin=*]
\item With the disadvantage of being more cumbersome to express. 
\item With the advantage that it will enable us to identify the equivalence classes of $\Lambda_{X}$ in \S\ref{clasesE(d)}.
\item Also note that the $(r,d)$--skeleton of $\Lambda_{X}$ describes the ``embedding'' of $\Lambda_{X}$ in 
$\overline{\CC}_{z}\times\CC_{t}$. 
\end{enumerate}

\noindent
Figure \ref{proyeccionArbol} presents a particular example that will help the reader follow 
the construction.

A priori, 
$\Lambda_{X}$ has two types of vertices: essential vertices 
$\circled{\scalebox{0.75}{$n \text{$+$} \sigma$}} 
=(\infty_{\sigma}, a_{\sigma}, -\infty)$
and pole vertices 
$\circled{\, \iota \, }=(p_{\iota},\widetilde{p}_{\iota},-\nu_{\iota})$.


\smallskip
Before proceeding to the construction of the $(r,d)$--skeleton of $\Lambda_{X}$, 
we shall need the following
two definitions/constructions.

\smallskip
Associated to the pole 
vertices, recalling Definition \ref{helicoide-ciclico},
the following construction is natural.
See right hand side of Figures \ref{figPiezasBasicas}.c, 
\ref{figejemplo-un-polo}, \ref{figejemplo-dos-polos} 
and the blow up of vertices 
$\circled{1}$ and $\circled{4}$ in Figure \ref{proyeccionArbol}.

\begin{remark}\label{def-vertical-cycle}
For each pole vertex 
$\circled{\, \iota \,}=
(p_{\iota},\widetilde{p}_{\iota}, -\nu_{\iota})$ of $\Lambda_{X}$, the

\centerline{
vertical cycle of length $\nu_\iota+1$ associated to $\circled{\iota}$}

\noindent
is a cyclic graph consisting of exactly $\nu_\iota+1$  copies of the vertex $\circled{\iota}$ 
joined by $\nu_{\iota}+1$ vertical edges (without weights).
The vertices on the vertical cycle are also assigned a local level: in this case arithmetic modulo 
$(\nu_{\iota}+1)$ is to be used.
\\
The vertical cycle of length $\nu_{\iota}+1$ will only have vertices of valence 2.
Once again, call one direction of the vertical cycle \emph{up} and the other direction \emph{down}. 
To be precise, up corresponds to the anti--clockwise direction of 
$\beta(\theta)=\pi_{X,2}^{-1}\big( t_{\msigma}+\rho\,\e^{i2\pi\theta})$ considered 
in Remark \ref{pesos-como-info-geometrica}.3.
\end{remark}

Similarly, associated to essential vertices, 
recalling Definition \ref{helicoide-finito},
the following construction is natural.

\begin{remark}\label{def-vertical-tower}
For each essential vertex 
$\circled{\scalebox{0.75}{$n \text{$+$} \sigma$}} 
=(\infty_{\sigma}, a_{\sigma}, - \infty)$, 
of $\Lambda_{X}$, let
\\ \centerline{
$K_{\max}(\sigma)=\max\limits_{\mrho}\{0,K(\sigma,\mrho)\}$ \quad and \quad 
$K_{\min}(\sigma)=\min\limits_{\mrho}\{0,K(\sigma,\mrho)\}$,
}\\  
where the maximum and minimum are taken over all the edges that start at 
$\circled{\scalebox{0.75}{$n \text{$+$} \sigma$}} $ and 
end at the respective $\{\circled{\mrho}\}$.
Then by letting

\centerline{
$K(\sigma)\doteq K_{\max}(\sigma) - K_{\min}(\sigma)$}

\noindent
we shall say that the 

\centerline{
vertical tower of length $K(\sigma)$ associated to 
$\circled{\scalebox{0.75}{$n \text{$+$} \sigma$}} $} 

\noindent
is a linear graph consisting of exactly 
$K(\sigma) +1\geq1$ copies of the vertex 
$\circled{\scalebox{0.75}{$n \text{$+$} \sigma$}} $ joined by 
$K(\sigma)$ 
\emph{vertical edges} (without weights). 
Each of the vertices of the vertical tower will have a \emph{local level} assigned to it: 
the local level assigned to the first vertex of the vertical tower will be $K_{\min}(\sigma)$, 
the local level assigned to the second vertex of the vertical tower will be $K_{\min}(\sigma)+1$,
continuing in this way the local level assigned to the last vertex of the vertical tower will be
$K_{\max}(\sigma)$.  
We shall call the increasing direction, 
using $\beta(\theta)=\pi_{X,2}^{-1}\big( t_{\msigma}+\rho\,\e^{i2\pi\theta})$, 
of the local level \emph{up} and the decreasing 
direction \emph{down}.
\\
The vertical tower will have vertices of valence 1 at the extreme local levels 
$K_{\min}(\sigma)$ 
and $K_{\max}(\sigma)$, otherwise of valence 2.
\end{remark}

\noindent
On the right hand side of Figure \ref{figPiezasBasicas}.b a simple example of a vertical 
tower is presented.
However in Example \ref{EjemploPH} and 
Figure \ref{proyeccionArbol} a more complex $(r,4)$--skeleton is shown: 
vertex $\circled{2}\in\Lambda_{X}$ blows up into
the vertical tower of length $- K(2,6)$. 
Moreover,

\noindent 
the local zero level is assigned to 
the vertex where the edge $\Delta_{1\,2}$ is attached;
 
\noindent 
the local $K(2,3)\leq -1$ level is assigned to the vertex where the edge $\Delta_{2\,3}$ is attached;

\noindent 
the local $K(2,4)\leq -2$ level is assigned to the vertex where the edge $\Delta_{2\,4}$ is attached; and

\noindent
the local $K(2,6)\leq -3$ level is assigned to the vertex where the edge $\Delta_{2\,6}$ is attached.

\smallskip
We can now define the associated combinatorial object.

\begin{definition}\label{skeleton}
Let $\Lambda_{X}$ be an $(r,d)$--configuration tree.
The \emph{$(r,d)$--skeleton of $\Lambda_{X}$} is the undirected graph obtained by:
\begin{enumerate}[label=\alph*),leftmargin=*]
		\item Replacing each essential and pole vertices 
		$\circled{\msigma}=(z_\msigma, t_\msigma, -\nu_\msigma)\in\Lambda_{X}$, 
		with their associated vertical tower 
		or vertical cycle respectively.
		
		\item 
		For each
		directed weighted edge, 
		$(\Delta_{\msigma \mrho}, K(\msigma, \mrho))\in\Lambda_{X}$,  
		eliminate the weight and consider it an undirected \emph{horizontal edge} 
		$\pm\Delta_{\msigma \mrho}$.
		
		\item The undirected $\pm\Delta_{\msigma \mrho}$ edge is to have as its ends 
		the local level 0 vertex of the vertical tower or vertical cycle associated to 
		$\circled{\mrho}$, 
		and the local level $K(\msigma,\mrho)$ vertex of the vertical tower 
		or vertical cycle associated to $\circled{\msigma}$; noting that if $\circled{\msigma}$ 
		is a pole vertex, modular arithmetic is to be used.
		
	\end{enumerate}
\end{definition}

\begin{remark}\label{propiedadesdelesqueleto}
The $(r,d)$--skeleton of $\Lambda_{X}$ has the following properties 
(also see Diagram \ref{diagramaSKConf}): 
\begin{enumerate}[label=\arabic*.,leftmargin=*]
\item The edges of the $(r,d)$--skeleton of $\Lambda_{X}$ are divided in two sets: 

\noindent
$\bullet$ the vertical edges 
(alluded to in Definitions \ref{def-vertical-tower} and \ref{def-vertical-cycle} above), and 

\noindent
$\bullet$ the \emph{horizontal edges} 
$\pm\Delta_{\msigma \mrho}$,  
which form a finite set of connected subtrees, 
that correspond precisely to the horizontal subtrees of the $(r,d)$--configuration 
tree $\Lambda_X$
(recall Definition \ref{d-confTree}.7).

\item Consider two horizontal edges $\Delta_{\mvarrho \msigma}$ and 
$\Delta_{\msigma \mrho}$ in the $(r,d)$--skeleton of $\Lambda_{X}$ that share the 
vertex $\circled{\msigma}$ in the original $(r,d)$--configuration tree $\Lambda_{X}$. 
We shall say that:

In the $(r,d)$--skeleton of $\Lambda_{X}$,
the (horizontal) edge $\Delta_{\msigma \mrho}$, 
relative to the (horizontal) edge $\Delta_{\mvarrho \msigma}$, 
is 
$\left\{ \hspace{-5pt}
\begin{array}{l}
\text{at the \emph{same (global) level} if } K(\msigma,\mrho) = 0, \\
K(\msigma,\mrho) \text{ levels \emph{up} if }K(\msigma,\mrho)>0, \\
K(\msigma,\mrho) \text{ levels \emph{down} if }K(\msigma,\mrho)<0.
\end{array}
\right.$

\noindent
Hence, using geometry and combinatorics, 
$K(\msigma,\mrho)$ can be recognized as 

\noindent 
$\bigcdot$ the number of sheets in $\R_{X}$ separating 
the diagonals $\Delta_{\mvarrho\msigma}$ and $\Delta_{\msigma\mrho}$, 
or equivalently

\noindent 
$\bigcdot$ the number of levels in the 
$(r,d)$--skeleton of $\Lambda_{X}$ separating the edges 
$\Delta_{\mvarrho\msigma}$ and $\Delta_{\msigma\mrho}$.

\item 
Note that even though one can recognize which is the GZL sheet on the $(r,d)$--skeleton 
of $\Lambda_{X}$, this is a property inherited from $\Lambda_{X}$: it is not intrinsic to 
the $(r,d)$--skeleton.

\item Roughly speaking, the 
$(r,d)$--configuration tree $\Lambda_{X}$ is a \emph{blow--down} of the 
$(r,d)$--skeleton of $\Lambda_{X}$, see Diagram \ref{diagramaSKConf} and Figure \ref{proyeccionArbol}
for an example.
\end{enumerate}
\end{remark}

\subsection{From the $(r,d)$--skeleton of $\Lambda_{X}$ to the Riemann surface $\R_{X}$}
\label{skeleton-to-RX}

\noindent
We proceed in two steps:
In the first step, from the $(r,d)$--skeleton of $\Lambda_{X}$ we will construct a connected Riemann surface 
with boundary, the $(r,d)$--soul of $\Lambda_{X}$ (see Definition \ref{soul}).

\noindent
As the second and final step, we shall glue $2d$ infinite helicoids on the 
boundaries of the $(r,d)$--soul of $\Lambda_{X}$ to obtain 
the simply connected 
Riemann surface $\R_X$ (without boundary).

\begin{definition}\label{soul}
Given an $(r,d)$--skeleton of $\Lambda_{X}$ as above, the associated 
\emph{$(r,d)$--soul of $\Lambda_{X}$}
is a flat Riemann surface 

\noindent
1) constructed from the gluing of 
half planes $(\HH^2_{\pm}, \del{}{t})$
and
finite height horizontal strips 
$\big(\{ 0 < \Im{z} < h\}, \del{}{t} \big)$,

\noindent
2) having two families of cone points:
\begin{enumerate}[label=\alph*),leftmargin=*]
\item the first with $n$ conic points $\{ (p_{\iota},\widetilde{p}_{\iota}) \}_{\iota=1}^{n}$ each with cone angle 
$2(\nu_{\iota}+1)\pi$ and $r=\sum_{\iota}\nu_{\iota}$,

\item the second with $d$ conic points $\{(z_{\sigma}, a_{\sigma})\}_{\sigma=1}^{d}$ each with cone 
angle 
$2 \big(K(\sigma)+1\big) \pi$, $K(\sigma)\geq 0$,
\end{enumerate}

\noindent  
3)
$d$ horizontal branch cuts starting at the cone points of the second family.

\noindent
The branch cuts determine the boundary of the $(r,d)$--soul, $2d$ horizontal boundaries, 
recall Equation \eqref{sheetminuscuts}.
\end{definition}
For $X(z)=\frac{1}{P(z)}\e^{E(z)}\del{}{z}\in\E(r,d)$ 
the $(r,d)$--soul of the $(r,d)$--configuration tree $\Lambda_{X}$ 
is a simply connected Riemann surface 
with $n$ cone points $\{(p_{\iota},\widetilde{p}_{\iota})\}_{\iota=1}^{n}$ each with cone angle
$2(\nu_{\iota}+1)\pi$, corresponding to the
$n$ finitely ramified branch points  
associated to the poles of $X$, 
where $-\nu_{\iota}$ is the order of the pole $p_{\iota}$;
$d$ cone points $\{ (\infty_{\sigma}, a_{\sigma}) \}_{\sigma=1}^{d}$ 
each with cone angle 
$2\big(K(\sigma)+1 \big) \pi$,
corresponding to the logarithmic branch points of $X$;
with boundary 
the $2d$ horizontal segments
$[a_\sigma,\infty)_{-}$ $\cup\ [a_\sigma,\infty)_{+}$.

\medskip 

\noindent
\textbf{1. Construction of the \emph{$(r,d)$--soul of $\Lambda_X$} 
from the \emph{$(r,d)$--skeleton of $\Lambda_{X}$}.}

From Remark \ref{propiedadesdelesqueleto}.2, we see that 
the $(r,d)$--skeleton of $\Lambda_{X}$ 
has two types of vertices: 
those that do not share horizontal edges 
and those that share horizontal edges (the vertices that belong 
to a horizontal subtree).

\smallskip
\emph{Generic horizontal subtrees assumption}.
For simplicity, let us first assume that on any given horizontal subtree
the asymptotic and critical values, 
associated to the vertices of the horizontal subtree,    
all lie on different horizontal trajectories of $(\CC_{t},\del{}{t})$.
Figure \ref{fig-ejemplo-z9} illustrates this assumption.

\smallskip
Starting from the $(r,d)$--skeleton of $\Lambda_{X}$ consider the following construction.

\begin{enumerate}[label=\alph*),leftmargin=*]
	\item 
	Replace each vertex\footnote{
	Recall that all the vertices of the $(r,d)$--skeleton of 
	$\Lambda_{X}$ are either the original vertices 
	$\circled{\msigma}$ of the 
	original $(r,d)$--configuration tree $\Lambda_{X}$, or copies of them. 
	Thus any vertex in the $(r,d)$--skeleton of $\Lambda_{X}$ projects to 
	a unique vertex on $\Lambda_{X}$.
	} 
	of the $(r,d)$--skeleton of $\Lambda_{X}$ that does not share a horizontal edge with a sheet 
	$\CC_{t}\backslash L_{\msigma}$. 
	
	\item 
	Given a horizontal subtree with $s$ vertices, say
	$\big\{\circled{\msigma_1}, \ldots, \circled{\msigma_s}\big\}$, 
	replace the given horizontal subtree with a sheet 
	\\
	\centerline{
	$\CC_{t}\backslash \{ L_{\msigma_\ell} \}_{\ell=1} ^{s} ,$}
	where each
	$L_{\msigma_\ell} $ is the horizontal branch cut associated to the vertex 
	$\circled{\msigma_\ell}$,
	recall Equation \eqref{sheetminuscuts}.
	By the generic horizontal subtrees assumption, 
	all the values $\{t_{\msigma_\ell}\}$ lie on different horizontal trajectories of $\del{}{t}$, 
	then none of the horizontal branch cuts $L_{\msigma_\ell}$ intersect in $\CC_{t}$.
	
	\noindent
	Continue this replacement process for every horizontal subtree.
	
	\noindent
	Note that we obtain stacked copies of $\CC_{t}\backslash L_{\msigma}$ and 
	$\CC_{t}\backslash\{ L_{\msigma_\ell} \}_{\ell=1} ^{s}$, 
	but they retain their relative position respect to the $(r,d)$--skeleton of $\Lambda_{X}$, 
	by the fact that we still have not removed the vertical edges of the $d$--skeleton of 
	$\Lambda_{X}$.
	
	\item 
	We now replace the vertical towers and vertical cycles in the 
	$(r,d)$--skeleton of $\Lambda_{X}$ with $\abs{K}$--helicoids or $(\nu+1)$--cyclic helicoids 
	respectively
	(recall Definitions \ref{helicoide-finito} and \ref{helicoide-ciclico}).
	On each vertical tower or vertical cycle, say the one associated to the vertex 
$\circled{\msigma}$, 
apply Corollary \ref{pegado-isometrico} to glue
the horizontal branch cuts by alternating the boundaries of 
$\CC_{t}\backslash L_{\msigma}$; so as to form finite helicoids or cyclic helicoids
over the vertex $\circled{\msigma}$, 
making sure that all the finite helicoids go up 
when turning 
counter--clockwise around the vertex. 

        \noindent
	In the case where a vertical tower is involved, the finite helicoid has two boundaries 
	consisting of $[t_{\msigma},\infty)_{+}$ and $[t_{\msigma},\infty)_{-}$; in the case where 
	a vertical cycle is 
	involved we obtain a cyclic helicoid, that is a finite helicoid whose boundaries have been 
	identified/glued.
	 
\end{enumerate}
\smallskip

\emph{Non--generic (degenerate) horizontal subtrees}.
We now turn our attention to the
particular case 
when on some horizontal subtree there are at least two 
asymptotic or critical values 
$\{t_{\msigma}\}_{\msigma=1}^{d+n}\subset(\CC,\del{}{t})$ 
arising from the vertices 
$\circled{\msigma}$, 
that lie on the same horizontal trajectory  of $\del{}{t}$.
Since there are only a finite set of asymptotic or critical values, 
then for any small enough angle $\theta > 0$, 
the set of values $\{t_{\msigma}\}\subset(\CC,\e^{i\theta}\del{}{t})$ 
lie on $m+n$ different trajectories of 
the rotated vector field
$\Re{\e^{i\theta}\del{}{t}}$. 
Proceed with the construction (a)--(e) as above but using 
$\e^{i\theta} L_{\msigma}$ instead of $L_{\msigma}$ for 
the construction. Note that for small enough $\theta>0$ 
all the surfaces obtained are homeomorphic. 
Finally let $\theta\to 0^{+}$ and consider the limiting 
surface.

\medskip
According to Definition \ref{soul} we have:

\begin{proposition}\label{soul-of-LambdaX} 
Every $\Lambda_X$ has a canonically associated $(r,d)$--soul.
\end{proposition}

\begin{proof}
The \emph{$(r,d)$--soul of the $(r,d)$--configuration 
tree $\Lambda_{X}$}
is the Riemann surface with boundary, described by (a)--(c) above.
\end{proof}

\begin{example}
The $(r,d)$--soul is shaded blue in all the figures. 

\noindent
1) $X\in\E(r,0)$, so $\Psi_{X}$ is a polynomial, 
in which case the $(r,0)$--soul of $\Lambda_{X}$ is $\R_{X}$. 
See Figure \ref{figejemplo-dos-polos}.

\noindent
2) $X(z)=\e^{z}\del{}{z}\in\E(0,1)$, so $\Psi_{X}$ is an exponential, 
in which case the $(0,1)$--soul of $\Lambda_{X}$ 
consists of $\CC_{t}\backslash L_{1}$, a single sheet with exactly one branch cut. 
See Figure \ref{figejemplo-un-va} and figure 11.a in \cite{AlvarezMucino}.

\noindent
3) $X(z)=\e^{z^{2}}\del{}{z}\in\E(0,2)$, so $\Psi_{X}$ is the error function, 
in which case the $(0,2)$--soul of $\Lambda_{X}$ 
consists of $\CC_{t}\backslash (L_{1}\cup L_{2})$, a single sheet with exactly two branch cuts. 
See Figure \ref{tabla-funcion-error} and figure 11.b in \cite{AlvarezMucino}.

\noindent
4) An $(r,d)$--configuration tree has all $K(\msigma,\mrho)\equiv 0$ if and only if on the corresponding 
Riemann surface $\R_{X}$ all the diagonals share the same sheet 
$\CC_{t}\backslash \{ L_\msigma \}_{\msigma=1}^{d+n}$. 
In this case the $(r,d)$--soul of $\Lambda_{X}$ is this sheet.
See Figure \ref{fig-ejemplo-z9}.b.
\end{example}

\noindent
\textbf{2. Construction of $\R_{X}$ from the $(r,d)$--soul of $\Lambda_X$.}
To each of the $2d$ boundaries of the $(r,d)$--soul of $\Lambda_{X}$, 
glue a semi--infinite helicoid to obtain a simply connected Riemann surface $\R_{X}$. 
This surface has exactly $d$ logarithmic 
branch points over $d$ finite asymptotic values 
of $\Psi_X$
and $n$ finitely ramified branch points with ramification indices 
that add up to $r+n$.

We can recognize that our isometric glueing
Corollary \ref{pegado-isometrico} in the above 
cases coincides with the Maskit surgery 
as is defined by 
M. Taniguchi
\cite{Taniguchi1} p. 68, 
\cite{Taniguchi2} p. 110--115.
In fact, $\R_{X}$ is realized via Maskit 
surgeries with
 
\noindent $\bullet$
$d$ exp--blocks (in our language $2d$ semi--infinite helicoids) and 

\noindent $\bullet$
$r$ quadratic blocks,

\noindent  
hence following\footnote{
It is to be noted that in the case of $X\in\E(r,0)$, $r\geq1$, there is no need to use M. Taniguchi's
results involving 
exponential blocks.
} 
\cite{Taniguchi1} theorem 1 and
\cite{Taniguchi2} theorem 2.14,
there exist polynomials $E(z)$ of degree $d$ and $P(z)$ of degree $r$ arising from $\Lambda_{X}$, 
which characterize the function 
$$
\Psi_{X}\in SF_{r,d}=
\left\{\int_{z_0}^{z}P(\zeta)\,\e^{-E(\zeta)} d\zeta\ + b \ \Big{\vert} \ 
P,E\in\CC[z], \ \deg{P}=r, \ \deg{E}=d\right\}.
$$

\noindent
Alternatively\footnote{
Note that these results are classical, in particular 
R. Nevanlinna \cite{Nevanlinna1}, \cite{Nevanlinna2} and 
G. Elfving \cite{Elfving} essentially proved the 
correspondence between $\R_{X}$ and $\Psi_{X}$.},
since $\R_{X}$ is a log--Riemann surface with $d$ logarithmic branch points over 
$d$ finite asymptotic values
and $n$ finitely ramified branch points with ramification indices that add up to $r+n$ 
whose finite completion is simply connected, 
by theorem 1.1 of \cite{Biswas-PerezMarco-2}, it follows that $\Psi_{X}\in SF_{r,d}$.

Finally assign to $\R_{X}$ a flat metric $\big(\R_{X},\pi_{X,2}^{*}(\del{}{t})\big)$ 
induced by $\pi_{X,2}$.
By the Dictionary Proposition \ref{basic-correspondence}, 
our sought after vector field is 
\\ \centerline{
$X(z)=\Psi_{X}^{*}(\del{}{t})(z)= \frac{1}{P(z)}\, \e^{E(z)} \del{}{z}\in \E(r,d)$ }
\\
as required.

We have essentially proved the following.

\begin{proposition}
Consider the following set of $(r,d)$--configuration
trees

\centerline{
$
\left\{ 
\begin{array}{l}
\Lambda_X 
\text{ has at least two branch points }
\\
\text{over  
different values in  }\CC_t 
\end{array}
\right \}
$,}

\noindent 
then 
the $(r,d)$--soul of $\Lambda_{X}$ determines 
a unique vector field $X \in \E(r,d)$.
\end{proposition}
\begin{proof}
If $\Lambda_X$ has only one ramification value in $\CC_t$, then 
by simple inspection it is

\centerline{
$\Lambda_{X}=\Big\{ \circled{1}=(p_{1},\widetilde{p}_{1}, -r);
\raiz{1}\, ;\,
\varnothing
\Big\}$ 
\ or \
$\Lambda_{X}=\Big\{ \circled{1}=(\infty_{1},a_{1},-\infty); 
\raiz{1}\, ;\,
\varnothing \Big\}$,}

\noindent 
see Examples \ref{ejemplo-un-polo}, \ref{ejemplo-un-va}.
The corresponding vector fields 

\centerline{
$X(z)=\frac{\mu}{ (z-p_1)^{r}} \del{}{z} \in \E(r,0)$ 
\ and \ 
$X(z)=\mu^{-1} \e^z \del{}{z} \in \E(0,1)$,  }

\noindent 
are not uniquely determined,
since $\mu\neq0$ is not unique.

On the other hand, if $\Lambda_X$ 
has at least two branch points 
over  different values, say  $t_1, t_2 \in \CC_t$, 
then the diagonal $\Delta_{1\,2 }$ satisfies the Equation 
\eqref{diagonalsemiresidue},

\centerline{$
\int_{z_1}^{z_2 } 
P(\zeta)\e^{-E(\zeta)} d\zeta
= 
t_{2}-t_{1}. 
$}

\noindent 
It allows us the computation of the factors
$\lambda$, $\mu$ in Equation \eqref{def-lambda}, 
obtaining the uniqness of $P(z)$ and $E(z)$.
\end{proof}

\begin{remark}\label{encajedeLambdaenR}
Note that the $(r,d)$--configuration tree $\Lambda_{X}$  
is a tree ``embedded'' in $\R_{X}\subset\CW_{z}\times\CW_{t}$;
however it is an embedding as a subset of $\overline{\CC}_{z} \times \CW_{t}$. 
It is not a genuine embedding in $\R_{X}$ since the logarithmic branch points are 
not in fact part of the surface $\R_{X}\subset\CC_{z}\times\CW_{t}$ 
(see also Definition \ref{singesen}).
On the other hand, on the $(r,d)$--skeleton of $\Lambda_{X}$, the branch points of $\R_{X}$
are 
replaced by a vertical tower or vertical cycle during the blow--up process of $\Lambda_{X}$
(the vertical edges of the $(r,d)$--skeleton of $\Lambda_{X}$ indicate how 
many sheets separate the diagonals).
\\
In this sense, both the $(r,d)$--configuration tree $\Lambda_{X}$ and the $(r,d)$--skeleton of $\Lambda_{X}$ 
project to a graph $\pi_{X,2}(\Lambda_{X})\subset\CC_{t}$. 
See Figures \ref{figejemplo-dos-polos}--\ref{figejemplo-E33} and \ref{proyeccionArbol},
in particular $\pi_{X,2}(\Lambda_{X})$ need not be a tree as in Figure \ref{figejemplo-E33}.
This is represented by the diagram: 
\begin{center}
\begin{picture}(180,85)(0,10)

\put(-90,40){\vbox{\begin{equation}\label{diagramaSKConf}\end{equation}}}

\put(55,90){$(r,d)$--skeleton of $\Lambda_{X}$}

\put(100,83){\vector(0,-1){24}}
\put(105,70){blow--down}

\put(95,59){\vector(0,1){24}}
\put(55,70){blow--up}

\put(-60,46){$\CW_{z}\times\CW_{t}\ \hookleftarrow \
\R_{X}\ $``\,$\hookleftarrow$\,'' $(r,d)$--configuration tree 
$\Lambda_{X}\hookrightarrow\overline{\CC}_{z}\times\CW_{t}$}

\put(100,38){\vector(0,-1){24}}
\put(105,25){$ \pi_{X,2} $}

\put(70,5){$\pi_{X,2}(\Lambda_{X})\subset\CC_{t}$.}
\end{picture}
\end{center}
\end{remark}

\subsection{The equivalence relation on $(r,d)$--configuration trees}
\label{clasesE(d)}

\begin{remark}\label{clase-equiv-arbol}
{\it Non--uniqueness of $(r,d)$--configuration trees $\Lambda_{X}$ associated to $\Psi_{X}$.} 

\noindent
1. Even though condition (4) of Definition \ref{d-confTree} provides a clear choice for the 
root $\raiz{\varrho}$ of $\Lambda(r,d)$, 
if the branch point corresponding to the root shares more than one sheet with other branch points, 
then each of these sheets could be the global zero level sheet. 
Hence each choice provides a different global zero level subtree and thus a different 
$(r,d)$--configuration tree $\Lambda_{X}$.

\noindent
2. The choice of the weight when considering an edge that connects a pole vertex 
with any other type of vertex is not unique because of the modular arithmetic involved.
For instance, if we have a weighted edge $(\Delta_{\iota\mrho},K(\iota,\mrho))$ 
connecting a pole vertex 
$\circled{\iota}=(p_{\iota},\widetilde{p}_{\iota},-\nu_{\iota})$ 
to any other vertex $\circled{\mrho}$, 
then changing $K(\iota,\mrho)$ to $K(\iota,\mrho)+\ell \nu$ for $\ell\in\ZZ$, will give rise to a 
different $(r,d)$--configuration tree associated to the same $\Psi_{X}$.

\noindent
3. However, 
recalling Remark \ref{propiedadesdelesqueleto}.2--3,
it is to be noted that for fixed $X\in\E(r,d)$, even though there are choices
in the construction of $\Lambda_{X}$, up to re--labelling of the vertices, the
$(r,d)$--skeletons associated to each choice
will be the same as undirected graphs.
\end{remark}

\noindent
Following is an example that illustrates (1)--(3). 

\begin{example}[Example \ref{ejemplo-E33} revisited]
\label{ejemplo-E33-nuevo}
Let us consider once again 
\\
\centerline{
$
X(z)=\dfrac{\e^{z^{3}}}{3z^{3}-1}\ddel{}{z} \in \E(3,3).
$}

\noindent
However we shall now re--label the vertices, choose a different global zero level and assign 
different weights to some edges.
\begin{equation}\label{vertices-nuevos-ejemplo-E33}
\begin{array}{rclrcl}
\circled{1}=(z_{1},t_{1},\nu_{1}) &=& (p_{2},\widetilde{p}_{2},-1),&
\quad\circled{4}=(z_4 ,t_4 ,\nu_4) &=& (p_{1},\widetilde{p}_{1},-1),
\\
\circled{5}=(z_5 ,t_5 ,\nu_5) &=& (p_{3},\widetilde{p}_{3},-1),&
\quad\circled{2}=(z_2 ,t_2 ,\nu_2) &=& (\infty_{1},a_{1},-\infty),
\\
\circled{3}=(z_3 ,t_3 ,\nu_3) &=& (\infty_{2},a_{1},-\infty),&
\quad\circled{6}=(z_{6},t_{6},\nu_{6}) &=& (\infty_{3},a_{1},-\infty).
\end{array}
\end{equation} 

\noindent
Thus the $(3,3)$--configuration tree (see Figure \ref{figejemplo-E33-nuevo}) is
\begin{multline}\label{arbolejemplo-E33-bis}
\Lambda_{X}=
\Big\{ \circled{1},\circled{2},\circled{3},\circled{4},\circled{5},\circled{6};
\raiz{5}\, ;\,
\\
(\Delta_{5\,4}, 0), 
(\Delta_{5\, 6}, -5), 
(\Delta_{4\, 2}, -1), 
(\Delta_{2\, 1}, 0),
(\Delta_{1\, 3}, -1) 
\Big\},
\end{multline}
with edges given by the diagonals/semi--residues
\begin{equation}\label{pesosejemplo-E33-bis}
\begin{array}{ll}
& \Delta_{5\, 4} = \int\limits^{p_{1}}_{p_{3}} \omega_{X} = \widetilde{p}_{1} - \widetilde{p}_{3} = 
-\Big(\dfrac{1-\e^{-i 2\pi/3}}{\sqrt[3]{3\e}}\Big), 
\\
& \Delta_{5\, 6} = \int\limits_{p_{3}}^{\infty_{3}} \omega_{X} = a_{1} - \widetilde{p}_{3} = 
\dfrac{\e^{- i 2\pi/3}}{\sqrt[3]{3\e}},
\\
& \Delta_{4\, 2} = \int\limits_{p_1}^{\infty_1} \omega_{X} = a_1 - \widetilde{p}_1 = 
\dfrac{1}{\sqrt[3]{3\e}},
\\
& \Delta_{2\, 1} = \int\limits_{\infty_1}^{p_{2}} \omega_{X} =  \widetilde{p}_{2} - a_1 =
-\dfrac{1}{\sqrt[3]{3\e}} \e^{i 2\pi/3},
\\
& \Delta_{1\, 3} = \int\limits_{p_{2}}^{\infty_{2}} \omega_{X} = a_{1} - \widetilde{p}_{2} = 
\dfrac{\e^{ i 2\pi/3}}{\sqrt[3]{3\e}}, 
\end{array}
\end{equation}
and weights $K(5,4)=0$, $K(5,6)=-5$, $K(4,2)=-1$, $K(2,1)=0$, $K(1,3)=-1$.

\noindent
Of course this $(3,3)$--configuration 
tree is different from the one of Example \ref{ejemplo-E33},
however \emph{their corresponding $(3,3)$--skeletons are the same up to re--labelling of the vertices}.
Compare Figures \ref{figejemplo-E33} 
and \ref{figejemplo-E33-nuevo}.
\begin{figure}[htbp]
\begin{center}
\includegraphics[width=\textwidth]{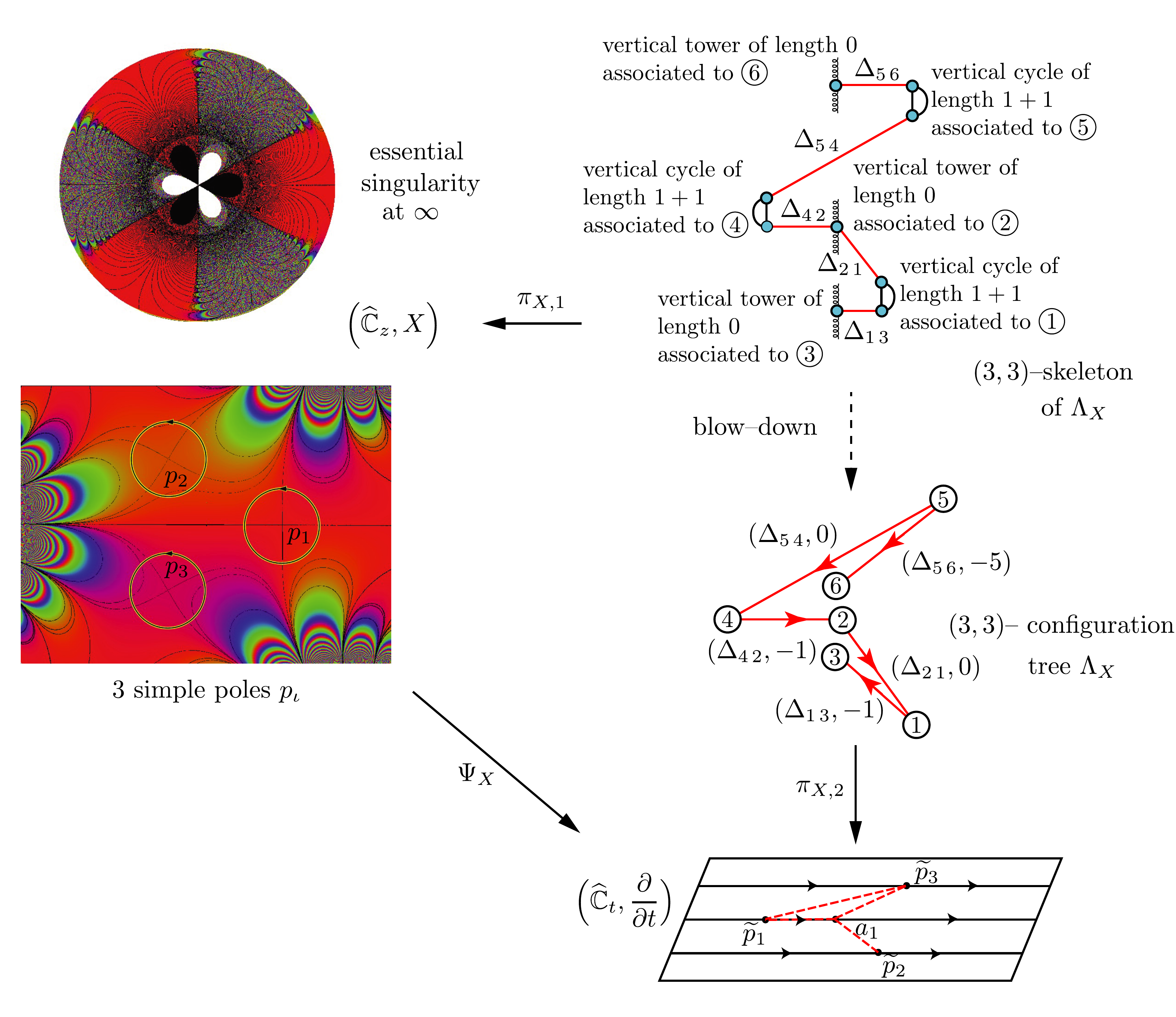}
\caption{{\bf Vector field $\frac{\e^{z^{3}}}{3z^{3}-1}\del{}{z}$ with an essential singularity at $\infty$ and 3 simple poles $p_{\iota}$.}
The five diagonals and their projections are shown in red.
The Riemann surface $\R_{X}$ is not drawn.
See Example \ref{ejemplo-E33-nuevo}.
}
\label{figejemplo-E33-nuevo}
\end{center}
\end{figure}
\end{example}

Summarizing, the following choices and/or conventions have been made 
for the $(r,d)$--skeleton of $\Lambda_{X}$.

\noindent
1) Because of condition (4) of Definition \ref{d-confTree} 
and Remark \ref{correspexponentialtract}.3 
the root $\raiz{\varrho}$ is unique.

\noindent
2) The choice of $z_0\in\CC_z$ as the initial point of integration of 
$\Psi_{X}(z)=\int\limits_{z_0}^{z}\omega_{X}$, 
allows the critical and asymptotic values to be well defined 
(and thus the vertices of the $(r,d)$--skeleton of $\Lambda_{X}$).

\noindent
3) Condition (7) of Definition \ref{d-confTree} provides a unique choice of the subset of diagonals 
needed to specify the $(r,d)$--skeleton of $\Lambda_{X}$.

\noindent
4) The ``blow--up'' of the pole and essential vertices (of the $(r,d)$--configuration tree $\Lambda_{X}$) 
into vertical cycles and vertical towers (of the $(r,d)$--skeleton of $\Lambda_{X}$) respectively, eliminates
the weights $K(\msigma,\mrho)$ for each diagonal $\Delta_{\msigma\,\mrho}$ 
on the $(r,d)$--skeleton of $\Lambda_{X}$. 

\noindent
Thus, as stated in Remark \ref{clase-equiv-arbol}.3, an immediate consequence is the following.

\begin{definition}\label{equivalenciaArboles}
\emph{Two $(r,d)$--configuration trees $\Lambda_{1}$ and $\Lambda_{2}$ are equivalent} 
if their corresponding $(r,d)$--skeletons
are the same up to re--labelling of the vertices.
\end{definition}

This finishes the proof of the Main Theorem. 
\qedhere
\end{proof}

\section{Decomposition of the phase portraits into invariant components}\label{Decomposicion}

Recall Lemma \ref{descomposicion-planos-bandas}
providing a decomposition in
half planes and finite height horizontal strips 
related to $\Re{X}$. 
The interior of these pieces
are invariant open components under 
the real vector field $\Re{X}$.

\begin{theorem}\label{HorizontalStripStructures}
For $X \in \E(r,d)$, 
its phase portrait  
decomposes into $\Re{X}$--invariant components
as follows
\begin{multline}\label{BetterStripDecomposition}
\big(\CC_{z},X\big)=
\underbrace{
\left(\bar{\HH}^2_\pm ,\del{}{z}\right)
\cup \ldots \cup \left(\bar{\HH}^2_\mp ,\del{}{z}\right)
}_{2(r+1) \leq N_p \leq 4r} 
\\
\bigcup_{a_\sigma} 
\Biggl[
\left(\Big\{0\leq |\Im{z}|\leq 2\pi (K_\sigma + 1) \Big\},\e^z \del{}{z}\right)_{a_{\sigma}}\qquad\qquad
\\
\qquad
\cup
\left(\bar{\HH}_{\pm}^2,\e^z\del{}{z}\right)_{a_{\sigma},upper}
\cup
\left(\bar{\HH}_{\pm}^2,\e^z\del{}{z}\right)_{a_{\sigma},lower}
\Biggr]
\\
\bigcup_ {\ell}^{M\leq\infty} 
\left( \Big\{ 0\leq\Im{z}\leq h_{\ell} \Big\},\del{}{z}\right),
\end{multline}
where $\{a_\sigma\}$ are the finite asymptotic values of $\Psi_{X}$, 
$N_p$ is the number of half planes associated to the poles of $X$ 
(equivalently the number of hyperbolic sectors around the poles of $X$) and 
$M$ is the number of diagonals (equivalently the number of 
finite height strips in the phase portrait of $\Re{X}$).

\noindent
There are an infinite number of half planes 
$\big(\bar{\HH}^2_\pm ,\del{}{z}\big)$ in the decomposition if and only if $d\geq 1$.
\end{theorem}
\begin{proof}
Decomposition \eqref{BetterStripDecomposition} follows 
by recalling Definition \ref{piezasFatou},
the biholomorphism $\pi_{X,1}$ presented in Diagram \ref{RX}
and the fine structure of the $(r,d)$--skeleton of $\Lambda_X$. 
It is an 
accurate description of the phase portrait decomposition 
of $\Re{X}$:

\noindent
The first row depicts the, at least $2(r+1)$ and at most $4r$, 
half planes associated to the 
$r$ poles.

\noindent
On the second row are the $d$ finite helicoids
arising from the $d$ finite asymptotic values 
$\{a_\sigma\}$, 
where it is to be noticed that this can be an empty 
collection. 

\noindent
On the third row are the $2d$ semi--infinite helicoids.

\noindent
Finally, on the fourth row the finite height strips 
associated to the diagonals in $\R_{X}$.
\end{proof}

\begin{corollary}
\label{trayectorias-incompletas}
Let $X \in \E(r,d)$, $d \geq 1$. 
Then,
the incomplete trajectories of $\Re{X}$ in $\CW_z$
are infinite, numerable and  Lebesgue measure zero. 
\hfill $\qed$
\end{corollary}

Obviously, the number of incomplete trajectories
is finite if and only of $r\geq1$ and $d=0$.
Compare with
\cite{LopezMucino} and  
\cite{Langley}. 

\section{On the topology of $\Re{X}$}\label{topoclassification}
Consider the group of orientation 
preserving 
homeomorphisms of $\CC$, 
$$
Homeo(\CC)^{+}
=\{h:\CW_{z}\rightarrow\CW_{z}
\ | \ 
\text{ preserving orientation and fixing } \infty\in\CW\}.$$

\begin{definition}\label{deftopequiv}
Let $X_{1}, X_{2}\in\E(r,d)$ be two singular analytic vector fields.

\noindent
They are \emph{topologically equivalent}
if there exists $h\in Homeo(\CC)^{+}$ which takes the 
trajectories of $\Re{X_{1}}$ to trajectories of $\Re{X_{2}}$, 
preserving real time orientation, but not 
necessarily the parametrization. 

\noindent 
A \emph{bifurcation for $\Re{X_{1}}$ occurs},
when the topology of its phase portrait topologically 
changes under small deformation of $X_1$ in the family 
$\E(r,d)$, otherwise $\Re{X_1}$
is \emph{structurally stable} in $\E(r,d)$.
\end{definition}

Let  $\Lambda_{X}=\Big\{ 
\big\{ \circled{\msigma}=\big(z_{\msigma},t_{\msigma},-\nu_{\msigma} \big) 
\big\}_{\sigma = 1}^{d+n} ;
\, \raiz{\varrho} \, ;
\big\{
(\Delta_{\msigma \mrho}, K(\msigma ,\mrho)) 
\big\} \Big\}$ 
be a $(r,d)$--configuration tree. 
By simple inspection we have

\begin{theorem}[Structural stability of $\Re{X}$ for $X\in\E(r,d)$]\label{estabilidadestructural}
\hfil\\
The real vector field
$\Re{X}$ is structurally stable in $\E(r,d)$ if and only if 
 
\noindent
i) $X$ has only simple poles and

\noindent
2) 
$\Im{\Delta_{\msigma\mrho} }\neq 0$ for all 
weighted edges of $\Lambda_{X}$.
\end{theorem}

\begin{proof}
A diagonal $\Delta_{\msigma\mrho}$ has 
$\Im{\Delta_{\msigma\mrho}} = 0$
and poles of $X$ at its two extreme points, 
if and only if 
$\Delta_{\msigma\mrho}$ determines a 
saddle connection of $\Re{X}$, 
see Lemma \ref{descomposicion-planos-bandas}.3.   
This is the unique bifurcation scenario.
\end{proof}

Recalling that an $(r,d)$--skeleton of $\Lambda_{X}$ is a graph embedded in 
$\overline{\CC}_z \times \CC_t$, with a specific complex parameter  
associated to each edge and
having horizontal and vertical attributes, we now make the following.

\noindent
As a direct consequence of the structure of the 
$(r,d)$--skeleton of $\Lambda_X$ we have.
%
\begin{theorem}
[Number of non topologically equivalent vector fields $\Re{X}$
for $X\in\E(r,d)$]
\label{numbertop}
\hfill
\\
Given a fixed pair $(r,d)$:
\begin{enumerate}[label=\arabic*),leftmargin=*] 
\item 
The number of topologies of $\Re{X}$ is 
infinite when 
\\
\centerline{
$(r,d)\in\big\{(r\geq2,1), (r\geq1,2),(r\geq0,d\geq3)\big\}$.}

\item The number of topologies is
\begin{enumerate}[label=\alph*),leftmargin=*]
\item one when $(r,d)=(0,1), (1,0)$;
\item two when $(r,d)=(0,2)$, $(1,1)$;
\item finite when $(r,0)$.
\end{enumerate}
\end{enumerate}
\end{theorem}

For $X \in \E(r,0)$, the phase portrait
$\Re{X}$ on $\CC_{z}$ 
only has a finite number $n \leq r$ of multiple saddle points. 
These phase portraits 
were first studied by W.~M.~Boothby \cite{Boothby1}, \cite{Boothby2}, 
showing that they appear as the real part of 
certain harmonic functions (non necessarily polynomials);
in our framework, 
the imaginary part of 
$\Psi_X(z)=\int^z P(\zeta) \e^{-E(\zeta)} d\zeta$.

\begin{proof}

For each $(r,d)\in\big\{(r\geq2,1), (r\geq1,2),(r\geq0,d\geq3)\big\}$ there will be at least one topological
$(r,d)$--skeleton of $\Lambda_{X}$ with at least one vertical tower with two horizontal subgraphs attached to 
the \emph{same} vertical tower. 
These horizontal subgraphs are vertically separated from each other by an integer number 
$K(\sigma,\rho)$, of degree 2 vertices on a vertical tower. 
Hence there are an infinite number of different ways, described by $\{K(\sigma,\rho)\geq 1\}$, 
we can attach these two subgraphs to the vertical tower, 
each of which represents a different configuration in $\R_{X}$.

\smallskip
The remaining cases are 
$(r,d)\in\big\{ (0,1), (0,2), (1,1), (r\geq1,0)\big\}$.
\\
The cases $(0,1), (1,0)$ are trivial by Lemma \ref{oneasymptoticvalue}.
For cases $(0,2)$ and $(1,1)$: $\R_{X}$ has two branch points 
hence they must share the same sheet. 
Thus each one of these cases have exactly two topologies.
Case (0,2) is illustrated in Figure
\ref{tabla-funcion-error}.
Case $(r\geq 2,0)$ corresponds to $\Psi_{X}$ being a polynomial, 
hence the number of topologies is finite. 
\end{proof}

Table \ref{drtopo} presents a summary of the 
possible non topologically equivalent vector fields $\Re{X}$, for $X\in\E(r,d)$, that arise 
for different pairs $(r,d)$.
\begin{table}[htp]
\caption{Topologies of $\Re{X}$ for different pairs $(r,d)$.}
{\small
\begin{center}
\begin{tabular}{|c|c|c|l|}
\hline
& & & \\[-8pt]
$r$ & $d$ & \# of topologies & $(r,d)$--configuration trees $\Lambda_X$ \\
& & of $\Re{X}$ & chosen as a representative for the topological class\\[2pt]
\hline 
\hline
& & & \\[-8pt]
1 & 0 & 1 & $\Lambda_X=\{(p_{1},\widetilde{p}_{1},-1);\varnothing\}$\\[2pt]
\hline
& & & \\[-8pt]
0 & 1 & 1 & $\Lambda_X=\{(\infty_{1},a_{1},-\infty);\varnothing\}$\\[2pt]
\hline
& & & \\[-8pt]
0 & 2 & 2 & $\Lambda_X=\{(\infty_{1},a_{1},-\infty), (\infty_{1},a_{2},-\infty); \big( \Delta_{1\, 2}, K(1, 2) \big) \},$\\
& & & \ with $\Delta_{1\, 2}\in\CC^{*}$, two topologies: $\Delta_{1\, 2}\in\RR, \Delta_{1\, 2}\not\in\RR$\\[2pt]
\hline
& & & \\[-8pt]
1 & 1 & 2 & $\Lambda_X=\{(\infty_1,a_1,-\infty), (p_1,\widetilde{p}_1,-1); \big(\Delta_{1\, 2}, K(1\, 2) \big) \},$\\
& & & \ with $\Delta_{1\, 2}\in\CC^{*}$, two topologies: $\Delta_{1\, 2}\in\RR, \Delta_{1\, 2}\not\in\RR$\\[2pt]
\hline
& & & \\[-8pt]
2 & 0 & 3 & $\Lambda_X=\{(p_{1},\widetilde{p}_{1},-2); \varnothing \},$ gives rise to one topology.\\
& & & $\Lambda_X=\{(p_1,\widetilde{p}_1,-1), (p_2,\widetilde{p}_2,-1); \big(\Delta_{1\, 2}, K(1\, 2) \big) \},$\\
& & & \ with $\Delta_{1\, 2}\in\CC^{*}$, two topologies: $\Delta_{1\, 2}\in\RR, \Delta_{1\, 2}\not\in\RR$\\[2pt]
\hline
& & & \\[-8pt]
$r\geq3$ & 0 & finite & 
$\Lambda_X=\Big\{(p_1,\widetilde{p}_1,-\nu_1), 
\dots, (p_n,\widetilde{p}_n,-\nu_n);$\\ 
& &  & 
\qquad \qquad \ \ $\{\big(\Delta_{\iota \kappa}, K(\iota, \kappa)\big) \ \vert\ \iota,\kappa\in\{1,\dots,n-1\} \}  \Big\},$\\
& & & $1\leq n\leq r$ being the number of distinct poles\\[2pt]
\hline
& & & \\[-8pt]
$r\geq2$ & $1$ & infinite & 
$\Lambda_X=\Big\{(\infty_1,a_1,-\infty),$\\
& & & 
\qquad \qquad \ \ $(p_1,\widetilde{p}_1,-\nu_1), \dots, (p_n,\widetilde{p}_n,-\nu_n);$\\ 
& & & 
\qquad \qquad \ \ $\{\big(\Delta_{\msigma \mrho}, K(\msigma, \mrho)\big) \ \vert\ 
\msigma,\mrho\in\{1,\dots,n+1\} \}  \Big\},$\\
& & & $1\leq n\leq r$ being the number of distinct poles\\[2pt]
\hline
& & & \\[-8pt]
$r\geq1$ & $2$ & infinite & 
$\Lambda_X=\Big\{(\infty_1,a_1,-\infty), (\infty_2,a_2,-\infty),$ \\
& & & 
\qquad \qquad \ \ $(p_1,\widetilde{p}_1,-\nu_1), \dots, (p_n,\widetilde{p}_n,-\nu_n);$\\ 
& & & 
\qquad \qquad \ \ $\{\big(\Delta_{\msigma \mrho}, K(\msigma, \mrho)\big) \ \vert\ 
\msigma,\mrho\in\{1,\dots,n+2\} \}  \Big\},$\\
& & & $1\leq n\leq r$ being the number of distinct poles\\[2pt]
\hline
& & & \\[-8pt]
$r\geq0$ & $d\geq3$ & infinite & 
$\Lambda_X=\Big\{(\infty_1,a_1,-\infty), (\infty_2,a_2,-\infty),$ \\
& & &
\qquad \qquad \ \ 
$(\infty_3,a_3,-\infty), \dots, (\infty_d,a_d,-\infty),$ \\
& & &
\qquad \qquad \ \ 
$(p_1,\widetilde{p}_1,-\nu_1), \dots, (p_n,\widetilde{p}_n,-\nu_n);$\\ 
& & & 
\qquad \qquad \ \ $\{\big(\Delta_{\msigma \mrho}, K(\msigma, \mrho)\big) \ \vert\ 
\msigma,\mrho\in\{1,\dots,d+n\} \}  \Big\},$\\
& &  & $0\leq n\leq r$ being the number of distinct poles\\[2pt]
\hline
\end{tabular}
\end{center}
\label{drtopo}
}
\end{table}

\begin{remark}
\label{function-coarser-vectorfield}
As a side note, 
Theorem \ref{numbertop} and 
Proposition \ref{lema-arbol-3-0} 
show that the topological classification of 
functions is coarser than the
topological classification of phase portraits of vector fields, 
even for $\Psi_X$ and $X$ in $\E(r,d)$.
\end{remark}


\section{The essential singularity at $\infty$}\label{singularidad-esencial-al-infinito}

Our naive question;

\centerline{
\emph{how can we describe 
the 
essential singularity of $X$ at $\infty\in\CW_{z}$, 
for $X\in\E(r,d)$?,}} 

\noindent 
is answered in this section. 

\smallskip 

In \cite{AlvarezMucino} \S5, germs of singular analytic vector fields $X$ are studied;
the Poincar\'e--Hopf index theory 
and a certain version of the 
decomposition in angular sectors 
for essential isolated singularities
are established for the phase portrait of
$\Re{X}$.

\noindent In fact, 
starting with a simple closed anticlockwise orientated
path $\gamma$
enclosing $z_{\vartheta} 
\in \{p_1, \ldots, p_r, \infty \} \subset  \CW_z$
a pole, zero or essential singularity,
the notion of an admissible cyclic word 
$\mathcal{W}_{X}$ in the alphabet 
$\{H, E, P, \ent{}{}\}$ is 
well defined,
\begin{equation}
\label{de-germen-a-palabra}
\big( (\CW, z_{\vartheta}),\Re{X} \big)
\longmapsto 
\mathcal{W}_{X}.
\end{equation}

\begin{remark}
\label{definicion-sectores}
The letters in the alphabet are the usual angular sectors 
for vector fields: hyperbolic $H$, elliptic $E$, 
parabolic $P$ (see 
\cite{Andronov-L-G-M} p.~304, 
\cite{Arnold-Ilyashenko} p.~86) and 
Figure \ref{los-cuatro-sectores}.
A new  
\emph{class 1 entire sector $\ent{}{}$} 
based upon $\e^z \del{}{z}$ at infinity: 
recalling Diagram \ref{diagramaRX}, 
$\ent{}{}$ can be thought as the image under $\pi_{X,1}$ of a semi--infinite helicoid
contained in $\R_{X}$,
see Example \ref{ejemplo-un-va}, Figures \ref{figPiezasBasicas}.a, \ref{figejemplo-un-va} 
and \ref{los-cuatro-sectores}. 
For full details see \cite{AlvarezMucino} p.~151.
\begin{figure}[htbp]
\begin{center}
\includegraphics[width=0.95\textwidth]{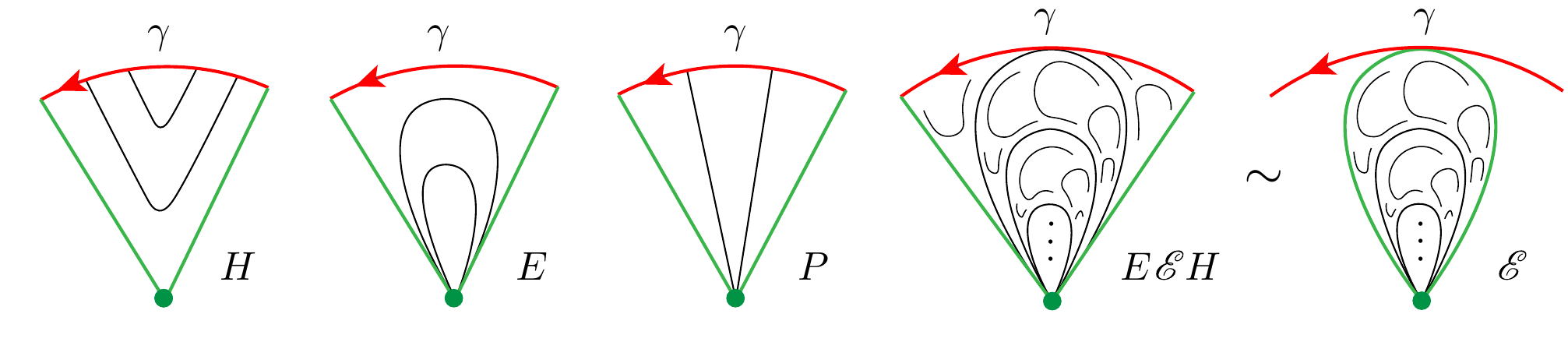}
\caption{ 
Hyperbolic $H$, elliptic $E$, parabolic $P$ and 
entire $\ent{}{}$ sectors at a singular point 
$(\CW_z, z_{\vartheta})$ of 
a vector field $\Re{X}$.
The point $z_{\vartheta}$ is green,  
the path $\gamma$ is shown in red.
The equivalent relation 
$E\ent{}{} H \sim \ent{}{}$
is illustrated on the right.
}
\label{los-cuatro-sectores} 
\end{center}
\end{figure}
\end{remark}

\smallskip
Specific attributions encoded by the word $\mathcal{W}_X$ 
in \eqref{de-germen-a-palabra}
are as follows.

\begin{enumerate}[label=\arabic*),leftmargin=*]
\item
\emph{Equivalence classes.}
The word $\mathcal{W}_X$ is well defined up to the
relations  

\centerline{ $E\ent{}{}H \sim \E $ 
\  and \ 
$H\ent{}{}E \sim \E $,}

\noindent 
according to \cite{AlvarezMucino} pp. 166--167.
Under this equivalence the word becomes independent of the 
choice of the path $\gamma$ enclosing the singularity.

\item
\emph{Poincar\'e--Hopf index.}
If
the number of letters $H$, $E$ and $\E$ that appear in a word 
$\mathcal{W}_{X}$ at $z_{\vartheta}$, 
is denoted by $h$,  $e$ and $\varepsilon$ respectively, then
the Poincar\'e--Hopf index formula is
\begin{equation}
\label{Poincare-Hopf-formula-a-la-Bendixon}
PH(X, z_{\vartheta}) = 1 + \frac{e-h + \varepsilon}{2}.
\end{equation}

\noindent 
Furthermore,
in theorem A p.~130 and \S6 of \cite{AlvarezMucino}, 
the Poincar\'e--Hopf index theorem
\begin{equation}
\label{Poincare-Hopf-formula-global}
\chi(\CW) = \sum PH(X, z_{\vartheta}) 
\end{equation}

\noindent 
is extended to include germs of 
singular analytic vector fields 
$X$ that determine an admissible cyclic word.

\item
\emph{Displacement of parabolic sectors.} 
As matter of record, each parabolic sector 
$P_\nu$ of $\mathcal{W}_X$
has a displacement number $\nu \in \CC \backslash \RR$,  
see \cite{AlvarezMucino} pp. 149--150. 

\item
\emph{The residue.} In fact the residue of 
the vector field germ is 

\centerline{$ Res(\omega_X ) \doteq 
Res(X, z_{\vartheta})= 
\frac{1}{2\pi i }\int_\gamma \omega_X$, }

\noindent recall 
that $\omega_X(X)=1$, also see
\cite{AlvarezMucino} p.~167.

\noindent 
Clearly for $X \in \E(r,d)$ all the residues are zero, since 
$\omega_X = P(z)\e^{-E(z)}dz $ is holomorphic on $\CC_z$.
\end{enumerate}

\smallskip

The action \eqref{LaAccion}, 

\centerline{
$
\mathcal{A}: Aut(\CC) \times \E(r,d) \longrightarrow \E(r,r), 
\ \ \
(T, X) \longmapsto T^*X$, 
}

\noindent 
is a valuable tool, see \cite{AlvarezMucino2}. 

\noindent 
Clearly, for each singularity 
$z_{\vartheta} \in \{p_1, \ldots , p_r, \infty\}$ of $X$,
the germ $\big( (\CW_z, z_{\vartheta }),X(z) \big)$ is

\noindent 
$\bigcdot$
a local analytic invariant (under the
local biholomorphisms of $(\CW_z, z_{\vartheta})$, and

\noindent 
$\bigcdot$
an analytic invariant under 
the above action of $Aut(\CC)$.

\begin{example}[Cyclic words at poles]
For $
X(z)=
\frac{1}{(z-p_{\iota})^{\nu_{\iota}} } \del{}{z}$,
the cyclic word
$\mathcal{W}_{X}$ consists of 
exactly $2(\nu_{\iota}+1)$ hyperbolic sectors $H$,
see Figure \ref{fig:forma-normal}:
\begin{equation}
\Big( (\CC_{z},p_{\iota}),
\Re{X}
\Big) \longmapsto 
\mathcal{W}_X=\underbrace{HH\cdots HH}_{2(\nu_{\iota}+1)}.
\end{equation}
\noindent 
The Poincar\'e--Hopf index 
is $PH(X,p_\iota)= - \nu_\iota$. 
\end{example}

\begin{example}[A cyclic 
word at $\infty$, from a zero of $X$]
Recall the rational vector field 
$
X(z)=\frac{1}{
(z-p_{1})^{\nu_{1}} (z-p_{1})^{\nu_{2}} 
} 
\del{}{z}$ in Example 
\ref{ejemplo-dos-polos}, 
in our language the description of the singularity 
at infinity is
\begin{equation}
\Big( (\CW_{z}, \infty),
\Re{X} \Big) \longmapsto 
\mathcal{W}_X=\underbrace{EE\cdots EE}_{\nu_1+ \nu_2 + 2 }.
\end{equation}
\noindent 
The Poincar\'e--Hopf index is 
$PH(X,\infty)= \nu_1+ \nu_2 +2 $,
also see Figure \ref{fig:forma-normal}.
\end{example}

\begin{example}
[The cyclic word at $\infty$
of the exponential vector field has two entire sectors]
Recall the exponential vector field 
$X(z)= \e^z \del{}{z}$
in Example 
\ref{ejemplo-un-va} 
and Figure \ref{figejemplo-un-va},
we have 
\begin{equation}
\Big( (\CW_{z}, \infty),
\Re{X} \Big) 
\longmapsto 
\mathcal{W}_X= 
E \ent{}{} H \ent{}{} \sim \ent{}{} \ent{}{}.
\end{equation}

\noindent 
The Poincar\'e--Hopf index of $X$ at $\infty$ is 
$2$.
\end{example}

\begin{example}[The error function]
\label{dos-campos-de-la-funcion-error}
The vector field

\centerline{
$X(z)= \mu \frac{\sqrt{\pi} }{4}  \e^{z^ 2} \del{}{z}$, 
\ \ \ $\mu \in \CC^*,$}

\noindent 
has associated the error function

\centerline{$\Psi(z) = \mu^{-1}\frac{2}{ \sqrt{\pi} }
\int_0 ^z \e^{-\zeta^2} d\zeta$.}

\noindent Case $\mu= 1$, the 
logarithmic branch points are 

\centerline{$\{ (\infty_1, -1, -\infty), (\infty_2, 1, -\infty),  
(\infty_3, \infty, -\infty), (\infty_4, \infty, -\infty) \}$,} 

\noindent 
using the notation in equations \eqref{essenvert}, and 
the $\Re{X}$--invariant decomposition is

\centerline{$
(\CW_z, X) = \bigcup_{\sigma=1}^\infty 
\big( \overline{\HH}^2_\sigma , \del{}{z} \big) . $}

\noindent The cyclic word is
$$
\Big( (\CW_{z}, \infty),
\Re{X} \Big) \longmapsto 
\mathcal{W}_X= 
E \ent{}{} H H \ent{}{} E \ent{}{} HH \ent{}{}.
$$
\noindent 
See  Figure \ref{tabla-funcion-error}. 

\noindent 
Case $\mu= i$, the 
logarithmic branch points are 

\centerline{$\{ (\infty_1, -i, -\infty), (\infty_2, i, -\infty),  
(\infty_3, \infty, -\infty), (\infty_4, \infty, -\infty) \}$,} 

\noindent and the $\Re{X}$--invariant decomposition is

\centerline{$
 (\CW_z, X) = \Big(\bigcup_{\sigma=1}^\infty 
 \big( \overline{\HH}^2_\sigma,  \del{}{z} \big)
 \Big) \
\cup \ \big( \{ -1 \leq \Im{z} \leq 1 \}, \del{}{z} \big).$}

\noindent The cyclic word is
$$
\Big( (\CW_{z}, \infty),
\Re{X} \Big) \longmapsto 
\mathcal{W}_X= 
E \ent{}{} H H  \ent{}{} P_{2i} E \ent{}{} HH \ent{}{} P_{-2i},
$$
note that 
the appearance of two opposite parabolic sectors 
having displacements $\pm 2i$
is due the
horizontal strip in the decomposition.
See  Figure \ref{tabla-funcion-error}. 

\noindent
In both cases the Poincar\'e--Hopf index of $X$ at $\infty$ is 
$2$.
\begin{figure}[htbp]
\begin{center}
\includegraphics[width=0.6\textwidth]{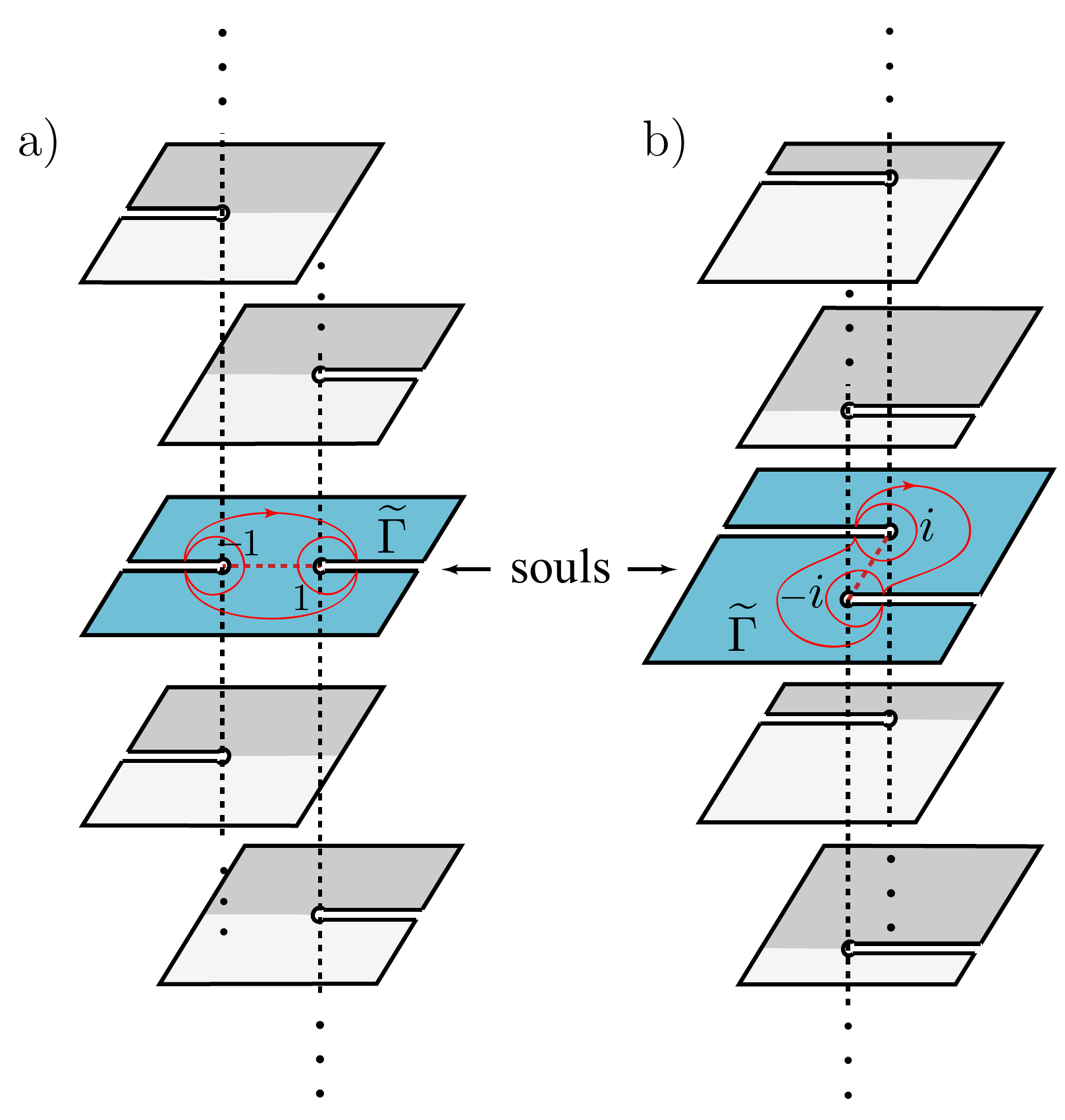}
\caption{ 
The Riemann surfaces of $X$ and $\e^{i \pi/2} X$
for $X \in \E(0,2)$. 
In (a) is the one associated to 
the error function 
$\Psi_{X_1}(z) = \frac{2}{ \sqrt{\pi} }
\int_0 ^z \e^{-\zeta^2} d\zeta$, having a diagonal
$\Delta_{1\,2}$ which determines a 
homoclinic trajectory
of $\Re{X}$.
In (b) is the one corresponding to 
$\Psi_{X_2}(z) = \frac{2i}{ \sqrt{\pi} }
\int_0 ^z \e^{-\zeta^2} d\zeta$.
The red curves represent taut $\widetilde{\Gamma}$'s that allow the 
recognition of the words.
The global topologies of the corresponding 
$\Re{X}$, $X \in \E(0,2)$, are described  in the third row of Table 2, 
and the germ of singularities at $\infty$ in Example 
\ref{dos-campos-de-la-funcion-error}.
}
\label{tabla-funcion-error} 
\end{center}
\end{figure}
\end{example}

\begin{remark}
\emph{The singularity
at infinity does not determine the analytic 
class of $X$.}
\label{el-germen-al-infinito-no-determina-la-clase-de-X}
1. For $X\in\E(3,0)$ having simple zeros,
all the  germs $\big( (\CW_z, \infty), X \big)$ 
are analytically equivalent. Thus the singularity
at infinity does not determine the analytic 
class of $X$ in $\E(3,0)/Aut(\CC)$.  

\noindent
2. We further note that the vector field 
$$
\widetilde{X}(p_3, z)=\mu \dfrac{2p_3 - 1}{12 z(z-1)(z- p_3)}
\e^{z^{4}}
\ddel{}{z}\in\E(3,4),
$$

\noindent
has the same behaviour in $\{z\in\CC\ \vert\ \abs{z}<R\}\subset\CC_{z}$, 
for adequate choices of $\mu\in\CC^{*}$ and  $R>0$,
as $X(p_3, z)$ given by \eqref{familia-polinomios-que-vienen-de-3-polos}, hence the singularity
at infinity does not determine the analytic 
class of $X$ in $\E(3,4)/Aut(\CC)$, see \cite{AlvarezMucino} \S9 and \S10.
\end{remark}
Thus, for vector fields  in $\E(r,d)$, $r+d \geq 3$,
the $Aut(\CC)$--equivalence notion from
\eqref{LaAccion} is very rigid.

We now have that for the essential singularity:

\begin{theorem}\label{corresp-germen-palabra}
\ 
\begin{enumerate}[label=\arabic*),leftmargin=*]
\item 
Let be $X\in\E(r,d)$,
the cyclic word $\mathcal{W}_{X}$ at $\infty$
is recognized as 
\begin{equation}
\big( (\CW_{z},\infty),\Re{X} \big) \longmapsto \mathcal{W}_X=W_{1} W_{2} \cdots W_{k}, 
\quad W_{\iota}\in\{ H,E,P,\ent{}{} \},
\end{equation}
with exactly $\varepsilon=2d$ letters $W_{\iota}=\ent{}{}$.

\noindent
Moreover, $h-e=2(d-r-1)$.

\item 
The word $\mathcal{W}_X$  
is a local topological invariant
of the germ $\big( (\CW_z, \infty), \Re{X} \big)$.

\item
Conversely, 
a germ of a singular complex analytic vector field 
$\big((\CC ,0), Y \big)$  
is the restriction of an $X \in \E(r,d)$ at $\infty$ 
if and only if 
the point $0$ is 
an isolated essential singularity 
of $Y$ 
and  satisfies that
\\
i) the residue of the word $Res(\omega_{Y} ) = 0$,
\\
ii) the Poincar\'e--Hopf index of the word $PH(Y, 0 ) = 2 + r$, 
\\
iii) $\mathcal{W}_{Y}$ is an admissible cyclic word
having exactly $2d$ entire sectors $\ent{}{}$.
\end{enumerate}
\end{theorem}

\begin{proof} 
The proof of statement (1) follows 
the arguments in \S5, \S9 and \S10 
of \cite{AlvarezMucino}.

\noindent
Step 1: Take a simple path $\gamma \subset (\CW_z, \infty)$ enclosing only 
$\infty$ 
($\gamma$ does not enclose any poles of $X$).

\noindent
Step 2: Lift $\gamma$ to $\Gamma$ in $\R_{X}\subset\CW_{z}\times\CW_{t}$.
Note that a priori,
$\Gamma$ does not lie completely in the soul of $\R_{X}$, recall Definition \ref{soul}.

\noindent
Step 3: 
The singularity at $\infty$ of $X$ has a certain self--similarity
(as the examples in \S \ref{concreteexamples} shown),
hence in order to recognize a simple word describing it, a suitable 
deformation of 
$\Gamma$ is required. 
That is, we
deform $\Gamma$ to a \emph{taut} deformation $\widetilde{\Gamma}$ 
in the soul of $\R_{X}$.
For examples of a taut deformation $\widetilde{\Gamma}$ see Figures \ref{tabla-funcion-error} and \ref{TresValoresAsintoticos}. 
For the appropriate technical definitions and another example see
pp. 211--212 of \cite{AlvarezMucino}, in particular figure 17. 

\smallskip 
The taut deformation $\widetilde{\Gamma}$ recognizes 
letters $W_{\iota}$ at $\infty$ as follows:
\begin{itemize}[leftmargin=*]
\item letters $P$ when $\widetilde{\Gamma}$ crosses
finite height strip flows,

\item letters $H$ when $\widetilde{\Gamma}$ 
makes a half circle around 
a branch point of $\R_{X}$,

\item letters $E$ when $\widetilde{\Gamma}$ 
makes a half circle around (the branch point at)
$\infty$ on a sheet of $\R_{X}$,

\item letters $\E$ when $\widetilde{\Gamma}$ 
touches a component of the boundary of the soul;
see Figures \ref{los-cuatro-sectores} and \ref{tabla-funcion-error}.
\end{itemize}

\smallskip 

As for the difference $h-e$ between the number of sectors $H$ and $E$ appearing in the 
cyclic word $\mathcal{W}_{X}$ at $\infty$, 
we shall use 
the Poincare--Hopf index theory extended to these kinds of
singularities 
(Theorem A in \S6 of \cite{AlvarezMucino} with 
$M=\CW_{z}$). 

\noindent
From the fact that $X\in\E(r,d)$ 
has exactly $r$ poles (counted with order) in $\CC_{z}$
and since $PH(X,p_{\iota})=-\nu_{\iota}$ for a pole $p_{\iota}$ of order $-\nu_{\iota}$,
then equation (6.6) of \cite{AlvarezMucino} gives us

\begin{equation}\label{66nueva}
2=\chi(\CW)=PH(X,\infty)+\sum_{p_{\iota}\in\MP} PH(X,p_{\iota}) \\
=PH(X,\infty)-r.
\end{equation}

On the other hand from equation (6.5) of \cite{AlvarezMucino}  

\centerline{
$PH(X,\infty)=1+\frac{e-h+2d}{2}$,
}

\noindent 
the result follows.

Assertion (2) follows by simple inspection. 

For assertion (3), use
a slight modification of corollary 10.1 of 
\cite{AlvarezMucino}. 
The only change arises from the fact that $X\in\E(r,d)$ 
has exactly $r$ poles (counted with order) in $\CC_{z}$. 
Once again, by \eqref{66nueva} the result follows.
\end{proof}

\begin{example}[Cyclic words at $\infty$]
1. Recall the vector field in Example \ref{ejemplo-alma-no-trivial}
$$X(z)=\dfrac{\e^{z}}{(z-9i\frac{\pi}{2}) (z+i\frac{\pi}{2})}\del{}{z}$$
Figure \ref{figejemplo-alma-no-trivial} shows the $(2,1)$--skeleton of $\Lambda_{X}$ 
together with the soul of $\Lambda_{X}$. Here we also show a \emph{taut} curve 
$\widetilde{\Gamma}_1\cup\cdots\cup\widetilde{\Gamma}_{10}=\widetilde{\Gamma}(\tau)
=\big( \gamma(\tau),(\Psi_{X}\circ\gamma)(\tau) \big) \subset\R_{X}$ 
where $\gamma(\tau)$ is a simple closed curve enclosing $\infty\in\CW_z$ with $p_1=9i\frac{\pi}{2}$
and $p_2=-i\frac{\pi}{2}$ lying in its exterior. 
As shown in \cite{AlvarezMucino} \S9.1, we can read the admissible cyclic word 
\begin{equation}
\mathcal{W}_X=
\underbrace{E E}_{\widetilde{\Gamma}_1}
\underbrace{P}_{\widetilde{\Gamma}_2}
\underbrace{EE}_{\widetilde{\Gamma}_3}
\underbrace{EE}_{\widetilde{\Gamma}_4}
\underbrace{P}_{\widetilde{\Gamma}_5}
\underbrace{EE}_{\widetilde{\Gamma}_6}
\underbrace{E P E \E H H}_{\widetilde{\Gamma}_7}
\underbrace{HH}_{\widetilde{\Gamma}_8}
\underbrace{HH}_{\widetilde{\Gamma}_9}
\underbrace{H H \E E P E}_{\widetilde{\Gamma}_{10}},
\end{equation}
so that the number of elliptic, hyperbolic and entire sectors are $e=12$, $h=8$ and $\varepsilon=2$
respectively.
Thus we may calculate the Poincar\'e--Hopf index at $\infty$, 
which in this case turns out to be 
$$
\begin{array}{ll}
PH(X,\infty)& = 
1+
\dfrac{e-h+\varepsilon}{2}
\\
&=1+3=4.
\end{array}
$$

\noindent 
2. Recall the vector field 
$X(z)=\frac{-\e^{z^3}}{ 3z^2} \del{}{z}$
in Example \ref{ejemplo-E23},  
\begin{equation}
\Big( (\CW_{z}, \infty),
\Re{X} \Big) 
\longmapsto 
\mathcal{W}_X= \ent{}{} \ent{}{}  \ent{}{} \ent{}{}\ent{}{} \ent{}{}
\end{equation}
\noindent 
The Poincar\'e--Hopf index of $X$ at $\infty$ is 
$4$.

\noindent 
3. Recall the vector field 
$X(z)=\frac{\e^{z^3}}{ 3z^3-1} \del{}{z}$
in Example \ref{ejemplo-E33},  
\begin{equation}
\Big( (\CW_{z}, \infty),
\Re{X}\Big) 
\longmapsto 
\mathcal{W}_X= \ent{}{} EE \ent{}{}  \ent{}{} \ent{}{}\ent{}{} \ent{}{}.
\end{equation}
\noindent 
The Poincar\'e--Hopf index of $X$ at $\infty$ is 
$5$.
\end{example}

In the next example we show that it is possible to calculate, among other things,
the Poincar\'e--Hopf index.

\begin{example}\label{EjemploPHrevisit}[Example \ref{EjemploPH} revisited.]
Considering Example \ref{EjemploPH} once again, 
Figure \ref{figPoincHopf-2} shows the $(r,4)$--skeleton of $\Lambda_{X}$ 
together with the soul of $\Lambda_{X}$. Here we also show a \emph{taut} curve 
$\widetilde{\Gamma}(\tau)=\big( \gamma(\tau),(\Psi_{X}\circ\gamma)(\tau) \big) \subset\R_{X}$ 
where $\gamma(\tau)$ is a simple closed curve enclosing $\infty\in\CW_z$ with $p_1$
and $p_2$ lying in its exterior. 
As shown in \cite{AlvarezMucino} \S9.1 and \S6, we can read the admissible cyclic word 
\begin{multline}
\mathcal{W}_X=
\underbrace{E E}_{\widetilde{\Gamma}_1}\
\underbrace{P \E H H \E E P E}_{\widetilde{\Gamma}_2}
\underbrace{EE}_{\substack{\widetilde{\Gamma}_3\\ (\nu_1 -2)\\ \text{copies}}}
\underbrace{P}_{\widetilde{\Gamma}_4}
\underbrace{EE}_{\substack{\widetilde{\Gamma}_5\\ -K(2,3)-1\\ \text{copies}}}
\underbrace{P \E H H \E E P E}_{\widetilde{\Gamma}_6}
\\
\underbrace{EE}_{\substack{\widetilde{\Gamma}_7\\ -K(2,4)+K(2,3)-1\\ \text{copies}}}
\underbrace{P}_{\widetilde{\Gamma}_8}\
\underbrace{EE}_{\substack{\widetilde{\Gamma}_9\\ \nu_2 \\ \text{copies}}}\
\underbrace{E P E}_{\widetilde{\Gamma}_{10}}
\underbrace{EE}_{\substack{\widetilde{\Gamma}_{11}\\ -K(2,6)+K(2,4)-1\\ \text{copies}}}
\underbrace{P \E H H \E E P E \E H H}_{\widetilde{\Gamma}_{12}}
\\
\underbrace{HH}_{\substack{\widetilde{\Gamma}_{13}\\ -K(2,6)+K(2,4)-1\\ \text{copies}}}
\underbrace{H H}_{\widetilde{\Gamma}_{14}}
\underbrace{HH}_{\substack{\widetilde{\Gamma}_{15}\\ -K(2,4)+K(2,3)-1\\ \text{copies}}}
\underbrace{H H}_{\widetilde{\Gamma}_{16}}
\underbrace{H H}_{\substack{\widetilde{\Gamma}_{17}\\ -K(2,3)-1\\ \text{copies}}}
\underbrace{H H \E E P E}_{\widetilde{\Gamma}_{18}},
\end{multline}
and thus calculate the Poincar\'e--Hopf index at $\infty$, 
which in this particular case turns out to be 
$$
\begin{array}{ll}
PH(X,\infty)& = 
1+
\dfrac{1}{2} 
\big(2\big( \nu_1+\nu_2 - K(2,6) + 1 \big)
-2\big( - K(2,6) + 4 \big) + 8 \big)
\\
&=2+\nu_1+\nu_2.
\end{array}
$$

\begin{figure}[htbp]
\begin{center}
\includegraphics[width=\textwidth]{23-PoincareHopf-2.pdf}
\caption{{\bf The soul of $\Lambda_{X}$ and the cyclic word $\mathcal{W}_{X}$.}
The left hand side shows the $(r,4)$--skeleton of $\Lambda_{X}$ of Example \ref{EjemploPH}, 
in the middle the corresponding soul together with the curve $\widetilde{\Gamma}$ 
decomposed into the corresponding $\widetilde{\Gamma}_{\iota}$ on each sheet of the soul. 
On the right hand side are the corresponding syllables of the cyclic word $\mathcal{W}_{X}$
around $\infty\in\CW_{z}$. See Example \ref{EjemploPHrevisit}.
}
\label{figPoincHopf-2}
\end{center}
\end{figure}
\end{example}

\section{Relations with other works}

\subsection{The case that all critical and asymptotic 
values are real}

Recall the following result. 
\begin{theorem}[Eremenko et al., 
\cite{EremenkoGabrielov}, \cite{EremenkoGabrielov2}]
\emph{
If all critical points of a rational function 
$\frac{Q}{P}(z)$ are real, 
then it is equivalent to a real rational function.}
\end{theorem}
\noindent
This immediately implies that for such a rational function all the critical values are also real.

Motivated by the above, 
we have.

\begin{corollary}[Real critical and asymptotic values]
\hfill
\begin{enumerate}[label=\arabic*),leftmargin=*]
\item If all critical and asymptotic values  
of $\Psi_X$ for $X \in \E(r,d)$ are in $\RR$, 
then the following assertions hold. 

\begin{enumerate}[label=\alph*),leftmargin=*]
\item  $\R_X$, as in \eqref{BetterStripDecomposition}, is the union of half planes.

\item  $\Psi_X : U \subset \CW \longrightarrow \HH^2$ 
is a Schwartz--Christoffel map, for each half plane $U$.

\item $X$ is unstable in $\E(r,d)$, thus a bifurcation for 
$\Re{X}$ occurs. 
\end{enumerate}

\item The 
critical and 
asymptotic values are in $\RR$ if and only if 
the family of rotated vector fields 
$\{ \Re{ \e^{i \theta} X} 
\ \vert \ 
\theta \in \RR/ 2\pi n \}$ bifurcates 
exactly at $\theta=n\pi$ for $n\in\ZZ$. 
\end{enumerate}
\hfill\qed
\end{corollary}

\subsection{Relations with Bely{\u\i}'s functions}

A \emph{rational function $\frac{Q}{P}(z)$ 
is Bely{\u\i}} if it has only three 
critical values $\{ 0,1, \infty \}$,
the original source is \cite{Belyi}, see
\cite{LandoZvonkin} Ch. 2 for current developments.
Recently Ch. J. Bishop \cite{Bishop} 
develops analogous combinatorial and analytic ideas
for entire functions. 

By Lemma \ref{pareja-finita-infinita}, 
we have that for $X\in\E(r,d)$, the distinguished parameter
$\Psi_X$ has an even number $2d$ of asymptotic values (counted 
with multiplicity). Hence,
the construction of a certain $\Psi_X(z)$ having  
three asymptotic values, say at $\{0, 1, \infty\}$ 
set theoretically, as in 
Bely{\u \i}'s theory, is possible.

\begin{example} Let
$X(z)= \frac{\sqrt{\pi} }{4} \e^{z^ 2} \del{}{z}$ 
be as in Example 
\ref{dos-campos-de-la-funcion-error},
having associated the error function
$\Psi(z) = \frac{2}{ \sqrt{\pi} }\int_0 ^z \e^{-\zeta^2} d\zeta$
with logarithmic branch points,
 
\centerline{$\{ (\infty_1, -1, -\infty), 
(\infty_2, 1, -\infty ),  
(\infty_3, \infty, -\infty), (\infty_4, \infty, -\infty) \}$,} 

\noindent 
using the notation in Equations \eqref{essenvert} and 
Definition \ref{divisor-reducido}. 
In set theoretically language, its asymptotic values are
$\{ -1 , 1 , \infty\} $, hence
it is a Bely{\u\i} function. 
\end{example}

\begin{example}{\it 
A transcendental Bely{\u \i} function
that does not belong to the family
$\Psi_X$, for $X \in \E(r,d)$.
}
\label{TresValoresAsintoticos-texto}
With the present techniques, we can describe the following vector field arising from a transcendental Bely{\u \i} function due to R. Nevanlinna
\cite{Nevanlinna1} p.\,292.
Let $\R$ be the Riemann surface that consists of 
half a Riemann sphere (cut along the extended real line 
$\RR\cup\{\infty\}\subset\CW$)
glued to three semi--infinite towers of copies of 
$\CW\backslash (a,b]$ where
$(a,b]\in \{ (-\infty,0], (0,1],(1,\infty] \}$, as in 
Figure \ref{TresValoresAsintoticos}.
The general version of the dictionary 
(\cite{AlvarezMucino} Lemma 2.6)
shows 
that a transcendental function 
$\Upsilon(z) :\CC_z \longrightarrow \CW_{t}$ and a 
vector field $X(z)= \frac{1}{\Upsilon^\prime (z)} \del{}{z}$ are
associated to $\R$. 

\noindent 
The (logarithmic) branch points
of $\Upsilon(z)$ are 

\centerline{$\{ (\infty_1, 0), (\infty_2, 1),  
(\infty_3, \infty) \}$.} 

\noindent 
Of course there is only one such possible Riemann surface 
(up to M\"obius transformation).
Compare also with the line complex description as 
in p.~292 of \cite{Nevanlinna1}.

\noindent The cyclic word is
$$
\Big( (\CW_{z}, \infty),
\Re{X} \Big) \longmapsto 
\mathcal{W}_X= 
H \ent{}{} E  \ent{}{} H \mathcal{T}.
$$
\noindent 
Note the appearance of a new angular sector 
$\mathcal{T}$ having an accumulation point
of double zeros of $X$:
the phase portrait of $\Re{X}$ is obtained by considering the pullback 
of $\Re{\del{}{t}}$ via $\Upsilon$,   
see Figure \ref{tres-sectores-2}.c.
The 1--order of $X$ is finite 
and at least 1.

\begin{figure}[htbp]
\begin{center}
\includegraphics[width=0.80\textwidth]{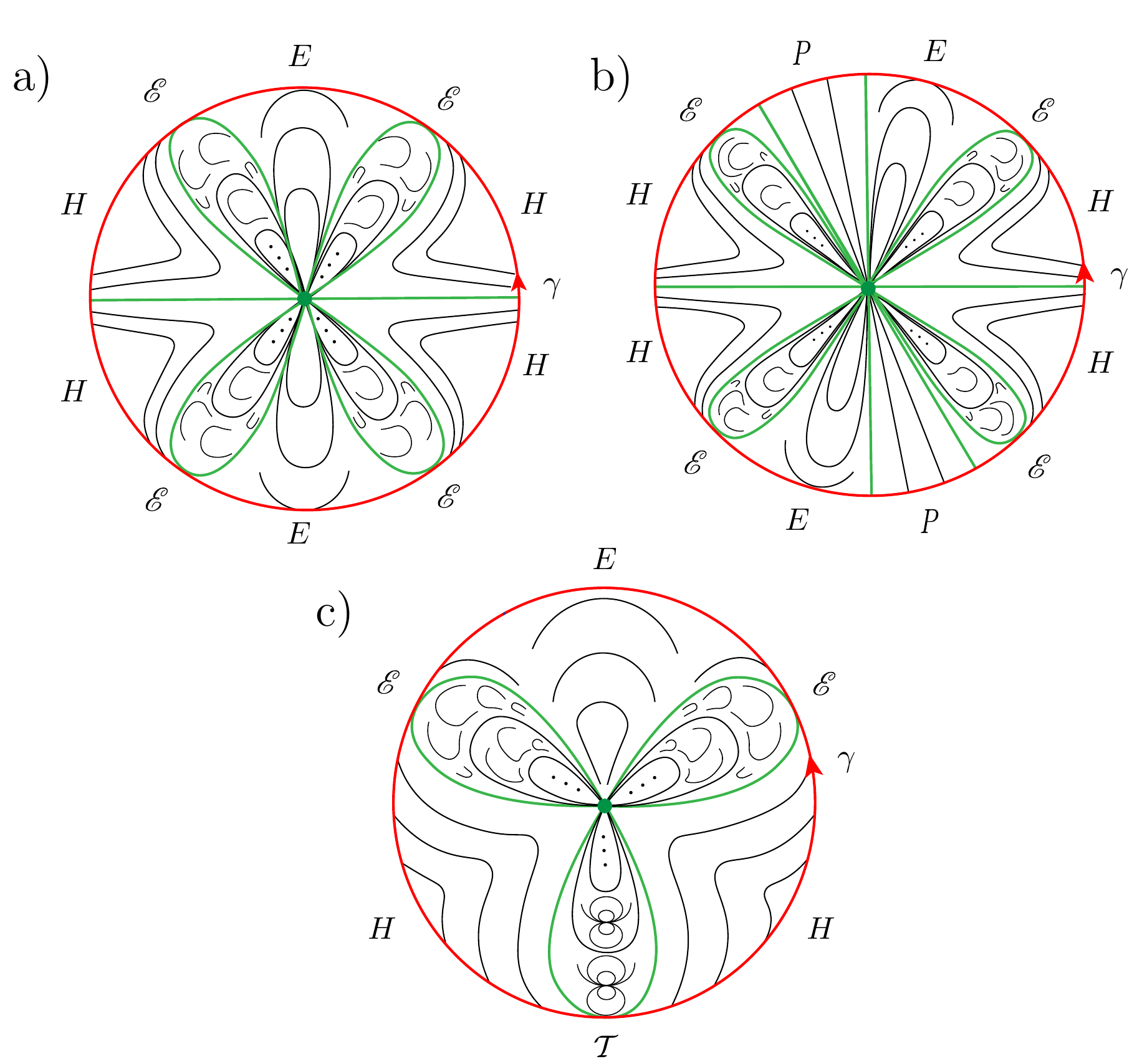}
\caption{
The cyclic words (a)--(b)
appearing in Examples \ref{dos-campos-de-la-funcion-error} 
and (c) in \ref{TresValoresAsintoticos-texto}.
Numerical models for (a)--(b)  
appeared as figures 15 and 16 in \cite{AlvarezMucino}.
}
\label{tres-sectores-2}
\end{center}
\end{figure}

\begin{figure}[htbp]
\begin{center}
\includegraphics[width=0.85\textwidth]{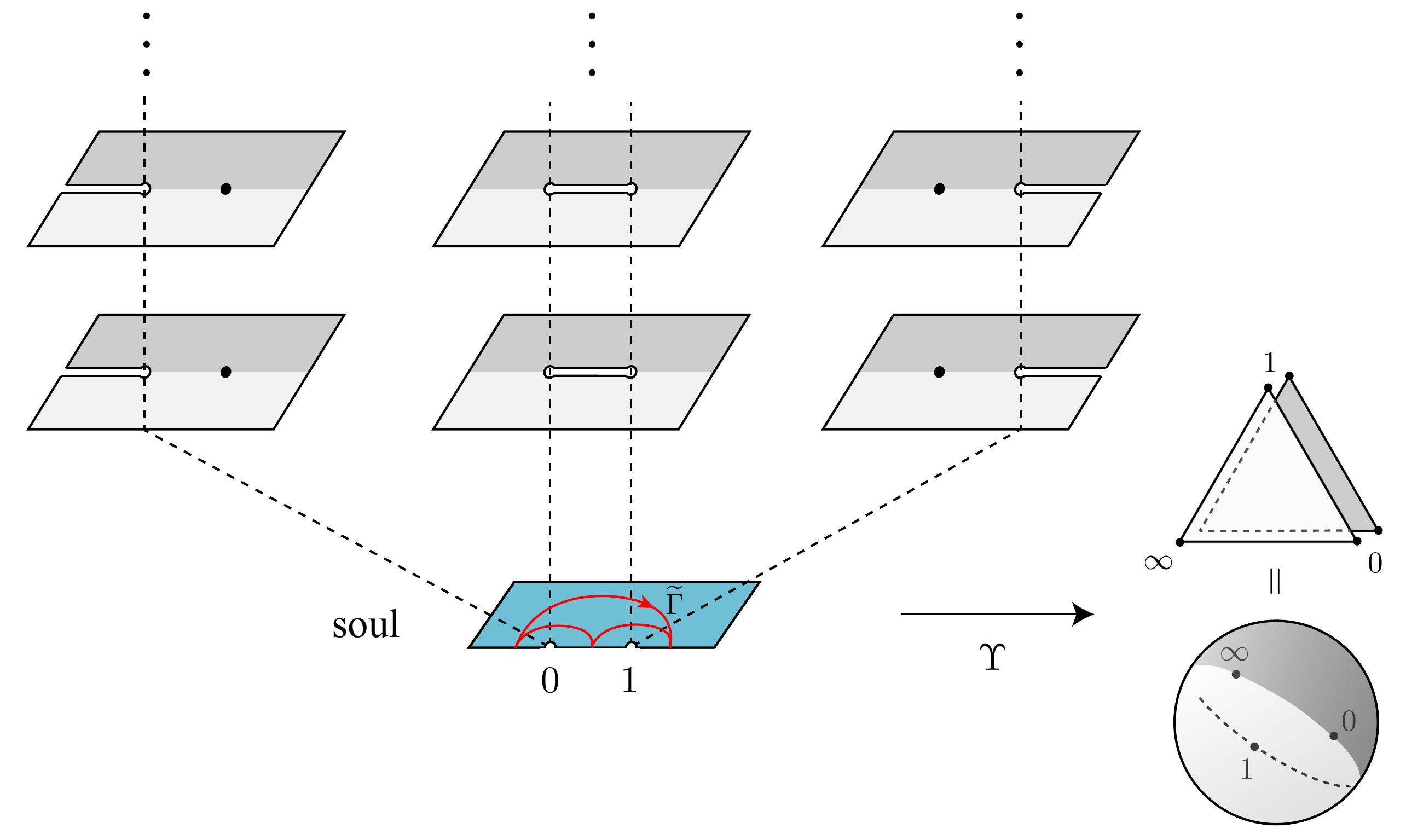}
\caption{
Riemann surface corresponding to a 
transcendental 
Bely{\u\i} function $\Upsilon$. 
The path $\widetilde{\Gamma}$ is a taut deformation of 
$\Gamma= (\Psi_X \circ \gamma)$ originated by a
$\gamma$ bounding
the singularity $\big( (\CW, \infty ), X_\Upsilon \big)$.
Note that topologically this is the only 
possible surface with exactly three logarithmic branch 
points. 
Compare with Lemma \ref{pareja-finita-infinita} and 
note that
the singularities in the central column are 
topologically different
from semi--infinite helicoids in 
Figure \ref{figPiezasBasicas}.
}
\label{TresValoresAsintoticos}
\end{center}
\end{figure}
\end{example}

\subsection{Relation with complex correspondence principle in mechanics}
C. Bender \emph{et al.} studies the relation between 
classical and quantum mechanics using a $\CC$
complex framework, see \cite{Bender-1}, 
\cite{Bender-2} and \cite{Bender-3}.
Motivated by the correspondence principle, asserting that quantum mechanics
resembles classical mechanics in the high--quantum--number limit. 
These works introduce the
concept of a local quantum probability density $\rho(z)$ 
in the complex plane.
C. Bender proposes the novel
approach of constructing a complex contour $C$ 
on which
$\rho(z)dz$ is an infinitesimal probability measure. 
Thus, $C$ must satisfies
$$
\begin{array}{rr}
\hbox{condition I:} & \ \ \Im{\rho(z) dz} =0,
\\
\hbox{condition II:} & \ \ \Re{\rho(z) dz}  > 0,
\\
\hbox{condition III:} & \ \ \int_C \rho(z)dz = 1,
\end{array}
$$

\noindent 
see \cite{Bender-1}.
In our language, we consider $\rho(z) dz$ as 
entire  1--form
on $\CC$. By using the singular complex analytic Dictionary Proposition 
\ref{basic-correspondence} and conditions I--II, we have that
$C$ can be interpreted as a trajectory of 

\centerline{
$\Re{\frac{1}{\rho(z)} \del{}{z}}$, \ where
$\frac{1}{\rho  (z)} \del{}{z} \in \E(r, 2)$, }

\noindent 
condition III is a certain normalization. 
The works of C. Bender {\it et al.} illustrate the application 
of the corresponding trajectory structures.

\subsection{Future work}
\label{future-work}

\subsubsection{Relations with  Dessin's d'enfants}
\label{Dessin's-d'-enfants}
In the combinatorial framework, 
dessin's d'enfants are well known 
plane bipartite trees associated to Bely{\u \i} functions;
in their modern form were promoted and named by A. Grothendieck.
Among other things, they have been used to study the action of the absolute Galois group of $\QQ$, see 
\cite{Grothendieck}.

By the singular complex analytic Dictionary Proposition 
\ref{basic-correspondence}, 
we may consider a polynomial Bely{\u \i} 
function $\Psi_{X}$ as the distinguished parameter 
of the associated vector field 
$X \in \E(r, 0)$, $r\geq 2$.  
That is, up to action of $Aut(\CC)$ on the target, 
we assume that the critical values of $\Psi_{X}$
are $\{0,1,\infty \}$. 
By assigning the color black to the vertices associated to the critical value 0 and 
the color white to the vertices associated to the critical value 1, we see that $\Lambda_{X}$
is a bipartite graph, as in the usual theory.

The extension of the theory for $\Psi_X$,
$X \in \E(r, d)$ with $r\geq 1$ is possible
and is an interesting subject. 
In fact, 
\emph{$(r,d)$--configuration trees $\Lambda_{X}$ that lie over exactly three critical or asymptotic
values are a natural extension of dessins d'enfant of the 
structurally finite Bely{\u\i} functions $\Psi_{X}$.}

\subsubsection{Topological classification of $\Re{X}$ for $X\in\E(r,d)$}
As suggested by the results of \S\ref{topoclassification}; 
a careful study of the $(r,d)$--skeleton of $\Lambda_X$ 
allows for a global topological classification of $\Re{X}$ for $X\in\E(r,d)$, 
in terms of the placement of the critical and asymptotic values. 
This study is in progress, however the technical language needed to 
provide a clear exposition would require too much space, thus it 
has been left for a future work. A particularly interesting case is
the bound for the number of 
topological classes $\{ \Re{X} \ \vert \ \E(r, 0)\}$, $r\geq 3$.

\subsubsection{Dynamical coordinates for other families of vector fields}
As Example \ref{TresValoresAsintoticos-texto} suggests, there are other families 
of vector fields where the construction of the dynamical coordinates $\Lambda_{X}$ 
is certainly possible.

\noindent
For instance, when considering the family 
\begin{equation*}
\E(s,r,d)=
\left\{ X(z)=\frac{Q(z)}{P(z)}\ \e^{E(z)}\del{}{z} \ \Big\vert \ 
\begin{array}{l}
Q,\ P,\ E\in\CC[z],  \\
\deg{Q}=s,\ \deg{P}=r,\ \deg{E}=d 
\end{array}
\right\},
\end{equation*} 

\noindent
as in
\cite{AlvarezMucino2}, 
we are presented with two intrinsically different cases:

\noindent
1) If $\Psi_X$ is single valued
(this is equivalent to requiring that the 
associated 1--form $\omega_X$ have all its residues 
zero), 
then vertices of the form 
$(q_\iota,\infty,\nu_\iota)$, corresponding to the zeros 
$\MZ=\{q_\iota\}_{\iota=1}^s$ of $X$, 
need to be added to the description of $\Lambda_X$.

\noindent
2) If 
$\Psi_X$ is multivalued  
(there appear at least two
non zero residues for $\omega_X$)
then extra structure will be required, 
because of the appearance of logarithmic singularities 
over those $q_\iota\in\CC_z$ where the associated 1--form 
has non--zero residue.

\subsubsection{On cyclic words}
\hfill\\
\textbf{Cyclic words as topological or analytical invariants for germs.}
The word $\mathcal{W}_X$ (as in Theorem \ref{corresp-germen-palabra}), 
is a local topological invariant
of a germ $\big( (\CW_z, \infty), \Re{X} \big)$, $X \in \E(r,d)$.

\noindent 
Moreover,
the word $\mathcal{W}_X$ in general, 
is not a global 
topological invariant of $X \in \E(r, d)$.
For example all the vector fields $X\in \E(r, 0)$, $r \geq 3$, 
with all critical and asymptotic values in $\RR$,
have the same word $\mathcal{W}_X = \underbrace{EE \cdots EE}_{2r +2}$ 
at $\infty$.

\noindent
However, it is possible to modify the definitions of angular sectors $P_\nu$, $E$ and $\ent{}{}$ so that 
in fact the corresponding $\mathcal{W}_X$ is a \emph{global analytic invariant} of $X$ modulo
$Aut(\CC)$. This is left for a future project.

\noindent
\textbf{Other angular sectors as letters for cyclic words.}
As shown in Example \ref{TresValoresAsintoticos-texto} and in 
examples 5.9, 5.12 and figures 2, 5 of \cite{AlvarezMucino}; there are certainly other possible 
angular sectors that can be used as letters for cyclic words.
In this context and considering the above examples, it is clear that there are an infinite number 
of topologically different angular sectors (letters) that can appear in a cyclic word associated to an 
essential singularity for a vector field $X$. 

\noindent
However, it is not immediately clear 
\emph{how many topologically different letters there are 
when we specify the $p$--order of $X$}, 
that is 
the coarse analytic invariant of functions and vector fields.
For instance, by once again considering Example \ref{TresValoresAsintoticos-texto}, 
$X(z)=\Upsilon^{*}(\del{}{t})(z)=\frac{1}{\Upsilon'(z)}\del{}{z}$;
we may also consider 
$Y(z)=\Upsilon^{*}(\mu t\del{}{t})(z)=\mu \frac{\Upsilon(z)}{\Upsilon'(z)}\del{}{z}$ 
which provides a (very) different vector field.


\section{Epilogue}\label{epilogo}
The second part of the proof of the Main Theorem provides a proof 
of the following 
result of independent interest.
Recalling Equations \eqref{poleenumeration}, \eqref{essenenumeration},
\eqref{puntos-criticos-sing-trascen} and \eqref{valores-criticos-asintoticos} 
we have.

\begin{corollary}
\label{Psi-con-valores-criticos-asintoticos-preasignados}
Given $r+d$ values

\centerline{
$\widetilde{p}_1, \ldots, \widetilde{p}_r, a_1, \ldots, a_d\in\CC_{t}$,} 

\noindent
possibly repeated, with the exception that  
if $d\geq 1$ there are at least two non--repeated values.
Then there exists a (non unique) vector filed
$X \in \E(r,d)$ having these $r+d$ ramification values, 
{\it i.e.} a
collection of 
$n+d$ 
realizable vertices, see \eqref{sistema-realizable},  

\centerline{
$\left\{ (p_{\iota},\widetilde{p}_{\iota},-\nu_{\iota})
\right\}_{\iota=1}^n \cup
\left\{
(\infty_{\sigma},a_{\sigma},-\infty)\right\}_{\msigma=1}^{d}$
}

\noindent 
for the corresponding $\Psi_{X}$. 
\end{corollary}

\begin{remark}
In our case, the complete collection of ramification values are 
$\{\widetilde{p}_1,  \ldots , \widetilde{p}_n, a_1,\ldots, a_d, 
\infty\} \subset \CW_t $.
Moreover, the use of ramification values $t_{\msigma}$ provides information 
of the moduli space of $\Psi_{X}$.
Recalling the classical Riemann's idea that 
for ramified cover maps over $\CW_t$
with $\mathfrak{n} \geq 4$ ramification values, we can 
specify three 
of them, and the other $\mathfrak{n}-3$ determine
holomorphic deformations of the cover maps
$\pi_{X, 2}$.
\end{remark}

Because of the singular complex analytic Dictionary 
Proposition \ref{basic-correspondence}, 
the works of R. Thom \cite{Thom} and J. Mycielski \cite{Mycielski},
describes the situation for polynomials 
$\Psi_{X}$, {\it i.e.} the case $d=0$.
However the answer is not unique, that is given a set 
of preassigned critical values 
$\{  \widetilde{p}_1,\ldots , \widetilde{p}_n \}$ 
there are a finite number of polynomials $\Psi_{X}$ 
with the above set as critical values, namely

\begin{theorem*}[Mycielski--Thom]\label{Mycielski}
Given $r$ points 
$\widetilde{p}_1, \ldots, \widetilde{p}_r
\in\CC_t$, there exist
$r$ points 
$p_1, \ldots, p_r \in\CC_z$ such that 
the (monic) polynomial of degree $r+1$ 

\centerline{
$\Psi(z)=(r+1)\int^z \prod\limits_{\iota=1}^r (\zeta-p_\iota) d\zeta$
}

\noindent
satisfies 
\begin{enumerate}[label=\arabic*),leftmargin=*]
\item 
$\Psi(p_\iota)=\widetilde{p}_\iota$ and
$\Psi^\prime(p_\iota) = 0$, for $\iota=1,\ldots,r$.
\item 
If $\beta$ occurs $k$ times in the collection 
$p_1, \ldots, p_r$ then 
$(z-\beta)^{k+1}$ divides $\Psi(z)-\Psi(\beta)$.
\hfill
\qed
\end{enumerate}
\end{theorem*}
\begin{proof}[Proof of Corollary
\ref{Psi-con-valores-criticos-asintoticos-preasignados}]
Recalling Equation \eqref{sistema-realizable},
we want to show that there exists: 

\begin{enumerate}[label=\alph*),leftmargin=*]
\item $r$ points $p_1,\ldots,p_r\in\CC_{z}$, determining a (monic) polynomial 
$P(z) = \prod\limits_{\iota=1}^r (z-p_\iota)$,
\item a polynomial $E(z)$ of degree $d$,  
\item $d$ asymptotic paths $\alpha_\sigma(\tau)$, 
\end{enumerate}
such that the distinguished parameter
$\Psi_{X}(z)=
\int^{z}  P(\zeta) \e^{-E(\zeta)} \, d\zeta$ satisfies 

\noindent 
i) 
$\Psi_{X}(p_\iota)=\widetilde{p}_\iota$
and $\Psi_{X}^{(\ell )} (p_\iota) = 0$, 
for $1 \leq \ell \leq \nu_{\iota}=1$,

\noindent 
ii) 
$\lim\limits_{\tau\to\infty} \Psi_{X}(\alpha_\sigma (\tau))=a_\sigma$, 
for $\sigma=1,\ldots,d$.

\noindent
Furthermore,
the polynomials $P(z)$ and $E(z)$ are non--unique.

Note that the \emph{geometrical} construction carried out in \S\ref{confTree-to-skeleton}
(the second part of the proof of the 
Main Theorem) can be carried out
\emph{by only specifying the critical and asymptotic values and the corresponding $K(\msigma,\mrho)$}
of the $(r,d)$--configuration tree, 
this uses Lemma \ref{oneasymptoticvalue} actively.

That is we can construct a Riemann surface $\R$ by glueing sheets $\CC_{t}$ with branch cuts starting at 
$\{ \widetilde{p}_1, \ldots, \widetilde{p}_r, a_1, \ldots, a_d \}$. 
As before, $\R$ is recognized as a simply connected Riemann surface $\R_{X}$ 
corresponding to some $X\in\E(r,d)$; thus showing that for every possible choice of 
$(\widetilde{p}_1, \ldots, \widetilde{p}_r, a_1, \ldots, a_d)\in\CC^{r+d}$ there are polynomials
$P(z)$ and $E(z)$ of degrees $r$ and $d$ respectively such that $\Psi_{X}$ has precisely
$\{ \widetilde{p}_1, \ldots, \widetilde{p}_r, a_1, \ldots, a_d \}$ as critical and asymptotic values.

For the non--uniqueness of the polynomials $P(z)$ and $E(z)$ note that: 
when $d=0$ generically, by Bezout's
Theorem, there are $(r+1)^{r}$ solutions of the system \eqref{sistema-realizable}. 
For the case 
$r=3,$ $d=0$ see Equation \eqref{tres-puntos-de-ramificacion}
and for
$r=0$, $d=3$, recall Example \ref{ejemplo-Exp3} where there are an
infinite number of solutions for each vertex of finite asymptotic values $a_1, a_2, a_3\in\CC_t$: 
each parameter $K(1,3)\in\ZZ$ provides a different $X\in\E(0,3)$. 
The general case now follows easily from the above examples.
\end{proof}
\begin{remark}
1. Corollary \ref{Psi-con-valores-criticos-asintoticos-preasignados} 
gives rise to a 
complex analytic set
in $\overline{\CC}_{z}^{\,r+d} \times \CC_{t}^{r+d}$ 
consisting of the sets of branch points that determine 
$\R_{X}$ with $X\in\E(r,d)$.

\noindent
2. Corollary \ref{Psi-con-valores-criticos-asintoticos-preasignados} can be interpreted as 
saying that the map from $\CC[z]_{=r}\times\CC[z]_{=d}$ to $\CC^{r+d}$ 
given by the 
Equation \eqref{sistema-realizable} is surjective.
\end{remark}

Recalling \eqref{campo-X-con-P-y-E} and \eqref{coeficientes-P-y-E}, 
there are two obvious ways of parametrizing $\E(r,d)$: 

\smallskip

\noindent
1) Specifying the coefficients of $P(z)$ and $E(z)$.

\smallskip 

\noindent 
2) Specifying the roots of $P(z)$ and $E(z)$ together with the non zero coefficient $\mu$.

\smallskip

\noindent 
Noting that the roots of $P(z)$ correspond to the poles of $X$, 
equivalently to the critical points of $\Psi_{X}$.
In the particular case of $d=0$, 
the usual geometrical/dynamical 
interpretation of (2) arises.
However, 
we are not aware of a geometrical/dynamical 
interpretation of the roots of $E(z)$;
compare with \cite{AlvarezMucino2} where a study of the 
discrete symmetries of $X$ is provided. 

\smallskip

\noindent
3) A third kind of ``parametrization'' is given by 
Corollary \ref{Psi-con-valores-criticos-asintoticos-preasignados}: 
given a set of critical and asymptotic values 
$\{  \widetilde{p}_{1},\ldots , \widetilde{p}_{n}, a_{1},\ldots, a_{d} \}$,
there are non--unique $X\in\E(r,d)$ such that the above set are precisely 
the critical and asymptotic values of $\Psi_{X}$.

\noindent
Note that the non--uniqueness arises from the solution of the
system of  transcendental equations \eqref{sistema-realizable}.

\smallskip 

Recalling Diagram \ref{diagramaRX}, parametrizations (2) 
case $d=0$
and (3)
can be represented in Diagram \ref{diagramaDificultades} by 
specifying the left hand side or the bottom part, respectively:

\begin{center}
\begin{picture}(180,60)(40,4)
\put(-51,15){\vbox{\begin{equation}\label{diagramaDificultades}\end{equation}}}

\put(-50,50){$\CC_{z} \supset \{ p_1,\ldots p_n, \infty_1, \ldots, \infty_d\} $}

\put(108,50){$  
\big\{ (p_{\iota},\widetilde{p}_{\iota},-\nu_{\iota})
\big\}_{\iota =1}^{n}
\cup
\big\{(\infty_{\sigma},a_{\sigma},-\infty) 
\big\}_{\sigma =1}^{d} \subset \R_X$}

\put(107,53){\vector(-1,0){27}}
\put(85,60){$\pi_{X,1}$}

\put(175,40){\vector(0,-1){25}}
\put(180,25){$ \pi_{X,2} $}

\put(25,42){\vector(3,-1){100}}
\put(55,17){$ \Psi_X $}

\put(128,6){$\{  \widetilde{p}_{1},\ldots , \widetilde{p}_{r},
a_{1},\ldots, a_{d} \} \subset \CC_t .$ }

\end{picture}
\end{center}

?`Is there another way of parametrizing $X\in\E(r,d)$?

Striving for a unique geometrical/dynamical solution 
in the general case $r, d\geq1$,
$(r,d)$--configuration trees provide a 
``mixed approach''. 
Further study of the above question
and effective parameters
from Diagram \ref{diagramaDificultades} is the goal 
of a future project.




\end{document}